\documentclass[12pt]{article}
\usepackage{etex} 
\usepackage{a4wide}
\usepackage{amsmath, amsthm, amsfonts, amssymb, bbm, bm}
\usepackage{graphics}
\usepackage{tikz-cd}
\usepackage{upgreek}
\usepackage{accents}
\usepackage{diagbox}
\usetikzlibrary{arrows}
\usepackage{hyperref}
\hypersetup{       
   pdftitle={Area-dependent quantum field theories with defects - I. Runkel and L. Szegedy},
   colorlinks=true,        
   pdfstartview=FitH,      
   allcolors=[rgb]{0,0.5,0},        
	 bookmarksdepth=section, 
   bookmarksopen=false,     
   bookmarksnumbered=true 
}
\usepackage{epstopdf}

\usepackage{xypic}
\usepackage[all]{xy}
\usepackage[english]{babel}
\usepackage[font=small,format=plain,labelfont=bf,up]{caption}
\usepackage{enumitem}
\usepackage{footmisc}

\usepackage{mathtools}

\usepackage{soul}

\usepackage{extpfeil}

\makeatletter
\newcommand{\sbullet}{\hbox{\fontfamily{lmr}\fontsize{.4\dimexpr(\f@size pt)}{0}\selectfont\textbullet}}

\makeatother
\DeclarePairedDelimiterX\setc[2]{\{}{\}}{\,#1 \;\delimsize\vert\; #2\,}

\newtheoremstyle{important-thm}
     {3pt}
     {3pt}
     {\slshape}
     {}
     {\bfseries}
     {.}
     {.5em}
     {}

\theoremstyle{plain}

\theoremstyle{important-thm}
\newtheorem{theorem}{Theorem}
\newtheorem{lemma}[theorem]{Lemma}
\newtheorem{proposition}[theorem]{Proposition}
\newtheorem{corollary}[theorem]{Corollary}
\theoremstyle{definition}
\newtheorem{remark}[theorem]{Remark}
\newtheorem{example}[theorem]{Example}
\newtheorem{definition}[theorem]{Definition}

 \numberwithin{equation}{section}
 \numberwithin{theorem}{section}

\DeclareMathOperator{\Aut}{Aut}
\DeclareMathOperator{\Out}{Out}
\DeclareMathOperator{\End}{End}
\DeclareMathOperator{\Hom}{Hom}

\DeclareMathOperator{\Sp}{Sp}
\DeclareMathOperator{\id}{id}
\DeclareMathOperator{\im}{im}
\DeclareMathOperator{\coker}{coker}

\DeclareMathOperator{\tr}{tr}
\DeclareMathOperator{\rank}{rank}
\renewcommand{\sp}{\mathrm{sp}}

\def\be#1\ee{\begin{equation}#1\end{equation}}
\def\ba#1\ea{\begin{align}#1\end{align}}

\newcommand\mapsfrom{\mathrel{\reflectbox{\ensuremath{\mapsto}}}}
\newcommand\Vectfd{{\mathcal{V}\hspace{-.5pt}ect^\mathrm{fd}}}
\newcommand\Bordarea{{\mathcal{B}\hspace{-.5pt}ord_2^{\hspace{2pt}area}}}
\newcommand\Bordna{{\mathcal{B}\hspace{-.5pt}ord_2}}

\newcommand\Borddef[1]{{\mathcal{B}\hspace{-.5pt}ord_{2, #1}^{\hspace{2pt}area,\hspace{1pt} def}}}
\newcommand\Borddefna[1]{{\mathcal{B}\hspace{-.5pt}ord_{2, #1}^{\hspace{2pt}def}}}

\newcommand\AQFT[1]{{\mathrm{aQFT}\hspace{-.5pt}\left( #1 \right)}}
\newcommand\aQFT{\text{aQFT}}
\newcommand\Hilb{{\mathcal{H}\hspace{-.5pt}ilb}}
\newcommand\Top{{\mathcal{T}\hspace{-2.0pt}op}}

\newcommand\RFrob[1]{{\mathcal{RF}\hspace{-.5pt}rob\hspace{-.5pt}\left( #1 \right)}}
\newcommand\cRFrob[1]{{c\mathcal{RF}\hspace{-.5pt}rob\hspace{-.5pt}\left( #1 \right)}}
\newcommand\dRFrob[1]{\dagger\text{-}{\mathcal{RF}\hspace{-.5pt}rob\hspace{-.5pt}\left( #1 \right)}}
\newcommand\dFrobf[1]{\dagger\text{-}{\mathcal{F}\hspace{-.5pt}rob^F\hspace{-.5pt}\left( #1 \right)}}
\newcommand\RAlg[1]{{\mathcal{RA}\hspace{-.5pt}lg\hspace{-.5pt}\left( #1 \right)}}
\newcommand\Mod[1]{{\mathcal{M}\hspace{-.5pt}od\hspace{-.5pt}\left( #1 \right)}}
\newcommand\Frobz[1]{{\mathcal{F}\hspace{-.5pt}rob^Z\hspace{-2.5pt}\left( #1 \right)}}
\newcommand\Algz[1]{{\mathcal{A}\hspace{-.0pt}lg^Z\hspace{-2.5pt}\left( #1 \right)}}
\newcommand\Modz[1]{{\mathcal{M}\hspace{-.0pt}od^Z\hspace{-2.5pt}\left( #1 \right)}}
\newcommand{\norm}[1]{ \left\|{#1}\right\| }
\newcommand{\ctimes}{ \circlearrowleft }

\newcommand\eps           {\varepsilon}

\newcommand\Ab            {\mathbb{A}}
\newcommand\Ac            {\mathcal{A}}
\newcommand\Bc            {\mathcal{B}}
\newcommand\Cc            {\mathcal{C}}
\newcommand\Cb            {\mathbb{C}}
\renewcommand\Dc            {\mathcal{D}}
\newcommand\Db            {\mathbb{D}}
\newcommand\Ec            {\mathcal{E}}
\newcommand\Fc            {\mathcal{F}}
\newcommand\Hc            {\mathcal{H}}

\newcommand\Ib            {\mathbb{I}}

\newcommand\Kc            {\mathcal{K}}
\renewcommand\Lc            {\mathcal{L}}
\newcommand\Nb            {\mathbb{N}}
\newcommand\Oc            {\mathcal{O}}
\newcommand\Rb            {\mathbb{R}}

\newcommand\Sc            {\mathcal{S}}
\newcommand\Sb            {\mathbb{S}^1}
\newcommand\Sbb            {\mathbb{S}^2}
\renewcommand\Yc            {\mathcal{Y}}

\newcommand\funZ            {\mathcal{Z}}
\newcommand\Zb            {\mathbb{Z}}

\newcommand{\void}[1]{}

\newcommand\doi[2]        {\href{http://dx.doi.org/#1}{#2}}

\renewcommand{\restriction}{\mathord{|}}

\begin{document}

\thispagestyle{empty}
\def\thefootnote{\fnsymbol{footnote}}
\begin{flushright}
ZMP-HH/18-13\\
Hamburger Beitr\"age zur Mathematik 743
\end{flushright}
\vskip 3em
\begin{center}\LARGE
Area-dependent quantum field theory with defects
\end{center}

\vskip 2em
\begin{center}
{\large 
Ingo Runkel~~and~~L\'or\'ant Szegedy~\footnote{Emails: {\tt ingo.runkel@uni-hamburg.de}~,~{\tt lorant.szegedy@uni-hamburg.de}}}
\\[1em]
Fachbereich Mathematik, Universit\"at Hamburg\\
Bundesstra\ss e 55, 20146 Hamburg, Germany
\end{center}

\vskip 2em
\begin{center}
  July 2018
\end{center}
\vskip 2em

\begin{abstract}
Area-dependent quantum field theory is a modification of two-dimensional topological quantum field theory, 
where one equips each connected component of a bordism with a positive real number -- interpreted as area -- which behaves additively under glueing. 
As opposed to topological theories, in area-dependent theories the state spaces can be 
infinite-dimensional. 

We introduce the notion of regularised Frobenius algebras and show that area-dependent theories are in one-to-one correspondence to 
commutative regularised Frobenius algebras.
We provide a state-sum construction for area-dependent theories, which includes theories with defects. 
Defect lines are labeled by dualisable bimodules over regularised algebras. We show that the tensor product of such bimodules agrees with the fusion of defect lines, which is defined as the limit where the area separating two defect lines is taken to zero.

All these constructions are exemplified by two-dimensional Yang-Mills theory with compact gauge group and with Wilson lines as defects, which we treat in detail.
\end{abstract}

\setcounter{footnote}{0}
\def\thefootnote{\arabic{footnote}}

\newpage

{\small

\tableofcontents

}

\section{Introduction and summary}

\newcommand\Bordn{{\mathcal{B}\hspace{-.5pt}ord_n}}
\newcommand\Bordnvol{{\mathcal{B}\hspace{-.5pt}ord_n^{\hspace{2pt}vol}}}

Volume-dependent quantum field theory is a modification of topological quantum field theory. 
The latter are symmetric monoidal functors from the category $\Bordn$ of $(n-1)$-dimensional (closed, compact, oriented) manifolds and $n$-dimensional bordisms to vector spaces, see e.g.\ \cite{Carqueville:2016def,Carqueville:2017tft} for a review. 
The modification consists of changing the morphism spaces of $\Bordn$ while keeping the objects the same. Namely, we equip each connected component of a bordism with a positive real number which we interpret as the volume of that component. One way to think about this is to start from Riemannian manifolds as bordisms and then -- instead of forgetting the metric entirely as one does in the topological case -- to remember the integral of the associated volume form.\footnote{
	This is equivalent to remembering the induced volume form up to diffeomorphism, see \cite{Moser:1965vol,Banyaga:1974vol}.}

Formally, morphisms in $\Bordnvol$ are pairs $(M,v)$, where $M$ is a morphism in $\Bordn$ and $v$ is a function $v : \pi_0(M) \to \Rb_{>0}$. The volumes of connected components add under glueing.
In order to have identities in $\Bordnvol$, we allow $v$ to take the value $0$ on connected components which are cylinders $U \times [0,1]$. A volume-dependent QFT is defined to be a symmetric monoidal functor
\be\label{eq:vol-dep-QFT}
	\funZ : \Bordnvol \to \Hilb \ ,
\ee
which is continuous on $\Hom$-spaces.
We will allow for more general target categories in the main text, but for the introduction let us stick with $\Hilb$, the category of complex Hilbert spaces and bounded linear maps, equipped with strong operator topology.
The topology on $\Bordnvol$ is that of the disjoint union over $M \in \Bordn$ of $\Rb_{>0}^{|\pi_0(M)|}$ (or $\Rb_{\ge 0}$ for components of $M$ that are cylinders). In other words, 
$\funZ(U \xrightarrow{M} U' , v)$
is jointly continuous as a function in the volume parameters assigned to the connected components of $M$ with values in bounded linear maps $B(\funZ(U),\funZ(U'))$.

The main change when passing from $\Bordn$ to $\Bordnvol$, and indeed the main motivation to look at this generalisation in the first place, is that the vector spaces $\funZ(U)$ can now be infinite-dimensional. This is in contrast to topological QFTs, where the usual zig-zag argument forces $\funZ(U)$ to be finite-dimensional for all objects $U \in \Bordn$ 
(see e.g.\ \cite[Sec.\,2.4]{Carqueville:2017tft}). 
The same argument in the case of $\Bordnvol$ merely requires each $\funZ(U)$ to be a separable Hilbert space (cf.\ Lemma~\ref{lem:patrclass} and Theorem~\ref{thm:aqftrfaequiv}). 

\medskip

In this paper we analyse in detail the case $n=2$. Since bordisms are now two-dimensional, we refer to such theories as \textsl{area-dependent QFTs}, or \textsl{{\aQFT}s}\footnote{We use the small `a' in \aQFT\ to set it apart from Algebraic QFT or Axiomatic QFT, which are often abbreviated as AQFT.
} 
for short, and we write $\funZ : \Bordarea \to \Hilb$. The precise definition can be found in Section~\ref{sec:aqft}. The main example of an {\aQFT} is two-dimensional Yang-Mills theory for a compact semisimple Lie group $G$ \cite{Migdal:1975re,Rusakov:1990wilson,Witten:1991gt}, 
in which case the Hilbert space assigned to a circle is $Cl^2(G)$, that is, square integrable class functions on $G$. We treat this example in detail in Section~\ref{sec:2dym}.
Area-dependent theories in general have been considered in \cite{Brunner:1995diploma} and briefly in
\cite[Sec.\,1.4]{Segal:1999sn} (see also \cite[Sec.\,4.5]{Bartlett:2005mt}).
A construction of area-dependent theories using triangulations with equal triangle area has been given in \cite{Cunha:1997qtop}.

Two-dimensional TQFTs are of course a special case of {\aQFT}, namely they are {\aQFT}s which are independent of the area parameters. Conversely one can show that if for all bordisms $\Sigma$ the zero area limit of $\funZ(\Sigma)$ exists, then all state spaces $\funZ(U)$ are necessarily finite dimensional, and the zero area limit of $\funZ$ is a TQFT (Remark~\ref{rem:aqft-all-zero-area-limits}).

Recall that 2d\,TQFTs are in one-to-one correspondence to commutative Frobenius algebras 
\cite{Dijkgraaf:1989phd, Abrams:1996fa}, 
and that there is a state-sum construction of 2d\,TQFTs which starts from a strongly separable (not necessarily commutative) Frobenius algebra $A$ as an input \cite{Bachas:1993lat,Fukuma:1994sts,Lauda:2007oc}. The commutative Frobenius algebra defining the TQFT obtained from this state-sum construction is just the centre $Z(A)$.

The generalisation of these results to {\aQFT}s is for the most part straightforward to the point of being mechanical: just add a positive real parameter to all maps in sight (``area parameters'') and impose the condition that everything just depends on the sum of these areas.  

Take, for example, the notion of an algebra. An algebra is an object $A$ in $\Hilb$ (or in your favourite monoidal category) together with morphisms (bounded linear maps in this case) $\mu : A \otimes A \to A$, the multiplication, and $\eta : \Cb \to A$, the unit. These have to satisfy
\be
	\mu \circ (\id_A \otimes \mu)  =  \mu \circ (\mu \otimes \id_A)
	\quad , \quad
	\mu \circ (\id_A \otimes \eta)  =  \id_A = \mu \circ (\eta \otimes \id_A) \ .
\ee
A \textsl{regularised algebra} is then defined as follows (see Section~\ref{sec:RA-RFA-def}). It is an object $A \in \Hilb$ together with two families of morphisms $\mu_a : A \otimes A \to A$ and $\eta_a : \Cb \to A$, for $a \in \Rb_{>0}$, such that, for all $a_1,a_2,b_1,b_2 \in \Rb_{>0}$ with $a_1+a_2=b_1+b_2$,
\be\label{eq:intro-regalg}
	\mu_{a_1} \circ (\id_A \otimes \mu_{a_2})  =  \mu_{b_1} \circ (\mu_{b_2} \otimes \id_A)
	\quad , \quad
	\mu_{a_1} \circ (\id_A \otimes \eta_{a_2})  =  \mu_{b_1} \circ (\eta_{b_1} \otimes \id_A) \ .
\ee
The unit condition is one of the places where a little more thought is required: note that we do not demand that we obtain $\id_A$. Instead we \textsl{define} $P_a := \mu_{a_1} \circ (\id_A \otimes \eta_{a_2}) : A \to A$, where $a = a_1+a_2$. By the second condition in \eqref{eq:intro-regalg} this is indeed independent of the choice of $a_1,a_2$ in the decomposition $a = a_1+a_2$. We now impose two conditions: $P_a$ has to be continuous\footnote{
For more general monoidal categories than $\Hilb$ we need to add another continuity condition. We refer to the main text for details and restrict ourselves to $\Hilb$ for the introduction, where this condition is automatic -- see \eqref{eq:racont} 
and Corollary~\ref{cor:semigrp}.
} 
in $a$ and it has to satisfy $\lim_{a \to 0} P_a  = \id_A$. 
It is important for the formalism to \textsl{not} require $\mu_a$ and $\eta_a$ to have zero-area limits on their own.
Simple consequences of this definition are that $\mu_a$ and $\eta_a$ are continuous in $a$, and that $P_a$ is a semigroup, $P_a \circ P_b = P_{a+b}$.

The algebraic cornerstone of this work is the notion of a \textsl{regularised Frobenius algebra} (RFA), which is a regularised algebra and coalgebra (with families $\Delta_a$ and $\eps_a$ for area-dependent coproduct and counit), subject to the usual compatibility condition, suitably decorated with area parameters 
(Definition~\ref{def:rfa}).
One small surprise is that RFAs, as opposed to Frobenius algebras, do not form a groupoid: it may happen that the inverse to a homomorphism of Frobenius algebras is not bounded, hence not a morphism in $\Hilb$ 
(Remark~\ref{rem:not-groupoid}).

An important example of an RFA is $L^2(G)$, the square integrable functions on a compact semisimple Lie group $G$. 
Here, the structure maps $\mu_a$ and $\Delta_a$ do have zero are limits given by the convolution product and by $\Delta_0(f)(g,h) := f(gh)$. The unit and counit families $\eta_a$ and $\eps_a$ on the other hand do not have $a\to0$ limits, see Section~\ref{sec:twoRFAsfromG} for details.
In fact, by the Peter-Weyl theorem we have $L^2(G) = \bigoplus_{V \in \widehat G} V \otimes V^*$, where the sum is a Hilbert space direct sum over isomorphism classes of irreducible unitary representations of $G$, and the RFA structure on $L^2(G)$ restricts to an infinite direct sum of finite-dimensional RFAs on $V \otimes V^*$.
This is a general result for Hermitian RFAs, i.e.\ RFAs for which 
$\mu_a^\dagger = \Delta_a$ and $\eta_a^\dagger = \eps_a$ for every $a\in\Rb_{>0}$
(Theorem~\ref{thm:daggerclassification}):

\begin{theorem}
Every Hermitian RFA is a Hilbert space direct sum of finite dimensional Hermitian RFAs.
\end{theorem}

Finite dimensional RFAs in turn are very simple: they are just usual (by definition finite dimensional) Frobenius algebras $A$ together with an element $H$ in the centre $Z(A)$ of $A$. The area-dependence is obtained by setting $P_a := \exp(a H)$ and defining $\mu_a := P_a \circ \mu$, etc., see  Proposition~\ref{prop:findimraclass}.
This makes RFAs sound not very interesting, but note that, conversely, for an infinite direct sum of finite-dimensional RFAs to again define an RFA one has to satisfy non-trivial bounds, as detailed in Proposition~\ref{prop:directsumrfa}.
And as the example of $L^2(G)$ shows, the direct sum decomposition may not always be the most natural perspective.

Our first main theorem generalises the classification of 2d\,TQFTs in terms of commutative Frobenius algebras as given in \cite{Dijkgraaf:1989phd, Abrams:1996fa}. Namely, in Theorem~\ref{thm:aqftrfaequiv} we show:\footnote{
	In \cite{Segal:1999sn,Bartlett:2005mt} the classification is instead in terms of algebras with a non-degenerate trace and an approximate unit. However, it is implicitly assumed there that the zero-area limit of the pair of pants with two in-going and one out-going boundary circles exists. This is not true for all examples as the commutative RFAs in Remark~\ref{rem:no-zero-area-limit} illustrate.}

\begin{theorem}\label{thm:introAQFTclass}
	There is a one-to-one correspondence between {\aQFT}s $\Bordarea \to\Hilb$ and commutative RFAs in $\Hilb$.
\end{theorem}

In Sections~\ref{sec:latticedata} and \ref{sec:data-from-rfa} we furthermore generalise the state-sum construction of \cite{Bachas:1993lat,Fukuma:1994sts,Lauda:2007oc}. We find that a strongly separable symmetric RFA $A$ (as defined in Section~\ref{sec:RA-RFA-def}) provides the data for the state-sum construction of an {\aQFT}, and the resulting {\aQFT} corresponds, via Theorem~\ref{thm:introAQFTclass}, to the commutative RFA given by the centre of $A$, see Theorem~\ref{thm:latticecenter}.

\begin{remark}
The category $\Bordarea$, or more generally the category $\Bordnvol$, is enriched in $\Top$.
One could thus define volume-dependent theories to be $\Top$-enriched symmetric monoidal functors $\Bordnvol \to \Sc$ for some $\Top$-enriched symmetric monoidal target category $\Sc$. This would make the explicit mention of continuity in \eqref{eq:vol-dep-QFT} unnecessary.
The reason we do not do this here is that it restricts the choice of target category. In particular, our main example -- $\Hilb$ with strong operator topology -- is not $\Top$-enriched (Remark~\ref{rem:hilbcont}).
On the other hand, $\Hilb$ with norm topology is $\Top$-enriched, but this leads to another problem. Namely, the version of $\Bordnvol$ we use has identities in the form of zero-area cylinders (recall that only cylinder components are allowed to have zero area). 
This can be shown to imply that a volume-dependent QFT $\Bordnvol \to (\Hilb \text{ with norm-top.})$ must take values in finite-dimensional Hilbert spaces (Corollary~\ref{cor:finite-in-norm}). Hence, to have an interesting theory one has to remove the zero-volume cylinders. This can be done, but we do not pursue this further in the present work.
\end{remark}

We have seen that the first new feature one encounters when passing from 2d\,TQFTs to {\aQFT}s is the possibility of infinite-dimensional state spaces. When one develops the theory in the presence of line defects one encounters a second new feature, namely that line defects can be transmissive to area or not. Let us explain this point in more detail.

A 2d\,TQFT with defects is again a symmetric monoidal functor whose source is now a more complicated bordism category. Bordisms are decorated by 
an oriented 1-dimensional submanifold. 
Connected surface
components in the complement of this submanifold
are labeled by elements from a set $D_2$ (``world sheet phases''), and components of the submanifold itself 
are labeled by elements from $D_1$ (``defect conditions''). For details we refer to 
\cite{Davydov:2011dt,Carqueville:2016def}, as well as to Section~\ref{sec:daqft}. 

In the area-dependent case, it is natural to equip the connected components of the defect-submanifold with a length parameter $l \in \Rb_{>0}$. This is suggested by the motto: ``If in an $n$-dimensional volume-dependent theory with defects the surrounding $n$-dimensional theory is trivial, one should end up with an $(n-1)$-dimensional volume-dependent theory.''
We will attach an independent area parameter to each connected surface component of the complement of the defect-submanifold.
	
A \textsl{defect {\aQFT}} is defined to be a symmetric monoidal functor
\be
	\funZ:\Borddef{D_2,D_1} \to \Hilb \ ,
\ee
where $\Borddef{D_2,D_1}$ is the bordism category as just outlined, and where $\funZ$ is demanded to be continuous in the area and length parameters 
(Definition~\ref{def:daqft}).
An important example of a defect {\aQFT} is again provided by 2d Yang-Mills 
(2d~YM)
theory with
$G$ as above. In this case, the label set for two-dimensional connected components is just 
$D_2 = \{ G \}$ (corresponding to the 2d~YM theory without defects given by $G$),
and a possible choice for $D_1$ is finite-dimensional unitary representations of $G$. A defect line labeled by $R \in D_1$ corresponds to a Wilson line observable labeled by $R$. We investigate this situation in Section~\ref{sec:2dym:wilson} where we also consider symmetry defects in addition to Wilson lines.

Let $\funZ$ be defect {\aQFT} and consider a surface $\Sigma$ with defect lines where one such line (or circle) is labeled by $x \in D_1$. Suppose the area of the connected surface component to the right of that line is $a$ and that to the left is $b$. The defect condition $x$ is called \textsl{transmissive} if for all such surfaces $\Sigma$, $\funZ(\Sigma)$ only depends on $a+b$, and not on $a$ and $b$ separately (i.e.\ not on $a-b$). We interpret this as area flowing through the defect line labeled $x$ without affecting the value of $\funZ$. 
For example, in 2d~YM theory with $G$ connected, 
Wilson lines are transmissive if and only if the $G$-representation $R$ labeling it is 
a direct sum of trivial representations (Section~\ref{sec:2dym:wilson}). 

To construct examples of defect {\aQFT}s in a more systematic way, in Sections~\ref{sec:lattice:daqft}--\ref{sec:fusion-of-defects} we generalise the state-sum construction of defect TQFTs given in \cite{Davydov:2011dt} to accommodate area- and length-dependence. If defect {\aQFT}s are evaluated on bordisms without defects, one just obtains an {\aQFT} as before, though one which 
depends on the label from $D_2$ attached to the surface. Indeed, we will choose 
\be
	D_2^\mathrm{ss} = \{ \text{ strongly separable symmetric RFAs } \} \ .
\ee	
A defect line separating connected components of $\Sigma$ labeled by $A$ and $B$ in $D_2$ is in turn labeled by an \textsl{$A$-$B$-bimodule} $M$, which is in addition \textsl{dualisable}. 
Bimodules over regularised algebras are defined in Section~\ref{sec:modules}: they are objects $M \in \Hilb$ together with a bounded linear map $\rho_{a,l,b} : A \otimes M \otimes B \to M$, which now depends on
three parameters $a,l,b \in \Rb_{>0}$, subject to some natural conditions, see Definition~\ref{def:AB-bimodule-def}. In the state-sum construction $a,b$ are interpreted as area and $l$ as length in a rectangular plaquette bisected by the defect line. A bimodule is dualisable if it forms part of a dual pair of bimodules, we refer to Definition~\ref{def:dual-pair} for details. 
Altogether:
\be
D_1^\mathrm{ss} = \{ \text{ dualisable bimodules over regularised RFAs } \} \ .
\ee
Our main result with regard to defect {\aQFT}s is 
(Theorem~\ref{thm:state-sum-defect-aqft} and Proposition~\ref{prop:pdata-from-action}):

\begin{theorem}
	The state-sum construction defines a defect {\aQFT}
$$
\funZ^\mathrm{ss} : \Borddef{D_2^\mathrm{ss},D_1^\mathrm{ss}} \to \Hilb \ .
$$
\end{theorem}

\begin{figure}[tb]
	\centering
	\def\svgwidth{6cm}
	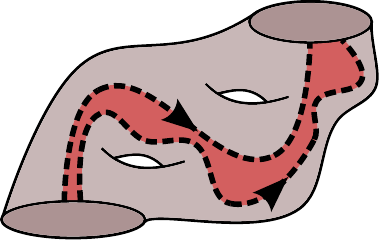
	\caption{Surface with parallel defect lines. The defect lines are the dotted lines in the figure. 
		In this figure they both start and end on a boundary component. 
		The defect lines both have length $l$ and the area of the surface component between them is $a$.}
	\label{fig:parallel-defect-lines}
\end{figure}

Crucially, one can define the tensor product $M \otimes_A N$ of 
bimodules. It satisfies a universal property (Definition~\ref{def:tensor-product-bimodule}),
and it can be shown to exist in some natural 
cases\footnote{These bimodules need to be left and right modules as well (which is not automatic). For details see Remark~\ref{rem:left-right-module-gives-bimodule} and Lemma~\ref{lem:left-right-module-bimodule}.\label{foot:bimod-lr-mod}}
at least in $\Hilb$ (Proposition~\ref{prop:tens-prod-lr-bimodules}).
The tensor product of bimodules is designed to model the ``fusion of defect lines'' in a defect {\aQFT} in the following sense. Let $\Sigma(a)$ be a bordism with two parallel defect lines, one labeled by $M \in D_1$ and the other by $N \in D_1$, and assume the connected surface component separating them is labeled by $A \in D_2$ (Figure~\ref{fig:parallel-defect-lines}).\footnote{
	These bimodules also need to be such that their tensor product has a dual, for more details see Lemma~\ref{lem:dual-pair-tensor-product}, the precise conditions of Theorem~\ref{thm:fusion} and Remark~\ref{rem:all-works-in-hilb}.
} 
Denote the area assigned to this component by $a$ and assume that the two defect lines have the same length label $l$. Let $\Sigma'$ be equal to $\Sigma(a)$, except that the component separating $M$ and $N$ has been collapsed, resulting in a single defect line which is now labeled by $M \otimes_A N$.
Then (Theorem~\ref{thm:fusion} and Remark~\ref{rem:all-works-in-hilb}):

\begin{theorem}
$\lim_{a \to 0} \funZ^\mathrm{ss}(\Sigma(a)) = \funZ^\mathrm{ss}(\Sigma')$.
\end{theorem}

In the example of 2d~YM theory, the relevant strongly separable symmetric RFA is $L^2(G)$, and the defect describing a Wilson line labeled by a unitary $G$-representation $R$ is obtained from the bimodule $R \otimes L^2(G)$ (see Section~\ref{sec:2dym:wilson} for details). As expected, the fusion of Wilson lines labeled $R$ and $S$ is given by the $G$-representation $R \otimes S$, which in terms of the above theorem follows from the bimodule tensor product $(R \otimes L^2(G)) \otimes_{L^2(G)} (S \otimes L^2(G)) \cong (R \otimes S) \otimes L^2(G)$ (Proposition~\ref{prop:Wilson-bimodules}).

Examples of defects that are not Wilson lines can be obtained by twisting the action on the regular bimodule $L^2(G)$ by appropriate automorphisms of $G$ (Lemma~\ref{lem:bimodule-isomorphisms}).

A special case of 2d~YM theory is when the group $G$ is finite. Here all zero area limits exist and the theory is given by the state-sum construction using the group algebra $\Cb[G]$. The resulting topological field theory is described by the centre of $\Cb[G]$ which is the same as class functions on $G$.
The relation of this TFT to orbifolds, see e.g.\ \cite[Ex.\,1]{Brunner:2013orb}, is that state-sum models are orbifolds of the trivial theories.
The present investigation suggests that including area-dependence may be useful to treat orbifolds by compact Lie groups, such as the one investigated in \cite{Gaberdiel:2012corb}.

\bigskip

This paper is organized as follows.
In Section~\ref{sec:RFA}
we collect all the required algebraic preliminaries about regularised algebras and RFAs, as well as their modules. 
In Section~\ref{sec:aqft-with-and-without} 
we state the definition of an {\aQFT} without and with defects, and we show that {\aQFT}s without defects correspond to commutative RFAs. 
Section~\ref{sec:state-sum-construction} 
contains the state-sum constructions, first the one without defects and then the version with defects. It is shown how the data needed for the state-sum construction can be obtained from RFAs and from dualisable bimodules, and how the tensor product of bimodules and the fusion of defect lines are related.
Finally, in Section~\ref{sec:2dym} we give a detailed treatment of our main example, 2d~YM
theory with Wilson lines as defects.

\subsubsection*{Acknowledgments}
The authors thank 
	Yuki Arano, 
	Nils Carqueville,
	Alexei Davydov,
	Reiner Lauterbach,
	Pau Enrique Moliner, 
	Chris Heunen, 
	Andr\'e Henriques,
	Ehud Meir,
	Catherine Meusburger,
	Gregor Schaumann,
	Richard Szabo
and
	Stefan Wagner
for helpful discussions
	and comments.
LS is supported by the DFG Research Training Group 1670 ``Mathematics Inspired by String Theory and Quantum Field Theory''.

\section{Regularised Frobenius algebras}\label{sec:RFA}

\subsection{Definition of regularised algebras and Frobenius algebras}\label{sec:RA-RFA-def}

Let $(\Sc,\otimes,\Ib)$ be a 
	strict\footnote{Although our examples of such categories will not be strict, one can always find an equivalent strict monoidal category with these properties
such that the equivalence functor is a homeomorphism on hom-sets.
	}
monoidal category whose hom-sets are topological spaces
such that composition is separately continuous.

We stress that we do not require the composition of $\Sc$ to be jointly continuous, nor
the tensor product of $\Sc$ to be (jointly or separately) continuous.
The reason is that our main example -- the category of Hilbert spaces with bounded linear maps and strong operator topology -- has none of these properties, see  Remark~\ref{rem:hilbcont} below.

\begin{definition}\label{def:reg-alg}
A \textsl{regularised algebra} in $\Sc$ is an object $A\in\Sc$ 
together with families of morphisms 
\begin{align}
	\mu_a\in\Sc( A^{\otimes 2}, A)\quad\text{and}\quad\eta_a\in\Sc(\Ib, A)\vphantom{A^{\otimes 2}}
	\label{eq:ra:def}
\end{align}
for every $a\in\Rb_{>0}$, called \textsl{product} and \textsl{unit},
such that the following relations hold:
\begin{enumerate}
	\item for every $a,a_1,a_2,b_1,b_2\in\Rb_{>0}$, such that $a_1+a_2=b_1+b_2=a$, 
		\begin{align}
		\mu_{a_1}\circ\left( id_{A}\otimes\eta_{a_2} \right)
		&=\mu_{b_1}\circ\left(\eta_{b_2}\otimes id_{A}\right) \ ,
		\label{eq:ra:unit}\\
		\mu_{a_1}\circ\left( id_{A}\otimes\mu_{a_2} \right)
		&=\mu_{b_1}\circ\left(\mu_{b_2}\otimes id_{A}\right)\ , \label{eq:ra:assoc}
	\end{align}
	\item \label{def:reg-alg:2}
	Let $P_a \in \Sc(A,A)$ be given by \eqref{eq:ra:unit}, i.e.\ $P_a = \mu_{a_1}\circ\left( id_{A}\otimes\eta_{a_2} \right)$.
	We require:
	\begin{enumerate}
	\item $\lim_{a\to0}P_a=\id_A$.
	\item The assignments
		\begin{align}
			\Rb_{\ge0}^n&\to\Sc(A^{\otimes n},A^{\otimes n})\nonumber\\
       			 (a_1,\dots,a_n)&\mapsto P_{a_1}\otimes\dots\otimes P_{a_n}
			\label{eq:racont}
		\end{align}
		are jointly continuous for every $n\ge1$.
\end{enumerate}
\end{enumerate}
Let
$A,B\in\Sc$ be regularised algebras. A \textsl{morphism of regularised algebras} $A\xrightarrow{f}B$ is
a morphism in $\Sc$ such that for every $a\in\Rb_{>0}$
\begin{center}
\begin{tabular}[h]{c c}
	$\eta^{B}_a=f\circ\eta^{A}_a$,
	& $\mu^{B}_a\circ\left( f\otimes f \right)=f\circ\mu_{a}^{A}$.\\
\end{tabular}
\end{center}
\end{definition}

Note that continuity is imposed only on $P_a$ but not on $\mu_a$ or $\eta_a$. However,  we will see shortly that continuity of $\mu_a$ and $\eta_a$ is implied by the definition. On the other hand, it is important not to impose the existence of an $a\to 0$ limit on $\mu_a$ and $\eta_a$; in Section~\ref{sec:examples} we will see examples where this limit does not exist, which would then have been excluded.

\begin{figure}[tb]
	\centering
	\def\svgwidth{6cm}
	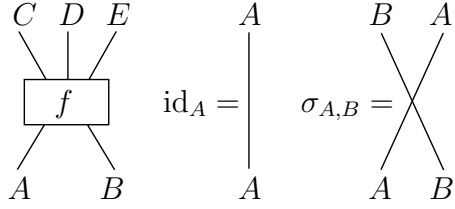
	\caption{
Graphical notation of morphisms in a strict (symmetric) monoidal category $\Sc$.
Here a morphism $f\in\Sc(A\otimes B,C\otimes D\otimes E)$, the identity $\id_{A}\in\Sc(A,A)$ and
the symmetric braiding $\sigma_{A,B}$ are shown.
The tensor product of morphisms is depicted by drawing the morphisms next to each other and 
composition of morphisms is stacking them on top of each other.
	}
	\label{fig:graphical-calc}
\end{figure}
We will often use string diagram notation to represent morphisms in strict monoidal categories, our conventions are given in Figure~\ref{fig:graphical-calc}.
The morphisms in \eqref{eq:ra:def} are drawn as 
\begin{align}
	\begin{aligned}
	\def\svgwidth{4.0cm}
\begingroup%
  \makeatletter%
  \providecommand\color[2][]{%
    \errmessage{(Inkscape) Color is used for the text in Inkscape, but the package 'color.sty' is not loaded}%
    \renewcommand\color[2][]{}%
  }%
  \providecommand\transparent[1]{%
    \errmessage{(Inkscape) Transparency is used (non-zero) for the text in Inkscape, but the package 'transparent.sty' is not loaded}%
    \renewcommand\transparent[1]{}%
  }%
  \providecommand\rotatebox[2]{#2}%
  \newcommand*\fsize{\dimexpr\f@size pt\relax}%
  \newcommand*\lineheight[1]{\fontsize{\fsize}{#1\fsize}\selectfont}%
  \ifx\svgwidth\undefined%
    \setlength{\unitlength}{164.51539567bp}%
    \ifx\svgscale\undefined%
      \relax%
    \else%
      \setlength{\unitlength}{\unitlength * \real{\svgscale}}%
    \fi%
  \else%
    \setlength{\unitlength}{\svgwidth}%
  \fi%
  \global\let\svgwidth\undefined%
  \global\let\svgscale\undefined%
  \makeatother%
  \begin{picture}(1,0.55564553)%
    \lineheight{1}%
    \setlength\tabcolsep{0pt}%
    \put(0,0){\includegraphics[width=\unitlength,page=1]{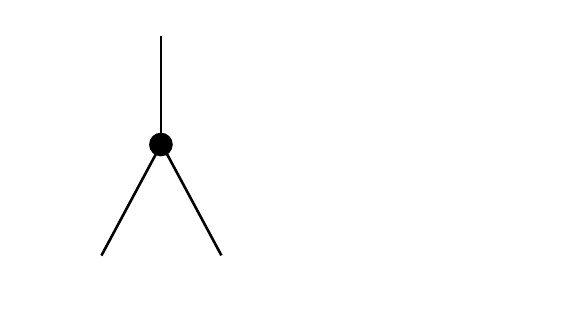}}%
    \put(-0.00454104,0.29131149){\color[rgb]{0,0,0}\makebox(0,0)[lt]{\lineheight{0}\smash{\begin{tabular}[t]{l}$\mu_a=$\end{tabular}}}}%
    \put(0,0){\includegraphics[width=\unitlength,page=2]{ra-morph.pdf}}%
    \put(0.58757931,0.29131149){\color[rgb]{0,0,0}\makebox(0,0)[lt]{\lineheight{0}\smash{\begin{tabular}[t]{l}$\eta_a=$\end{tabular}}}}%
    \put(0.90151591,0.16278382){\color[rgb]{0,0,0}\makebox(0,0)[lt]{\lineheight{0}\smash{\begin{tabular}[t]{l}\scriptsize{$a$}\end{tabular}}}}%
    \put(0.34939116,0.0080403){\color[rgb]{0,0,0}\makebox(0,0)[lt]{\lineheight{0}\smash{\begin{tabular}[t]{l}$A$\end{tabular}}}}%
    \put(0.1313396,0.0080403){\color[rgb]{0,0,0}\makebox(0,0)[lt]{\lineheight{0}\smash{\begin{tabular}[t]{l}$A$\end{tabular}}}}%
    \put(0.83340156,0.0080403){\color[rgb]{0,0,0}\makebox(0,0)[lt]{\lineheight{0}\smash{\begin{tabular}[t]{l}$\Ib$\end{tabular}}}}%
    \put(0.34051037,0.29257728){\color[rgb]{0,0,0}\makebox(0,0)[lt]{\lineheight{0}\smash{\begin{tabular}[t]{l}\scriptsize{$a$}\end{tabular}}}}%
    \put(0.24774848,0.51408159){\color[rgb]{0,0,0}\makebox(0,0)[lt]{\lineheight{0}\smash{\begin{tabular}[t]{l}$A$\end{tabular}}}}%
    \put(0.81304508,0.51408169){\color[rgb]{0,0,0}\makebox(0,0)[lt]{\lineheight{0}\smash{\begin{tabular}[t]{l}$A$\end{tabular}}}}%
  \end{picture}%
\endgroup%

	\end{aligned}
	\label{eq:ra:def-graphical}
\end{align}
and the relations in \eqref{eq:ra:unit} and \eqref{eq:ra:assoc} are
\begin{align}
	\begin{aligned}
	\def\svgwidth{11cm}
	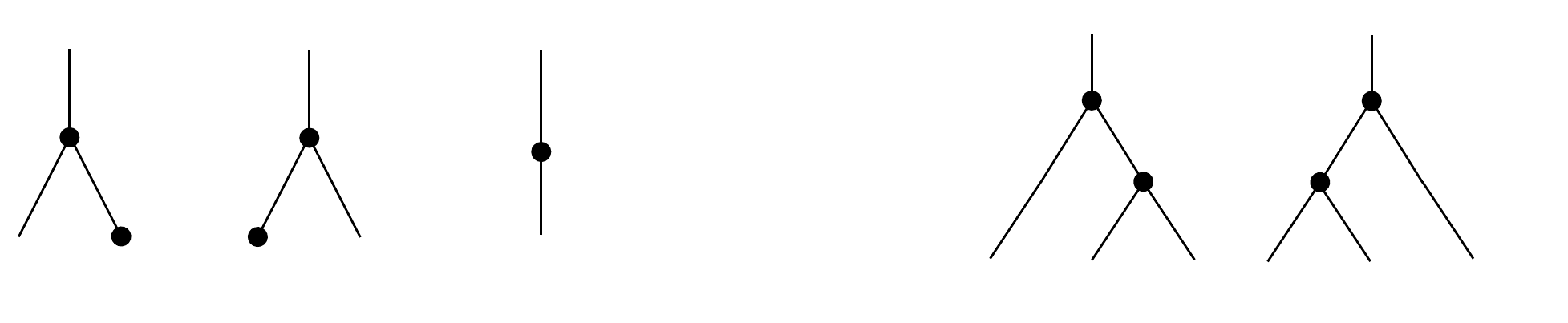
	\end{aligned}\ .
	\label{eq:ra:rel-graphical}
\end{align}

The next lemma gives some simple consequences of the above definition. In particular, part 4 shows that even though we imposed no continuity condition on the tensor product of $\Sc$, as far as morphisms built from a regularised algebra are concerned, everything is even jointly continuous.

\begin{lemma}\label{lem:ra:properties}
	Let $A$ be a regularised algebra. 
	Let $a_1,a_2,b_1,b_2,c_1,c_2\in\Rb_{>0}$
	such that $a_1+a_2=b_2+b_2=c_1+c_2$.
	\begin{enumerate}
		\item Let $\eta_a'\in\Sc(\Ib,A)$ be a family of morphisms which satisfy \eqref{eq:ra:unit}. 
			Then $\eta_a'=\eta_a$ for every $a\in\Rb_{>0}$. \label{lem:ra:uniqueunit}
		\item $P_{a_1}\circ\eta_{a_2}=\eta_{a_1+a_2}$ and $P_{a_1}\circ P_{a_2}=P_{a_1+a_2}$.\label{lem:ra:semigrp}
		\item $P_{a_1}\circ\mu_{a_2}=\mu_{b_1}\circ
			\left( P_{b_2}\otimes \id \right)
			=\mu_{c_1}\circ \left( \id\otimes P_{c_2}\right)=\mu_{a_1+a_2}$. \label{lem:ra:mupacommute}
		\item In the monoidal sub-category of $\Sc$ tensor generated by $A$, $\mu_a$ and $\eta_a$,
			every morphism is jointly continuous in the parameters.\label{lem:ra:generatedcat}
	\end{enumerate}
\end{lemma}
\begin{proof}
Let $a,b,c\in\Rb_{>0}$.

\smallskip

\noindent
\textsl{Part~\ref{lem:ra:uniqueunit}:} 
	Let us write $P'_{a+b}:=\mu_a\circ\left( \eta_b'\otimes\id_A \right)$ for the morphism in \eqref{eq:ra:unit}.
From \eqref{eq:ra:unit} we have that
	\begin{align}
		\mu_{a}\circ\left( \eta_b\otimes\eta_c' \right)=\mu_{a}\circ\left( \eta_c\otimes\eta_b' \right)
		\label{eq:pfuniqueunit1}
	\end{align}
	as both sides only depend on the sum of the parameters.
	We then have that 
	\begin{align}
		P_{a+b}\circ\eta_c'=P'_{a+b}\circ\eta_c\ ,
		\label{eq:pfuniqueunit2}
	\end{align}
	and using that the composition is separately continuous together with 
	$\lim_{a,b\to0}P_{a+b}=\lim_{a,b\to0}P'_{a+b}=\id_A$ 
	we get that $\eta_c'=\eta_c$ for every $c\in\Rb_{>0}$.

\medskip

\noindent
\textsl{Part~\ref{lem:ra:semigrp}:} 
	The first equation follows from Part~\ref{lem:ra:uniqueunit},
	because $P_{b}\circ\eta_a$ satisfies \eqref{eq:ra:unit}.
	The second equation follows from 
	associativity~\eqref{eq:ra:assoc} and from
	the first one:
	\begin{align}
		P_a\circ P_b=&\mu_{a_1}\circ
		\left( \eta_{a_2}\otimes\mu_{b_1} \right)\circ
		\left( \eta_{b_2}\otimes\id  \right)=
		\mu_{b_1}\circ \left(P_{a}\circ\eta_{b_2}\otimes \id\right)\nonumber\\
		=& \mu_{b_1}\circ \left(\eta_{a+b_2}\otimes \id\right)= P_{a+b}\ ,
		\label{eq:pfpasemigrp}
	\end{align}
	where $a=a_1+a_2$ and $b=b_1+b_2$.

\medskip

\noindent
\textsl{Part~\ref{lem:ra:mupacommute}:} 
	The first two equalities
	follow from the associativity of $\mu_a$ and the definition of $P_a$.
	For the last equality note that with $c=c_1+c_2$ we have
	\begin{align}
		P_c\circ P_b\circ\mu_a=
		P_{c+b}\circ\mu_a=
		\mu_c\circ\left( \eta_b\otimes\mu_a \right)=
		\mu_{c_1}\circ\left( \eta_{c_2}\otimes\mu_{a+b} \right)=
		P_{c}\circ\mu_{a+b} \ .
		\label{eq:pfmuparam}
	\end{align}
	Finally we use separate continuity of the composition and $\lim_{c\to0}P_c=\id_A$.

\medskip

\noindent
\textsl{Part~\ref{lem:ra:generatedcat}:} 
	Let $\varphi_{a_1,\dots,a_N} : A^{\otimes n} \to A^{\otimes m}$ be a morphism in $\Sc$ tensor generated by $\mu_a$ and $\eta_a$, involving a total of $N$ copies of the latter two morphisms, with parameters $a_1,\dots,a_N$.
	One can write 
$\varphi_{a_1,\dots,a_N}$ in the form 
	$\varphi^{(1)}_{\eps_1}\circ \left( \bigotimes_{i=1}^{N}
	P_{a_i-\eps_1-\eps_2} \right)\circ\varphi^{(2)}_{\eps_2}$
	for some $\eps_1,\eps_2\in\Rb_{>0}$ and morphisms
	$\varphi^{(1)}_{\eps_1}$ and $\varphi^{(2)}_{\eps_2}$. Then by separate
	continuity of the composition of $\Sc$ and joint continuity of $\bigotimes_{i=1}^{N}P_{a_i}$
in \eqref{eq:racont},
	the morphism $\varphi_{a_1,\dots,a_N}$ is jointly continuous in the parameters 
	$a_1,\dots,a_N$.
\end{proof}

As a special case of Part~\ref{lem:ra:generatedcat} of the above lemma we get:

\begin{corollary}
	In a regularised algebra the maps $a\mapsto \mu_a$ and $a\mapsto \eta_a$ are continuous.	
\end{corollary}

Next we introduce the dual concept to a regularised algebra.
A \textsl{regularised coalgebra} in $\Sc$ is an object $A\in\Sc$ 
together with families of morphisms 
\begin{align}
	\Delta_a: A\to A^{\otimes 2}\quad\text{and}\quad\eps_a:A\to \Ib\vphantom{A^{\otimes 2}}
	\label{eq:rca:def}
\end{align}
for $a\in\Rb_{>0}$, called \textsl{coproduct} and \textsl{counit},
such that the following relations hold: for all $a,a_1,a_2,b_1,b_2>0$, such that $a_1+a_2=b_1+b_2=a$,
\begin{align}
	\left( id_{A}\otimes\eps_{a_2} \right)\circ\Delta_{a_1}
	&=\left(\eps_{b_2}\otimes id_{A}\right)\circ\Delta_{b_1}=:P'_a\ ,
	\label{eq:rca:counit}\\
	\left( id_{A}\otimes\Delta_{a_2} \right)\circ\Delta_{a_1}
	&=\left(\Delta_{b_2}\otimes id_{A}\right)\circ\Delta_{b_1}\ ,
	\label{eq:rca:coassoc}
\end{align}
$\lim_{a\to0}P'_a=\id_A$, and 
the assignments
\begin{align}
	\Rb_{\ge0}^n&\to\Sc(A^{\otimes n},A^{\otimes n})\nonumber\\
        (a_1,\dots,a_n)&\mapsto P'_{a_1}\otimes\dots\otimes P'_{a_n}
	\label{eq:rcacont}
\end{align}
are jointly continuous for every $n\ge1$.
A \textsl{morphism of regularised coalgebras}, is a morphism of the objects which
is compatible with $\Delta_a$ and $\eps_a$ in the obvious way.
Note that for a regularised coalgebra the dual statements of Lemma~\ref{lem:ra:properties} hold.
For the morphisms in \eqref{eq:rca:def} we introduce the following graphical notation:
\begin{align}
	\begin{aligned}
	\def\svgwidth{4.0cm}
\begingroup%
  \makeatletter%
  \providecommand\color[2][]{%
    \errmessage{(Inkscape) Color is used for the text in Inkscape, but the package 'color.sty' is not loaded}%
    \renewcommand\color[2][]{}%
  }%
  \providecommand\transparent[1]{%
    \errmessage{(Inkscape) Transparency is used (non-zero) for the text in Inkscape, but the package 'transparent.sty' is not loaded}%
    \renewcommand\transparent[1]{}%
  }%
  \providecommand\rotatebox[2]{#2}%
  \newcommand*\fsize{\dimexpr\f@size pt\relax}%
  \newcommand*\lineheight[1]{\fontsize{\fsize}{#1\fsize}\selectfont}%
  \ifx\svgwidth\undefined%
    \setlength{\unitlength}{156.42446855bp}%
    \ifx\svgscale\undefined%
      \relax%
    \else%
      \setlength{\unitlength}{\unitlength * \real{\svgscale}}%
    \fi%
  \else%
    \setlength{\unitlength}{\svgwidth}%
  \fi%
  \global\let\svgwidth\undefined%
  \global\let\svgscale\undefined%
  \makeatother%
  \begin{picture}(1,0.5848442)%
    \lineheight{1}%
    \setlength\tabcolsep{0pt}%
    \put(0,0){\includegraphics[width=\unitlength,page=1]{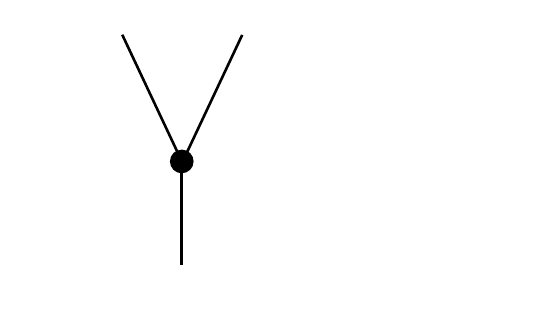}}%
    \put(0.376188,0.29017059){\color[rgb]{0,0,0}\makebox(0,0)[lt]{\lineheight{0}\smash{\begin{tabular}[t]{l}\scriptsize{$a$}\end{tabular}}}}%
    \put(-0.00477592,0.29017059){\color[rgb]{0,0,0}\makebox(0,0)[lt]{\lineheight{0}\smash{\begin{tabular}[t]{l}$\Delta_a=$\end{tabular}}}}%
    \put(0,0){\includegraphics[width=\unitlength,page=2]{rca-morph.pdf}}%
    \put(0.57002491,0.29017059){\color[rgb]{0,0,0}\makebox(0,0)[lt]{\lineheight{0}\smash{\begin{tabular}[t]{l}$\eps_a=$\end{tabular}}}}%
    \put(0.8964219,0.42321445){\color[rgb]{0,0,0}\makebox(0,0)[lt]{\lineheight{0}\smash{\begin{tabular}[t]{l}\scriptsize{$a$}\end{tabular}}}}%
    \put(0.18320425,0.54096627){\color[rgb]{0,0,0}\makebox(0,0)[lt]{\lineheight{0}\smash{\begin{tabular}[t]{l}$A$\end{tabular}}}}%
    \put(0.41039156,0.54113052){\color[rgb]{0,0,0}\makebox(0,0)[lt]{\lineheight{0}\smash{\begin{tabular}[t]{l}$A$\end{tabular}}}}%
    \put(0.29551977,0.00845618){\color[rgb]{0,0,0}\makebox(0,0)[lt]{\lineheight{0}\smash{\begin{tabular}[t]{l}$A$\end{tabular}}}}%
    \put(0.81038589,0.00862043){\color[rgb]{0,0,0}\makebox(0,0)[lt]{\lineheight{0}\smash{\begin{tabular}[t]{l}$A$\end{tabular}}}}%
    \put(0.82956447,0.47849573){\color[rgb]{0,0,0}\makebox(0,0)[lt]{\lineheight{0}\smash{\begin{tabular}[t]{l}$\Ib$\end{tabular}}}}%
  \end{picture}%
\endgroup%

	\end{aligned}\ .
	\label{eq:rca-morph-graphical}
\end{align}

A key notion in this paper is the following:

\begin{definition}\label{def:rfa}
A \textsl{regularised Frobenius algebra} (or RFA in short) in $\Sc$ is
a regularised algebra $A$, which is also a regularised coalgebra, such that
\begin{align}
	\begin{aligned}
		\def\svgwidth{8cm}
\begingroup%
  \makeatletter%
  \providecommand\color[2][]{%
    \errmessage{(Inkscape) Color is used for the text in Inkscape, but the package 'color.sty' is not loaded}%
    \renewcommand\color[2][]{}%
  }%
  \providecommand\transparent[1]{%
    \errmessage{(Inkscape) Transparency is used (non-zero) for the text in Inkscape, but the package 'transparent.sty' is not loaded}%
    \renewcommand\transparent[1]{}%
  }%
  \providecommand\rotatebox[2]{#2}%
  \ifx\svgwidth\undefined%
    \setlength{\unitlength}{257.67612305bp}%
    \ifx\svgscale\undefined%
      \relax%
    \else%
      \setlength{\unitlength}{\unitlength * \real{\svgscale}}%
    \fi%
  \else%
    \setlength{\unitlength}{\svgwidth}%
  \fi%
  \global\let\svgwidth\undefined%
  \global\let\svgscale\undefined%
  \makeatother%
  \begin{picture}(1,0.35161639)%
    \put(0,0){\includegraphics[width=\unitlength]{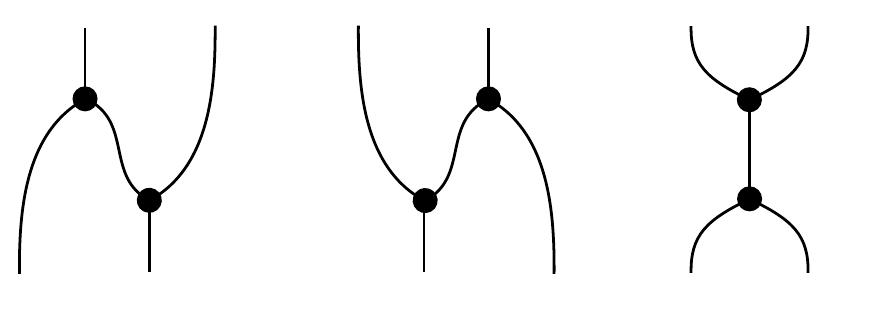}}%
    \put(0.1301166,0.23622226){\color[rgb]{0,0,0}\makebox(0,0)[lb]{\smash{\scriptsize{$a_1$}}}}%
    \put(0.30200142,0.16985526){\color[rgb]{0,0,0}\makebox(0,0)[lb]{\smash{$=$}}}%
    \put(0.49179793,0.2362223){\color[rgb]{0,0,0}\makebox(0,0)[lb]{\smash{\scriptsize{$b_1$}}}}%
    \put(0.5114265,0.12486634){\color[rgb]{0,0,0}\makebox(0,0)[lb]{\smash{\scriptsize{$b_2$}}}}%
    \put(0.67638668,0.16985531){\color[rgb]{0,0,0}\makebox(0,0)[lb]{\smash{$=$}}}%
    \put(0.76989807,0.23347447){\color[rgb]{0,0,0}\makebox(0,0)[lb]{\smash{\scriptsize{$c_1$}}}}%
    \put(0.10111068,0.12548842){\color[rgb]{0,0,0}\makebox(0,0)[lb]{\smash{\scriptsize{$a_2$}}}}%
    \put(0.77210975,0.12480148){\color[rgb]{0,0,0}\makebox(0,0)[lb]{\smash{\scriptsize{$c_2$}}}}%
    \put(0.07865979,0.33274569){\color[rgb]{0,0,0}\makebox(0,0)[lb]{\smash{$A$}}}%
    \put(0.22768408,0.33274581){\color[rgb]{0,0,0}\makebox(0,0)[lb]{\smash{$A$}}}%
    \put(0.38291771,0.33274569){\color[rgb]{0,0,0}\makebox(0,0)[lb]{\smash{$A$}}}%
    \put(0.53194199,0.33274581){\color[rgb]{0,0,0}\makebox(0,0)[lb]{\smash{$A$}}}%
    \put(0.76168776,0.33274569){\color[rgb]{0,0,0}\makebox(0,0)[lb]{\smash{$A$}}}%
    \put(0.88587467,0.33274581){\color[rgb]{0,0,0}\makebox(0,0)[lb]{\smash{$A$}}}%
    \put(-0.0020617,0.00365042){\color[rgb]{0,0,0}\makebox(0,0)[lb]{\smash{$A$}}}%
    \put(0.14696259,0.00365051){\color[rgb]{0,0,0}\makebox(0,0)[lb]{\smash{$A$}}}%
    \put(0.4512205,0.00365042){\color[rgb]{0,0,0}\makebox(0,0)[lb]{\smash{$A$}}}%
    \put(0.60024479,0.00365051){\color[rgb]{0,0,0}\makebox(0,0)[lb]{\smash{$A$}}}%
    \put(0.76168776,0.00365042){\color[rgb]{0,0,0}\makebox(0,0)[lb]{\smash{$A$}}}%
    \put(0.88587467,0.00365051){\color[rgb]{0,0,0}\makebox(0,0)[lb]{\smash{$A$}}}%
  \end{picture}%
\endgroup%

	\end{aligned}
	\label{eq:rfa:frobrel}
\end{align}
holds for all $a_1+a_2=b_1+b_2=c_1+c_2$.
A \textsl{morphism of RFAs}, 
is a morphism of regularised algebras and coalgebras.
\end{definition}

In an RFA the semigroup homomorphism $P_a$ from the algebra structure and $P_a'$ from the coalgebra structure coincide:

\begin{lemma}
For an RFA we have $P_a=P'_a$ for all $a>0$.
\end{lemma}
\begin{proof}
Let $a,b \in \Rb_{>0}$ be arbitrary. Choose $a_1,a_2,b_2,b_2\in\Rb_{>0}$ such that $a=a_1+a_2$, $b=b_1+b_2$ and $a>b_1$ and $b>a_1$ (e.g.\ $b_1 = \frac a2$, $a_1= \frac b2$).	
By relation \eqref{eq:rfa:frobrel} one has that
\begin{align}
	\left( \mu_{a_2}\otimes\id_A \right) \circ
	\left( \id_A \otimes \Delta_{b_2}\right)
	=\Delta_{a_1+a_2-b_1}\circ\mu_{b_1+b_2-a_1}
	\label{eq:frspec}
\end{align}
Composing \eqref{eq:frspec} with $\id_A\otimes\eps_{b_1}$ from the left
and with $\eta_{a_1}\otimes\id_A$ from the right yields
\begin{align}
	P_a\circ P'_b=P'_a\circ P_b.
	\label{eq:ppprime}
\end{align}
We can take the $b\to0$ limit on both sides of \eqref{eq:ppprime} and use
separate continuity of the composition in $\Sc$ to get $P_a=P'_a$.
\end{proof}

\begin{remark}\label{rem:quasi-invertible}
	Requiring that $\lim_{a\to0}P_a=\id$ does not imply that $P_a$ is mono or epi for every $a\in\Rb_{\ge0}$, 
	as the following example in $\Hilb$ illustrates.\footnote{We would like to thank Reiner Lauterbach for explaining this example to us.}
	Let $L\in\Rb_{>0}$ and $\Hc:=L^2([0,L])$. Define $P_a\in\Bc(\Hc)$ for $f\in\Hc$ to be right shift by $a$,
	\begin{align}
		\begin{aligned}
			(P_a(f))(x) =
			\begin{cases}
				0&\text{if $x< a$}\\
				f(x-a)&\text{if $x\ge a$}
			\end{cases}
		\end{aligned}\ .
		\label{eq:not-injective-nor-surjective-semigroup}
	\end{align}
	This is neither mono nor epi for any choice of $a\in\Rb_{>0}$, and for $a\ge L$ we even have $P_a=0$.
	\label{rem:not-injective-nor-surjective-semigroup}
\end{remark}

Usual (non-regularised) Frobenius algebras have an equivalent characterisation via a non-degenerate invariant pairing. The same is true in the regularised setting, as we now illustrate.
Let $A$ a regularised algebra $A$ together with a family of morphisms
$\eps_a:A\to\Ib$ for $a\in\Rb_{>0}$ such that for all $a_1+a_2 = b_1+b_2$ we have
$\eps_{a_1}\circ\mu_{a_2} = \eps_{b_1}\circ\mu_{b_2}$.
We call the pairing $\beta_a:=\eps_{a_1}\circ\mu_{a_2}$ 
\textsl{non-degenerate} if there is a family of morphisms
$\gamma_a:\Ib\to A^{\otimes 2}$ such that 
\begin{align}
	\left( \id_A\otimes \beta_{a_1} \right)\circ
	\left( \gamma_{a_2} \otimes\id_A\right)
	~=~P_a~=~
	\left( \beta_{b_1}\otimes\id_A\right)\circ
	\left(  \id_A\otimes\gamma_{b_2} \right)
	\label{eq:rfa:nondeg}
\end{align}
for all $a_1+a_2=b_1+b_2=a$.

\begin{lemma}\label{lem:copairing}
\begin{enumerate}
\item  \label{lem:copairing:1} For all $a,b>0$,
	\begin{align}
		\left(P_a\otimes\id_A\right)\circ\gamma_b
		=\gamma_{a+b} =
		\left(\id_A\otimes P_a\right)\circ\gamma_b\ ,
		\label{eq:copairing}
	\end{align}
\item \label{lem:copairing:2}
	The relation \eqref{eq:rfa:nondeg} defines $\gamma_a$ uniquely.
\item\label{lem:copairing:3}	The map $a \mapsto \gamma_a$ is continuous.
\end{enumerate}	
\end{lemma}
\begin{proof}
\textsl{Part~\ref{lem:copairing:1}:}
	From \eqref{eq:rfa:nondeg} one has that
\begin{align}
	\left( \id_A\otimes \beta_{b} \right)\circ
	\left( \gamma_{a+x} \otimes\id_A\right) 
	=P_{a+b+x}=
	\left( \id_A\otimes \beta_{b+x} \right)\circ
	\left( \gamma_{a} \otimes\id_A\right).
	\label{eq:copairing1}
\end{align}
Tensoring with $\id_A$ from the right and composing with $\gamma_c$
from the right gives
\begin{align}
	\left( \id_A\otimes P_{b+c} \right)\circ \gamma_{a+x}=
	\left( \id_A\otimes P_{b+c+x} \right)\circ \gamma_{a}.
	\label{eq:copairing2}
\end{align}
Taking the limit $b,c\to0$ gives the second equation of \eqref{eq:copairing}.
One obtains the first equation of \eqref{eq:copairing} similarly.

\noindent
\textsl{Part~\ref{lem:copairing:2}:}
If $\Gamma_a$ is a family of morphisms satisfying \eqref{eq:rfa:nondeg},
then it also satisfies \eqref{eq:copairing}. Then
\begin{align}
	\begin{aligned}
	\Gamma_{a+b+c}&=
	\left( \id_A\otimes P_{b+c} \right)\circ \Gamma_{a}=
	\left( \id_A\otimes \beta_{b}\otimes\id_A\right)\circ\left( \Gamma_{a}\otimes\gamma_{c}\right)
	=\left( P_{a+b} \otimes \id_A\right)\circ \gamma_{c}\\
	&=\gamma_{a+b+c}\ .
	\end{aligned}
	\label{eq:copairing3}
\end{align}

\noindent
\textsl{Part~\ref{lem:copairing:3}:}
Continuity of $\gamma_a$ is clear from Part~\ref{lem:copairing:1} by continuity of $P_a$ and separate continuity of the composition in $\Sc$.

\end{proof}

We can now give the alternative characterisation of an RFA.

\begin{proposition} \label{lem:rfa:altdef}
Let $A$ be a regularised algebra  
and let $\eps_a : A \to \Ib$ be a family of morphisms such that the pairing $\eps_{a_1} \circ \mu_{a_2}$ only depends on $a_1+a_2$. If $\eps_{a_1}\circ\mu_{a_2}$ 
is non-degenerate, then $A$ is a regularised Frobenius algebra with counit
$\eps_a$ and coproduct
\begin{align}
	\Delta_a:=\left( \mu_{a_1}\otimes\id_A \right)\circ
	\left( \id_A \otimes\gamma_{a_2}\right),
	\label{eq:rfa:deltaaltdef}
\end{align}
for some $a_1+a_2=a$.
\end{proposition}

\begin{proof}
	By \eqref{eq:copairing} and Part~\ref{lem:ra:mupacommute} of Lemma~\ref{lem:ra:properties}, $\Delta_a$ indeed only depends on $a = a_1+a_2$.
	Checking the algebraic relations \eqref{eq:ra:unit}, \eqref{eq:ra:assoc}, 
	\eqref{eq:rca:counit}, \eqref{eq:rca:coassoc} and \eqref{eq:rfa:frobrel} of an RFA
	is analogous as for ordinary Frobenius algebras. From these follows that $P'_a=P_a$,
	so in particular 
	$\lim_{a\to0}P'_a=\id_A$ and \eqref{eq:rcacont} holds.
\end{proof}

Note that the converse of the proposition holds trivially:
if $A$ is a regularised Frobenius algebra then $\eps_a$ is 
non-degenerate in the above sense with $\gamma_a=\Delta_{a_1}\circ\eta_{a_2}$.

\medskip

Let the category $\Sc$ be in addition symmetric with braiding $\sigma$.
Then we call a regularised algebra $A\in\Sc$
\textsl{commutative} if $\mu_a\circ \sigma=\mu_a$ for all $a\in\Rb_{>0}$. 
The \textsl{centre} of a regularised algebra $A$ is 
an object $B\in\Sc$ and a morphism $i_B:B\to A$ such that
	\begin{align}
\mu_a\circ \sigma \circ \left( i_B \otimes \id_A \right)=
\mu_a\circ\left( i_B \otimes \id_A \right)
		\label{eq:central}
	\end{align} 
	for all $a\in\Rb_{>0}$,
	which is universal in the following sense.
	If there is an object $C$ and a morphism $f:C\to A$ satisfying the above
	equation then there is a unique morphism $\tilde{f}:C\to B$ such that
	the diagram
	\begin{equation}
	\begin{tikzcd}
		B \ar{r}{i_B} &A\\
		\ar{u}{\tilde{f}}C \ar{ru}[swap]{f} &
	\end{tikzcd}
		\label{eq:central2}
	\end{equation}
	commutes. This implies in particular that $i_B$ is mono
	\cite{Davydov:2010fc}.
If the centre exists then one has the induced morphism
$\tilde{P}_{a}\in\Sc(B,B)$ such that $P_a\circ i_B=i_B\circ\tilde{P}_a$.

\begin{lemma}
	If the centre of a regularised algebra exists,
	$\lim_{a\to0}\tilde{P}_{a}$ exists
	and the maps
	\begin{align}
		(a_1,\dots,a_n)\mapsto \tilde{P}_{a_1}
		\otimes\dots\otimes \tilde{P}_{a_n}
		\label{eq:centerind}
	\end{align}
	are jointly continuous for every $n\ge 1$, 
	then it is a commutative regularised algebra.
	\label{lem:ra:center}
\end{lemma}
\begin{proof}
	Similarly as one gets $\tilde{P}_{a}$, one has 
	induced multiplication and unit $\tilde{\mu}$ and $\tilde{\eta}$.
	Checking associativity and unitality
	is now straightforward. 
	The limit $\lim_{a\to0}\tilde{P}_{a}=\id_{B}$ follows from $P_a\circ i_B=i_B\circ\tilde{P}_a$ by separate continuity of composition in $\Sc$.
\end{proof}

A regularised algebra is \textsl{separable} if there exists a family
of morphisms $e_a\in\Sc(\Ib,A\otimes A)$ for every $a\in\Rb_{>0}$ such that
\begin{enumerate}
	\item $(\mu_{a_1}\otimes\id_A)\circ(\id_A\otimes e_{a_2})=(\id_A\otimes\mu_{b_1})\circ(e_{b_2}\otimes\id_A)$
		and
	\item $\mu_{a_1}\circ e_{a_2}=\eta_a$.
\end{enumerate}
The $e_a$ are called
\textsl{separability idempotents}.
A regularised algebra $A$ is \textsl{strongly separable} if it is separable and
furthermore
\begin{enumerate}[resume]
	\item $\sigma_{A,A}\circ e_a=e_a$.
\end{enumerate}
These notions are direct generalisations of separability and strong separability for algebras, see e.g.\ \cite{Kanzaki:1964ss,Lauda:2007oc}. 

For an RFA $A$,
we call the family of morphisms 
$\tau_{a}:=\mu_{a_1}\circ\Delta_{a_2}\circ\eta_{a_3}$ 
for $a_1,a_2,a_3\in\Rb_{>0}$ with $a=a_1+a_2+a_3$ the 
\textsl{window element} of $A$, cf.~\cite[Def.\,2.12]{Lauda:2007oc}. 
We call the window element \textsl{invertible} if there exists
a family of morphisms $z_a\in\Sc(\Ib,A)$ for $a\in\Rb_{>0}$ 
(the \textsl{inverse}) such that
$\mu_{a_1}\circ(\tau_{a_2}\otimes z_{a_3})=\eta_{a_1+a_2+a_3}
= \mu_{a_1}\circ(z_{a_3}\otimes \tau_{a_3})$.
From a direct computation one can verify that if there exists another family
of morphisms $z'_a$ which satisfies the above equation then $z'_a=z_a$ for every 
$a\in\Rb_{>0}$, that is the inverse of the window element is unique.
In the following we write $\tau^{-1}_a$ for the inverse of $\tau_a$.
It is easy to check that the window element and its inverse satisfy 
\eqref{eq:central}, that is, they factorise through the centre if the centre exists.

An RFA is \textsl{symmetric} if $\eps_{a_1}\circ\mu_{a_2}\circ \sigma=\eps_{b_1}\circ\mu_{b_2}$.
The following is a direct translation of \cite[Thm.\,2.14]{Lauda:2007oc} 
for strong separability for
symmetric Frobenius algebras. 

\begin{proposition}\label{prop:symm-rfa-strongly-separable}
A symmetric RFA is  strongly separable
if and only if its window element is invertible.
\end{proposition}
\begin{proof}
	Set $e_a:=\Delta_{a_1}\circ\tau^{-1}_{a_2}$. 
	Conversely set $\tau^{-1}_a:=(\eps_{a_1}\otimes\id_A)\circ e_{a_2}$.
\end{proof}

\subsection{RFAs in the category of Hilbert spaces}

Let $\Hilb$ denote the symmetric monoidal category of
Hilbert spaces and bounded linear maps with the strong operator topology
on the hom-sets and the Hilbert space tensor product.
We write $\Bc(\Hc,\Kc):=\Hilb(\Hc,\Kc)$ and $\Bc(\Hc):=\Hilb(\Hc,\Hc)$
for the hom-sets.
\begin{remark}\label{rem:hilbcont}
\begin{enumerate}
\item
	In $\Hilb$ the composition of morphisms is separately continuous, but not jointly continuous.
	The tensor product is not separately continuous, in fact
	even tensoring with the identity morphism of an infinite-dimensional Hilbert space
	is not continuous.
	Furthermore, taking adjoints is not continuous in $\Hilb$.
	For more details see \cite[Prob.\,211]{Halmos:2012hpb} and 
	\cite[Sec.\,2.6]{Kadison:1983oa1}.
\item
	Instead of the strong operator topology one could use the so called
	ultrastrong-$*$ operator topology in which taking adjoint
	is continuous and tensoring is separately continuous
	\cite[Prop.\,I.8.6.4]{Blackadar:2006op}.\footnote{We thank Yuki Arano for explaining this to us.}
	However in this topology composition of morphisms is still
	not jointly continuous \cite[Prop.\,46.1-2]{Reinhard:2015ph}.
\end{enumerate}
\end{remark}

We give the following technical lemma which will be useful later.
\begin{lemma}
	Let $f:X\to X'$ and $g:Y\to Y'$ be morphisms in $\Hilb$, 
	both mono (resp.\ epi).
	Then $f\otimes g:X\otimes Y\to X'\otimes Y'$ is mono (resp.\ epi).
	\label{lem:monoepi}
\end{lemma}
\begin{proof}
	If $f$ and $g$ are both epi, then their image is dense. 
	Then the algebraic tensor product of $\im(f)$ and $\im(g)$ is dense
	in $X\otimes Y$ and it is contained in $\im(f\otimes g)$, which is hence dense. 
	This means that $f\otimes g$ is epi.
	
	If $f$ and $g$ are both mono, then $f^{\dagger}$ and $g^{\dagger}$ are both epi.
	But then $f^{\dagger}\otimes g^{\dagger}=(f\otimes g)^{\dagger}$ is epi
	and hence $f\otimes g$ is mono.	
\end{proof} 

The next lemma shows in particular that an RFA in $\Hilb$ has a Hilbert basis with at most countably many elements.

\begin{lemma} \label{lem:patrclass}
	Let $A\in\Hilb$ be an RFA. 
	\begin{enumerate}
		\item 
		The Hilbert space underlying $A$ is separable.
		\label{lem:patrclass:1}
		\item For all $a\in\Rb_{>0}$, $P_a$ is a trace class operator (and hence compact).\label{lem:patrclass:2}
	\end{enumerate}

\end{lemma}
\begin{proof}
	\textsl{Part~\ref{lem:patrclass:1}:}
	Let $\setc*{\phi_j}{j\in I}$ be a complete set of orthonormal vectors in $A$
and
	write $\gamma_{a_2}(1)=\sum_{k,l\in I}\phi_k\otimes\phi_l\,\gamma_{a_2}^{kl}$.
	By \cite[Cor.\,5.28]{Kubrusly:2001ot}, 
	independently of the countability of the indexing set $I$,
	there are at most countably many non-zero terms in the above sum. 
Thus for a given $a_2$ there is a countable set of pairs $(k,l) \in I \times I$ such that $\gamma_{a_2}^{kl}\neq0$. Define $I(a_2)\subseteq I$ to be the countable set of all elements of $I$ that appear in such a pair.
	Let
	\begin{align}
		J:=\bigcup_{n\in\Zb_{>0}}I(1/n)\subseteq I\quad\text{ and }\quad A_J:=\overline{\mathrm{span}\setc*{ \phi_j}{j\in J }}\subseteq A\ .
		\label{eq:union-of-index-sets}
	\end{align}
	Note that $J$ is countable and $A_J$ is separable.
	By \eqref{eq:rfa:nondeg}, for every $v\in A$ and $n\in\Zb_{>0}$ we have that
	\begin{align}
		P_{1/n}(v)\in A_J\quad\text{ and }\quad \lim_{n\to \infty}P_{1/n}(v)=v\ ,
		\label{eq:P-1-over-n-v}
	\end{align}
	since $\lim_{n\to \infty}P_{1/n}=\id_A$ in the strong operator topology.
	Since $A_J$ is closed, $v$ is an element of $A_J$. 
	We have shown that $A_J=A$, and hence that $A$ is separable.

\medskip

\noindent
	\textsl{Part~\ref{lem:patrclass:2}:}
	First let us compute the following expression for some $a_1,a_2\in\Rb_{>0}$:
	\begin{align}
		\beta_{a_1}\circ\sigma\circ\gamma_{a_2}(1)
		=\sum_{j,k\in I}\beta_{a_1}(\phi_j\otimes\phi_k)\gamma_{a_2}^{kj}\ .
		\label{eq:beta-sigma-gamma}
	\end{align}
	This is an absolutely convergent sum, since the lhs is a composition of bounded linear maps.
	We can rewrite this expression using \eqref{eq:rfa:nondeg} to get
	\begin{align}
		\begin{aligned}
		\beta_{a_1}\circ\sigma\circ\gamma_{a_2}(1)
		=&\sum_{j,k,l\in I}\beta_{a_1}(\phi_j\otimes\phi_k)\gamma_{a_2}^{kl}
		\langle\phi_j|\phi_l\rangle\\
		=&\sum_{j,k,l\in I}\langle\phi_j|(\beta_{a_1}\otimes\id_A)\phi_j\otimes\phi_k\otimes\phi_l
		\gamma_{a_2}^{kl}\rangle\rangle\\
		=&\sum_{j\in I}\langle\phi_j|(\beta_{a_1}\otimes\id_A)\phi_j\otimes\sum_{k,l\in I}\phi_k\otimes\phi_l
		\gamma_{a_2}^{kl}\rangle\rangle\\
		=&\sum_{j\in I}\langle\phi_j|(\beta_{a_1}\otimes\id_A)
		\circ(\id_A\otimes\gamma_{a_2}(1))\phi_j\rangle\\
		=&\sum_{j\in I}\langle\phi_j|P_a\phi_j\rangle\ ,
		\end{aligned}
		\label{eq:eq:beta-sigma-gamma-tr-pa}
	\end{align}
	which is again an absolutely convergent sum.
	By \cite[Ex.\,18.2]{Conway:2000ot} $P_a$ is a trace class operator
	if and only if $\sum_{j\in I}\langle\phi_j|P_a\phi_j\rangle$ is absolutely convergent
	for every choice of orthonormal basis $\{ \phi_j \}$,
	which we just have shown. In this case we have that
	\begin{align}
	\tr(P_a)=\sum_{j\in I}\langle\phi_j|P_a\phi_j\rangle\ .
		\label{eq:tr-pa}
	\end{align}
\end{proof}

Let $A\in\Hilb$ be an RFA.
By 
the Part~\ref{lem:patrclass:2} of Lemma~\ref{lem:patrclass} and
\cite[Thm.\,II.4.29]{Engel:1999sg}, 
$a\mapsto P_a$ (for $a>0$) is norm continuous.
The following corollary shows that if we had defined $\Hilb$ to have 
the norm operator topology on hom-sets all examples of
RFAs in $\Hilb$ 
with self adjoint $P_a$
would be finite dimensional.

\begin{corollary}\label{cor:finite-in-norm}
	Let $A\in\Hilb$ be an RFA such that 
	$\lim_{a\to0}P_a=\id_A$ in the norm topology
	on $\Bc(A)$. Then $A$ is finite dimensional.
	\label{cor:normcontisbad}
\end{corollary}
\begin{proof}
By \cite[Thm.\,I.3.7]{Engel:1999sg},
	$P_a=e^{aH}$ for some $H\in\Bc(A)$. In particular
	the spectrum $\sp(H)$ of $H$ is bounded.
	By the spectral mapping theorem 
	\cite[Thm.\,VII.4.10]{Conway:1994fa} 
	one has $\sp(P_a)=\sp(e^{aH})=e^{a \cdot \sp(H)}$, in particular $0\not\in\sp(P_a)$.
	So altogether we have that $P_a$ is invertible.
	From Lemma~\ref{lem:patrclass}\,\eqref{lem:patrclass:2} we know that $P_a$ is compact.
	Since a bounded operator composed with a compact operator is compact,
	the identity $\id_A=P_a^{-1}P_a$ shows that $\id_A$ is compact.
	The identity on $A$ is compact if and only if $\dim(A)<\infty$.
\end{proof}

The following lemma will be instrumental in showing that various joint continuity conditions hold automatically in $\Hilb$.
A similar statement can be found in \cite[Sec.\,2]{Khalil:2010tensor-semigroups}.

\begin{lemma}
Let $\Hc_i\in\Hilb$ ($i=1,2$).
Let $X$ be a subset of a 
	finite dimensional
normed vector space (e.g.\ $X = \Rb_{\ge 0}$). Equip $X$ with the induced topology  and
let $a\mapsto S^{(i)}_a$ be two
continuous maps $X\to\Bc(\Hc_i)$.
Then the map $X^2\to\Bc(\Hc_1\otimes\Hc_2)$, 
$(a,b)\mapsto S^{(1)}_a\otimes S^{(2)}_b$ is jointly
continuous.
\label{lem:semigrp}
\end{lemma}
\begin{proof}
We will first show that the map $a\mapsto\norm{S^{(i)}_a}$ is bounded on compact subsets of $X$.
	Let $K\subset X$ be compact.
	By strong continuity we have that for every $h\in \Hc_i$ the map $a\mapsto S^{(i)}_a(h)$ is continuous,
	so in particular the map $a\mapsto \norm{S^{(i)}_a(h)}$ is continuous, hence bounded on $K$.
	By the Uniform Boundedness Principle \cite[Ch.\,III.14]{Conway:1994fa} 
	the map $a\mapsto \norm{S^{(i)}_a}$ is bounded on $K$.

	\medskip

	Now we turn to the claim in the lemma.
	Let $a_0,b_0\in X$ and $\kappa,\eps\in\Rb_{>0}$ be fixed. 
	We will show that the map $(a,b)\mapsto S^{(1)}_a\otimes S^{(2)}_b$ is continuous at $(a_0,b_0)$.
	
		For $T\in\Hc_1\otimes\Hc_2$ take a sequence 
	$\left\{ T_n \right\}_{n}$ in the algebraic tensor product of 
	$\Hc_1$ and $\Hc_2$ such that $T_n\xrightarrow{n\to\infty} T$. 
	We have the estimate
	\begin{align}
		\begin{aligned}
		&\norm{(S^{(1)}_a\otimes S^{(2)}_b-S^{(1)}_{a_0}\otimes S^{(2)}_{b_0})T}\\
		\le& \norm{(S^{(1)}_a\otimes S^{(2)}_b-S^{(1)}_{a_0}\otimes S^{(2)}_{b_0})}\cdot
		\norm{T-T_n}+
		\norm{(S^{(1)}_a\otimes S^{(2)}_b-S^{(1)}_{a_0}\otimes S^{(2)}_{b_0})T_n}
		\ .
		\end{aligned}
		\label{eq:jcont1}
	\end{align}

	We give an estimate for the first term on the rhs of \eqref{eq:jcont1}.
Fix some $\delta_1>0$. Then by the above boundedness result there is a $\kappa>0$ such that
for every $a,b\in X$ with $|a-a_0|+|b-b_0|<\delta_1$ we have
	\begin{align}
		\norm{S^{(1)}_{a}}<\kappa\quad\text{ and }\quad \norm{S^{(2)}_{b}}<\kappa \ .
	\end{align}
	So we have
	\begin{align}
		\begin{aligned}
		\norm{S^{(1)}_a\otimes S^{(2)}_b-S^{(1)}_{a_0}\otimes S^{(2)}_{b_0}}
		\le &\norm{S^{(1)}_a}\cdot\norm{S^{(2)}_b}+\norm{S^{(1)}_{a_0}}\cdot\norm{S^{(2)}_{b_0}}\\
		\le &\ \kappa^2 +\norm{S^{(1)}_{a_0}}\cdot\norm{S^{(2)}_{b_0}}=:N_{a_0,b_0}^{\kappa}\ .
		\end{aligned}
		\label{eq:op-norm-estimate}
	\end{align}
	Since $T_n\xrightarrow{n\to\infty} T$, we can choose $n$ (which we keep fixed from now on) such that
	\begin{align}
		\norm{T-T_n}<\frac{\varepsilon}{2N_{a_0,b_0}^{\kappa}}\ .
		\label{eq:Tn-norm-estimate}
	\end{align}
	Putting \eqref{eq:op-norm-estimate} and \eqref{eq:Tn-norm-estimate} together we get
	\begin{align}
		\norm{(S^{(1)}_a\otimes S^{(2)}_b-S^{(1)}_{a_0}\otimes S^{(2)}_{b_0})}\cdot \norm{T-T_n}
		\le \frac{\varepsilon}{2}\ .
		\label{eq:middle-term-estimate}
	\end{align}

	We give an estimate for the last term in \eqref{eq:jcont1}.
	Recall that each $T_n$ was chosen in the algebraic tensor product of $\Hc_1$ and $\Hc_2$.
	Thus $T_n$ is a finite sum of elementary tensors, 
	\begin{align}
		T_n=\sum_{j=1}^{t_n}x_n^j\otimes y_n^j
		\label{eq:Tn-sum}
	\end{align}
	for $t_n\in\Zb_{\ge1}$, $x_n^j\in\Hc^{(1)}$ and $y_n^j\in\Hc^{(2)}$.
	Using this, we get:
	\begin{align}
		\begin{aligned}
		&\norm{(S^{(1)}_a-S^{(1)}_{a_0})\otimes S^{(2)}_b T_n +S^{(1)}_{a_0}\otimes (S^{(2)}_b-S^{(2)}_{b_0})T_n}\\
		\le&\sum_{j=1}^{t_n}\left( \norm{(S^{(1)}_a-S^{(1)}_{a_0})x_n^j}\cdot\norm{S^{(2)}_b}\cdot\norm{y_n^j}
		+\norm{(S^{(1)}_{a_0}}\cdot\norm{x_n^j}\cdot\norm{(S^{(2)}_b-S^{(2)}_{b_0})y_n^j}\right)\ .
		\end{aligned}
		\label{eq:last-term-estimate-1}
	\end{align}

	By strong continuity of $a\mapsto S^{(i)}_a$ we can chose $\delta_2 >0$ such that
	for every $a,b\in X$ with $|a-a_0|+|b-b_0|<\delta_2$ we have
	\begin{align}
		\norm{(S^{(1)}_a-S^{(1)}_{a_0})x_n^j}<\frac{\varepsilon}{4t_n \kappa \norm{y_n^j}} \quad\text{ and }\quad
		\norm{(S^{(2)}_b-S^{(2)}_{b_0})y_n^j}<\frac{\varepsilon}{4t_n \norm{S^{(1)}_{a_0}}\cdot \norm{x_n^j}}\ , 
	\end{align}
	for every $j=1,\dots,t_n$, since these are only finitely many conditions to satisfy. 
	Let $\delta:=\min\left\{ \delta_1,\delta_2 \right\}$. Then for every $a,b\in X$ with $|a-a_0|+|b-b_0|<\delta$ we have that 
	\begin{align}
		\begin{aligned}
			\norm{(S^{(1)}_a\otimes S^{(2)}_b-S^{(1)}_{a_0}\otimes S^{(2)}_{b_0})T_n} \le\frac{\varepsilon}{2}\ .
		\end{aligned}
		\label{eq:eq:last-term-estimate-2}
	\end{align}

	Finally, using \eqref{eq:middle-term-estimate} and \eqref{eq:eq:last-term-estimate-2} we have that
	\begin{align}
		\norm{(S^{(1)}_a\otimes S^{(2)}_b-S^{(1)}_{a_0}\otimes S^{(2)}_{b_0})T}<\varepsilon\ .
		\label{eq:final-estimate}
	\end{align}
	
\end{proof}

By iterating the previous lemma we see that
the definition of a regularised algebra simplifies in $\Hilb$. 
Namely it is enough to check that $a \mapsto P_a$ is continuous, 
rather than having to consider multiple tensor products.

\begin{corollary}\label{cor:semigrp}
	The continuity condition \eqref{eq:racont} in $\Hilb$
	is automatically satisfied for any $n\ge 2$ if it holds for $n=1$.
\end{corollary}

\subsection{Examples of regularised algebras and RFAs in $\Vectfd$ and $\Hilb$}\label{sec:examples}

Let $\Vectfd$ denote the symmetric monoidal category of
finite dimensional complex vector spaces with the usual tensor product of 
vector spaces and the topology on the hom-sets induced by any norm on the vector spaces.
In the following we list examples of regularised algebras and RFAs in $\Vectfd$ and $\Hilb$.
\begin{enumerate}
	\item Let $A$ be a algebra in $\Vectfd$ 
		with multiplication $\mu$ and unit $\eta$ and set
		$\mu_a:=\mu\cdot e^{a\sigma}$, $\eta_a:=\eta\cdot e^{a\sigma}$ 
		for some $\sigma\in\Cb$.  Then $A$ is a regularised algebra. 
		One can similarly obtain an RFA from a Frobenius algebra.
		\label{ex:trivial}
\end{enumerate}

A Frobenius algebra in $\Vectfd$ is always finite-dimensional. In Example~\ref{ex:trivial} we equipped them with an RFA structure for which all $a \to 0$ limits exist. The converse also holds in the following
sense.

\begin{proposition}\label{prop:rfafindim}
	Let $A\in\RFrob{\Hilb}$. The following are equivalent.
	\begin{enumerate}
		\item $A$ is finite dimensional.
		\item 
	All of the following limits exist:
	\begin{center}
	\begin{tabular}[h]{c c c c}
		$\lim_{a\to0}\eta_a$ , & $\lim_{a\to0}\mu_a$ , & 
		$\lim_{a\to0}\eps_a$ , & $\lim_{a\to0}\Delta_a$ .\\
	\end{tabular}
	\end{center}
	\end{enumerate}
\end{proposition}
\begin{proof}
	($1\Rightarrow2$): If $A$ is finite dimensional, then the map $a\mapsto P_a$
	is norm continuous, hence $P_a=e^{aH}$ for some $H\in\Bc(A)$.
	Then $\eta_0:=e^{-aH}\eta_a$ is independent of $a$ and $\eta_a=P_a \circ\eta_0$,
	hence $\lim_{a\to0}\eta_a=\eta_0$ exists. One similarly proves that
	the other limits exist as well.

	($1\Leftarrow2$): The morphisms given by these limits
	define a Frobenius algebra structure on $A$, hence $A$ is finite dimensional.
\end{proof}

\begin{enumerate}[resume]
	\item Consider the polynomial algebra $\Cb[x]$ and complete it
		with the Hilbert space structure
		given by $\langle x^n,x^m\rangle=\delta_{n,m}f(m)$ for some monotonously decreasing
		function $f: \Nb \to (0,1]\subset\Rb$ and denote by $\overline{\Cb[x]}$ its Hilbert space completion.			
		Let $P_a(x^n):=e^{a\sigma x}x^n$
	for $\sigma\in\Rb$
	(note the $x$ in the exponent).
	We now show that this defines a bounded operator. 
	Let $y\in\overline{\Cb[x]}$ with $y=\sum_{n\in\Nb}y_n x^n$.
	Then 
	\begin{align*}
		\norm{P_a(y)}^2=\sum_{n,m\in\Nb}\left( \frac{a^m}{m!} \right)^2 f(n+m)|y_n|^2\le
		\sum_{n,m\in\Nb}\frac{a^{2m}}{(2m)!} f(n)
		|y_n|^2\le e^a\norm{y}^2\ .
	\end{align*}
	where we used that $f$ is monotonously decreasing.

		Let us assume that 
		\begin{align}
                        \sup_{k\in\Nb}\left\{ \sum_{l=0}^{k}f(k)f(l)f(k-l) \right\}<\infty 
                        \label{eq:supcond}
                \end{align}
                holds, e.g.\ $f(m)=(1+m)^{-2}$ or $f(m)=e^{-m}$. 
                Then the operator
		\begin{align}
			\begin{aligned}
				M:\overline{\Cb[x]}&\to\overline{\Cb[x]}\otimes\overline{\Cb[x]}\\
				x^k&\mapsto \sum_{l=0}^{k} f(k) x^{k-l}\otimes x^l
			\end{aligned}
			\label{eq:mu-adjoint}
		\end{align}
		is bounded.  The adjoint of $M$ is the standard multiplication
		\begin{align}
			\begin{aligned}
				\mu:\overline{\Cb[x]}\otimes\overline{\Cb[x]}&\to\overline{\Cb[x]}\\
				x^k\otimes x^l&\mapsto x^{k+l}\ ,
			\end{aligned}
			\label{eq:mu-mult}
		\end{align}
		which is therefore also a bounded operator.
		Then defining $\mu_a:=P_a\circ\mu$ and $\eta_a(1):=P_a(1)$
		gives a regularised algebra in $\Hilb$.
		\label{ex:infdim-nonherm}
		Note, however, that this regularised algebra cannot be turned into
		a regularised Frobenius algebra because $P_a$ is not trace class,
		cf. Lemma~\ref{lem:patrclass}.

	\item Consider the Frobenius algebra
		$A:=\Cb[x]/\langle x^d\rangle$ in $\Vectfd$ with $\eps(x^k)=\delta_{k,d-1}$. Let $h\in A$ and define $P_a(f):=e^{ah}f$, 
		$\eps_a:=\eps\circ P_a$, $\eta_a:=P_a\circ\eta$ and $\mu_a:=P_a\circ\mu$.
		Then $\Cb[x]/\langle x^d\rangle$ is an RFA, denoted $A_{h}$.
		Unless $d=1$, this RFA is not separable.
		\label{ex:cxd}
\end{enumerate}
\begin{proposition}\label{prop:directsumrfa}
	Let $I$ be a countable (possibly infinite) set.
	For $k\in I$ let $F_k\in\Hilb$ be a (possibly infinite-dimensional) RFA. 
	Then $\bigoplus_{k\in I}F_k$ (the completed direct sum of Hilbert spaces) 
	is an RFA in $\Hilb$ if and only if, for every $a\in\Rb_{>0}$,
	\begin{align}
		\sup_{k\in I}\norm{\mu_a^k}<\infty\quad&\text{and}\quad \sup_{k\in I}\norm{\Delta_a^k}<\infty\ ,\\
		\sum_{k\in I}\norm{\eps_a^k}^2<\infty\quad&\text{and}\quad \sum_{k\in I}\norm{\eta_a^k}^2<\infty \ ,
	\end{align}
	where $\mu_a^k$, $\Delta_a^k$, $\eps_a^k$ and $\eta_a^k$
	denote the structure maps of $F_k$. 
\end{proposition}
\begin{proof}
	Let $F:=\bigoplus_{k\in I}F_k$ and fix the value of $a$.

\smallskip

\noindent
	($\Rightarrow$): 
	Let us write $x_k$ for the $k$'th component of $x\in F=\bigoplus_{k\in I}F_k$.
	Then for every $k\in I$
	\begin{align*}
		\norm{\Delta_a^k}=\sup_{\substack{x_k\in F_k\\ \norm{x_k}=1}}\norm{\Delta_a^k(x_k)}=
		\sup_{\substack{x_k\in F_k\\ \norm{x_k}=1}}\norm{\Delta_a(x_k)}\le
		\sup_{\substack{x_k\in F_k\\ \norm{x_k}=1}}\norm{\Delta_a}\cdot
	\norm{x_k}
	=\norm{\Delta_a}<\infty\ ,
	\end{align*}
	so in particular $\sup_{k}\norm{\Delta_a^k}<\infty$.
	A similar proof applies to the case of $\mu_a$.
	We calculate the norm of $\eta_a$:
	\begin{align*}
		\norm{\eta_a}^2=\norm{\eta_a(1)}^2=\sum_{k\in I}\norm{\eta_a^k(1)}^2=\sum_{k\in I}\norm{\eta_a^k}^2\ ,
	\end{align*}
	which is finite if and only if $\eta_a$ is a bounded operator.
	If $\eps_a$ is bounded, then by the Riesz Lemma there exists a unique $v\in F$ such that
	$\eps_a(x)=\langle v,x\rangle$ and $\norm{\eps_a}=\norm{v}$.
	Then $\langle v_k,x_k\rangle=\langle v,x_k\rangle=\eps_a(x_k)=\eps_a^k(x_k)$.
	So again by the Riesz Lemma $\norm{\eps_a^k}=\norm{v_k}$.
	We have that
	\begin{align*}
		\norm{\eps_a}^2=\norm{v}^2=\sum_{k\in I}\norm{v_k}^2=\sum_{k\in I}\norm{\eps_a^k}^2\ .
	\end{align*}

\smallskip

\noindent
	($\Leftarrow$): 
	The operators $\eta_a$ and $\eps_a$ are bounded by the previous discussion.
	For $\Delta_a$ one has that
	\begin{align*}
		\norm{\Delta_a}^2 &=
		\sup_{\substack{x\in F\\ \norm{x}=1}}\norm{\Delta_a(x)}^2=
		\sup_{\substack{x\in F\\ \norm{x}=1}}\norm{\sum_{k\in I}\Delta_a(x_k)}^2=
		\sup_{\substack{x\in F\\ \norm{x}=1}}\sum_{k\in I}\norm{\Delta^k_a(x_k)}^2\\&\le
		\sup_{\substack{x\in F\\ \norm{x}=1}}\sum_{k\in I}\norm{\Delta^k_a}^2\norm{x_k}^2
		\le
		\left(\sup_{l}\norm{\Delta_a^l}^2\right)\cdot
		\sup_{\substack{x\in F\\ \norm{x}=1}}\sum_{k\in I}\norm{x_k}^2=
		\sup_{l}\norm{\Delta_a^l}^2<\infty\ ,
	\end{align*}
	so $\Delta_a$ is bounded. For $\mu_a$ the proof is similar.

	Then one needs to check that $a\mapsto P_a:=\sum_{k\in I}P_a^k$ is continuous.
	Let $\eps\in\Rb_{>0}$, $a_0\in\Rb_{\ge0}$ and $f\in\bigoplus_{k\in I}F_k$ 
	with components $f_k$ be fixed.
	Let $a'>a_0$ and $0<E<\eps$ be arbitrary. Since $P_a-P_{a_0}$ is a bounded operator,
	one can find 
	$J_{a'}\subset I$ finite, 
	such that for every $a<a'$
	\begin{align*}
		\sum_{j\in I\setminus J_{a'}}\norm{(P_a^j-P_{a_0}^j)f_j}^2<E\ .
	\end{align*}
	Then let $\delta'>0$ be such that for every $|a-a_0|<\delta'$
	\begin{align*}
		\sum_{j\in J_{a'}}\norm{(P_a^j-P_{a_0}^j)f_j}^2<\eps-E\ ,
	\end{align*}
	which can be chosen since
	the sum is finite and each $P_a^j$ is continuous by assumption.
	Finally let $\delta:=\mathrm{min}\left\{ \delta',a'-a_0 \right\}$.
	By construction we have that for every $|a-a_0|<\delta$,
	\begin{align*}
		\norm{(P_a-P_{a_0})f}^2=\sum_{j\in I}\norm{(P_a^j-P_{a_0}^j)f_j}^2<\eps\ .
	\end{align*}
\end{proof}
	All examples of RFAs known to us are of the above form.
	For Hermitian RFAs, which we will introduce in Section~\ref{sec:dagger}, we can show that they are necessarily of the above form.
	Note that the same RFA
	$F_k$ cannot appear infinitely many times.

	Now we continue our list of examples with some special cases.
\begin{enumerate}[resume]
	\item Let $\left( \epsilon_k,\sigma_k \right)_{k\in I}$ be a countable family of
		pairs of complex numbers such that for all $a>0$
		\begin{align}
			\sup_{k\in I}\left|\epsilon_k e^{-a\sigma_k}\right|<\infty \quad \text{and}\quad
			\sum_{k\in I}\left|\frac{e^{-a\sigma_k}}{\epsilon_k}\right|^2<\infty\ .
			\label{eq:conditions-A-epsilon-example}
		\end{align}
		Then $A_{\epsilon,\sigma}:=\bigoplus_{k\in I}\Cb f_k$, 
		the Hilbert space generated by orthonormal vectors $f_k$, becomes an RFA
		by Proposition~\ref{prop:directsumrfa} via
		\begin{align}
			\mu_a(f_k\otimes f_j)&:=\delta_{k,j}\epsilon_k f_k e^{-a \sigma_k}\ ,&
			\eta_a(1)&:=\sum_{k\in I}\frac{f_k}{\epsilon_k}e^{-a\sigma_k}\ ,\\
			\Delta_a(f_k)&:=\frac{f_k\otimes f_k}{\epsilon_k}e^{-a\sigma_k}\ ,&
			\eps_a(f_k)&:=\epsilon_k e^{-a\sigma_k}\ .
		\end{align}
		This RFA is strongly separable (with $\tau_a=\eta_a$) and commutative. 
		\label{ex:dsep1dim}
	\item \label{ex:noninv}
		Let $I:=\Zb_{>0}$ and consider the one dimensional Hilbert spaces
		$\Cb f_k$ and $\Cb g_k$ 
		with $\norm{f_k}^2 = k^2$ and $\norm{g_k}^2 = k^{-1}$. 
		Let $F:=\bigoplus_{k=1}^{\infty}\Cb f_k$ and $G:=\bigoplus_{k=1}^{\infty}\Cb g_k$ 
		be the Hilbert space direct sums, so that
		\begin{align}
			\langle f_k,f_j\rangle_{F}=\delta_{k,j}k^2\quad\text{and}\quad
			\langle g_k,g_j\rangle_{G}=\delta_{k,j}k^{-1}\ .
		\end{align}
Define the maps
		\begin{align}
			\begin{aligned}
				\mu^F_a(f_k\otimes f_j):=&\delta_{k,j}e^{-ak^2}f_k\ ,&\eta^F_a(1):=&\sum_{k=1}^{\infty}e^{-ak^2}f_k\ ,\\
				\Delta^F_a(f_k):=&e^{-ak^2}f_k\otimes f_k\ ,&\eps^F_a(f_k):=&e^{-ak^2}\ ,
			\end{aligned}
			\label{eq:ex:noninv:structure-maps}
		\end{align}
		and similarly for $G$ by changing $f_k$ to $g_k$.
		These formulas define strongly separable (with $\tau_a=\eta_a$) 
		commutative RFAs by the previous example
		with $(\epsilon_k,\sigma_k)=(k^{-1},k^2)$ for $F$
		and with $(\epsilon_k,\sigma_k)=(k,k^2)$ for $G$. 
		Note that $\lim_{a\to0}\mu^F_a$ exists and has norm 1,
		but $\lim_{a\to0}\mu^G_a$ does not: the set 
		$\setc*{\norm{\mu^G_0(g_k\otimes g_k)}/\norm{g_k\otimes g_k}=k}{k\in\Zb_{>0}}$
		is not bounded.

	Define the morphism of RFAs $\psi:F\to G$ as
	\begin{align*}
		\psi(f_k)=g_k \quad \text{ for } ~~k=1,2,\dots \ .
	\end{align*}
	It is an operator with $\norm{\psi}=1$ and is mono and epi,
	but it does not have a bounded inverse,
	as the set $\setc*{\norm{\psi^{-1}(g_k)}/\norm{g_k}=k^{2}}{k\in\Zb_{>0}}$ is not bounded.
	This is an example illustrating that the category $\Hilb$ is not abelian: 
	a morphism can be mono and epi without being invertible.
	The example also shows that RFA morphisms which are mono and epi need
	not preserve the existence of zero-area limits. Isomorphisms, on the
	other hand, being continuous with continuous inverse, do preserve the existence of limits.

	\item Consider $L^2(G)$, the Hilbert space of square integrable functions
		on a compact semisimple Lie group $G$ with the following morphisms:
\begin{align}
	\begin{aligned}
		\eta_a(1)&:=\sum_{V\in \hat{G}}e^{-a\sigma_V}\dim(V) \chi_V\ ,\quad
		\mu(F)(x):=\int_G F(y,y^{-1}x)dy\ ,\\
		P_a(f)&:=\mu(\eta_a(1)\otimes f)\ ,\quad \mu_a:=P_a\circ\mu\ ,\\
		\eps_a(f)&:=\int_G \eta_a(1)(x)f(x^{-1})dx\ ,\quad
		\Delta(f)(x,y):=f(xy)\ ,\quad
		\Delta_a:=\Delta\circ P_a\ ,
	\end{aligned}
	\label{eq:L2G-structure-morphisms}
\end{align}
		where $f\in L^2(G)$, $F\in  L^2(G\times G)\cong L^2(G)\otimes  L^2(G)$, 
		$\hat{G}$ is a set of representatives 
		of isomorphism classes of finite dimensional 
		simple unitary $G$-modules,
		$\sigma_V$ is the value of the Casimir operator 
		of the Lie algebra of $G$ in the simple module $V$,
		$\chi_V$ is the character of $V$,
		and $\int_G$ denotes the Haar integral on $G$.
		These formulas define a strongly separable 
		RFA in $\Hilb$ (with $\tau_a=\eta_a$), 
		for details see Section~\ref{sec:twoRFAsfromG}.
		\label{ex:2dym:l2}

	\item The centre of the previous RFA is $Cl^2(G)$, the Hilbert space
		of square integrable class functions on $G$, with multiplication,
		unit and counit given by the same formulas, 
		but with the following coproduct:
		\begin{align}
			\Delta_a(f)=\sum_{V\in \hat{G}}e^{-a\sigma_V}
			\left( \dim(V) \right)^{-1} \chi_V\otimes\chi_V f_V\ ,
			\label{eq:Cl2G-coproduct}
		\end{align}
		where $f=\sum_{V\in \hat{G}}f_V \chi_V\in Cl^2(G)$.
		This is a strongly separable RFA in $\Hilb$ (with $\tau_a(1)=\sum_{V\in \hat{G}}e^{-a\sigma_V}
		\left( \dim(V) \right)^{-1}
		\chi_V$ and 
		$\tau^{-1}_a(1)=\sum_{V\in \hat{G}}e^{-a\sigma_V}
		\left( \dim(V) \right)^{3}
		\chi_V$). 
		For more details see Section~\ref{sec:twoRFAsfromG}.
		\label{ex:2dym:cl2}
\end{enumerate}

\subsection{Tensor products of RFAs and finite-dimensional RFAs}

We denote the \textsl{category of regularised algebras in $\Sc$} by $\RAlg{\Sc}$ and
the \textsl{category of RFAs in $\Sc$} by $\RFrob{\Sc}$. 
In this section we investigate under which conditions one can endow these categories with a monoidal structure. Then we describe the case $\Sc=\Vectfd$ in detail.

\begin{proposition}\label{prop:rfamor}
	Any morphism of RFAs is mono and epi.
\end{proposition}
\begin{proof}
	Let $\varphi:A\to B$ be a morphism of RFAs and
        let $\psi_{a,b}:=(\id_A\otimes\beta_b^B)\circ
        (\id_A\otimes\varphi\otimes\id_B)\circ(\gamma_a^A\otimes\id_B)$.
        Then $\varphi\circ\psi_{a,b}=P_{a+b}^B$
        and $\psi_{a,b}\circ\varphi=P_{a+b}^A$.
        We show that $\varphi$ is epi, showing that it is mono is similar.
        Let $f,g\in\Sc(B,X)$ for an object $X$ such that $f\circ\varphi=g\circ\varphi$.
        After composing with $\psi_{a,b}$ from the right for $a,b\in\Rb_{>0}$ we get
        $f\circ P^B_{a+b}=g\circ P^B_{a+b}$. This last equation holds for every $a,b\in\Rb_{>0}$,
        so we can take the limit $a,b\to0$ to get $f=g$.
\end{proof}

\begin{remark}\label{rem:not-groupoid}
		As we saw in 
		Example~\ref{ex:noninv}, 
		not every morphism of RFAs in $\Hilb$ is invertible,
	hence $\RFrob{\Hilb}$ is \textsl{not} a groupoid.
\end{remark}

However we have the following:
\begin{corollary}\label{cor:fdrfagroupoid}
	The category $\RFrob{\Vectfd}$ is a groupoid.
\end{corollary}

\begin{proposition}\label{prop:rfa-tensorprod}
Assume that $\Sc$ has a symmetric structure $\sigma$ and that for 
$A,B\in\RFrob{\Sc}$ the assignments 
\begin{align}
	(a_1,\dots,a_n)\mapsto P_{a_1}^A\otimes P_{a_1}^B\otimes\dots\otimes P_{a_n}^A\otimes P_{a_n}^B
	\label{eq:rfamoncatcond}
\end{align}
are jointly continuous for every $n\ge1$. Then $A\otimes B$ is an RFAs by
\begin{align}
	\begin{aligned}
		\mu_{a}^{A\otimes B}&:=\left( \mu_a^A\otimes\mu_a^B \right)
		\circ \left( \id\otimes \sigma\otimes \id \right)\ ,
		&\eta_a^{A\otimes B}&:=\eta_a^A\otimes \eta_a^B\ ,\\
		\Delta_a^{A\otimes B}&:=\left( \id\otimes \sigma\otimes \id \right)\circ
                \left( \Delta_a^A\otimes\Delta_a^B \right)\ ,
                &\eps_a^{A\otimes B}&:=\eps_a^A\otimes \eps_a^B\ .
	\end{aligned}
	\label{eq:tensor-product-RFA}
\end{align}
\end{proposition}
\begin{proof}
	Checking the algebraic relations is straightforward.
	The continuity of the maps in \eqref{eq:rfamoncatcond}
	assures that the continuity condition holds for the tensor product.
\end{proof}
	If condition \eqref{eq:rfamoncatcond} holds for every pair
	$A,B\in\RFrob{\Sc}$ we can define a symmetric monoidal structure on $\RFrob{\Sc}$,
	where the symmetric structure is inherited from $\Sc$.
	The tensor unit is the trivial RFA.

\begin{proposition}\label{prop:rfob-hilb-vect-mon-cat}
	$\RFrob{\Hilb}$ and $\RFrob{\Vectfd}$ are symmetric monoidal categories
	with the above tensor product.
\end{proposition}
\begin{proof}
	In $\Vectfd$ the tensor product is continuous, 
	so there the statement is trivial.	
	In $\Hilb$ Corollary~\ref{cor:semigrp} assures that
	the condition \eqref{eq:rfamoncatcond} holds for every pair
        $A,B\in\RFrob{\Hilb}$.
\end{proof}

\subsubsection*{Finite dimensional regularised algebras and regularised Frobenius algebras}

In the rest of this section we classify finite dimensional regularised (Frobenius) algebras.
The forgetful functor from finite dimensional Hilbert spaces to $\Vectfd$
is an equivalence of categories, therefore in this subsection we will only consider 
regularised (Frobenius) algebras in the latter.

Denote with $\Algz{\Vectfd}$ the category with objects pairs $(F,H)$, where
$F\in\Vectfd$ is an algebra and $H\in Z(H)$ is an element in the centre of $F$,
and morphisms $\phi:(F,H)\to(F',H')$ such that 
$\phi:F\to F'$
is a morphism of algebras and $\phi(H)=H'$.
Analogously, 
denote with $\Frobz{\Vectfd}$ the category of pairs of
Frobenius algebras and elements in their centre.

We define a functor $\Dc:\RAlg{\Vectfd}\to\Algz{\Vectfd}$ as follows:
on objects as $\Dc(A):=(A,\frac{d}{da}\eta_a(1)\restriction_{a=0})$ and on morphisms as identity. 
The same definition also gives a functor $\Dc:\RFrob{\Vectfd}\to\Frobz{\Vectfd}$.

\begin{proposition}\label{prop:findimraclass}
	The functors $\Dc:\RAlg{\Vectfd}\to\Algz{\Vectfd}$ 
	and $\Dc:\RFrob{\Vectfd}\to\Frobz{\Vectfd}$ are equivalences of categories.
\end{proposition}
\begin{proof}
	The inverse functor sends $(A,H)$ to 
	the regularised algebra $A$ with $P_a:=e^{aH}$, $\mu_a:=P_a\circ\mu$ and
	$\eta_a:=P_a\circ\eta$, where $\mu$ and $\eta$ are the multiplication and unit of $A$.
\end{proof}

\begin{remark}
	Let $(A,H)\in\Algz{Vect}$.  Then 
	\begin{align*}
		\Dc^{-1}(A,H)=\bigoplus_{\lambda\in \Sp H}\Dc^{-1}\left(A_{\lambda},Pr_{A_\lambda}(H)\right)
	\end{align*}
	as regularised algebras, where $A_{\lambda}$ denotes 
	the generalised eigenspace of $H$ 
	corresponding to the eigenvalue $\lambda$
	and $Pr_{A_\lambda}$ is the projection onto it.
	If $\Dc^{-1}(A,H)$ is furthermore an RFA then the above
	decomposition is valid as RFAs.
	\label{rem:fdRAdecomp}
\end{remark}

\subsection{Hermitian RFAs in $\Hilb$}\label{sec:dagger}

We start by recalling the notion of a dagger (or $\dagger$-) symmetric monoidal category $\Sc$, e.g.\ from \cite{Selinger:2007dc}.
A \textsl{dagger structure on $\Sc$} is a functor $(-)^{\dagger}:\Sc\to\Sc^{opp}$
which is identity on objects, $(-)^{\dagger\dagger}=\id_{\Sc}$,
$(f\otimes g)^{\dagger}=f^{\dagger}\otimes g^{\dagger}$ for any morphisms $f,g$ and
$\sigma_{U,V}^{\dagger}=\sigma_{V,U}$.

Let $\Sc$ be as in the beginning of Section~\ref{sec:RA-RFA-def} and
fix a $\dagger$-structure on $\Sc$. We do not require $(-)^\dagger$ to be continuous on hom-spaces, cf.\ Remark~\ref{rem:hilbcont}.

\begin{definition}\label{def:rfa:hermitean}
A \textsl{Hermitian regularised Frobenius algebra} (or $\dagger$-RFA for short) in $\Sc$ is an RFA in $\Sc$ for which
$\mu_a^{\dagger}=\Delta_a$ and $\eta_a^{\dagger}=\eps_a$ (and therefore $P_a=P_a^{\dagger}$).
We denote by $\dRFrob{\Sc}$ the full subcategory of $\RFrob{\Sc}$ with objects given by $\dagger$-RFAs.
\end{definition}

In the following we specialise to $\Sc=\Hilb$ with dagger structure given by the adjoint.
Note that $\dRFrob{\Hilb}$ is symmetric monoidal.

\begin{example}
Let us look at the examples from Section~\ref{sec:examples}.
In Example~\ref{ex:trivial}, if the Frobenius algebra $A\in\Hilb$ is a $\dagger$-Frobenius algebra (see e.g.\ \cite[Def.\,3.3]{Vicary:2008qa}) and if $\sigma\in\Rb$ then $P_a$ is self-adjoint and hence $A$ is a $\dagger$-RFA.
In Section~\ref{sec:twoRFAsfromG} we will show that the RFAs in Examples~\ref{ex:2dym:l2}~and~\ref{ex:2dym:cl2} are $\dagger$-RFAs.
The two RFAs in  Example~\ref{ex:noninv} are not $\dagger$-RFAs, as one can easily confirm that the summands $\Cb f_k$ and $\Cb g_k$ for $k>1$ are not $\dagger$-RFAs.
We compute e.g.\ for $\Cb f_k$ that
\begin{align}
	\langle f_k,\mu_a(f_k\otimes f_k)\rangle=e^{-ak^2}k^2\quad \text{ and }\quad \langle\Delta_a(f_k),f_k\otimes f_k\rangle=e^{-ak^2}k^4\ ,
	\label{eq:ex:noninv:notdagger}
\end{align}
so clearly, if $k>1$ then $\mu_a^{\dagger}\neq\Delta_a$. 
\end{example}

Let $\dFrobf{\Hilb}$ denote the category which has objects countable
families $\Phi=\left\{ F_j,\sigma_j \right\}_{j\in I}$ 
of $\dagger$-Frobenius algebras $F_j$
and real numbers $\sigma_j$, such that for every $a\in\Rb_{>0}$
\begin{align}
      \sup_{j\in I}\left\{ e^{-a\sigma_j}\norm{\mu_j} \right\}<\infty
      \quad\text{ and } \quad 
      \sum_{j\in I}e^{-2a\sigma_j}\norm{\eta_j}^2<\infty\ .
      \label{eq:drfaconvcond}
\end{align}
A morphism $\Psi:\Phi\to\Phi'$ consists of a bijection $f:I\xrightarrow{\sim}I'$ which satisfies $\sigma_j=\sigma_{f(j)}$ 
and a family of morphisms of Frobenius algebras $\psi_j:F_j\to F'_{f(j)}$
(which are automatically invertible \cite[Lem.\,2.4.5]{Kock:2004fa}).
We will write $\Psi=\left( f,\left\{ \psi_j \right\}_{j\in I} \right)$.

Let $\Phi\in\dFrobf{\Hilb}$ with the notation from above. Then by Proposition~\ref{prop:findimraclass}, $\Dc^{-1}(F_j,\sigma_j\id_{F_j})$ for $j\in I$ is an RFA.
Using Proposition~\ref{prop:directsumrfa}, we get an RFA structure on $\bigoplus_{j\in I}F_j$.
The next theorem shows that the resulting functor is an equivalence.

\begin{theorem} \label{thm:daggerclassification}
	There is an equivalence of categories $\dFrobf{\Hilb}\to\dRFrob{\Hilb}$
	given by $\Phi\mapsto \bigoplus_{j\in I}F_j$.\footnote{
	We would like to thank Andr\'e Henriques for explaining to us 
	this decomposition of $\dagger$-RFAs, or rather the corresponding decomposition of 
Hermitian area-dependent QFT's via Corollary~\ref{cor:hermaqftrfaequiv}.}
\end{theorem}
\begin{proof}
	We define the inverse functor.
	Let $F\in\dRFrob{\Hilb}$ and fix $a\in\Rb_{>0}$. 
	Then $P_a$ is self-adjoint and therefore can be diagonalised.
	Let $\sp_{\mathrm{pt}}(P_a)$ denote the point spectrum\footnote{
		The point spectrum of a bounded operator is the set of eigenvalues. Every compact operator on an infinite-dimensional Hilbert space has 0 in its spectrum, but it need not be an eigenvalue.} of $P_a$.
	Furthermore, by Lemma~\ref{lem:patrclass} $P_a$ is of trace class, and hence compact. 
	Thus it has at most countably many 
	eigenvalues and the eigenspaces with non-zero eigenvalues
	are finite dimensional. Let
	\begin{align}\label{eq:fdec}
		F=\bigoplus_{\alpha\in \sp_{\mathrm{pt}}(P_a) }F_{\alpha}
	\end{align}
	be the corresponding eigenspace decomposition 
	of $P_a$.
	
\medskip

\noindent	
\textsl{Claim:} The eigenvalue $\alpha$
	of $P_a$ on $F_{\alpha}$ 
	is of the form 
	$e^{-a \sigma_{\alpha}}$ for some $\sigma_{\alpha}\in\Rb$. 
	In particular 0 is not an eigenvalue.
\\
	To show this, first 
	assume that $c(a):=\alpha\neq0$, so that $F_{\alpha}$ is finite dimensional, and  
	simultaneously diagonalise $P_a$, $P_b$ and $P_{a+b}$ on $F_{\alpha}$.
	Then on a subspace where all three operators are constant with values $c(a)$, $c(b)$ and $c(a+b)$ one has that $c(a)c(b)=c(a+b)$. 
	Furthermore $a\mapsto c(a)$ is a continuous function $\Rb_{\ge0}\to\Rb$ and $c(0)=1$ since $a\mapsto P_a$ is strongly continuous at every $a\in\Rb_{\ge0}$ and $\lim_{a\to0}P_a=\id_F$. 
	So the unique solution to the above functional equation is $c(a)=e^{-a \sigma_{\alpha}}$ for some $\sigma_{\alpha}\in\Rb$.\\
	Finally let us assume that $\alpha=0$. Clearly, $\ker(P_{a})\subseteq \ker(P_{a+b})$ for every $b\in\Rb_{\ge0}$.
	Since $P_a$ is self adjoint, we have for $v\in F_0$ that $0=P_a(v)=P_{a/2}^{\dagger}\circ P_{a/2}(v)$.
	But then $P_{a/2}(v)=0$ and similarly, for every $n\in\Zb_{\ge0}$ we have that $P_{a/2^n}(v)=0$.
	Altogether we have that $F_{0}=\ker(P_{a})=\ker(P_{b})$ for every $b\in\Rb_{\ge0}$.
	So $\lim_{a\to0}P_a=\id_F$ implies that $F_0=\left\{ 0 \right\}$.

\medskip

\noindent	
	\textsl{Claim:} The eigenspaces are $\dagger$-Frobenius algebras by
	restricting and projecting the structure maps of $F$.
\\	
	To show this, first confirm that the structure maps do not mix
	eigenspaces of $P_a$, because $P_a$ commutes with them.
	Then checking $\dagger$-RFA relations is straightforward and 
	these are $\dagger$-Frobenius algebras,
	cf.\ Proposition~\ref{prop:rfafindim}.

\medskip

\noindent	
	\textsl{Claim:} The convergence conditions in \eqref{eq:drfaconvcond}
	are satisfied by the above obtained family of $\dagger$-Frobenius algebras
	$F_{\alpha}$ and real numbers 
	$\sigma_{\alpha}$.\\
This can be shown directly
	by computing the norm of the structure maps.

\medskip

	Showing that the two functors give an equivalence of categories is now straightforward.

\end{proof}

\begin{corollary}\label{cor:dagger-RFA-Pa-mono-epi}
	Let $A\in\Hilb$ be a $\dagger$-RFA. Then $P_a$ is mono and epi.
\end{corollary}

\begin{proof}
	From the proof of Theorem~\ref{thm:daggerclassification} we see that $P_a$ is mono.
	Since $P_a$ is self-adjoint we get that $P_a$ is epi.
\end{proof}

\begin{lemma}
	Every $\dagger$-Frobenius algebra in $\Hilb$ is semisimple.
	\label{lem:dfa:ssi}
\end{lemma}

\begin{proof}
Let $F$ denote a $\dagger$-Frobenius algebra in $\Hilb$ and let
$\zeta:=\mu\circ\Delta=\Delta^{*}\circ\Delta$, which is an $F$-$F$-bimodule morphism
and an $F$-$F$-bicomodule morphism. 
It is a self-adjoint operator, so it can be
diagonalised and $F$ decomposes into Hilbert spaces as
\begin{align}\label{eq:edec}
        F=\bigoplus_{\alpha\in\sp(\zeta)}F_{\alpha}\ ,
\end{align}
where $F_{\alpha}$ is the eigenspace of $\zeta$ with eigenvalue $\alpha$.

Now we show that \eqref{eq:edec} is a direct sum of Frobenius algebras.
Let $\alpha\ne\beta$ and take $a\in F_{\alpha}$, $b\in F_{\beta}$.
We have
\begin{align}
	\begin{aligned}
        \zeta(ab)=&a\zeta(b)={\beta}ab\\
        =&\zeta(a)b={\alpha}ab 
	\end{aligned}
        \label{eq:tfactalg}
\end{align}
since $\zeta$
is a bimodule morphism. Then \eqref{eq:tfactalg} shows that $ab=0$,
so \eqref{eq:edec} is a decomposition as algebras.

Similarly one shows that equation \ref{eq:edec} is a decomposition as coalgebras.
We have for $\forall a\in F_{\alpha}$, using Sweedler notation:
\begin{align}
	\begin{aligned}
        \Delta(\zeta(a))=&\zeta(a_{(1)})\otimes a_{(2)}=a_{(1)}\otimes \zeta(a_{(2)})\\
        =&\alpha\Delta(a)=\alpha a_{(1)} \otimes a_{(2)}\ ,
	\end{aligned}
        \label{eq:tfactcoalg}
\end{align}
which shows that the comultiplication restricted to $F_{\alpha}$ lands in
$F_{\alpha}\otimes F_{\alpha}$.

	We now show that $0$ is not in the spectrum. Let us assume otherwise. 
	Then $F_0$ is a Frobenius algebra.
We have $\zeta(x)=\Delta^*\circ\Delta(x)=0$ for every $x\in F_0$, 
and so also $\Delta(x)=0$, which is a contradiction to counitality. 
Therefore 0 is not in the spectrum of $\zeta$, i.e.\ $\zeta$ is injective.

Now the only thing left to show is that each summand $F_{\alpha}$ is semisimple.
Take $\Delta(1)\cdot \alpha^{-1}$ 
projected on $F_{\alpha}\otimes F_{\alpha}$.
This is a separability idempotent for the algebra $F_{\alpha}$, hence $F_{\alpha}$
is separable, hence semisimple.
\end{proof}

Let $\epsilon\in\Cb\setminus\left\{ 0 \right\}$, $\sigma\in\Rb$ and let $\Cb_{\epsilon,\sigma}$ denote the one dimensional 
$\dagger$-RFA structure on $\Cb$ given by
\begin{align}
	\begin{aligned}
		\eps_a(1)&=e^{-a\sigma}\epsilon\ , 
       		&\Delta_a(1)&=\frac{e^{-a\sigma}}{\epsilon}1\otimes 1\ , \\
		 \eta_a(1)&=e^{-a\sigma} \epsilon^* 1\ ,
       		& \mu_a(1\otimes 1)&=\frac{e^{-a\sigma}}{\epsilon^*}1\ . 
	\end{aligned}
	\label{eq:c-eps-sigma}
\end{align}
Let $C\in\Hilb$ be a 1-dimensional $\dagger$-RFA and $c\in C$.
	Then by Proposition~\ref{prop:findimraclass}, $\varepsilon_a=\varepsilon_0\circ P_a$.
	Set $\epsilon:=\varepsilon_0(c)\in\Cb$ and $\sigma\in\Rb$ 
	to be such that $P_a(c)=e^{-a\sigma}c$.
	Then 
	\begin{align}
		\begin{aligned}
		C&\to\Cb_{\epsilon,\sigma}\\
		c&\mapsto 1
		\end{aligned}
		\label{eq:1d-d-rfa-mor}
	\end{align}
	is an isomorphism of RFAs.

\begin{corollary}
	Let $C$ be a commutative $\dagger$-RFA in $\Hilb$. 
	Then there is a family of numbers 
	$\left\{ \epsilon_j,\sigma_j \right\}_{j\in I}$, 
	where $\epsilon_j\in\Cb$ and $\sigma_j\in\Rb$, satisfying
	\begin{align}
		\sup_{j\in I}\left\{ e^{-a\sigma_j}|\epsilon_j|^{-1} \right\}<\infty
       		\quad\text {and }\quad 
       		\sum_{j\in I}e^{-2a\sigma_j}|\epsilon_j|^2<\infty
		\label{eq:cor:commdaggerclassification}
	\end{align}
	for every $a\in\Rb_{>0}$ such that 
	$C\cong\bigoplus_{j\in I}\Cb_{\epsilon_j,\sigma_j}$
	as RFAs.
	\label{cor:commdaggerclassification}
\end{corollary}
\begin{proof}
	By Theorem \ref{thm:daggerclassification} and Lemma~\ref{lem:dfa:ssi},
	$C$ is a direct sum of semisimple algebras.
	By the Wedderburn-Artin theorem every semisimple commutative algebra is 
	a direct sum of 1-dimensional algebras.
	Using  the isomorphism \eqref{eq:1d-d-rfa-mor} we get the above family of numbers.
	The finiteness conditions come from \eqref{eq:drfaconvcond}.
\end{proof}
\begin{remark} \label{rem:no-zero-area-limit}
	In some cases none of the structure maps of a commutative Hermitian RFA
	admit an $a\to0$ limit. A concrete example can be given as follows.
	Fix $1/2>\delta>0$. Then the family of numbers
	$\left\{ n^{1/2+\delta},n \right\}_{n\in\Zb_{>0}}$
	satisfies \eqref{eq:cor:commdaggerclassification} and
	the structure maps $\mu_a$, $\Delta_a$, $\eta_a$, $\eps_a$ of the corresponding 
	commutative $\dagger$-RFA from Corollary~\ref{cor:commdaggerclassification}
	do not have an $a\to0$ limit. 
\end{remark}

\begin{lemma}
	Let $\varphi:\Cb_{\epsilon,\sigma}\to\Cb_{\epsilon',\sigma'}$ be a morphism of 
	RFAs. Then $\varphi(1)=\epsilon/\epsilon'\in U(1)$ and $\sigma=\sigma'$.
	\label{lem:1drfamor}
\end{lemma}
\begin{proof}
	From $\varphi\circ\eta_a=\eta_a'$ one has that for every $a\in\Rb_{\ge0}$,
$\varphi(1)\epsilon^* e^{-a\sigma}= (\epsilon')^* e^{-a\sigma'}$.
	Since $\epsilon\neq0$, $\epsilon'\neq0$ and $\varphi(1)\neq0$,
	one must have $\sigma=\sigma'$ and hence $\varphi(1)\epsilon^*  = (\epsilon')^*$.
	One similarly obtains from $\eps'_a\circ\varphi=\eps_a$ that
$\epsilon' \varphi(1)=\epsilon$.
	Combining these we get that $|\varphi(1)|=1$ and that $\varphi(1)=\epsilon/\epsilon'$.
\end{proof}

\begin{proposition}\label{prop:commhermrfacat}
	Every morphism of commutative $\dagger$-RFAs in $\Hilb$ is unitary, in particular
	the category of commutative $\dagger$-RFAs in $\Hilb$ is a groupoid.
\end{proposition}
\begin{proof}
	Let $\phi:C\to C'$ be a morphism of commutative $\dagger$-RFAs.
	By Corollary~\ref{cor:commdaggerclassification} we assume that
	$C=\bigoplus_{j\in I}\Cb_{\epsilon_j,\sigma_j}$ and
	$C'=\bigoplus_{j\in I'}\Cb_{\epsilon'_j,\sigma'_j}$.
	By a similar argument as in the proof of Lemma~\ref{lem:1drfamor},
	we see that $\phi$ does not mix the $\Cb_{\epsilon_j,\sigma_j}$'s with different $\sigma$'s.
	Let $C_{\sigma}:=\bigoplus_{\substack{j\in I\\ \sigma_j=\sigma}}\Cb_{\epsilon_j,\sigma_j}$ and
	define $C'_{\sigma}$ similarly. These are both finite dimensional, since
	these are eigenspaces of the $P_a$'s with eigenvalue $e^{-a\sigma}$.
	Let $\varphi:=\phi\restriction_{C_{\sigma}}$.
	Then $\varphi$ is a morphism of finite dimensional RFAs so it is a bijection by 
	Corollary~\ref{cor:fdrfagroupoid}.
	Let us write $g_j=1$ $(j=1,n_{\sigma})$ for the generator of $\Cb_{\epsilon_j,\sigma_j}$ in $C_{\sigma}$
	and $g_j'=1$ $(j=1,n_{\sigma})$ for the generator of $\Cb_{\epsilon_j',\sigma_j'}$ in $C'_{\sigma}$
	and write $\varphi(g_j)=\sum_{k=1}^{n_\sigma}\varphi^{jk}g_k'$.

	From the equation $\varphi\circ\mu=\mu'\circ(\varphi\otimes\varphi)$ 
	one has for every $j,k,l$ that
	\begin{align*}
		\delta_{jk}(\epsilon_j^*)^{-1}\varphi^{jl}=
		\varphi^{jl}\varphi^{kl}\left( (\epsilon'_l)^* \right)^{-1}.
	\end{align*}
	\begin{itemize}
		\item If $j\neq k$ then $\varphi^{jl}\varphi^{kl}=0$ 
			for every such $k$ and for every $l$.
			This means that in the matrix $\varphi^{jl}$ 
			in every row there might be at most one nonzero element.
			Since $\varphi$ is bijective there is also at least one 
			nonzero element in every row in the latter matrix
			and the same holds for every column. 
			We conclude that the matrix of $\varphi$ is
			the product of a permutation matrix $\pi$ and a diagonal matrix $D$.

		\item If $j=k$ and if $\varphi^{jl}\neq 0$
			then $\varphi^{jl}=\left( \epsilon'_l/\epsilon_j \right)^*$,
			which give the nonzero elements of the diagonal matrix.
	\end{itemize}

	Now $\pi^{-1}\circ\varphi$ restricts to RFA morphisms of the 1-dimensional
	components, hence by Lemma~\ref{lem:1drfamor} the diagonal matrix $D$ is unitary.
	Therefore $\varphi$ is unitary,
	$\phi$ is the direct sum of unitary matrices so $\phi$ is unitary and in particular invertible.
\end{proof}

\subsection{Modules over regularised algebras}\label{sec:modules}

We define modules over a regularised algebra in such a way that the action map now depends on two real parameters. This may seem odd at first sight but is motivated by the application to area-dependent field theory later on, see Section~\ref{sec:daqft}.

\begin{definition}
A \textsl{left module} over a regularised algebra $A$ (or left $A$-module) in
$\Sc$ is an object $U\in\Sc$ together with a family of morphisms
\begin{align}
	\begin{aligned}
	\def\svgwidth{4cm}
\begingroup%
  \makeatletter%
  \providecommand\color[2][]{%
    \errmessage{(Inkscape) Color is used for the text in Inkscape, but the package 'color.sty' is not loaded}%
    \renewcommand\color[2][]{}%
  }%
  \providecommand\transparent[1]{%
    \errmessage{(Inkscape) Transparency is used (non-zero) for the text in Inkscape, but the package 'transparent.sty' is not loaded}%
    \renewcommand\transparent[1]{}%
  }%
  \providecommand\rotatebox[2]{#2}%
  \ifx\svgwidth\undefined%
    \setlength{\unitlength}{139.91796875bp}%
    \ifx\svgscale\undefined%
      \relax%
    \else%
      \setlength{\unitlength}{\unitlength * \real{\svgscale}}%
    \fi%
  \else%
    \setlength{\unitlength}{\svgwidth}%
  \fi%
  \global\let\svgwidth\undefined%
  \global\let\svgscale\undefined%
  \makeatother%
  \begin{picture}(1,0.57680002)%
    \put(0,0){\includegraphics[width=\unitlength]{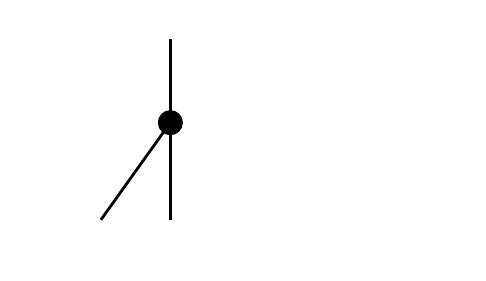}}%
    \put(0.32210838,0.01008403){\color[rgb]{0,0,0}\makebox(0,0)[lb]{\smash{$U$}}}%
    \put(0.38500237,0.35885982){\color[rgb]{0,0,0}\makebox(0,0)[lb]{\smash{\scriptsize
{$(a,l)$}}}}%
    \put(0.17916748,0.01008403){\color[rgb]{0,0,0}\makebox(0,0)[lb]{\smash{$A$}}}%
    \put(-0.00379687,0.29596583){\color[rgb]{0,0,0}\makebox(0,0)[lb]{\smash{$\rho_{a,l}=$}}}%
    \put(0.32210838,0.52467126){\color[rgb]{0,0,0}\makebox(0,0)[lb]{\smash{$U$}}}%
    \put(0.61370781,0.29596583){\color[rgb]{0,0,0}\makebox(0,0)[lb]{\smash{$\in\Sc(A\otimes U,U)$}}}%
  \end{picture}%
\endgroup%

	\end{aligned}
	\label{eq:modaction}
\end{align}
for every $a,l\in\Rb_{>0}$ called the \textsl{action}, such that they satisfy the following conditions.
\begin{enumerate}
	\item For every $a=a_1+a_2=b_1+b_2$ and $l=l_1+l_2$
\begin{align}
	\rho_{a_1,l_1}\circ\left( \id_A \otimes\rho_{a_2,l_2}\right)=
	\rho_{b_1,l}\circ\left( \mu_{b_2}\otimes\id_U \right)\ . 
	\label{eq:ra:module}
\end{align}
and
the morphisms
\be\label{eq:QU-leftmod-def}
Q^U_{a,l}:=\rho_{a_1,l_1}\circ \left( \eta_{a_2,l_2}\otimes id_U \right)
\ee
satisfy $\lim_{a,l\to0}Q^U_{a,l}=\id_U$.
\item The assignment
\begin{align}
\begin{aligned}
	(\Rb_{>0}^{2} \cup \{0\})
	&\to\Sc(U,U)\\
	(a,l)&\mapsto Q^U_{a,l}
\end{aligned}\label{eq:ra-module-cont}
\end{align} 
is jointly continuous.
\end{enumerate}
One similarly defines right modules. 
\end{definition}
Note that the morphisms $Q^U_{a,l}$ form a semigroup,
\begin{align}
	Q^U_{a_1,l_1}\circ Q^U_{a_2,l_2}=Q^U_{a_1+a_2,l_1+l_2} ,
	\label{eq:ra:modsemigrp}
\end{align}
and we have a continuous semigroup homomorphism $\Rb_{\ge 0}^2 \to \Sc(U,U)$, $(a,l) \to Q^U_{a,l}$.

\begin{remark}\label{rem:bimodule-continuity}
	As in the case of regularised algebras, one would want to impose \eqref{eq:ra-module-cont} for $n$-fold tensor products for every $n\ge1$.
	However in Section~\ref{sec:lattice:bimoduledata} we will see that the natural condition would be to have this for a set of different modules,
	which would lead to the notion of ``sets of mutually jointly continuous modules'', which is cumbersome to define.
	Instead, we will impose this condition later in Section~\ref{sec:lattice:bimoduledata}.
	When considering regularised algebras and modules in $\Hilb$, this continuity condition will be automatic, 
	see Lemma~\ref{lem:semigrp}.
\end{remark}

The definition of bimodules in terms of left and right modules 
requires an extra continuity assumption, so we spell it out in detail:

\begin{definition}\label{def:AB-bimodule-def}
An \textsl{$A$-$B$-bimodule} over 
regularised algebras $A$ and $B$ is an object $U\in\Sc$ together with a family of morphisms
\begin{align}
	\begin{aligned}
	\def\svgwidth{5cm}
\begingroup%
  \makeatletter%
  \providecommand\color[2][]{%
    \errmessage{(Inkscape) Color is used for the text in Inkscape, but the package 'color.sty' is not loaded}%
    \renewcommand\color[2][]{}%
  }%
  \providecommand\transparent[1]{%
    \errmessage{(Inkscape) Transparency is used (non-zero) for the text in Inkscape, but the package 'transparent.sty' is not loaded}%
    \renewcommand\transparent[1]{}%
  }%
  \providecommand\rotatebox[2]{#2}%
  \ifx\svgwidth\undefined%
    \setlength{\unitlength}{162.64296875bp}%
    \ifx\svgscale\undefined%
      \relax%
    \else%
      \setlength{\unitlength}{\unitlength * \real{\svgscale}}%
    \fi%
  \else%
    \setlength{\unitlength}{\svgwidth}%
  \fi%
  \global\let\svgwidth\undefined%
  \global\let\svgscale\undefined%
  \makeatother%
  \begin{picture}(1,0.49620766)%
    \put(0,0){\includegraphics[width=\unitlength]{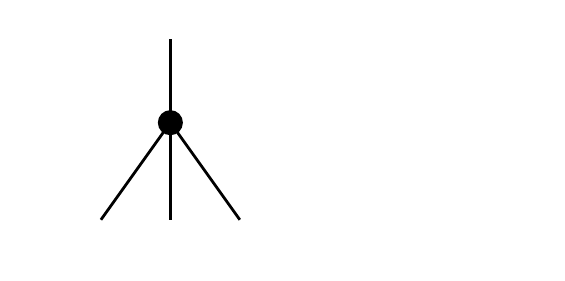}}%
    \put(0.27710236,0.00867506){\color[rgb]{0,0,0}\makebox(0,0)[lb]{\smash{$U$}}}%
    \put(0.3213711,0.30871877){\color[rgb]{0,0,0}\makebox(0,0)[lb]{\smash{\scriptsize
{$(a,l,b)$}}}}%
    \put(0.15413362,0.00867506){\color[rgb]{0,0,0}\makebox(0,0)[lb]{\smash{$A$}}}%
    \put(-0.00326636,0.25461253){\color[rgb]{0,0,0}\makebox(0,0)[lb]{\smash{$\rho_{a,l,b}=$}}}%
    \put(0.27710236,0.4513625){\color[rgb]{0,0,0}\makebox(0,0)[lb]{\smash{$U$}}}%
    \put(0.40007109,0.00867506){\color[rgb]{0,0,0}\makebox(0,0)[lb]{\smash{$B$}}}%
    \put(0.52795857,0.25461253){\color[rgb]{0,0,0}\makebox(0,0)[lb]{\smash{$\in\Sc(A\otimes U\otimes B,U)$}}}%
  \end{picture}%
\endgroup%

	\end{aligned}
	\label{eq:bimodaction}
\end{align}
for every $a,l,b\in\Rb_{>0}$ such that the following conditions hold.
\begin{enumerate}
	\item For every $a=a_1+a_2=a_1'+a_2'$, $b=b_1+b_2=b_1'+b_2'$ and $l=l_1+l_2$
\begin{align}
	\rho_{a_1,l_1,b_1}\circ\left( \id_A \otimes\rho_{a_2,l_2,b_2}\otimes \id_B\right)=
	\rho_{a_1',l,b_1'}\circ\left( \mu^A_{a_2'}\otimes\id_U\otimes\mu^B_{b_2'} \right)\ ,
	\label{eq:ra:bimodule}
\end{align}
the morphisms $Q^U_{a,l,b}:=\rho_{a_1,l,b_1}\circ \left( \eta_{a_2}^A\otimes id_U \otimes \eta_{b_2}^B\right)$ satisfies that
$\lim_{a,l,b\to0}Q^U_{a,l,b}=\id_U$ and
\item the map
\begin{align}
\begin{aligned}
	(\Rb_{> 0}^{3}\cup \left\{ 0 \right\})
	&\to\Sc(U,U)\\
	(a,l,b)&\mapsto Q^U_{a,l,b}
\end{aligned}\label{eq:ra-bimodule-cont}
\end{align} 
is jointly continuous.
\end{enumerate}
\end{definition}

A bimodule is called \textsl{transmissive}, 
if $\rho_{a,l,b}$ depends only on $a+b$, or, in other words, if $\rho_{a+u,l,a-u}$ is independent of $u$. As with the inclusion of the extra parameter $l$, the notion of transmissivity is motivated by the application to area-dependent quantum field theory, see Section~\ref{sec:daqft}.

\begin{remark}\label{rem:left-right-module-gives-bimodule}
Let $U$ be a left $A$-module and a right $B$-module such that 
the left and right actions $\rho_{a,l}^L$ and $\rho_{b,m}^R$ commute.
That is for every $a,b\in\Rb_{>0}$ and $l_1+l_2=m_1+m_2$
\begin{align}
	\rho_{a,l,b}:=\rho_{a,l_1}^L \circ \left( \rho_{b,l_2}^R\otimes \id_{B} \right)
	=\rho_{a,m_1}^R \circ \left( \id_{A}\otimes\rho_{b,m_2}^L \right)\ .
	\label{eq:ra:bimod}
\end{align}
If $\rho_{a,l,b}$ is jointly continuous in the parameters,
then $U$ is an $A$-$B$-bimodule. 
	Note that in contrast to the case of usual bimodules over associative algebras, 
	which are defined to be left- and right modules with commuting actions, 
	here we have to impose the extra condition of joint continuity.
\\
	Conversely, let $U$ be an $A$-$B$-bimodule with action $\rho_{a,l,b}$.
	If the limit 
		\begin{align}
			\rho_{a,l}^L:=\lim_{b\to0}\rho_{a,l,b_1}
			\circ(\id_{A\otimes U}\otimes \eta_{b_2}^{B})
			\label{eq:bimod-left-module-limit}
		\end{align}
		with $b=b_1+b_2$
	exists for every $a,l\in\Rb_{>0}$
	and remains jointly continuous in the limit,	 
	then $U$ becomes a left $A$-module with action $\rho_{a,l}^L$. 
	Similarly, if the analogous $a\to0$ limit 
	exists then $U$ becomes a right $B$-module.
	In Appendix~\ref{app:bimod} we give an example which illustrates that these limits do not always exist.
\end{remark}

\begin{example}\label{ex:twistedaction}
	Let $\Aut_{\RFrob{\Sc}}(A)$ denote the invertible morphisms 
	in $\RFrob{\Sc}(A,A)$. Then for $\alpha,\beta\in\Aut_{\RFrob{\Sc}}(A)$ we can define
	a transmissive bimodule structure ${}_{\alpha}A_{\beta}$ on $A$
	by twisting the multiplication from the two sides and letting the $l$-dependence be trivial.
	That is, for every $a,b,l\in\Rb_{>0}$ we define the action to be
	\begin{align}
		\rho_{a,l,b}:=\mu_a\circ(\id\otimes \mu_b)\circ(\alpha\otimes\id_A\otimes\beta)\ ,
		\label{eq:twistedaction}
	\end{align}
	which is jointly continuous in the parameters, since the composition in $\Sc$ is separately continuous,
	and since we can rewrite $\rho_{a,l,b} = P_c \circ \rho_{a',l,b'}$ with $a'+b'+c = a+b$.
	Note that $\beta^{-1}: {}_{\alpha}A_{\beta}\to {}_{\beta^{-1}\circ\alpha}A_{\id_A}$ is a bimodule isomorphism,
	so it is enough to consider twisting on one side.
\end{example}

The proof of the following proposition is similar to that of Proposition~\ref{prop:directsumrfa}.
\begin{proposition}
	Let $F=\bigoplus_{k\in I}F_k$ and $G=\bigoplus_{j\in J}G_j$ be RFAs in $\Hilb$ 
	as in Proposition~\ref{prop:directsumrfa}.
	Let $M_{kj}\in\Hilb$ be a 
$F_k$-$G_j$-bimodule with action $\rho_{a,l,b}^{M_{kj}}$ for $k\in I$ and $j\in J$. Then $M:=\bigoplus_{k\in I,\,j\in J}M_{kj}$ is a $F$-$G$-bimodule in $\Hilb$
	if and only if for every $a,l,b\in\Rb_{>0}$
	\begin{align}
		\sup_{k\in I,\,j\in J}\left\{ \norm{\rho_{a,l,b}^{M_{kj}}} \right\}<\infty\ .
		\label{eq:prop:directsum-bimodule}
	\end{align}
	\label{prop:directsum-bimodule}
\end{proposition}

A \textsl{morphism $U\xrightarrow{\phi}V$ 
of left modules over a regularised algebra $A$} is a morphism in
$\Sc$ which respects the action:
\begin{align}
	\phi\circ\rho_{a,l}^U=\rho_{a,l}^V\circ\left( \id_A\otimes\phi \right),
	\label{eq:ramodmorph}
\end{align}
for all $a,l\in\Rb_{>0}$. 
One similarly defines morphisms of right modules and bimodules.
Denote with $A\text{-}\Mod{\Sc}$ 
the \textsl{category of left modules over $A$ in $\Sc$}.

Recall from Proposition~\ref{prop:findimraclass} that
for a regularised algebra $A\in\Vectfd$ the pair
$\Dc(A)=(A,H)$ consists of the underlying algebra of $A$ and an element $H$ in its centre.
Let $A\text{-}\Modz{\Vectfd}$ denote the following category.
Its objects are pairs $(U,H_U)$, where $U$ is a left $A$-module in $\Vectfd$ and 
$H_U\in\End_A(U)$. 
Its morphisms are left $A$-module morphisms 
$\phi:U\to V$, such that $H_V\circ\phi=\phi\circ H_U$.

As in the case of regularised algebras in $\Vectfd$
(cf.\ Proposition~\ref{prop:rfafindim}),
the semigroup $(a,l)\mapsto Q^U_{a,l}$
is norm continuous and hence $Q^U_{a,l}=e^{a \cdot H_A+l \cdot H_U}$
for $H_A,H_U\in\End_A(U)$ such that $H_A\circ H_U=H_U\circ H_A$.

Let us define a functor 
$\Dc:A\text{-}\Mod{\Vectfd}\to A\text{-}\Modz{\Vectfd}$ 
as follows.  The $A$-module structure on $\Dc(U)$ is given by
$\rho^U=Q_{-a,-l}^U\circ\rho_{a,l}^U$ and 
$H_U$ is defined as above.
On morphisms $\Dc$ is the identity.

\begin{proposition}\label{prop:ramodule:findim}
	The functor $\Dc:A\text{-}\Mod{\Vectfd}\to A\text{-}\Modz{\Vectfd}$ is an equivalence of categories.
\end{proposition}
\begin{proof}
	The $A$-module structure on $\Dc^{-1}(U,H_U)$ is given by
	$\rho_{a,l}^U:=e^{aH_A+lH_U}\circ\rho^U$ 
	with $H_A=\rho^U(H\otimes -)\in\End_A(U)$,
	where $\rho^U$ is the action on $U$.
\end{proof}

\begin{remark}
	Let $A,B\in\Vectfd$ be regularised algebras and $U\in\Vectfd$ an $A$-$B$-bimodule $U\in\Vectfd$.
	As before we have $Q_{a,l,b}^{U}=e^{aH_A+lH_U+bH_B}$, where $H_A,H_U,H_B\in\End_{A,B}(U)$ are bimodule homomorphisms.
	Then $U$ is transmissive if and only if $H_A=H_B$.
\end{remark}

Let us assume now that $\Sc$ is symmetric.
We now introduce a notion of duals for bimodules.

\begin{definition}\label{def:dual-pair}
	Let $A,B\in\Sc$ be regularised algebras. 
	A \textsl{dual pair of bimodules} is an $A$-$B$-bimodule $U\in\Sc$ 
	and a $B$-$A$-bimodule $V\in\Sc$ together with families of morphisms
	for every $a,l,b\in\Rb_{>0}$
	\begin{align}
		\gamma_{a,l,b}&\in\Sc(\Ib,V\otimes U)\ ,&\beta_{a,l,b}&\in\Sc(U\otimes V,\Ib)\ ,
		\label{eq:ra-duality-morphisms}
	\end{align}
jointly continuous in the parameters,
	which we denote with
\begin{align}
	\begin{aligned}
	\def\svgwidth{6.5cm}
\begingroup%
  \makeatletter%
  \providecommand\color[2][]{%
    \errmessage{(Inkscape) Color is used for the text in Inkscape, but the package 'color.sty' is not loaded}%
    \renewcommand\color[2][]{}%
  }%
  \providecommand\transparent[1]{%
    \errmessage{(Inkscape) Transparency is used (non-zero) for the text in Inkscape, but the package 'transparent.sty' is not loaded}%
    \renewcommand\transparent[1]{}%
  }%
  \providecommand\rotatebox[2]{#2}%
  \ifx\svgwidth\undefined%
    \setlength{\unitlength}{229.74978027bp}%
    \ifx\svgscale\undefined%
      \relax%
    \else%
      \setlength{\unitlength}{\unitlength * \real{\svgscale}}%
    \fi%
  \else%
    \setlength{\unitlength}{\svgwidth}%
  \fi%
  \global\let\svgwidth\undefined%
  \global\let\svgscale\undefined%
  \makeatother%
  \begin{picture}(1,0.33743173)%
    \put(0,0){\includegraphics[width=\unitlength]{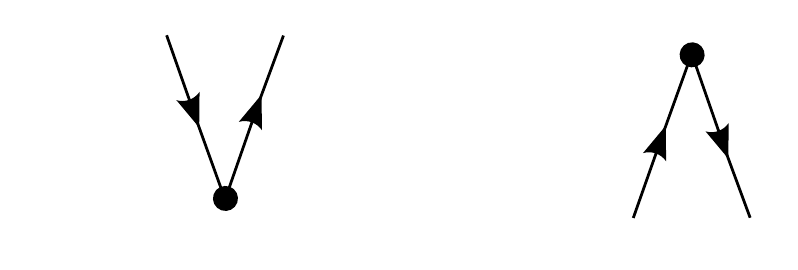}}%
    \put(0.90527643,0.2629486){\color[rgb]{0,0,0}\makebox(0,0)[lb]{\smash{\scriptsize
{$(a,l,b)$}}}}%
    \put(0.58267194,0.15039752){\color[rgb]{0,0,0}\makebox(0,0)[lb]{\smash{$\beta_{a,l,b}=$}}}%
    \put(0.76502285,0.00614119){\color[rgb]{0,0,0}\makebox(0,0)[lb]{\smash{$U$}}}%
    \put(0.9217151,0.00614119){\color[rgb]{0,0,0}\makebox(0,0)[lb]{\smash{$V$}}}%
    \put(0.32725629,0.0783325){\color[rgb]{0,0,0}\makebox(0,0)[lb]{\smash{\scriptsize{$(a,l,b)$}}}}%
    \put(-0.0023123,0.15039752){\color[rgb]{0,0,0}\makebox(0,0)[lb]{\smash{$\gamma_{a,l,b}=$}}}%
    \put(0.18003861,0.30473319){\color[rgb]{0,0,0}\makebox(0,0)[lb]{\smash{$V$}}}%
    \put(0.33673087,0.30473319){\color[rgb]{0,0,0}\makebox(0,0)[lb]{\smash{$U$}}}%
    \put(0.3528567,0.15039752){\color[rgb]{0,0,0}\makebox(0,0)[lb]{\smash{,}}}%
    \put(0.27057188,0.00614119){\color[rgb]{0,0,0}\makebox(0,0)[lb]{\smash{$\Ib$}}}%
    \put(0.85254275,0.30568524){\color[rgb]{0,0,0}\makebox(0,0)[lb]{\smash{$\Ib$}}}%
  \end{picture}%
\endgroup%

	\end{aligned}\ ,
	\label{eq:ra-duality-morphsisms-graphical}
\end{align}
	such that for $a_1+a_2=a$, $b_1+b_2=b$ and $l_1+l_2=l$ we have
\begin{align}
	\begin{aligned}
	\def\svgwidth{8cm}
\begingroup%
  \makeatletter%
  \providecommand\color[2][]{%
    \errmessage{(Inkscape) Color is used for the text in Inkscape, but the package 'color.sty' is not loaded}%
    \renewcommand\color[2][]{}%
  }%
  \providecommand\transparent[1]{%
    \errmessage{(Inkscape) Transparency is used (non-zero) for the text in Inkscape, but the package 'transparent.sty' is not loaded}%
    \renewcommand\transparent[1]{}%
  }%
  \providecommand\rotatebox[2]{#2}%
  \ifx\svgwidth\undefined%
    \setlength{\unitlength}{269.67312012bp}%
    \ifx\svgscale\undefined%
      \relax%
    \else%
      \setlength{\unitlength}{\unitlength * \real{\svgscale}}%
    \fi%
  \else%
    \setlength{\unitlength}{\svgwidth}%
  \fi%
  \global\let\svgwidth\undefined%
  \global\let\svgscale\undefined%
  \makeatother%
  \begin{picture}(1,0.31633162)%
    \put(0,0){\includegraphics[width=\unitlength]{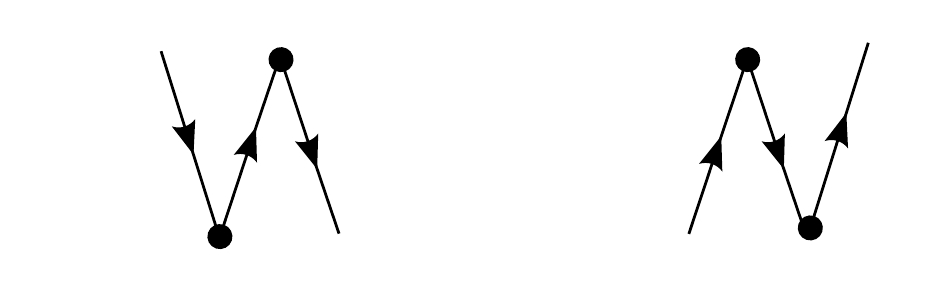}}%
    \put(0.61492099,0.26411445){\color[rgb]{0,0,0}\makebox(0,0)[lb]{\smash{\scriptsize
{$(a_2,l_2,b_2)$}}}}%
    \put(0.56760848,0.15779769){\color[rgb]{0,0,0}\makebox(0,0)[lb]{\smash{$Q^U_{a,l,b}=$}}}%
    \put(0.71109731,0.00523203){\color[rgb]{0,0,0}\makebox(0,0)[lb]{\smash{$U$}}}%
    \put(0.35214427,0.00523203){\color[rgb]{0,0,0}\makebox(0,0)[lb]{\smash{$V$}}}%
    \put(0.04291431,0.05642822){\color[rgb]{0,0,0}\makebox(0,0)[lb]{\smash{\scriptsize{$(a_1,l_1,b_1)$}}}}%
    \put(-0.00196998,0.15779769){\color[rgb]{0,0,0}\makebox(0,0)[lb]{\smash{$Q^V_{a,l,b}=$}}}%
    \put(0.14745196,0.28928498){\color[rgb]{0,0,0}\makebox(0,0)[lb]{\smash{$V$}}}%
    \put(0.91155172,0.28928498){\color[rgb]{0,0,0}\makebox(0,0)[lb]{\smash{$U$}}}%
    \put(0.38961521,0.15779769){\color[rgb]{0,0,0}\makebox(0,0)[lb]{\smash{,}}}%
    \put(0.8958751,0.06837728){\color[rgb]{0,0,0}\makebox(0,0)[lb]{\smash{\scriptsize
{$(a_1,l_1,b_1)$}}}}%
    \put(0.32627301,0.25961942){\color[rgb]{0,0,0}\makebox(0,0)[lb]{\smash{\scriptsize
{$(a_2,l_2,b_2)$}}}}%
    \put(0.28391347,0.28928498){\color[rgb]{0,0,0}\makebox(0,0)[lb]{\smash{$\Ib$}}}%
    \put(0.78229466,0.28928498){\color[rgb]{0,0,0}\makebox(0,0)[lb]{\smash{$\Ib$}}}%
    \put(0.22161587,0.00523203){\color[rgb]{0,0,0}\makebox(0,0)[lb]{\smash{$\Ib$}}}%
    \put(0.85052542,0.00523203){\color[rgb]{0,0,0}\makebox(0,0)[lb]{\smash{$\Ib$}}}%
  \end{picture}%
\endgroup%

	\end{aligned}\ ,
	\label{eq:ra:dual}
\end{align}
and for every $a_1+a_2=a_3+a_4$, $b_1+b_2=b_3+b_4$ and $l_1+l_2=l_3+l_4$  we have
	\begin{align}
	\begin{aligned}
	\def\svgwidth{11.5cm}
\begingroup%
  \makeatletter%
  \providecommand\color[2][]{%
    \errmessage{(Inkscape) Color is used for the text in Inkscape, but the package 'color.sty' is not loaded}%
    \renewcommand\color[2][]{}%
  }%
  \providecommand\transparent[1]{%
    \errmessage{(Inkscape) Transparency is used (non-zero) for the text in Inkscape, but the package 'transparent.sty' is not loaded}%
    \renewcommand\transparent[1]{}%
  }%
  \providecommand\rotatebox[2]{#2}%
  \ifx\svgwidth\undefined%
    \setlength{\unitlength}{359.98261719bp}%
    \ifx\svgscale\undefined%
      \relax%
    \else%
      \setlength{\unitlength}{\unitlength * \real{\svgscale}}%
    \fi%
  \else%
    \setlength{\unitlength}{\svgwidth}%
  \fi%
  \global\let\svgwidth\undefined%
  \global\let\svgscale\undefined%
  \makeatother%
  \begin{picture}(1,0.23351515)%
    \put(0,0){\includegraphics[width=\unitlength]{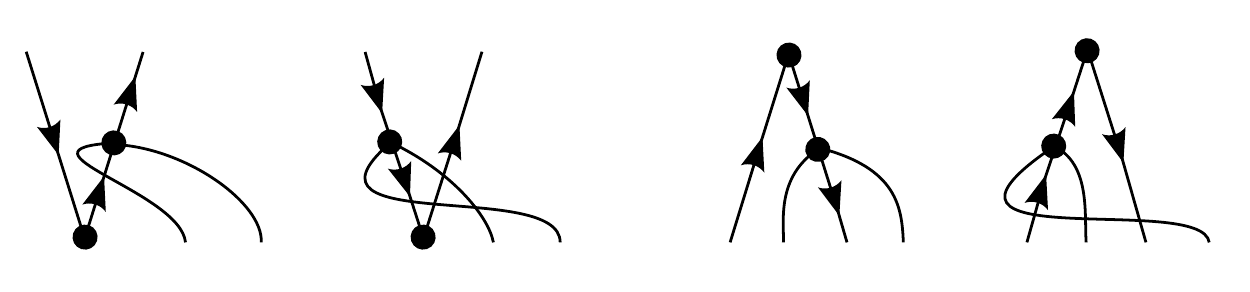}}%
    \put(0.10457502,0.21325376){\color[rgb]{0,0,0}\makebox(0,0)[lb]{\smash{$U$}}}%
    \put(0,0.00589245){\color[rgb]{0,0,0}\makebox(0,0)[lb]{\smash{\scriptsize{$(a_1,l_1,b_1)$}}}}%
    \put(0.00258025,0.21282934){\color[rgb]{0,0,0}\makebox(0,0)[lb]{\smash{$V$}}}%
    \put(0.1318655,0.00391946){\color[rgb]{0,0,0}\makebox(0,0)[lb]{\smash{$A$}}}%
    \put(0.1196202,0.117331){\color[rgb]{0,0,0}\rotatebox{58.82512787}{\makebox(0,0)[lb]{\smash{\scriptsize
{$(a_2,l_2,b_2)$}}}}}%
    \put(0.24541984,0.00483557){\color[rgb]{0,0,0}\makebox(0,0)[lb]{\smash{\scriptsize{$(a_4,l_4,b_4)$}}}}%
    \put(0.27370445,0.21282934){\color[rgb]{0,0,0}\makebox(0,0)[lb]{\smash{$V$}}}%
    \put(0.37165993,0.00391946){\color[rgb]{0,0,0}\makebox(0,0)[lb]{\smash{$A$}}}%
    \put(0.32915937,0.13237386){\color[rgb]{0,0,0}\rotatebox{73.11505517}{\makebox(0,0)[lb]{\smash{\scriptsize
{$(a_3,l_3,b_3)$}}}}}%
    \put(0.38373771,0.21322751){\color[rgb]{0,0,0}\makebox(0,0)[lb]{\smash{$U$}}}%
    \put(0.21774604,0.11183757){\color[rgb]{0,0,0}\makebox(0,0)[lb]{\smash{=}}}%
    \put(0.47266721,0.11183757){\color[rgb]{0,0,0}\makebox(0,0)[lb]{\smash{and}}}%
    \put(0.66774214,0.00391946){\color[rgb]{0,0,0}\makebox(0,0)[lb]{\smash{$V$}}}%
    \put(0.55580117,0.2086809){\color[rgb]{0,0,0}\makebox(0,0)[lb]{\smash{\scriptsize{$(a_1,l_1,b_1)$}}}}%
    \put(0.5657474,0.00391946){\color[rgb]{0,0,0}\makebox(0,0)[lb]{\smash{$U$}}}%
    \put(0.62087294,0.00391946){\color[rgb]{0,0,0}\makebox(0,0)[lb]{\smash{$B$}}}%
    \put(0.68251969,0.11315125){\color[rgb]{0,0,0}\rotatebox{67.90114179}{\makebox(0,0)[lb]{\smash{\scriptsize
{$(a_2,l_2,b_2)$}}}}}%
    \put(0.80780319,0.21202469){\color[rgb]{0,0,0}\rotatebox{-0.35358424}{\makebox(0,0)[lb]{\smash{\scriptsize{$(a_4,l_4,b_4)$}}}}}%
    \put(0.8044858,0.00391946){\color[rgb]{0,0,0}\makebox(0,0)[lb]{\smash{$U$}}}%
    \put(0.79014685,0.07893542){\color[rgb]{0,0,0}\rotatebox{71.89547785}{\makebox(0,0)[lb]{\smash{\scriptsize
{$(a_3,l_3,b_3)$}}}}}%
    \put(0.90581617,0.00391946){\color[rgb]{0,0,0}\makebox(0,0)[lb]{\smash{$V$}}}%
    \put(0.73544146,0.11183757){\color[rgb]{0,0,0}\makebox(0,0)[lb]{\smash{=}}}%
    \put(0.85620786,0.00391946){\color[rgb]{0,0,0}\makebox(0,0)[lb]{\smash{$B$}}}%
    \put(0.18768427,0.00391946){\color[rgb]{0,0,0}\makebox(0,0)[lb]{\smash{$B$}}}%
    \put(0.42909122,0.00436236){\color[rgb]{0,0,0}\makebox(0,0)[lb]{\smash{$B$}}}%
    \put(0.7057703,0.00391946){\color[rgb]{0,0,0}\makebox(0,0)[lb]{\smash{$A$}}}%
    \put(0.95003665,0.00836412){\color[rgb]{0,0,0}\makebox(0,0)[lb]{\smash{$A$}}}%
  \end{picture}%
\endgroup%

	\end{aligned}\ .
		\label{eq:duality-compatibility}
	\end{align}
\end{definition}

Let us compare this situation to Lemma~\ref{lem:copairing}. There the continuity of $\gamma_a$ in the parameter was automatic, but in Definition~\ref{def:dual-pair} we demanded continuity explicitly. The reason for this is that 
the argument in the proof of Lemma~\ref{lem:copairing} does not apply,
	as we have not required that $\id_V\otimes Q_{a,l,b}^U$ is continuous in the parameters, see Remark~\ref{rem:bimodule-continuity}.
However one can easily check that for every $a_1+a_2=a_3+a_4$, $l_1+l_2=l_3+l_4$ and $b_1+b_2=b_3+b_4$
\begin{align}
	(\id_V\otimes Q_{a_1,l_1,b_1}^U)\circ\gamma_{a_2,l_2,b_2}=(Q_{a_3,l_3,b_3}^V\otimes \id_U)\circ\gamma_{a_4,l_4,b_4}\ .
	\label{eq:Q-gamma-compatibility}
\end{align}
	Furthermore, in $\Hilb$ this is equal to $\gamma_{a_1+a_2,l_1+l_2,b_1+b_2}$.

Note that \eqref{eq:duality-compatibility} implies that the action on $V$ is determined by the action on $U$:
\begin{align}
	\begin{aligned}
	\def\svgwidth{4cm}
\begingroup%
  \makeatletter%
  \providecommand\color[2][]{%
    \errmessage{(Inkscape) Color is used for the text in Inkscape, but the package 'color.sty' is not loaded}%
    \renewcommand\color[2][]{}%
  }%
  \providecommand\transparent[1]{%
    \errmessage{(Inkscape) Transparency is used (non-zero) for the text in Inkscape, but the package 'transparent.sty' is not loaded}%
    \renewcommand\transparent[1]{}%
  }%
  \providecommand\rotatebox[2]{#2}%
  \ifx\svgwidth\undefined%
    \setlength{\unitlength}{138.54978027bp}%
    \ifx\svgscale\undefined%
      \relax%
    \else%
      \setlength{\unitlength}{\unitlength * \real{\svgscale}}%
    \fi%
  \else%
    \setlength{\unitlength}{\svgwidth}%
  \fi%
  \global\let\svgwidth\undefined%
  \global\let\svgscale\undefined%
  \makeatother%
  \begin{picture}(1,0.61570747)%
    \put(0,0){\includegraphics[width=\unitlength]{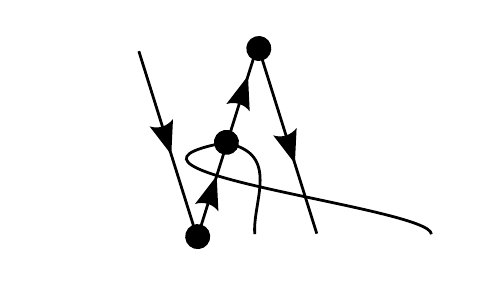}}%
    \put(0.63922038,0.01018361){\color[rgb]{0,0,0}\makebox(0,0)[lb]{\smash{$V$}}}%
    \put(0.01129817,0.10405227){\color[rgb]{0,0,0}\makebox(0,0)[lb]{\smash{\scriptsize{$(a_1,l_1,b_1)$}}}}%
    \put(-0.00383436,0.30713723){\color[rgb]{0,0,0}\makebox(0,0)[lb]{\smash{$\rho^V_{a,l,b}=$}}}%
    \put(0.24080753,0.56306393){\color[rgb]{0,0,0}\makebox(0,0)[lb]{\smash{$V$}}}%
    \put(0.5888646,0.50532291){\color[rgb]{0,0,0}\makebox(0,0)[lb]{\smash{\scriptsize
{$(a_3,l_3,b_3)$}}}}%
    \put(0.87018429,0.01018361){\color[rgb]{0,0,0}\makebox(0,0)[lb]{\smash{$A$}}}%
    \put(0.63347853,0.3099382){\color[rgb]{0,0,0}\makebox(0,0)[lb]{\smash{\scriptsize
{$(a_2,l_2,b_2)$}}}}%
    \put(0.46599744,0.01018361){\color[rgb]{0,0,0}\makebox(0,0)[lb]{\smash{$B$}}}%
  \end{picture}%
\endgroup%

	\end{aligned}\ .
	\label{eq:dualaction-bimodule}
\end{align}
We similarly define dual pairs of left and right modules and we omit the details here. 

\begin{example}\label{ex:twisted-bimodule-duals}
	Let $A\in\Sc$ be a symmetric RFA, $\alpha\in\Aut_{\RFrob{\Sc}}(A)$ and ${}_{\alpha}A_{\id}$ be the twisted transmissive bimodule from Example~\ref{ex:twistedaction}.
	Then $({}_{\alpha}A_{\id},{}_{\alpha^{-1}}A_{\id})$ is a dual pair of bimodules with duality morphisms
	\begin{align}
		\beta_{a,l,b}=\varepsilon_{a}\circ\mu_b\circ(\id_A\otimes\alpha)\quad\text{ and }\quad
		\gamma_{a,l,b}=(\alpha^{-1}\otimes\id_A)\circ\Delta_{a}\circ\eta_b
		\label{eq:twisted-bimodule-duals}
	\end{align}
	for $a,l,b\in\Rb_{>0}$.  Note that these morphisms only depend on $a+b$.
\end{example}

\begin{remark}
	If $(U,V)$ is a dual pair of bimodules with duality morphisms $\gamma_{a,l,b}$ and $\beta_{a,l,b}$, 
	then $(V,U)$ is also a dual pair of bimodules with duality morphisms 
	$\sigma_{V,U}\circ\gamma_{a,l,b}$ 
	and $\beta_{a,l,b}\circ\sigma_{V,U}$.
\end{remark}

Duals of bimodules
over associative algebras are unique up to unique isomorphism.
In the following we will see that under some assumptions this will be true for duals of 
bimodules over regularised algebras too. 
Let $(U,V)$ and $(U,W)$ be two dual pairs of bimodules and define
\begin{align}
\begin{aligned}
\def\svgwidth{8cm}
\begingroup%
  \makeatletter%
  \providecommand\color[2][]{%
    \errmessage{(Inkscape) Color is used for the text in Inkscape, but the package 'color.sty' is not loaded}%
    \renewcommand\color[2][]{}%
  }%
  \providecommand\transparent[1]{%
    \errmessage{(Inkscape) Transparency is used (non-zero) for the text in Inkscape, but the package 'transparent.sty' is not loaded}%
    \renewcommand\transparent[1]{}%
  }%
  \providecommand\rotatebox[2]{#2}%
  \ifx\svgwidth\undefined%
    \setlength{\unitlength}{216.2300354bp}%
    \ifx\svgscale\undefined%
      \relax%
    \else%
      \setlength{\unitlength}{\unitlength * \real{\svgscale}}%
    \fi%
  \else%
    \setlength{\unitlength}{\svgwidth}%
  \fi%
  \global\let\svgwidth\undefined%
  \global\let\svgscale\undefined%
  \makeatother%
  \begin{picture}(1,0.40298531)%
    \put(0,0){\includegraphics[width=\unitlength]{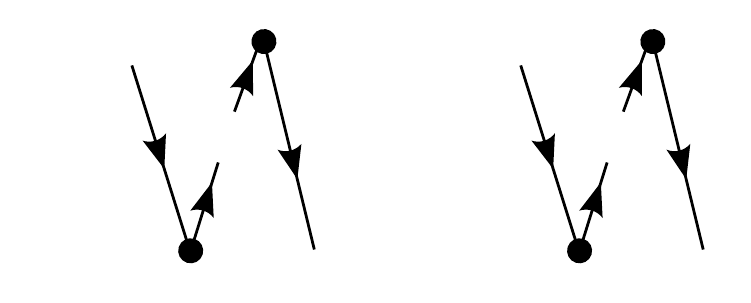}}%
    \put(0.28725229,0.2090485){\color[rgb]{0,0,0}\makebox(0,0)[lb]{\smash{$U$}}}%
    \put(0.15946638,0.34308397){\color[rgb]{0,0,0}\makebox(0,0)[lb]{\smash{$V$}}}%
    \put(0.26030425,0.3908188){\color[rgb]{0,0,0}\rotatebox{0.93255863}{\makebox(0,0)[lb]{\smash{\scriptsize
{$(a_2,l_2,b_2)$}}}}}%
    \put(-0.00245687,0.19507988){\color[rgb]{0,0,0}\makebox(0,0)[lb]{\smash{$\varphi_{a,l,b}:=$}}}%
    \put(0.15352935,0.00698041){\color[rgb]{0,0,0}\makebox(0,0)[lb]{\smash{\scriptsize{$(a_1,l_1,b_1)$}}}}%
    \put(0.39885346,0.01599847){\color[rgb]{0,0,0}\makebox(0,0)[lb]{\smash{$W$}}}%
    \put(0.80521921,0.2090485){\color[rgb]{0,0,0}\makebox(0,0)[lb]{\smash{$U$}}}%
    \put(0.6774333,0.34308397){\color[rgb]{0,0,0}\makebox(0,0)[lb]{\smash{$W$}}}%
    \put(0.51551005,0.19507988){\color[rgb]{0,0,0}\makebox(0,0)[lb]{\smash{$\psi_{a,l,b}:=$}}}%
    \put(0.91682036,0.01599847){\color[rgb]{0,0,0}\makebox(0,0)[lb]{\smash{$V$}}}%
    \put(0.47111289,0.19507988){\color[rgb]{0,0,0}\makebox(0,0)[lb]{\smash{,}}}%
    \put(0.78567185,0.39081988){\color[rgb]{0,0,0}\rotatebox{0.93255863}{\makebox(0,0)[lb]{\smash{\scriptsize
{$(a_2,l_2,b_2)$}}}}}%
    \put(0.6788958,0.00698041){\color[rgb]{0,0,0}\makebox(0,0)[lb]{\smash{\scriptsize{$(a_1,l_1,b_1)$}}}}%
  \end{picture}%
\endgroup%

\end{aligned}\ ,
	\label{eq:unique-dual}
\end{align}
which satisfy for $a=a_1+a_2$, $b=b_1+b_2$ and $l=l_1+l_2$ that
\begin{align}
	\varphi_{a_1,l_1,b_1}\circ\psi_{a_2,l_2,b_2}=Q_{a,l,b}^V
	\quad\text{ and }\quad
	\psi_{a_1,l_1,b_1}\circ\varphi_{a_2,l_2,b_2}=Q_{a,l,b}^W\ .
	\label{eq:phi-psi-comp}
\end{align}
Using separate continuity of the composition and \eqref{eq:phi-psi-comp} one can show the following
(we omit the details):
\begin{lemma} \label{lem:unique-dual}
	If the limits 
	\begin{align}
		\lim_{a,l,b\to0}\varphi_{a,l,b}\quad\text{ and }\quad\lim_{a,l,b\to0}\psi_{a,l,b}
		\label{eq:eq:unique-dual-limits}
	\end{align}
	exist, then $\varphi_{0,0,0}$ and $\psi_{0,0,0}$ are mutually inverse bimodule isomorphisms between $V$ and $W$.
\end{lemma}

In general we do not know if $V \cong W$, not even in $\Hilb$.

\begin{remark}
A related concept of duals was introduced in \cite{Abramsky:1999nt}
where duals are parametrised by Hilbert-Schmidt maps. 
The authors introduced the notion of a nuclear ideal in a symmetric monoidal category, which in $\Hilb$ consists of Hilbert-Schmidt maps $\mathbf{HSO}(\Hc,\Kc)$ for $\Hc,\Kc\in\Hilb$ \cite[Thm.\,5.9]{Abramsky:1999nt}.
Part of the data is an isomorphism $\theta:\mathbf{HSO}(\Hc,\Kc)\xrightarrow{\cong}\Bc(\Cb,\overline{\Hc}\otimes\Kc)$, where $\overline{\Hc}$ now denotes the conjugate Hilbert space. For $f\in\mathbf{HSO}(\Hc,\Kc)$ and $g\in\mathbf{HSO}(\Kc,\Lc)$ in a nuclear ideal the ``compactness'' relation holds:
\begin{align}
	(\id_{\Lc}\otimes\theta(f^{\dagger})^{\dagger})\circ(\theta(g)\otimes\id_{\Hc})=g\circ f\ .
	\label{eq:HSO-compactness}
\end{align}
Our definition of duals fits into this framework as follows.
Let $A,B$ be a regularised algebras and $U$ an $A$-$B$-bimodule in $\Hilb$ with dual $V$.
Then one can show that $Q_{a,l,b}^U$ is a trace class map, and hence Hilbert-Schmidt, cf.\ Lemma~\ref{lem:patrclass}.
Using the above notation let 
$\Hc=\Kc=\Lc:=U$,
\begin{align}
	\begin{aligned}
		f:=&Q_{a,l,b}^U\ , &g:=&Q_{a',l',b'}^U\ ,&
		\beta_{a,l,b}^U:=&\theta(f^{\dagger})^{\dagger} &\gamma_{a',l',b'}^U:=&\theta(g)\ .
	\end{aligned}
	\label{eq:HSO-compare-duality}
\end{align}
Then \eqref{eq:HSO-compactness} is exactly one half of the duality relation \eqref{eq:ra:dual} and $(U,\overline{U})$ is a dual pair of bimodules in the sense of Definition~\ref{def:dual-pair}.
\end{remark}

\subsection{Tensor product of modules over regularised algebras}\label{sec:tensor}

Let $A\in\Sc$ be a regularised algebra,
$M,N\in\Sc$ right and left $A$-modules
respectively and $U\in\Sc$ an $A$-$A$-bimodule.
Let $\rho^R_{a,l}:=\rho^U_{a',l,a}\circ(\eta_{a''}^A\otimes\id_{U\otimes A})$ and
$\rho^L_{a,l}:=\rho^U_{a,l,a'}\circ(\id_{A\otimes U}\otimes\eta_{a''}^A)$
for $a'+a''=a$.
\begin{definition} \label{def:tensor-prod-modules-cyclic}
	The \textsl{tensor product of $M$ and $N$ over $A$} 
	is an object $M\otimes_A N$ together with a morphism
	$\pi_{M \otimes_A N}:M\otimes N\to M\otimes_A N$ in $\Sc$,
	which is a coequaliser of $\rho^M_{a,l}\otimes Q_{a,l}^N$ and
	$Q_{a,l}^M\otimes\rho^N_{a,l}$ for every pair of parameters $(a,l)\in\Rb_{>0}^2$.

	If $\Sc$ is symmetric with braiding $\sigma$,
	one similarly defines the \textsl{cyclic tensor product} 
	$\left(\pi_{\ctimes_A U}:U\to\ctimes_A U\right)$
	to be a coequaliser of $\rho^L_{a,l}$ and $\rho^R_{a,l}\circ\sigma_{A,U}$
	for every $(a,l)\in\Rb_{>0}^2$.
\end{definition}

	Let $A,B,C\in\Sc$ be regularised algebras, 
	$V$ a $B$-$A$-module and $W$ an
	$A$-$C$-module.
	Let $\pi:V\otimes W\to V\otimes_A W$ denote a coequaliser of the morphisms
	\begin{align}
	\begin{aligned}
	\def\svgwidth{10cm}
\begingroup%
  \makeatletter%
  \providecommand\color[2][]{%
    \errmessage{(Inkscape) Color is used for the text in Inkscape, but the package 'color.sty' is not loaded}%
    \renewcommand\color[2][]{}%
  }%
  \providecommand\transparent[1]{%
    \errmessage{(Inkscape) Transparency is used (non-zero) for the text in Inkscape, but the package 'transparent.sty' is not loaded}%
    \renewcommand\transparent[1]{}%
  }%
  \providecommand\rotatebox[2]{#2}%
  \ifx\svgwidth\undefined%
    \setlength{\unitlength}{233.1109375bp}%
    \ifx\svgscale\undefined%
      \relax%
    \else%
      \setlength{\unitlength}{\unitlength * \real{\svgscale}}%
    \fi%
  \else%
    \setlength{\unitlength}{\svgwidth}%
  \fi%
  \global\let\svgwidth\undefined%
  \global\let\svgscale\undefined%
  \makeatother%
  \begin{picture}(1,0.26041115)%
    \put(0,0){\includegraphics[width=\unitlength]{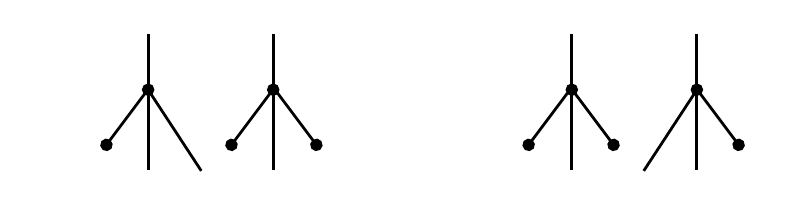}}%
    \put(0.16588132,0.00605264){\color[rgb]{0,0,0}\makebox(0,0)[lb]{\smash{$V$}}}%
    \put(0.23108633,0.00605264){\color[rgb]{0,0,0}\makebox(0,0)[lb]{\smash{$A$}}}%
    \put(0.32031423,0.00605264){\color[rgb]{0,0,0}\makebox(0,0)[lb]{\smash{$W$}}}%
    \put(0.69095321,0.00605264){\color[rgb]{0,0,0}\makebox(0,0)[lb]{\smash{$V$}}}%
    \put(0.77674927,0.00605264){\color[rgb]{0,0,0}\makebox(0,0)[lb]{\smash{$A$}}}%
    \put(0.84538612,0.00605264){\color[rgb]{0,0,0}\makebox(0,0)[lb]{\smash{$W$}}}%
    \put(0.16588132,0.2291224){\color[rgb]{0,0,0}\makebox(0,0)[lb]{\smash{$V$}}}%
    \put(0.32031423,0.2291224){\color[rgb]{0,0,0}\makebox(0,0)[lb]{\smash{$W$}}}%
    \put(0.69095321,0.2291224){\color[rgb]{0,0,0}\makebox(0,0)[lb]{\smash{$V$}}}%
    \put(0.84538612,0.2291224){\color[rgb]{0,0,0}\makebox(0,0)[lb]{\smash{$W$}}}%
    \put(0.1967679,0.16734924){\color[rgb]{0,0,0}\makebox(0,0)[lb]{\smash{\scriptsize
{$(b_2,l,a)$}}}}%
    \put(0.35120081,0.16734924){\color[rgb]{0,0,0}\makebox(0,0)[lb]{\smash{\scriptsize
{$(a_2,l,c_2)$}}}}%
    \put(0.72183979,0.16734924){\color[rgb]{0,0,0}\makebox(0,0)[lb]{\smash{\scriptsize
{$(b_2,l,a_2)$}}}}%
    \put(0.8762727,0.16734924){\color[rgb]{0,0,0}\makebox(0,0)[lb]{\smash{\scriptsize
{$(a,l,c_2)$}}}}%
    \put(0.93118218,0.09871239){\color[rgb]{0,0,0}\makebox(0,0)[lb]{\smash{\scriptsize
{$c_1$}}}}%
    \put(0.0835171,0.07812133){\color[rgb]{0,0,0}\makebox(0,0)[lb]{\smash{\scriptsize
{$b_1$}}}}%
    \put(0.45415604,0.13303081){\color[rgb]{0,0,0}\makebox(0,0)[lb]{\smash{,}}}%
    \put(0.25167738,0.10557607){\color[rgb]{0,0,0}\makebox(0,0)[lb]{\smash{\scriptsize
{$a_1$}}}}%
    \put(0.60858899,0.07812133){\color[rgb]{0,0,0}\makebox(0,0)[lb]{\smash{\scriptsize
{$b_1$}}}}%
    \put(0.76988558,0.10557607){\color[rgb]{0,0,0}\makebox(0,0)[lb]{\smash{\scriptsize
{$a_1$}}}}%
    \put(0.40611029,0.09871239){\color[rgb]{0,0,0}\makebox(0,0)[lb]{\smash{\scriptsize
{$c_1$}}}}%
    \put(-0.00227896,0.13303081){\color[rgb]{0,0,0}\makebox(0,0)[lb]{\smash{$\rho_{a,b,c,l}^{L}:=$}}}%
    \put(0.52279293,0.13303081){\color[rgb]{0,0,0}\makebox(0,0)[lb]{\smash{$\rho_{a,b,c,l}^{R}:=$}}}%
  \end{picture}%
\endgroup%

	\end{aligned}\ ,
		\label{eq:coeq-tensor-product}
	\end{align}
	for every parameter $a_1,a_2,b_1,b_2,c_1,c_2,l\in\Rb_{>0}$ with 
	$a=a_1+a_2$,
	$b=b_1+b_2$ and
	$c=c_1+c_2$
	(which of course may or may not exist).
	If the tensor product $B \otimes (-) \otimes C$ preserves 
	coequalisers of families of morphisms,
	then the universal property of the 
coequaliser
	induces a morphism $\bar{\rho}_{a,b,c,l}:B\otimes(V\otimes_A W)\otimes C\to V\otimes_A W$ form the morphism
	\begin{align}
	\begin{aligned}
	\def\svgwidth{4cm}
\begingroup%
  \makeatletter%
  \providecommand\color[2][]{%
    \errmessage{(Inkscape) Color is used for the text in Inkscape, but the package 'color.sty' is not loaded}%
    \renewcommand\color[2][]{}%
  }%
  \providecommand\transparent[1]{%
    \errmessage{(Inkscape) Transparency is used (non-zero) for the text in Inkscape, but the package 'transparent.sty' is not loaded}%
    \renewcommand\transparent[1]{}%
  }%
  \providecommand\rotatebox[2]{#2}%
  \ifx\svgwidth\undefined%
    \setlength{\unitlength}{90.75548096bp}%
    \ifx\svgscale\undefined%
      \relax%
    \else%
      \setlength{\unitlength}{\unitlength * \real{\svgscale}}%
    \fi%
  \else%
    \setlength{\unitlength}{\svgwidth}%
  \fi%
  \global\let\svgwidth\undefined%
  \global\let\svgscale\undefined%
  \makeatother%
  \begin{picture}(1,0.778935)%
    \put(0,0){\includegraphics[width=\unitlength]{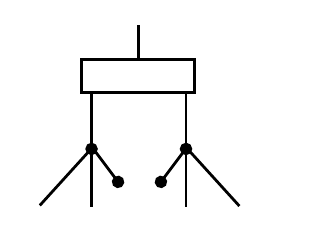}}%
    \put(0.66058352,0.29902623){\color[rgb]{0,0,0}\makebox(0,0)[lb]{\smash{\scriptsize
{$(a_3,l,c)$}}}}%
    \put(0,0.31918174){\color[rgb]{0,0,0}\makebox(0,0)[lb]{\smash{\scriptsize
{$(b,l,a_1)$}}}}%
    \put(0.40754644,0.50880097){\color[rgb]{0,0,0}\makebox(0,0)[lb]{\smash{$\pi$}}}%
    \put(0.2935909,0.72535697){\color[rgb]{0,0,0}\makebox(0,0)[lb]{\smash{$V\otimes_A W$}}}%
    \put(0.07082193,0.01036439){\color[rgb]{0,0,0}\makebox(0,0)[lb]{\smash{$B$}}}%
    \put(0.72637134,0.01036439){\color[rgb]{0,0,0}\makebox(0,0)[lb]{\smash{$C$}}}%
    \put(0.24603665,0.01036439){\color[rgb]{0,0,0}\makebox(0,0)[lb]{\smash{$V$}}}%
    \put(0.541739,0.01036439){\color[rgb]{0,0,0}\makebox(0,0)[lb]{\smash{$W$}}}%
    \put(0.48082939,0.12548595){\color[rgb]{0,0,0}\makebox(0,0)[lb]{\smash{\scriptsize{$a_4$}}}}%
    \put(0.33942926,0.12234029){\color[rgb]{0,0,0}\makebox(0,0)[lb]{\smash{\scriptsize{$a_2$}}}}%
  \end{picture}%
\endgroup%

	\end{aligned}\ ,
	\label{eq:def-tensor-product-action-before}
	\end{align}
where $a_1+a_2 = a_3 + a_4$.
	If the limit $\bar{\rho}_{b,l,c}:=\lim_{a\to0}\bar{\rho}_{a,b,c,l}$ exists and is jointly continuous in all three parameters, 
	then it gives a $B$-$C$-bimodule structure on $V\otimes_A W$.
	\begin{definition}\label{def:tensor-product-bimodule}
		The \textsl{tensor product of $V$ and $W$ over $A$} is the $B$-$C$-bimodule $V\otimes_A W$ with 
		the action $\bar{\rho}_{b,c,l}$ together with the coequaliser $\pi:V\otimes W\to V\otimes_A W$.
	\end{definition}

	The following proposition shows that in $\Vectfd$ the tensor product of modules over regularised algebras reduces to tensor product over ordinary associative algebras. The proof is straightforward and we omit it.
	
\begin{proposition}\label{prop:tensor-product-in-vect}
	Let $A$ be a regularised algebra in $\Vectfd$, 
	$M$ and $N$ right and left $A$-modules, respectively.
	Let $\Dc(M)=(M,H_M)$ and $\Dc(N)=(N,H_N)$ be the corresponding underlying modules and module morphisms from Proposition~\ref{prop:ramodule:findim}.
	Then $\Dc(M \otimes_A N)=(M\otimes_A N,H_M\otimes_A H_N)$, 
		where $M\otimes_A N$ is the tensor product of the underlying modules over the underlying algebra
		and $H_M\otimes_A H_N$ 
		is the induced morphism on the tensor product.
\end{proposition}

\medskip

For the rest of the section let $\Sc$ be symmetric monoidal and idempotent complete, and $A\in\Sc$ a strongly separable regularised algebra with separability idempotents $e_a$.
We define the following morphisms
\begin{align}
	\begin{aligned}
	\def\svgwidth{14cm}
\begingroup%
  \makeatletter%
  \providecommand\color[2][]{%
    \errmessage{(Inkscape) Color is used for the text in Inkscape, but the package 'color.sty' is not loaded}%
    \renewcommand\color[2][]{}%
  }%
  \providecommand\transparent[1]{%
    \errmessage{(Inkscape) Transparency is used (non-zero) for the text in Inkscape, but the package 'transparent.sty' is not loaded}%
    \renewcommand\transparent[1]{}%
  }%
  \providecommand\rotatebox[2]{#2}%
  \ifx\svgwidth\undefined%
    \setlength{\unitlength}{255.33024902bp}%
    \ifx\svgscale\undefined%
      \relax%
    \else%
      \setlength{\unitlength}{\unitlength * \real{\svgscale}}%
    \fi%
  \else%
    \setlength{\unitlength}{\svgwidth}%
  \fi%
  \global\let\svgwidth\undefined%
  \global\let\svgscale\undefined%
  \makeatother%
  \begin{picture}(1,0.25654887)%
    \put(0,0){\includegraphics[width=\unitlength]{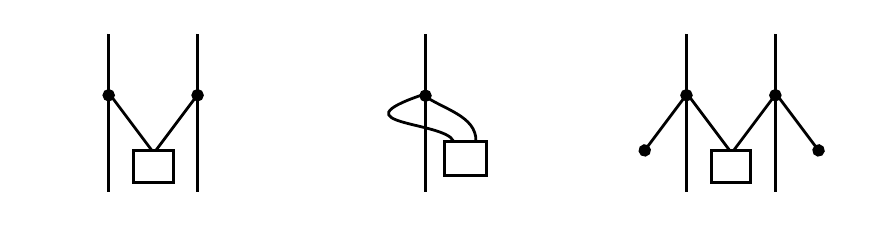}}%
    \put(0.8588409,0.00552593){\color[rgb]{0,0,0}\makebox(0,0)[lb]{\smash{$W$}}}%
    \put(0.8588409,0.22798293){\color[rgb]{0,0,0}\makebox(0,0)[lb]{\smash{$W$}}}%
    \put(0.88703968,0.16531898){\color[rgb]{0,0,0}\makebox(0,0)[lb]{\smash{\scriptsize
{$(a_3,l,c_1)$}}}}%
    \put(0.93717083,0.09012225){\color[rgb]{0,0,0}\makebox(0,0)[lb]{\smash{\scriptsize
{$c_2$}}}}%
    \put(0.80557655,0.06192348){\color[rgb]{0,0,0}\makebox(0,0)[lb]{\smash{$e_{a_2}$}}}%
    \put(0.78051097,0.17785177){\color[rgb]{0,0,0}\makebox(0,0)[lb]{\smash{\scriptsize
{$(b_2,l,a_1)$}}}}%
    \put(0.2071359,0.00552593){\color[rgb]{0,0,0}\makebox(0,0)[lb]{\smash{$N$}}}%
    \put(0.2071359,0.22798293){\color[rgb]{0,0,0}\makebox(0,0)[lb]{\smash{$N$}}}%
    \put(0.23533468,0.16531898){\color[rgb]{0,0,0}\makebox(0,0)[lb]{\smash{\scriptsize
{$(a_3,l)$}}}}%
    \put(0.26039953,0.12145419){\color[rgb]{0,0,0}\makebox(0,0)[lb]{\smash{,}}}%
    \put(0.69591465,0.06818987){\color[rgb]{0,0,0}\makebox(0,0)[lb]{\smash{\scriptsize
{$b_1$}}}}%
    \put(0.14133876,0.16531898){\color[rgb]{0,0,0}\makebox(0,0)[lb]{\smash{\scriptsize
{$(a_1,l)$}}}}%
    \put(0.10687359,0.00552593){\color[rgb]{0,0,0}\makebox(0,0)[lb]{\smash{$M$}}}%
    \put(0.10687359,0.22798293){\color[rgb]{0,0,0}\makebox(0,0)[lb]{\smash{$M$}}}%
    \put(0.76797819,0.00552593){\color[rgb]{0,0,0}\makebox(0,0)[lb]{\smash{$V$}}}%
    \put(0.76797819,0.22798293){\color[rgb]{0,0,0}\makebox(0,0)[lb]{\smash{$V$}}}%
    \put(0.49852323,0.16531898){\color[rgb]{0,0,0}\makebox(0,0)[lb]{\smash{\scriptsize
{$(a_1,l,a_3)$}}}}%
    \put(0.50427411,0.07201443){\color[rgb]{0,0,0}\makebox(0,0)[lb]{\smash{$e_{a_2}$}}}%
    \put(0.1537113,0.06186807){\color[rgb]{0,0,0}\makebox(0,0)[lb]{\smash{$e_{a_2}$}}}%
    \put(0.55820249,0.12101935){\color[rgb]{0,0,0}\makebox(0,0)[lb]{\smash{,}}}%
    \put(-0.00208064,0.11758807){\color[rgb]{0,0,0}\makebox(0,0)[lb]{\smash{$D_{a,l}^{M,N}:=$}}}%
    \put(0.33630465,0.11758807){\color[rgb]{0,0,0}\makebox(0,0)[lb]{\smash{$D_{a,l}^{U}:=$}}}%
    \put(0.62455878,0.11758807){\color[rgb]{0,0,0}\makebox(0,0)[lb]{\smash{$D_{a,b,c,l}^{V,W}:=$}}}%
    \put(0.47032446,0.22798293){\color[rgb]{0,0,0}\makebox(0,0)[lb]{\smash{$U$}}}%
    \put(0.47032446,0.00865913){\color[rgb]{0,0,0}\makebox(0,0)[lb]{\smash{$U$}}}%
  \end{picture}%
\endgroup%

	\end{aligned}\ ,
	\label{eq:D-definition-bimodule}
\end{align}
with $\sum_{i=1}^3 a_i=a$, $b=b_1+b_2$, $c=c_1+c_2$ and $l_1+l_2=l$.
{}From a direct computation it follows that
\begin{align}
D_{a_1,l_1}^{M,N}\circ D_{a_2,l_2}^{M,N}&= D_{a_1+a_2,l_1+l_2}^{M,N}
	\label{eq:D-MN-semigroup}\\
D_{a_1,l_1}^{U}\circ D_{a_2,l_2}^{U}&= D_{a_1+a_2,l_1+l_2}^{U}
	\label{eq:D-U-semigroup}\\
D_{a_1,b_1,c_1,l_1}^{V,W}\circ D_{a_2,b_2,c_2,l_1}^{V,W}&= D_{a_1+a_2,b_1+b_2,c_1+c_2,l_1+l_2}^{V,W}
	\label{eq:D-VW-semigroup}
\end{align}
for every $a_1,a_2,l_1,l_2,b_1,b_2\in\Rb_{>0}$.
So if $D_0^{M,N}:=\lim_{a,l\to0}D_{a,l}^{M,N}$
	exists, then it is an idempotent.
	In this case we write
	\begin{align}
		\begin{aligned}
			D_0^{M,N}&=\left[ M\otimes N\xrightarrow {\pi}\im(D_0^{M,N})\xrightarrow{\iota}M\otimes N \right]\\
			\id_{\im(D_0^{M,N})}&=\left[ \im(D_0^{M,N})\xrightarrow{\iota}M\otimes N \xrightarrow {\pi}\im(D_0^{M,N}) \right]
		\end{aligned}
		\label{eq:d0-image}
	\end{align}
	for the projection $\pi$ and embedding $\iota$ of its image $\im(D_0^{M,N})$.
	Similarly, if $D_0^U:=\lim_{a,l\to0}D_{a,l}^{U}$ 
	(resp.\ $D_0^{V,W}:=\lim_{a,b,c,l\to0}D_{a,b,c,l}^{V,W}$) exists, 
	then it is also an idempotent and we similarly write
	$\pi$, $\iota$ and $\im(D_0^U)$ (resp.\ $\im(D_0^{V,W})$).

	Let us assume that $\lim_{a,b,c,l\to0}D_{a,b,c,l}^{V,W}$ exists and define
	\begin{align}
	\begin{aligned}
	\def\svgwidth{6.5cm}
\begingroup%
  \makeatletter%
  \providecommand\color[2][]{%
    \errmessage{(Inkscape) Color is used for the text in Inkscape, but the package 'color.sty' is not loaded}%
    \renewcommand\color[2][]{}%
  }%
  \providecommand\transparent[1]{%
    \errmessage{(Inkscape) Transparency is used (non-zero) for the text in Inkscape, but the package 'transparent.sty' is not loaded}%
    \renewcommand\transparent[1]{}%
  }%
  \providecommand\rotatebox[2]{#2}%
  \ifx\svgwidth\undefined%
    \setlength{\unitlength}{138.42003174bp}%
    \ifx\svgscale\undefined%
      \relax%
    \else%
      \setlength{\unitlength}{\unitlength * \real{\svgscale}}%
    \fi%
  \else%
    \setlength{\unitlength}{\svgwidth}%
  \fi%
  \global\let\svgwidth\undefined%
  \global\let\svgscale\undefined%
  \makeatother%
  \begin{picture}(1,0.67645858)%
    \put(0,0){\includegraphics[width=\unitlength]{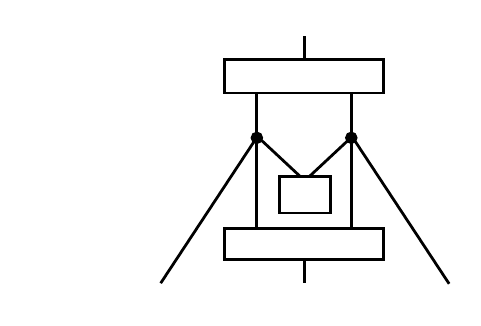}}%
    \put(0.7770465,0.38492309){\color[rgb]{0,0,0}\makebox(0,0)[lb]{\smash{\scriptsize
{$(a_3,l,c)$}}}}%
    \put(0.59210217,0.26355338){\color[rgb]{0,0,0}\makebox(0,0)[lb]{\smash{$e_{a_2}$}}}%
    \put(0.32081503,0.38657909){\color[rgb]{0,0,0}\makebox(0,0)[lb]{\smash{\scriptsize
{$(b,l,a_1)$}}}}%
    \put(0.61114203,0.49934445){\color[rgb]{0,0,0}\makebox(0,0)[lb]{\smash{$\pi$}}}%
    \put(0.61546837,0.1463921){\color[rgb]{0,0,0}\makebox(0,0)[lb]{\smash{$\iota$}}}%
    \put(0.53642676,0.64133){\color[rgb]{0,0,0}\makebox(0,0)[lb]{\smash{$\im(D_0^{V,W})$}}}%
    \put(0.52145891,0.01090431){\color[rgb]{0,0,0}\makebox(0,0)[lb]{\smash{$\im(D_0^{V,W})$}}}%
    \put(-0.00191898,0.34852063){\color[rgb]{0,0,0}\makebox(0,0)[lb]{\smash{$\tilde{\rho}_{a,b,c,l}^{V, W}:=$}}}%
    \put(0.30945452,0.01847775){\color[rgb]{0,0,0}\makebox(0,0)[lb]{\smash{$B$}}}%
    \put(0.91265264,0.01668576){\color[rgb]{0,0,0}\makebox(0,0)[lb]{\smash{$C$}}}%
  \end{picture}%
\endgroup%

	\end{aligned}
	\label{eq:tensor-product-action}
	\end{align}
	for $a_1,a_2,a_3,b,c,l\in\Rb_{>0}$ with $a=a_1+a_2+a_3$.
	
\begin{proposition}~
\label{prop:d0proj}
\begin{enumerate}
	\item If $D_0^{M,N}$ exists then $(\pi,\im(D_0)^{M,N})$ is the tensor product $M\otimes_A N$.
	\item If $D_0^{U}$ exists then $(\pi,\im(D_0^{U}))$ is the cyclic tensor product $\ctimes_A U$.
	\item If $D_0^{V,W}$ and 
	$\tilde{\rho}_{b,l,c}:=\lim_{a\to0}\tilde{\rho}_{a,b,c,l}^{V,W}$ 
		exists for every $b,c,l\in\Rb_{>0}$
		and is jointly continuous in the parameters
		then $(\pi,\im(D_0)^{V,W})$ with action 
		$\tilde{\rho}_{b,l,c}$ 
		is the tensor product $V\otimes_A W$.
\end{enumerate}
\end{proposition}
\begin{proof}
	We will only treat the third
	case, in the other two cases
	one proceeds analogously.
	We show that $(\pi,\im(D_0^{V,W}))$ 
	is a coequaliser of the morphisms in \eqref{eq:coeq-tensor-product}.
        Let $p:=(a,b,c,l)$, $p':=(a',b',c',l')$
        and $\varphi:V\otimes W\to Y$ be such that
	\begin{align}
		\varphi\circ \rho^L_{p} =
		\varphi\circ \rho^R_{p}\ .
		\label{eq:phiq}
	\end{align}
	Let $\tilde{\varphi}:=\varphi\circ\iota$. We need to show that $\varphi=\tilde{\varphi}\circ\pi$ and that $\tilde{\varphi}$ is unique.
	Compose both sides of \eqref{eq:phiq} with 
	\begin{align*}
	\begin{aligned}
	\def\svgwidth{5cm}
\begingroup%
  \makeatletter%
  \providecommand\color[2][]{%
    \errmessage{(Inkscape) Color is used for the text in Inkscape, but the package 'color.sty' is not loaded}%
    \renewcommand\color[2][]{}%
  }%
  \providecommand\transparent[1]{%
    \errmessage{(Inkscape) Transparency is used (non-zero) for the text in Inkscape, but the package 'transparent.sty' is not loaded}%
    \renewcommand\transparent[1]{}%
  }%
  \providecommand\rotatebox[2]{#2}%
  \ifx\svgwidth\undefined%
    \setlength{\unitlength}{84.4323044bp}%
    \ifx\svgscale\undefined%
      \relax%
    \else%
      \setlength{\unitlength}{\unitlength * \real{\svgscale}}%
    \fi%
  \else%
    \setlength{\unitlength}{\svgwidth}%
  \fi%
  \global\let\svgwidth\undefined%
  \global\let\svgscale\undefined%
  \makeatother%
  \begin{picture}(1,0.60110109)%
    \put(0,0){\includegraphics[width=\unitlength]{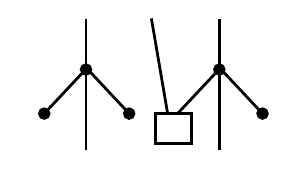}}%
    \put(0.8048054,0.35608946){\color[rgb]{0,0,0}\makebox(0,0)[lb]{\smash{\scriptsize{$(a_1',l',c_1')$}}}}%
    \put(0.88060577,0.12868835){\color[rgb]{0,0,0}\makebox(0,0)[lb]{\smash{\scriptsize{$c_2'$}}}}%
    \put(0.5395041,0.16191275){\color[rgb]{0,0,0}\makebox(0,0)[lb]{\smash{$e_{a_2'}$}}}%
    \put(0,0.35880435){\color[rgb]{0,0,0}\makebox(0,0)[lb]{\smash{\scriptsize{$(b_1',l',a_1')$}}}}%
    \put(0.03732665,0.21396377){\color[rgb]{0,0,0}\makebox(0,0)[lb]{\smash{\scriptsize{$b_2'$}}}}%
    \put(0.42459936,0.12078472){\color[rgb]{0,0,0}\makebox(0,0)[lb]{\smash{\scriptsize{$a_2'$}}}}%
    \put(0.25274299,0.55790805){\color[rgb]{0,0,0}\makebox(0,0)[lb]{\smash{$V$}}}%
    \put(0.48014396,0.5579082){\color[rgb]{0,0,0}\makebox(0,0)[lb]{\smash{$A$}}}%
    \put(0.70754507,0.5579082){\color[rgb]{0,0,0}\makebox(0,0)[lb]{\smash{$W$}}}%
    \put(0.25274299,0.00835544){\color[rgb]{0,0,0}\makebox(0,0)[lb]{\smash{$V$}}}%
    \put(0.70754507,0.00835544){\color[rgb]{0,0,0}\makebox(0,0)[lb]{\smash{$W$}}}%
  \end{picture}%
\endgroup%

	\end{aligned}
	\end{align*}
	with $a_1'+a_2'=a'$, $b_1'+b_2'=b'$ and $c_1'+c_2'=c'$ to get
	\begin{align*}
		\varphi\circ D_{p+p'}^{V,W}=\varphi\circ \left( Q_{p+p'}^V\otimes Q_{p+p'}^W\right).
	\end{align*}
	Now taking the limit $p,p'\to 0$ gives $\varphi\circ D_0^{V,W}=\varphi$, which we needed to show.
	Uniqueness of $\tilde{\varphi}$ follows from $\pi\circ\iota=\id_{\im(D_0^{V,W})}$.

	It is easy to see that the morphism $\bar{\rho}_{a,b,c,l}$ induced by \eqref{eq:def-tensor-product-action-before} is the morphism in \eqref{eq:tensor-product-action}.

\end{proof}

\begin{corollary}\label{cor:centercycltens}
	Consider $A$ as a 
	bimodule over itself.
	If $D_0^A$ exists then 
	$(\iota:\ctimes_A A\to A)$ is the centre of $A$.
\end{corollary}
\begin{proof}
	Using the previous notation we show that $\iota:\ctimes_A A\to A$ satisfies the
	universal property of the centre. So let $\varphi:Y\to A$ be such that
	\begin{align}
		\mu_a\circ\left( \id_{A}\otimes\varphi \right)=
		\mu_a\circ\sigma\circ\left( \id_{A}\otimes\varphi \right).
		\label{eq:phip}
	\end{align}
	Set $\tilde{\varphi}:=\pi\circ\varphi$. We need to show that $\iota\circ\tilde{\varphi}=\varphi$.
	From \eqref{eq:phip} one obtains that $D_a\circ\varphi=P_a\circ\varphi$. 
	Then taking the limit $a\to0$ gives $D_0\circ\varphi=\varphi$ which is what we needed to show.
	Uniqueness of $\tilde{\varphi}$ follows again from $\pi\circ\iota=\id_{\ctimes_A A}$.
\end{proof}

\begin{example} \label{ex:twisted-bimodule-tensor-products}
	Let $A\in\Sc$ be a strongly separable symmetric
	RFA, $\alpha,\beta\in\Aut_{\RFrob{\Sc}}(A)$ 
	${}_{\alpha}A_{\id},{}_{\beta}A_{\id}$ be the twisted transmissive bimodules from Example~\ref{ex:twistedaction}.
	Let us assume that $\lim_{a\to0}D_a^{{}_{\alpha}A_{\id},{}_{\beta}A_{\id}}$ from \eqref{eq:D-definition-bimodule},
	$\lim_{a\to0}\mu_a$ and $\lim_{a\to0}\Delta_a$ exist. Then
	$({}_{\alpha}A_{\id})\otimes_A ({}_{\beta}A_{\id})={}_{\alpha\circ\beta}A_{\id}$ and
	the projection $\pi:{}_{\alpha}A_{\id}\otimes{}_{\beta}A_{\id}\to{}_{\alpha\circ\beta}A_{\id}$ is given by $\pi=\mu_0\circ(\beta\otimes\id_A)$.
\end{example}

\begin{remark}
For $A$ strongly separable symmetric Frobenius, the tensor product over $A$ actually automatically satisfies a stronger coequaliser condition. Let us illustrate this in the first case of Proposition~\ref{prop:d0proj}: define $L_{a_1,a_2,l} := \rho^M_{a_1,l} \otimes Q^M_{a_2,l}$ and $R_{a_3,a_4,l} := Q_{a_3,l}^M\otimes\rho^N_{a_4,l}$ for $a_i>0$. Then $\pi : M \otimes N \to M \otimes_A N$ is defined as the coequaliser of $L_{a,a,l}$ and $R_{a,a,l}$, but it is straightforward to verify that also
$\pi \circ L_{a_1,a_2,l} = \pi \circ R_{a_3,a_4,l}$ holds for all $a_i>0$ such that $a_1+a_2 = a_3+a_4$.
\end{remark}

Using Proposition~\ref{prop:d0proj} and the dual action in \eqref{eq:dualaction-bimodule} one can show the following.
\begin{lemma} \label{lem:dual-pair-tensor-product}
	Let $(V,\bar{V})$ be a dual pair of $B$-$A$-bimodules and $(W,\bar{W})$ a dual pair of $A$-$C$-bimodules.
	Let us assume that $D_0^{V,W}$ and $D_0^{\bar{W},\bar{V}}$ exist and that
	$\lim_{a\to0}\tilde{\rho}_{a,b,c,l}^{V,W}$ and 
	$\lim_{a\to0}\tilde{\rho}_{a,b,c,l}^{\bar{W},\bar{V}}$ exist and are jointly continuous in their parameters.
	Let for $a,b,c,l\in\Rb_{>0}$ 
	\begin{align}
		\begin{aligned}
		\def\svgwidth{12cm}
\begingroup%
  \makeatletter%
  \providecommand\color[2][]{%
    \errmessage{(Inkscape) Color is used for the text in Inkscape, but the package 'color.sty' is not loaded}%
    \renewcommand\color[2][]{}%
  }%
  \providecommand\transparent[1]{%
    \errmessage{(Inkscape) Transparency is used (non-zero) for the text in Inkscape, but the package 'transparent.sty' is not loaded}%
    \renewcommand\transparent[1]{}%
  }%
  \providecommand\rotatebox[2]{#2}%
  \ifx\svgwidth\undefined%
    \setlength{\unitlength}{298.27284069bp}%
    \ifx\svgscale\undefined%
      \relax%
    \else%
      \setlength{\unitlength}{\unitlength * \real{\svgscale}}%
    \fi%
  \else%
    \setlength{\unitlength}{\svgwidth}%
  \fi%
  \global\let\svgwidth\undefined%
  \global\let\svgscale\undefined%
  \makeatother%
  \begin{picture}(1,0.28562204)%
    \put(0,0){\includegraphics[width=\unitlength]{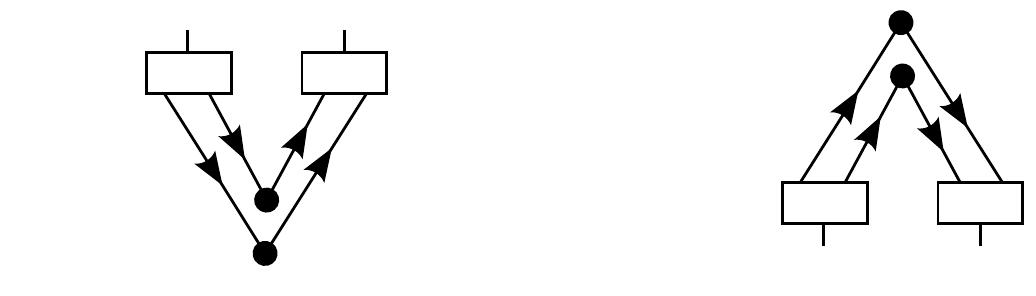}}%
    \put(0.60196445,0.13357685){\color[rgb]{0,0,0}\makebox(0,0)[lb]{\smash{$\tilde{\beta}_{a,b,c,l}^{V,W}:=$}}}%
    \put(0.89502364,0.26762905){\color[rgb]{0,0,0}\makebox(0,0)[lb]{\smash{\scriptsize
{$(b,l,a;V)$}}}}%
    \put(0.78221492,0.07740544){\color[rgb]{0,0,0}\makebox(0,0)[lb]{\smash{$\iota$}}}%
    \put(0.92656681,0.07565864){\color[rgb]{0,0,0}\makebox(0,0)[lb]{\smash{$\iota$}}}%
    \put(0.74981763,0.01594073){\color[rgb]{0,0,0}\makebox(0,0)[lb]{\smash{$V\otimes_A W$}}}%
    \put(0.89469763,0.01594073){\color[rgb]{0,0,0}\makebox(0,0)[lb]{\smash{$\bar{W}
\otimes_A \bar{V}$}}}%
    \put(0.91645775,0.21607958){\color[rgb]{0,0,0}\makebox(0,0)[lb]{\smash{\scriptsize
{$(a,l,c;W)$}}}}%
    \put(-0.00178109,0.13357685){\color[rgb]{0,0,0}\makebox(0,0)[lb]{\smash{$\tilde{\gamma}_{a,b,c,l}^{V,W}:=$}}}%
    \put(0.28101958,0.03567385){\color[rgb]{0,0,0}\makebox(0,0)[lb]{\smash{\scriptsize
{$(b,l,a;V)$}}}}%
    \put(0.15781323,0.20448739){\color[rgb]{0,0,0}\makebox(0,0)[lb]{\smash{$\pi$}}}%
    \put(0.3081066,0.20549153){\color[rgb]{0,0,0}\makebox(0,0)[lb]{\smash{$\pi$}}}%
    \put(0.27550638,0.26696773){\color[rgb]{0,0,0}\makebox(0,0)[lb]{\smash{$V\otimes_A W$}}}%
    \put(0.13364843,0.27747095){\color[rgb]{0,0,0}\makebox(0,0)[lb]{\smash{$\bar{W}
\otimes_A \bar{V}$}}}%
    \put(0.30245369,0.08617304){\color[rgb]{0,0,0}\makebox(0,0)[lb]{\smash{\scriptsize
{$(a,l,c;W)$}}}}%
    \put(0.47322326,0.13357685){\color[rgb]{0,0,0}\makebox(0,0)[lb]{\smash{and}}}%
  \end{picture}%
\endgroup%

		\end{aligned}\ .
		\label{eq:dual-pair-tensor-product}
	\end{align}
	If $\lim_{a\to0}\tilde{\gamma}_{a,b,c,l}^{V,W}$ and $\lim_{a\to0}\tilde{\beta}_{a,b,c,l}^{V,W}$
	exist for every $b,c,l\in\Rb_{>0}$ and are jointly continuous,
	then $(V\otimes_A W,\bar{W}\otimes_A \bar{V})$ is a dual pair of $B$-$C$-bimodules. 
\end{lemma}

\subsection{Tensor products in $\Hilb$}

We now consider the case $\Sc=\Hilb$. 
Note that in $\Hilb$ cokernels exist. If $f:X\to Y$ is a morphism in $\Hilb$ then 
$\pi:Y\to Y/\overline{\im(f)}$ is a cokernel of $f$ where $\pi$ is the canonical projection
and $\overline{\im(f)}$ is the closure of $\im(f)$.

After some preparatory lemmas 
we will discuss tensor products of modules over regularised algebras.

\begin{lemma}
	Tensoring with identity in $\Hilb$ preserves cokernels.
	\label{lem:hilb-tensor-preserves-coker}
\end{lemma}
\begin{proof}
	Let $f:X\to Y$, $\pi_{f}=\coker(f):Y\to Y/\overline{\im(f)}$, $Z\in\Hilb$ and $\pi_{f\otimes\id_Z}:=\coker(f\otimes\id_Z):Y\otimes Z\to Y\otimes Z/\overline{\im(f\otimes\id_Z)}$. 
The claim of the lemma boils down to the observation that 
$\overline{\im(f)} \otimes Z = \overline{\im(f\otimes\id_Z)}$, which in turn follows since both sides are closed and contain $\im(f) \otimes Z$ as a dense subset.
\end{proof}

\begin{lemma} \label{lem:equalimage}
Let $A$ be a regularised algebra, and let
	$M$ and $N$  be right and left $A$-modules and $U$ an $A$-$A$-bimodule.
Let $p,q\in (\Rb_{>0})^2$ arbitrary and set
	$\varphi_p:=\rho^M_{p}\otimes Q_{p}^N -Q_{p}^M\otimes\rho^N_{p}$.
	If $Q_r^N$ and $Q_r^M$ are epi for every $r\in (\Rb_{>0})^2$, 
	then $\overline{\im(\varphi_p)}=\overline{\im(\varphi_q)}$.
\end{lemma}
\begin{proof}
	Let $p:=(p_1,p_2)$ and $q:=(q_1,q_2)$.
	It is enough to show that $\overline{\im(\varphi_p)}=\overline{\im(\varphi_q)}$ 
	in the case when $p_1>q_1$ and $p_2>q_2$.  Then we have that
	\begin{align}
		\varphi_p&=\varphi_q\circ\left( Q^M_{p-q}\otimes\id_A\otimes Q^N_{p-q} \right)\ ,
		\label{eq:phi-phi-q}
	\end{align}
	from which we directly get that $\im(\varphi_p)\subset\im(\varphi_q)$.
	
	Now we show that $\im(\varphi_q)\subset \overline{\im(\varphi_p)}$.
	We write $R:=Q^M_{p-q}\otimes\id_A\otimes Q^N_{p-q}$ and choose an arbitrary
	$y\in M\otimes A\otimes N$. 
Let $x = \varphi_q(y)$.
	By Lemma~\ref{lem:monoepi}, $R$ is epi, 
so we can choose a sequence $(z_n)_{n\in\Nb}$ in $M\otimes A\otimes N$, for which
$\lim_{n\to\infty}R(z_n)=y$.
Applying $\varphi_q$ to both sides gives
$\lim_{n\to\infty}\varphi_p(z_n)=x$, and thus $x \in \overline{\im(\varphi_p)}$.
\end{proof}

\begin{proposition}\label{prop:hilb-tensor-prod-MNU}
Let $A$ be a regularised algebra, and let
$M$ and $N$  be right and left $A$-modules and $U$ an $A$-$A$-bimodule.
If $Q_{l,b}^M$, $Q_{a,l}^N$ and $Q_{a,l,b}^U$ are epi for every $a,l,b\in\Rb_{>0}$ 
then the following tensor products exist:
	\begin{align}
		M\otimes_A N\ ,\quad \ctimes_A U\ .
		\label{eq:hilb-tensor-prod-1}
	\end{align}
\end{proposition}
\begin{proof}
	Let us use the notation of Lemma~\ref{lem:equalimage}.
	Let $\pi$ be the projection 
	\begin{align}
		M\otimes N\to M\otimes N/\overline{\im(\varphi_p)}
		\label{eq:tensor-product-reg-alg-projector}
	\end{align}
	for some $p\in(\Rb_{>0})^2$,
	which is independent of $p$ by Lemma~\ref{lem:equalimage}.
	But this means exactly that $\pi$ is the cokernel of $\varphi_p$ for every $p\in(\Rb_{>0})^2$, and is hence a tensor product $M\otimes_A N$.
	
	A similar argument shows that $\ctimes_A U$ exists.
\end{proof}

\begin{lemma}
	If $V\in\Hilb$ is a left $B$-module and a right $A$-module such that
	the two actions commute as in \eqref{eq:ra:bimod}
	then it is a $B$-$A$-bimodule via $\rho_{a,l,b}^V$ as in \eqref{eq:ra:bimod}.
	\label{lem:left-right-module-bimodule}
\end{lemma}
	\begin{proof}
The algebraic conditions are clear, and it remains to show
that the two sided action $\rho_{a,l,b}^V$ is jointly continuous in all three parameters.
		Since the composition is separately continuous 
		and we have $Q_{a,l,b}^V\circ\rho_{a',l',b'}^V=\rho_{a+a',l+l',b+b'}^V$,
		it is enough to show that $Q_{a,l,b}^V$ is jointly continuous in all 3 parameters.
		Let $\rho_{a,l}^L$ be the left action and $\rho_{l,b}^R$ be the right action.
		Then we have $Q_{a,l,b}^V=Q_{a,l_1}^L\circ Q_{l_2,b}^R$ for any $a,b,l,l_1,l_2\in\Rb_{>0}$ with $l_1+l_2=l$.

		Let $\varepsilon>0$ and $v\in V$. We show that $Q_{a,l,b}^V$ is continuous at $(a_0,l_0,b_0)\in(\Rb_{>0})^3$.
Let us fix $0<l_0'<l_0$. For $l>l_0'$ we have the estimate
		\begin{align}
			\begin{aligned}
				\norm{(Q_{a,l,b}^V-Q_{a_0,l_0,b_0}^V)v}
				&=
				\norm{\left((Q_{a,l-l_0'}^L-Q_{a_0,l_0-l_0'}^L)Q_{l_0',b_0}^R +
				 Q_{a,l-l_0'}^L (Q_{l_0',b}^R-Q_{l_0',b_0}^R) \right)v}
\\
				&\le
\norm{(Q_{a,l-l_0'}^L-Q_{a_0,l_0-l_0'}^L)Q_{l_0',b_0}^R v}+
\norm{Q_{a,l-l_0'}^L}\cdot\norm{(Q_{l_0',b}^R-Q_{l_0',b_0}^R)v}\ .
			\end{aligned}
			\label{eq:joint-cont-estimate}
		\end{align}
Using the joint continuity of $Q_{a,l}^L$ at the point $(a_0,l_0-l_0')$ we can find $\delta_1>0$ such that
for every $a,l\in\Rb_{>0}$ with $|a-a_0|+|(l-l_0')-(l_0-l_0')|<\delta_1$ the first term in the second line of \eqref{eq:joint-cont-estimate} is smaller than $\varepsilon/2$.
For $a,l\in\Rb_{>0}$ with $|a-a_0|+|l-l_0|\le \delta_1$ and $l_0'\le l$ 
there exists a $K>0$ such that $\norm{Q_{a,l-l_0'}^L}<K$ since $(a,l)\to\norm{Q_{a,l}^L}$ is continuous.
Finally since $Q_{l_0',b}^R$ is continuous in $b$ we can choose $\delta_2>0$ such that $\norm{(Q_{l_0',b}^R-Q_{l_0',b_0}^R)v}<\varepsilon/(2K)$ 
		for every $b\in\Rb_{>0}$ with $|b-b_0|<\delta_2$.
		Altogether we have that $\norm{(Q_{a,l,b}^V-Q_{a_0,l_0,b_0}^V)v}<\varepsilon$ for every $a,l,b\in\Rb_{>0}$ with 
	$|a-a_0|+|l-l_0|+|b-b_0|<\min\left\{ \delta_1,\delta_2,l_0-l_0' \right\}$.
	\end{proof}

	Recall that the converse statement of Lemma~\ref{lem:left-right-module-bimodule} is not true. In Appendix~\ref{app:bimod} we give an example of a bimodule in $\Hilb$ which is not a left module.

\begin{proposition}
	Let $A,B,C$ be regularised algebras in $\Hilb$, 
	$V$ a $B$-$A$-bimodule, $W$ an $A$-$C$-bimodule,
	both coming from left/right modules	
	as in Lemma~\ref{lem:left-right-module-bimodule}.
	If $Q_{b,l,a}^V$ and $Q_{a,l,c}^W$ are epi for every $a,b,c,l\in\Rb_{>0}$,
	then the tensor product of bimodules exists:
	\begin{align}
		V\otimes_A W\ .
		\label{eq:tens-prod-lr-bimodules}
	\end{align}
	\label{prop:tens-prod-lr-bimodules}
\end{proposition}
\begin{proof}
By the assumption $V$ is a right $A$-module and $W$ is a left $A$-module.
	Since 
	$Q_{b,l,a}^V$ and $Q_{a,l,c}^W$ are epi, 
	the $Q$'s for the corresponding right and left module structures on $V$ and $W$ are epi as well.
	Let $\pi:V\otimes W\to V\otimes_A W$ be the tensor product of these right and left modules, which exists by Proposition~\ref{prop:hilb-tensor-prod-MNU}.
	Since the left and right actions for $V$ and $W$ commute as in \eqref{eq:ra:bimod}, $\pi$ is a coequaliser for \eqref{eq:coeq-tensor-product}.

	By Lemma~\ref{lem:hilb-tensor-preserves-coker}, tensoring with identity preserves cokernels, so the universal property of the cokernel induces a morphism 
	$\bar{\rho}_{a,b,c,l}:B\otimes(V\otimes_A W)\otimes C\to V\otimes_A W$ from
\eqref{eq:def-tensor-product-action-before}. Since $V$ is a left $B$-module and $W$ is a right $C$ module, the morphism in \eqref{eq:def-tensor-product-action-before} with parameter $a=0$ exists and induces the morphism $\bar{\rho}_{0,b,c,l}$, which is clearly the $a\to0$ limit of $\bar{\rho}_{a,b,c,l}$ and which is clearly jointly continuous in its parameters. So altogether we have shown that $V\otimes_A W$ as the tensor product of bimodules exists.
\end{proof}

\medskip

For the rest of this section we restrict our attention to tensor products over strongly separable regularised algebras.

\begin{lemma}
Let $A,M,N,U$ be as in Lemma~\ref{lem:equalimage} and suppose that $A$ is strongly separable.
	If $Q_{a,l}^M$, $Q_{a,l}^N$ and  $Q_{a,l}^U$
	are epi for every $a,l,b\in\Rb_{>0}$, then
	the following limits of the maps in \eqref{eq:D-definition-bimodule} exist:
	\begin{align}
		\lim_{a,l\to0}D_{a,l}^{M,N}\ ,\quad\lim_{a,l\to0}D_{a,l}^{U}\ .
		\label{eq:D-limits}
	\end{align}
The above idempotents are projectors onto the respective tensor products.
\label{lem:D-limits}
\end{lemma}
\begin{proof}
We will only show that the first limit exists, the second
can be shown similarly.
Abbreviate $p=(a,l)$ and $D_p = D_{a,l}^{M,N}$.
Recall the morphism $\varphi_p$ from Lemma~\ref{lem:equalimage}.
Let us identify $M\otimes N/\overline{\im(\varphi_p)}$ with the orthogonal complement of $\overline{\im(\varphi_p)}$ in $M\otimes N$ and let $D_0$ be the projection onto this closed subspace.
	Since on this subspace the left and right actions are identified, one has that
	$D_p= D_0 \circ \left( Q_{p_1}^M\otimes Q_{p_2}^N \right)$, 
	for appropriate $p,p_1,p_2\in(\Rb_{>0})^2$,
	so $\lim_{p\to0}D_p=D_0$.
	By Proposition~\ref{prop:hilb-tensor-prod-MNU} we know that the image of $D_0$ is the tensor product $M\otimes_A N$.
\end{proof}
\begin{proposition}\label{prop:hilb-tensor-prod}
Let $A$ be a strongly separable algebra and let
$V$ be a $B$-$A$-bimodule and $W$ an $A$-$C$-bimodule,
such that $Q_{b,l,a}^V$ and $Q_{a,l,c}^W$ are epi for every $a,b,c,l\in\Rb_{>0}$.
\begin{enumerate}
\item
If the limit $\lim_{a\to0}\tilde{\rho}_{a,b,c,l}$ of the morphism in \eqref{eq:tensor-product-action} exists, then
the tensor product 
	\begin{align}
		V\otimes_A W
		\label{eq:hilb-tensor-prod-2}
	\end{align}
	exists	and is a $B$-$C$-bimodule via $\tilde{\rho}_{0,b,c,l}^{V,W}$ as in Proposition~\ref{prop:d0proj}.
\item
If $V$ and $W$ are transmissive then $V\otimes_A W$ exists and
is transmissive as well.
\end{enumerate}
\end{proposition}
\begin{proof}
Similarly as in Lemma~\ref{lem:D-limits},
	for $V$ and $W$ we again have that $V\otimes_A W$ exists as a cokernel.
	What is left to be shown is that we get an induced action on $V\otimes_A W$.
	By Lemma~\ref{lem:hilb-tensor-preserves-coker}, tensoring with identity preserves cokernels, so we get an induced morphism from \eqref{eq:def-tensor-product-action-before}, which coincides with the morphism in \eqref{eq:tensor-product-action}. By our assumptions the $a\to 0$ limit of this morphism exists.
	It can be shown by iterating Lemma~\ref{lem:semigrp} that 
	the action on $V\otimes_A W$ is jointly continuous in the parameters,
	so $V\otimes_A W$ is a bimodule.

	If the bimodules were transmissive then the morphism $\tilde{\rho}_{a,b,c,l}^{V,W}$ from \eqref{eq:tensor-product-action} depends only on $a+b+c$ and not on the differences of these parameters. In particular its $a\to0$ limit and hence the tensor product $V\otimes_A W$ exists and the tensor product is transmissive.
\end{proof}

We have seen two different conditions for the existence of tensor product of bimodules. In the state sum construction we will use both and our main examples will satisfy both of these conditions, too.
Note that the existence of tensor product does not automatically mean that it closes on bimodules with duals. For the natural candidate for the dual of $V\otimes_A W$ to exist one would need to establish the existence of the limits in \eqref{eq:dual-pair-tensor-product}.

\section{Area-dependent QFTs with and without defects as functors}\label{sec:aqft-with-and-without}

In this section we define the symmetric monoidal categories of two-dimensional bordisms with area, with and without defects. Using these, area-dependent QFTs are defined as symmetric monoidal functors from such bordisms to a suitable target category $\Sc$. 
In the case without defects we classify such functors in terms of commutative regularised Frobenius algebras in $\Sc$, mirroring the result for two-dimensional topological field theories.

\medskip

Below, by \textsl{manifold} we will always mean an oriented smooth manifold.

\subsection{Bordisms with area and {\aQFT}s}\label{sec:aqft}

We first recall the definition of the category of 2-dimensional oriented bordisms, 
and then extend this definition to include an assignment of an area to each connected component of a bordism. 
Using these notions we define area-dependent QFT as a symmetric monoidal functor with depends continuously on the area.

\medskip

Let $S$ be a compact closed 1-manifold. 
A \textsl{collar of $S$} is an open neighbourhood of $S$ in $S\times\Rb$.
An \textsl{ingoing (outgoing) collar of $S$} is the intersection of 
$S\times[0,+\infty)$ (respectively $S\times(-\infty,0]$) with
a collar of $S$. Let $S^{\mathrm{rev}}$ denote $S$ with the reversed orientation.
A \textsl{surface} is a compact 2-dimensional manifold.
A \textsl{boundary parametrisation} of a surface $\Sigma$ is:
\begin{enumerate}
	\item A pair of compact closed 1-manifolds 
		$S$ and $T$. 
	\item 
	A choice of an ingoing collar $U$ of $S$ and an outgoing collar $V$ of $T$. 	
	\item A pair of orientation preserving smooth	embeddings
\begin{align}
	\phi_\mathrm{in}:U\hookrightarrow\Sigma\hookleftarrow
	V:\phi_\mathrm{out} \ ,
	\label{eq:bdrparam-aqft}
\end{align}
	such that $\phi_\mathrm{in}\sqcup\phi_\mathrm{out}$ maps 
$(S\times\{0\})^{\mathrm{rev}}\sqcup (T\times\{0\})$ diffeomorphically to $\partial\Sigma$.
\end{enumerate}
For two compact closed 1-manifolds $S$, $T$, 	
	a \textsl{bordism} $\Sigma:S\to T$ is a surface $\Sigma$ together with a boundary parametrisation. The \textsl{in-out cylinder over $S$} 
is the bordism $S\times[0,1]:S\to S$.

Let $\Sigma:S\to T$ be a bordism as in \eqref{eq:bdrparam-aqft} and let
$\Sigma':S\to T$ with
\begin{align}
	\psi_{in}: X
	\hookrightarrow\Sigma'\hookleftarrow
	Y :\psi_{out} 
	\label{eq:bdrparam-2nd}
\end{align} 
be another bordism.
The two bordisms $\Sigma,\Sigma':S\to T$ are \textsl{equivalent} if there exist 
an orientation preserving diffeomorphism $f: \Sigma \to \Sigma'$, as well as
ingoing and outgoing collars $C$ and $D$
of $S$ and $T$ contained in the collars $U$, $X$ and $V$, $Y$, respectively, such that the diagram
\begin{equation}
	\begin{tikzcd}[row sep=small]
		& U \ar[hookrightarrow]{r}{\varphi_\mathrm{in}} & \Sigma \ar{dd}{f} &V\ar[left hook ->,swap]{l}{\varphi_\mathrm{out}}&\\
		C \ar[hookrightarrow]{ru} \ar[hookrightarrow]{rd} &&&& \ar[left hook ->]{lu} \ar[left hook ->]{ld} D\\
		& X \ar[hookrightarrow]{r}{\psi_\mathrm{in}} & \Sigma'&\ar[left hook ->,swap]{l}{\varphi_\mathrm{out}}Y&
	\end{tikzcd}
\end{equation}
commutes.

Given two bordisms $\Sigma :S \to T$ and $\Xi : T \to W$, we define 
	$\Xi \circ \Sigma : S \to W$ to be the surface glued using the boundary parametrisations
	$\phi_\mathrm{out}^{\Sigma}$ and $\phi_\mathrm{in}^{\Xi}$ and with 
	$\phi_\mathrm{in}^{\Sigma}$ and $\phi_\mathrm{out}^{\Xi}$ parametrising the remaining boundary.
	The composition $[\Xi]\circ[\Sigma]:=[\Xi\circ\Sigma]:S\to W$ is well defined, 
	that is, it is independent of the choice of representatives $\Xi$, $\Sigma$ of the classes to be glued.

The \textsl{category of bordisms} $\Bordna$ has compact closed
1-manifolds as objects and equivalence classes of bordisms as morphisms.

The category $\Bordna$ becomes a $\dagger$-category as follows.
Let the functor $\left( - \right)^{\dagger}:\Bordna\to\Bordna$ be
identity on objects. 
Let $S\in\Bordna$ and let us define the inversions
\begin{align}
\begin{aligned}
	\iota_{S}:S\times\Rb&\to S\times\Rb\\
	(s,t)&\mapsto(s,-t)\ .
\end{aligned}
	\label{eq:inversion-def}
\end{align}
Let $\Sigma:S\to T$ be a  bordism with boundary parametrisation maps 
as in \eqref{eq:bdrparam-aqft}. 
We define $\left( \Sigma \right)^{\dagger}:T\to S$ to be the bordism
\begin{align}
	\phi'_\mathrm{in}:=\phi_\mathrm{out}\circ\iota_T:\iota(V)\hookrightarrow\Sigma^{rev}\hookleftarrow
	\iota(U):\phi'_\mathrm{out}:=\phi_\mathrm{in}\circ\iota_U \ ,
	\label{eq:bordism-dagger}
\end{align}
with reversed orientation and new boundary parametrisation maps
$\phi'_\mathrm{in}$ and $\phi'_\mathrm{out}$.

\medskip

After this quick review we can introduce the bordism category we are interested in:

\begin{definition}
A \textsl{bordism with area} 
$(\Sigma,\Ac:\pi_0(\Sigma)\to\Rb_{\ge0}):S\to T$ consists of a bordism
$\Sigma:S\to T$ and 
an \textsl{area map} $\Ac$, which is only allowed to take value 0 
on connected components equivalent to in-out cylinders.
The value $\Ac(c)$ for $c\in\pi_0(\Sigma)$ is called the
\textsl{area} of the component $c$.
\end{definition}

Two bordisms with area $(\Sigma,\Ac),(\Sigma',\Ac'):S\to T$ 
are \textsl{equivalent} if
the underlying bordisms are equivalent
with diffeomorphism $f:\Sigma\to\Sigma'$ and if
the following diagram commutes:
\begin{equation}
	\begin{tikzcd}[row sep=small]
		\pi_0(\Sigma)\ar{dd}[swap]{f_*} \ar{dr}{\Ac}& \\
		 &\Rb_{\ge0}\ , \\
		 \pi_0(\Sigma') \ar{ru}[swap]{\Ac'} &
	\end{tikzcd}
	\label{eq:amapequiv}
\end{equation}
where $f_*:\pi_0(\Sigma)\to\pi_0(\Sigma')$ is the map induced by $f$.

\begin{remark}
Allowing zero area for connected components which are equivalent to in-out
cylinders will ensure that the category of bordisms with area defined below has identities. 
Allowing zero area for all surfaces,
in particular for ``in-in'' and ``out-out'' cylinders,
        would make state spaces of corresponding 
        area-dependent quantum field theories
        finite dimensional, see Remark~\ref{rem:aqft-all-zero-area-limits} below. 
        Requiring all surface components to have
        strictly positive area and adding identities to the category by hand
        would -- 
at least in the example that the area-dependent theory takes values in $\Hilb$ and under some natural assumptions        
        -- not give a richer theory.
        Hence we opted for the definition above.
\end{remark}

Given two bordisms with area $(\Sigma,\Ac_{\Sigma}) :X \to Y$ and $(\Xi,\Ac_{\Xi}) : Y \to Z$, 
the \textsl{glued bordism with area} 
$(\Xi \circ \Sigma,\Ac_{\Xi \circ \Sigma}) : X \to Z$ is the glued bordism
together with the new area map $\Ac_{\Xi \circ \Sigma}$ defined by
assigning to each new connected component
the sum of areas of the connected components which were glued together to 
build up the new connected component. 

Let $[(\Xi,\Ac_{\Xi})]:T\to W$ and $[(\Sigma,\Ac_{\Sigma})]:S\to T$ 
be equivalence classes of bordisms with area.
The composition $[(\Xi,\Ac_{\Xi})]\circ[(\Sigma,\Ac_{\Sigma})]:=[(\Xi\circ\Sigma,\Ac_{\Xi\circ\Sigma})]:S\to W$ is again independent of the choice of representatives $(\Xi,\Ac_{\Xi})$, 
$(\Sigma,\Ac_{\Sigma})$ of the classes to be glued.
In the following we will by abuse of notation write the same symbol 
$(\Sigma,\Ac)$ for a bordism with area $(\Sigma,\Ac)$ 
and its equivalence class $[(\Sigma,\Ac)]$.

\begin{definition}
The \textsl{category of bordisms with area} $\Bordarea$ has the same objects 
as $\Bordna$
and equivalence classes of bordisms with area as morphisms.
\end{definition}

Both $\Bordna$ and $\Bordarea$ are symmetric monoidal categories with tensor product 
on objects and morphisms given by disjoint union.
The identities and the symmetric structure are given by equivalence classes of in-out cylinders
(with zero area).
There is a forgetful functor 
\be
F:\Bordarea\to\Bordna \ ,
\ee
which forgets the area map.

Next we introduce the following topology on hom-sets of $\Bordarea$.
Fix a bordism $\Sigma:S\to T$ in $\Bordna$. 
Define the subset $U_\Sigma \subset \Bordarea(S,T)$ as
\be
	U_{\Sigma}
	~:=~ F^{-1}(\Sigma)
	~=~ \big\{ (\Sigma,\Ac) \,\big|\, \Ac : \pi_0(\Sigma) \to \Rb_{\ge 0} \big\}
	~\cong~ \left( \Rb_{>0} \right)^{N_n}\times
\left( \Rb_{\ge0} \right)^{N_c} \ ,
\ee
where $N_c$ is the number of 
connected components of $\Sigma$ equivalent to a cylinder over a connected 1-manifold and $N_n=|\pi_0(\Sigma)|-N_c$. 
The topology on $U_\Sigma$ is that of 
$\left( \Rb_{>0} \right)^{N_n}\times\left( \Rb_{\ge0} \right)^{N_c}$.
We define the topology on
$\Bordarea(S,T)$ to be the disjoint union topology of the sets $U_{\Sigma}$.
One can quickly convince oneself of the following fact:

\begin{lemma}
The composition and the tensor product of $\Bordarea$ are jointly continuous.
\end{lemma}

\medskip

After these preparations we can finally define:

\begin{definition}
	Let $\Sc$ be a symmetric monoidal category whose hom-sets 
	are topological spaces and whose composition is separately continuous.
	An \textsl{area-dependent quantum field theory with values in $\Sc$} or 
	\textsl{{\aQFT}} in short is a symmetric monoidal functor
	$\funZ:\Bordarea\to \Sc$, such that for every $S,T\in\Bordarea$ the map
	\begin{align}
		\funZ_{S,T}:\Bordarea(S,T)&\to
		\Sc(\funZ(S),\funZ(T))
		\label{eq:aqft:cont}\\
		(\Sigma,\Ac)&\mapsto \funZ(\Sigma,\Ac)\nonumber
	\end{align}
	is continuous.
\end{definition}

The continuity requirement can equivalently be stated as follows.
For every bordism $\Sigma:S\to T$ in $\Bordna$, the map 
\begin{align}
	 U_{\Sigma} \cong \left( \Rb_{>0} \right)^{ N_n}\times
	\left( \Rb_{\ge0} \right)^{ N_c} &\to \Sc(\funZ(S),\funZ(T))
	\label{eq:aqft:cont2}\\
	\left( \Ac(x) \right)_{x\in\pi_0(\Sigma)}&\mapsto
	\funZ(\Sigma,\Ac)\nonumber
\end{align}
is continuous.

The following lemma shows that it is enough to require this continuity condition to hold for cylinders with area.
The proof is similar to the proof of Part~\ref{lem:ra:generatedcat} of Lemma~\ref{lem:ra:properties} and we omit it.

\begin{lemma} \label{lem:enough-to-check-cylinders}
	Let $\funZ:\Bordarea\to\Sc$ be a symmetric monoidal functor and 
	for every $S\in\Bordarea$ let 
	$(S\times[0,1],\Ac)$ denote a cylinder with area.
	If for every $S\in\Bordarea$ the assignment
	\begin{align}
		\begin{aligned}
			(\Rb_{\ge0})^{|\pi_0(S)|}&\to\Sc(\funZ(S),\funZ(S))\\
			(\Ac(x))_{x\in\pi_0(S)}&\mapsto \funZ(S\times[0,1],\Ac) ,
		\end{aligned}
		\label{eq:lem:enough-to-check-cylinders}
	\end{align}
	is continuous, then $\funZ$ is an {\aQFT}.
\end{lemma}

{\aQFT}s together with natural
transformations form a category $\AQFT{\Sc}$.
Assume that for $\funZ_1,\funZ_2\in\AQFT{\Sc}$ 
and all $S,T \in \Bordarea$ the map
\begin{align}
	(\funZ_1\otimes \funZ_2)_{S,T}:\Bordarea(S,T)&\to
	\Sc(\funZ_1(S)\otimes \funZ_2(S),\funZ_1(T)\otimes \funZ_2(T))
	\label{eq:aqft:tensor-cont}\\
	(\Sigma,\Ac)&\mapsto \funZ_1(\Sigma,\Ac)\otimes \funZ_2(\Sigma,\Ac)\nonumber
\end{align}
is continuous.  Then \eqref{eq:aqft:tensor-cont} defines an {\aQFT} which we denote with $\funZ_1\otimes \funZ_2$.
If the continuity condition
\eqref{eq:aqft:tensor-cont} holds for every $\funZ_1,\funZ_2\in\AQFT{\Sc}$ then $\AQFT{\Sc}$ becomes a symmetric monoidal category.
For example, combining 
Lemma~\ref{lem:semigrp} 
	and Lemma~\ref{lem:enough-to-check-cylinders}
	we see that $\AQFT{\Hilb}$ is symmetric monoidal.

The category $\Bordarea$ becomes a $\dagger$-category via the functor
which is the same as \eqref{eq:bordism-dagger} on the bordisms and which does not change the area maps.
Following the terminology of \cite[Sec.\,5.2]{Turaev:1994qi} we define:

\begin{definition}
	Let us assume that $\Sc$ is a $\dagger$-category.
	We call an {\aQFT} $\funZ:\Bordarea\to\Sc$ \textsl{Hermitian}, if 
	the diagram 
\begin{equation}
	\begin{tikzcd}
		\Bordarea\ar{d}[swap]{(-)^{\dagger}}\ar{r}{\funZ}&\Sc\ar{d}{(-)^{\dagger}}\\
		\Bordarea\ar{r}{\funZ}&\Sc
	\end{tikzcd}
	\label{eq:hermitean-aqft}
\end{equation}
	commutes.

\end{definition}

\subsection{Equivalence of {\aQFT}s and commutative RFAs}

Let $\Sbb_{n,m}:(\Sb)^{\sqcup m}\to(\Sb)^{\sqcup n}$ denote the $(n+m)$-holed sphere with $m$ ingoing 
and $n$ outgoing boundary components and
let $(\Sbb_{n,m},a):(\Sb)^{\sqcup m}\to(\Sb)^{\sqcup n}$ denote the corresponding bordism with area $a$.
Let us consider the following family of bordisms:
\begin{align}
	\bar{\eta}_a:=(\Sbb_{1,0},a)\ ,\quad\bar{\varepsilon}_a:=(\Sbb_{0,1},a)\ ,\quad\bar{\mu}_a:=(\Sbb_{1,2},a)\ ,\quad\bar{\Delta}_a:=(\Sbb_{2,1},a)\ ,
	\label{eq:bord_generators}
\end{align}
for $a\in\Rb_{>0}$.
In addition, it is useful to set
\be\label{eq:P_a-in-Bordarea}
	\bar{P}_a := \left( \Sbb_{1,1},a \right) \ .
\ee

\begin{lemma} \label{lem:rfa-in-bord}
	The morphisms in \eqref{eq:bord_generators} endow $\Sb\in\Bordarea$
	with the structure of a commutative regularised Frobenius algebra in $\Bordarea$.
\end{lemma}

\begin{proof}
	Checking the algebraic relations \eqref{eq:ra:unit}, \eqref{eq:ra:assoc}, 
	\eqref{eq:rca:counit}, \eqref{eq:rca:coassoc} and \eqref{eq:rfa:frobrel} of an RFA
	and commutativity is analogous to the case of ordinary Frobenius algebras, see e.g.\ \cite[Sec.\,3.1]{Kock:2004fa}.
	The morphism $P_a$ from 
	Part~\ref{def:reg-alg:2} in 
	Definition~\ref{def:reg-alg} 
	is now given by $\bar P_a$ in \eqref{eq:P_a-in-Bordarea}.	
	
	The limit $\lim_{a\to0}\bar{P}_a=\id_{\Sb}$ is immediate as
the identities in $\Bordarea$ are cylinders with 0 area.
	The continuity condition in \eqref{eq:racont} follows equally directly	
	from the definition of the topology on hom-sets in $\Bordarea$.
\end{proof}

{}From the above lemma it is maybe not surprising that {\aQFT}s with values in $\Sc$ are in one-to-one correspondence with
commutative RFAs in $\Sc$, in complete analogy to topological field theory. To give the precise statement and the equivalence functors, we need to introduce some notation.

Let $A\in\Sc$ be a commutative RFA. 
Let $a\in\Rb_{>0}$, $\mu_a^{(0)}:=\eta_a$, $\mu_a^{(1)}:=P_a$, $\mu_a^{(2)}:=\mu_a$ and for $n\ge3$ 
\begin{align}
	\mu_a^{(n)}:=\mu_{a/2}^{(n-1)}\circ(\id_{A^{\otimes(n-2)}}\otimes\mu_{a/2})\ .
	\label{eq:nfold-mult}
\end{align}
Let $\Delta_a^{(0)}:=\eps_a$, $\Delta_a^{(1)}:=P_a$, $\Delta_a^{(2)}:=\Delta_a$ and for $n\ge3$ 
\begin{align}
\Delta_a^{(n)}:=(\id_{A^{\otimes(n-2)}}\otimes\Delta_{a/2})\circ\Delta_{a/2}^{(n-1)}\ \ .
	\label{eq:nfold-comult}
\end{align}
We will use the same notation for the structure maps 
\eqref{eq:bord_generators} of the commutative RFA $\Sb\in\Bordarea$.

For an object $S\in\Bordarea$ let
\begin{align}
	\funZ^A(S):=\bigotimes_{x\in \pi_0(S)}A^{(x)},
	\label{eq:ZA-on-objects}
\end{align}
where for every $x\in \pi_0(S)$ $A^{(x)}=A$ and the superscript keeps track of tensor factors.

Let $S,T\in\Bordarea$, $(\Sigma_{g,b},a)\in\Bordarea(S,T)$ a connected bordism with area $a$
whose underlying surface is of genus $g$ and has $b=|\pi_0(S)|+|\pi_0(T)|$ many boundary components.
We say that $(\Sigma_{g,b},a)$ is of \textsl{normal form}, if
\begin{align}
	(\Sigma_{g,b},a)=\left[ S\xrightarrow{\psi_S}(\Sb)^{\sqcup|\pi_0(S)|}\xrightarrow{\bar{\mu}_{a_1}^{(|\pi_0(S)|)}}\Sb\xrightarrow{(\bar{\Delta}_{a_2/(2g)}\circ\bar{\mu}_{a_2/(2g)})^g}\Sb\xrightarrow{\bar{\Delta}_{a_3}^{(|\pi_0(T)|)}}(\Sb)^{\sqcup|\pi_0(T)|}\xrightarrow{\psi_T}T \right]
	\label{eq:normal-form}
\end{align}
for some $a_1,a_2,a_3\in\Rb_{>0}$ such that $a_1+a_2+a_3=a$ and 
orientation preserving diffeomorphisms
$\psi_S$ and $\psi_T$.
Every connected bordism with area is equivalent to 
a bordism with area of normal form with the same area.
By forgetting about the area, this is the normal form for ordinary bordisms,
see e.g.\ \cite[Sec.\,1.4.16]{Kock:2004fa}.
Let us pick a representative of $(\Sigma_{g,b},a)$ which is of normal form and let
\begin{align}
	\funZ^A(\Sigma_{g,b},a):=
	\left[ \funZ^A(S)\xrightarrow{\Psi_S}A^{\otimes|\pi_0(S)|}\xrightarrow{\mu_{a_1}^{(|\pi_0(S)|)}}A\xrightarrow{(\Delta_{a_2/(2g)}\circ\mu_{a_2/(2g)})^g}A\xrightarrow{\Delta_{a_3}^{(|\pi_0(T)|)}}A^{\otimes|\pi_0(T)|}\xrightarrow{\Psi_T}\funZ^A(T) \right]\ ,
	\label{eq:ZA-on-conn-morphisms}
\end{align}
where $\Psi_S$ and $\Psi_T$ denote the permutation of 
tensor factors induced by the bijections $\psi_S$ and $\psi_T$ respectively.
For $g=0$ the morphisms in the middle of the compositions in 
\eqref{eq:normal-form} and \eqref{eq:ZA-on-conn-morphisms} are $\id_{\Sb}$ and $\id_A$ respectively.
For a bordism with area $(\Sigma,\Ac)$, where $\Sigma$ is not necessarily connected we define 
\begin{align}
	\funZ^A(\Sigma,\Ac):=\otimes_{c\in\pi_0(\Sigma)}\funZ^A(c,\Ac(c))\ .
	\label{eq:ZA-on-morph}
\end{align}

\begin{lemma}\label{lem:aqft-crfa}\leavevmode
	\begin{enumerate}
		\item Let 
			$\funZ\in\AQFT{\Sc}$
			Then $\funZ(\Sb)$, with structure maps given by 
			the images of the bordisms \eqref{eq:bord_generators} under $\funZ$,			
			is a commutative RFA.\label{lem:aqft-crfa:1}
			
		\item Let $A\in\Sc$ be a commutative RFA. Then the assignments in
			\eqref{eq:ZA-on-objects} and \eqref{eq:ZA-on-morph} define an {\aQFT}
			$\funZ^A$.\label{lem:aqft-crfa:2}
	\end{enumerate}
\end{lemma}
\begin{proof}~\\
\noindent
	\textsl{Part~\ref{lem:aqft-crfa:1}} 
	follows directly from Lemma~\ref{lem:rfa-in-bord}
	and the continuity condition for $\funZ$.

\medskip
	
\noindent
	\textsl{Part~\ref{lem:aqft-crfa:2}:}
	Proving that $\funZ^A$ is a symmetric monoidal functor is similar to the
	case of topological field theories $\Bordna\to\Sc$ \cite[Thm.\,3]{Abrams:1996fa},
	and we omit it.
	The continuity condition \eqref{eq:aqft:cont} amounts to
	Lemma~\ref{lem:ra:properties} Part~\ref{lem:ra:generatedcat}.
\end{proof}

Now consider the functor 
\begin{align}
\begin{aligned}
	G: \AQFT{\Sc} &\to\cRFrob{\Sc}\\ T&\mapsto T(\Sb), \\
	\left(T\xrightarrow{\theta}T'\right)&\mapsto
	\left( T(\Sb)\xrightarrow{\theta_{T(\Sb)}}T'(\Sb) \right).
\end{aligned}\label{eq:aqftrfaequiv}
\end{align}

\begin{theorem}
	The functor $G$ defined in \eqref{eq:aqftrfaequiv} is an equivalence of categories.
	\label{thm:aqftrfaequiv}
\end{theorem}

\begin{proof}
	Let the inverse functor $H$ 
	be given by the assignments in Lemma~\ref{lem:aqft-crfa}
	Part~\ref{lem:aqft-crfa:2}. Then it is easy to see that 
	$G\circ H\cong \id_{\cRFrob{\Sc}}$.
	The rest of the proof is very similar to the proof of \cite[Thm.\,3]{Abrams:1996fa},
	on the equivalence between 2-dimensional topological field theories and commutative Frobenius algebras, and we omit it.
\end{proof}

\begin{remark}\label{rem:aqft-all-zero-area-limits}
	If all zero area limits of 
	$\funZ\in\AQFT{\Hilb}$ 
	exist, then the RFA $\funZ(\Sb)$ 
	is finite dimensional. This follows from 
	Theorem~\ref{thm:aqftrfaequiv} and Proposition~\ref{prop:rfafindim}.
\end{remark}

\begin{proposition}\label{prop:aqft-sym-mon-cat}
	Assume that $\Sc$ is a symmetric monoidal category and that 
	the conditions of Proposition~\ref{prop:rfa-tensorprod} hold for every pair 
	$A_1,A_2\in\cRFrob{\Sc}$. Then
	\begin{enumerate}
		\item the categories $\cRFrob{\Sc}$ and $\AQFT{\Sc}$ 
			are symmetric monoidal,
		\item the functor $G$ in \eqref{eq:aqftrfaequiv} is an equivalence of symmetric monoidal categories.
	\end{enumerate}
\end{proposition}
\begin{proof}
	As we already discussed after Proposition~\ref{prop:rfa-tensorprod}, $\cRFrob{\Sc}$ is a symmetric monoidal category.
	The equivalence from Theorem~\ref{thm:aqftrfaequiv} shows that the tensor product of {\aQFT}s in \eqref{eq:aqft:tensor-cont} equally satisfies the continuity condition. Hence $\AQFT{\Sc}$ 
is monoidal (and clearly symmetric).
It is easy to see that the equivalence $G$ is symmetric monoidal.
\end{proof}

Combining the above proposition with Proposition~\ref{prop:rfob-hilb-vect-mon-cat}, 
we get:

\begin{corollary}
The categories $\AQFT{\Vectfd}$ and $\AQFT{\Hilb}$ are symmetric monoidal.
\end{corollary}

\begin{corollary}
	The restriction of the functor $G$ in \eqref{eq:aqftrfaequiv}
	to the category of Hermitian {\aQFT}s with values in $\Sc$ 
	gives an equivalence to the category of $\dagger$-RFAs in $\Sc$.
	\label{cor:hermaqftrfaequiv}
\end{corollary}

Corollary~\ref{cor:commdaggerclassification} together with 
Corollary~\ref{cor:hermaqftrfaequiv} 
shows that a Hermitian {\aQFT} in $\Hilb$ is determined by a countable family of numbers
$\left\{ \epsilon_i,\sigma_i \right\}_{i\in I}$
satisfying convergence conditions given in Corollary~\ref{cor:commdaggerclassification}.

\subsection{Bordisms and {\aQFT}s with defects}\label{sec:daqft}

We start by recalling some notions from field theories with defects  
\cite{Davydov:2011dt,Carqueville:2016def,Carqueville:2017orb}.
Let $D_1$ and $D_2$ denote sets, which we call 
\textsl{labels for defect lines} and \textsl{phases},
and $s,t:D_1\to D_2$ maps of sets which we call \textsl{source} and \textsl{target} respectively.
These maps describe the possible geometric configurations of defect lines and 
surface components,
which we will explain in the following in more detail.
We refer to this set of data as a \textsl{set of defect conditions} and write 
$\Db := (D_1,D_2,s,t)$.

\begin{figure}[tb]
	\centering
	\def\svgwidth{16cm}
	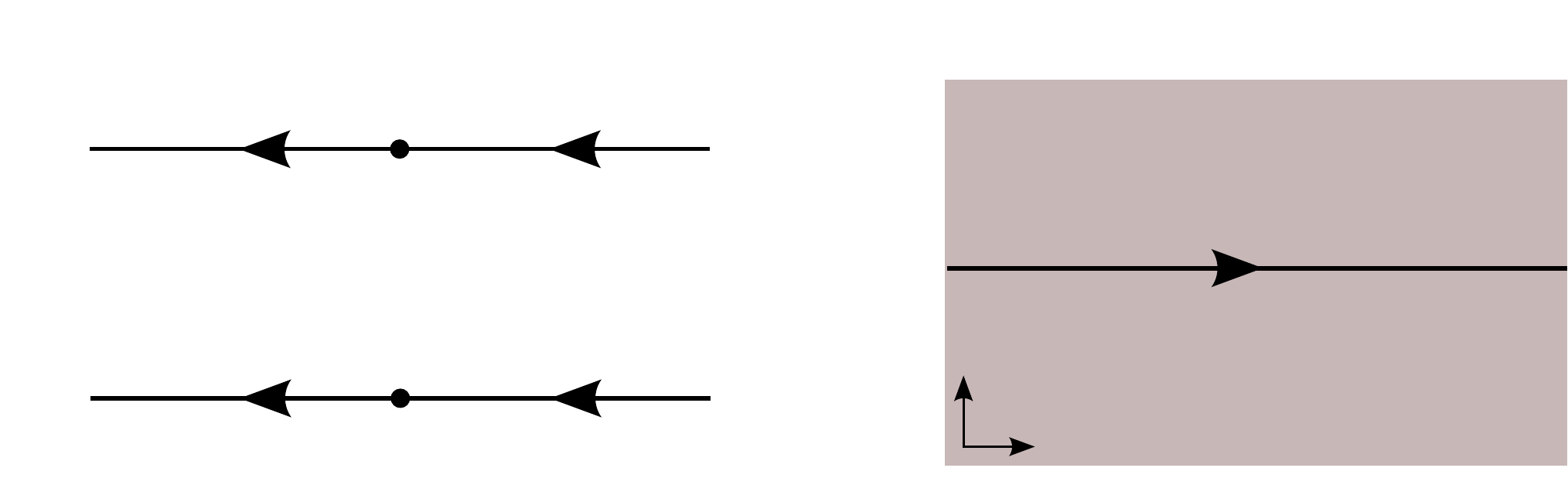
	\caption{
A neighbourhood of the submanifold $X_{[k-1]}$ in a $k$-manifold with defects.	
\\
$a)$ case $k=1$:	
	The arrows show the orientation of the 
	1-manifold $S$,
	$(p,+)$ denotes a positively oriented point $p\in S_{[0]}$ and $(p,-)$ denotes a negatively oriented point.
	These orientations allow us to define a left and right side $l_p,r_p\in \pi_0(S_{[1]})$ of $p$.
		We require for $(p,+)$
		that $t({d}_1(p)) ={d}_2(l_p)$ and that $s({d}_1(p)) ={d}_2(r_p)$
		and for $(p,-)$ that the $s$ and $t$ are exchanged:
		$s({d}_1(p)) ={d}_2(l_p)$ and $t({d}_1(p)) ={d}_2(r_p)$.
\\
$b)$ case $k=2$:	 The arrows marked with 1 and 2 show the orientation of the surface $\Sigma$,
	the arrow on the line shows the orientation of $c\in\pi_0(\Sigma_{[1]})$.
	The orientations of $\Sigma_{[1]}$ and $\Sigma_{[2]}$ allow us to define
	left and right side $l_c,r_c\in\pi_0(\Sigma_{[2]})$ of $c$.
	We require that for a defect line ${d}_1(c)$ the phase label
	on its right side is $s({d}_1(c))={d}_2(r_c)$ and
	that the phase label
	on its left side is $t({d}_1(c))={d}_2(l_c)$.
	}
	\label{fig:defcond2}
\end{figure}

Using a fixed set of defect conditions $\Db$ 
we introduce some notions.
Let $k\in\{1,2\}$. A \textsl{$k$-manifold with defects} is a compact $k$-manifold $X$, together with
	(see Figure~\ref{fig:defcond2})
\begin{enumerate}
	\item a finite decomposition into $(k-1)$- and $k$-dimensional submanifolds $X=X_{[k-1]}\cup X_{[k]}$ and 
	\item maps $d_l:\pi_0(X_{[k+l-2]})\to D_l$ for $l=1,2$,
\end{enumerate}
such that the following hold.
\begin{itemize}
	\item $X_{[k-1]}\cap X_{[k]}=\emptyset$,
	\item $X_{[k-1]}$ is an embedded oriented $(k-1)$-dimensional submanifold, which is either closed or $\partial X_{[k-1]}\subset \partial X$
	\item $X_{[k]}$ is a $k$-dimensional submanifold with orientation induced from $X$ and
	\item $d_1$ and $d_2$ are compatible with the maps $s$ and $t$ as shown in Figure~\ref{fig:defcond2}.
\end{itemize}
We call a closed 1-manifold with defects a \textsl{defect object}
and a 2-manifold with defects a \textsl{surface with defects}.
In particular, 
for a defect object $S$ the set $S_{[0]}$ is a finite set of distinct oriented points.
	For a surface with defects $\Sigma$,
	every connected component of $\Sigma_{[1]}$ is the image of a smooth embedding $[-1,1]\to\Sigma$ or $\Sb\to\Sigma$.

A \textsl{morphism of surfaces with defects} $f:\Sigma\to\Sigma'$
is an orientation preserving smooth map 
of surfaces such that the restrictions $f|_{\Sigma_{[k]}}$
map the submanifolds $\Sigma_{[k]}$ onto $\Sigma'_{[k]}$, they 
are diffeomorphisms onto their image, and they make the diagrams
\begin{equation}
\begin{tikzcd}
	\pi_0(  \Sigma_{[k]} )\ar{dr}[swap]{{d}_k}\ar{rr}{f_*}&&\pi_0( \Sigma'_{[k]} )\ar{dl}{{d}'_k}\\
	&D_k&
\end{tikzcd}
\label{eq:same-defect-labels}
\end{equation}
commute for $k=1,2$.

\begin{figure}[tb]
	\centering
	\def\svgwidth{7cm}
	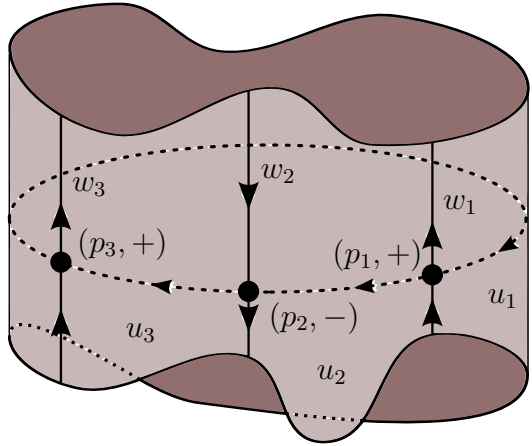
	\caption{A collar $C=C_{[1]}\cup C_{[2]}$ of $S$.
		The dotted circle in the middle shows $S\times\{0\}$ with its orientation, 
		the dots with labels $(p_i,\pm)$ for $i=1,\dots,3$ show $S_{[1]}$ with orientations,
		the straight lines with the arrows show the submanifold $C_{[1]}$ with its orientation.
		In the current figure
		both $C_{[1]}$ and $C_{[2]}$ have 3 connected components,
		the labels $w_i$ and $u_i$ for $i=1,\dots,3$ show 
		the values of ${d}_1$ and ${d}_2$ respectively.
	}
	\label{fig:collar}
\end{figure}

Let $S$ be a defect object.
A \textsl{collar of $S$} is a surface with defects $C=C_{[1]}\cup C_{[2]}$ such that
\begin{itemize}
	\item $C$ is an open neighbourhood of $S\times\{0\}$ in $S\times \Rb$ and
	\item $C_{[1]}$ is the intersection of $S_{[1]}\times \Rb$ with $C$ with orientation induced from the orientation of $S_{[1]}$ as shown in Figure~\ref{fig:collar},
	\item ${d}_k(c)=d_k\left( c \cap
	(S\times\{0\}) \right)$ for $c\in\pi_0(C_{[k]})$ and $k=1,2$.
\end{itemize}

An example of a collar is shown in Figure~\ref{fig:collar}.
An \textsl{ingoing (outgoing) collar with defects} is the intersection of 
a collar with defects and 
$S\times[0,+\infty)$ (respectively $S\times(-\infty,0]$).

A \textsl{boundary parametrisation} of a surface with defects $\Sigma$ consists of the following:
\begin{enumerate}
	\item A pair of defect objects $S$ and $T$.
	\item An ingoing collar $U$ of $S$ and an outgoing collar $V$ of $T$.
	\item A pair of morphisms of surfaces with defects
\begin{align}
	\phi_\mathrm{in}:U\hookrightarrow\Sigma\hookleftarrow
	V:\phi_\mathrm{out} \ ,
	\label{eq:bdrparam-maps}
\end{align}
We require that $\phi_\mathrm{in}\sqcup\phi_\mathrm{out}$ maps 
$(S\times\{0\})^{\mathrm{rev}}\sqcup T\times\{0\}$ diffeomorphically onto $\partial\Sigma$.
\end{enumerate}

A \textsl{bordism with defects} $\Sigma:S\to T$ is a surface $\Sigma$ together with a boundary parametrisation.
	The \textsl{in-out cylinder over $S$} is the bordism with defects $S\times[0,1]:S\to S$.
We define the equivalence of bordisms with defects similarly as in Section~\ref{sec:aqft},
now using diffeomorphisms of surfaces with defects that
are compatible with the boundary parametrisation on common collars of defect objects.
Given two bordisms with defects $\Sigma :S \to T$ and $\Xi : T \to W$, 
	we can glue them along the boundary parametrisations
	to obtain a bordism with defects $\Xi\circ\Sigma:S\to W$. 
	This glueing procedure is compatible with the above notion of equivalence.
	The \textsl{category of bordisms with defects} $\Borddefna{\Db}$ has defect objects
	as objects and equivalence classes of bordisms with defects as morphisms.

\medskip

After this preparation we turn to bordisms with area and defects.
\begin{definition}
A \textsl{bordism with area and defects} 
$(\Sigma,\Ac,\Lc):S\to T$ consists of a bordism with defects
$\Sigma:S\to T$,
an \textsl{area map} $\Ac:\pi_0(\Sigma_{[2]})\to\Rb_{\ge0}$
and a \textsl{length map} $\Lc:\pi_0(\Sigma_{[1]})\to\Rb_{\ge0}$,
which are only allowed to take value 0 
on connected components of $\Sigma$ equivalent to in-out cylinders with defects.
The value $\Ac(c)$ for $c\in\pi_0(\Sigma_{[2]})$ is called the
\textsl{area} of the component $c$
and the value of $\Lc(x)$ for $x\in\pi_0(\Sigma_{[1]})$ is called the
\textsl{length} of the defect line $x$.
\end{definition}

Two bordisms with area and defects $(\Sigma,\Ac,\Lc),(\Sigma',\Ac',\Lc'):S\to T$ 
are \textsl{equivalent} if
the underlying bordisms with defects are equivalent
with diffeomorphism $f:\Sigma\to\Sigma'$ and if
the following diagrams commute:
\begin{equation}
	\begin{tikzcd}[row sep=small]
		\pi_0(\Sigma_{[2]})\ar{dd}[swap]{f_*} \ar{dr}{\Ac}& \\
		 &\Rb_{\ge0} \\
		 \pi_0(\Sigma'_{[2]}) \ar{ru}[swap]{\Ac'} &
	\end{tikzcd}
	\quad\text{and}\quad
	\begin{tikzcd}[row sep=small]
		\pi_0(\Sigma_{[1]})\ar{dd}[swap]{f_*} \ar{dr}{\Lc}& \\
		 &\Rb_{\ge0}\ . \\
		 \pi_0(\Sigma'_{[1]}) \ar{ru}[swap]{\Lc'} &
	\end{tikzcd}
	\label{eq:amap-lmap-equiv}
\end{equation}

Given two bordisms with area and defects $(\Sigma,\Ac_{\Sigma},\Lc_{\Sigma}) :X \to Y$ and $(\Xi,\Ac_{\Xi},\Lc_{\Xi}) : Y \to Z$, 
the \textsl{glued bordism with area and defects} 
$(\Xi \circ \Sigma,\Ac_{\Xi \circ \Sigma},\Lc_{\Xi \circ \Sigma}
) : X \to Z$ 
is the the glued bordism with defects
together with the new area map $\Ac_{\Xi \circ \Sigma}$ defined by
assigning to each new connected component of $\left( \Xi \circ \Sigma \right)_{[2]}$
the sum of areas of the connected components which were glued together to 
build up the new connected component and with a similarly defined new length map
$\Lc_{\Xi \circ \Sigma}$.
	As before, 
this glueing procedure is compatible with the above notion of equivalence.

\begin{definition}
	The \textsl{category of bordisms with area and defects} 
	$\Borddef{\Db}$ 
	has the same objects
	as $\Borddefna{\Db}$
and equivalence classes of bordisms with area and defects as morphisms.
\end{definition}

Both $\Borddefna{\Db}$ and $\Borddef{\Db}$ are symmetric monoidal categories with tensor product 
on objects and morphisms given by disjoint union.
The identities and the symmetric structure are given by equivalence classes of in-out cylinders
(with zero area and length).

We introduce a similar topology on hom-sets of $\Borddef{\Db}$ as for $\Bordarea$ only that we now need to take into account the topology related to the lengths.
\begin{definition}\label{def:daqft}
	Let $\Sc$ be a symmetric monoidal category whose hom-sets 
	are topological spaces and composition is separately continuous.
	A \textsl{defect area-dependent quantum field theory with values in $\Sc$} (or 
	\textsl{defect {\aQFT}} for short) is a symmetric monoidal functor
	$\funZ:\Borddef{\Db} \to \Sc$, such that for every $S,T\in\Borddef{\Db}$ the map
	\begin{align}
		\funZ_{S,T}:\Borddef{\Db}(S,T)&\to
		\Sc(\funZ(S),\funZ(T))
		\label{eq:daqft:cont}\\
		(\Sigma,\Ac,\Lc)&\mapsto \funZ(\Sigma,\Ac,\Lc)\nonumber
	\end{align}
	is continuous.
\end{definition}

\begin{remark}
	Checking the continuity condition in \eqref{eq:daqft:cont} can be done by checking only for cylinders, similarly as in Lemma~\ref{lem:enough-to-check-cylinders} for {\aQFT}s without defects. To see this, one needs to cut surfaces with defects along circles which intersect with every defect line.
	\label{rem:daqft-enough-to-check-cylinders}
\end{remark}

\begin{figure}[tb]
	\centering
	\def\svgwidth{8cm}
	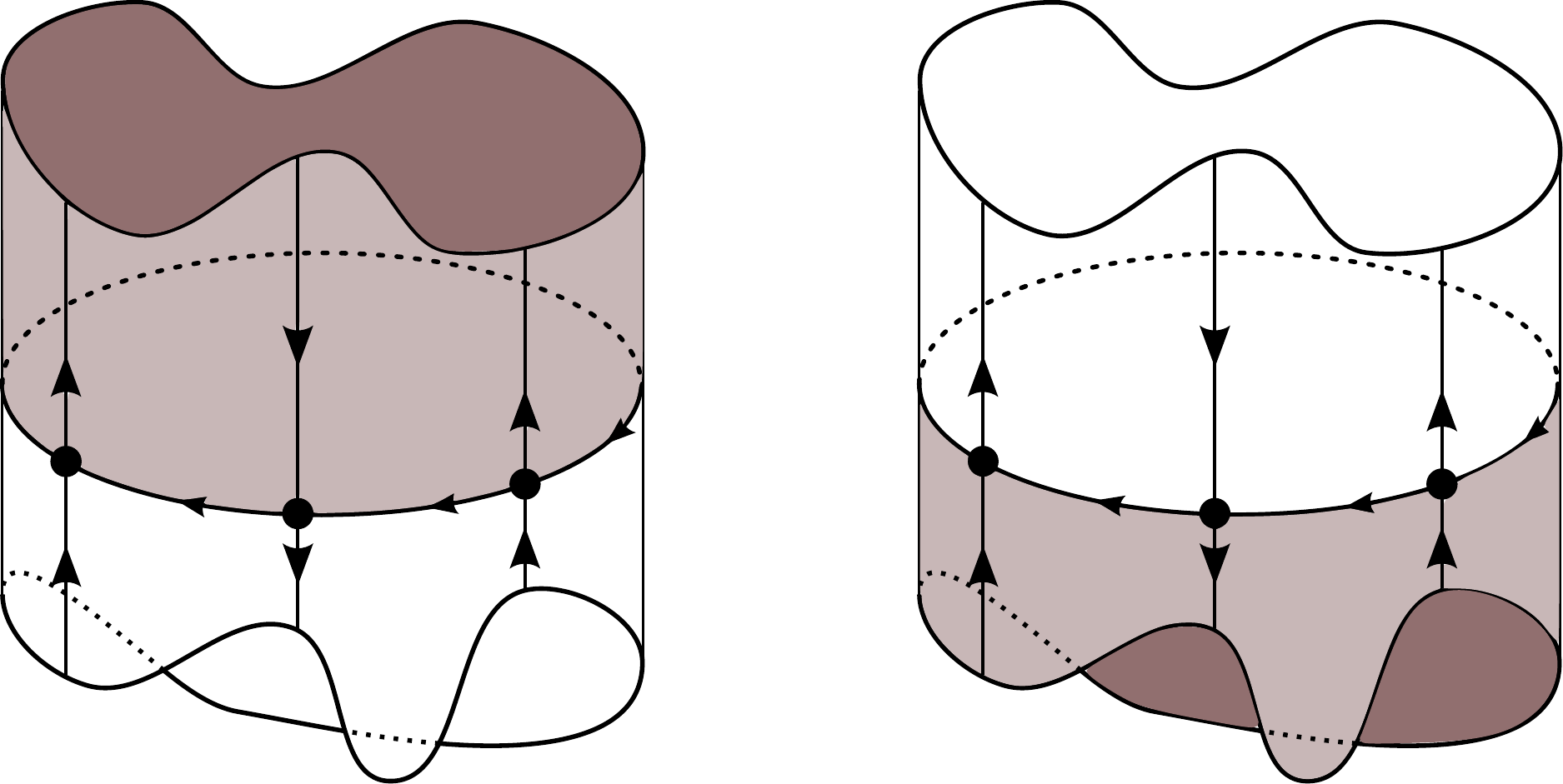
	\caption{Collars on the opposite sides of $\Sb\times\{0\}$.}
	\label{fig:dagger-collar}
\end{figure}

We turn the categories $\Borddefna{\Db}$ and$\Borddef{\Db}$ into $\dagger$-categories
in a similar way as $\Bordna$ and 
$\Bordarea$ in Section~\ref{sec:aqft}.
That is, if $M:S\to T$ is a bordism with area and defects, then $M^{\dagger}:T\to S$ is a bordism with area and defects with $(M^{\dagger})_{[k]}=M_{[k]}$ for $k=1,2$ with opposite orientation and with the same area maps and same defect labels. The boundary parametrisation is changed in the following way. The new collars are obtained from the old collars by extending the old ones and restricting them to the other side of $\Sb\times\{0\}$ as illustrated in Figure~\ref{fig:dagger-collar}. The boundary parametrisation maps are the old ones composed with the maps
$\iota_S$ and $\iota_T$ from \eqref{eq:inversion-def}.
We stress that in the definition of the dagger structure on $\Borddef{\Db}$ we have not included an involution on the set of defect labels $\Db$. This is important since we want the dagger to act as identity on objects. 
With these conventions it makes sense to consider $M^{\dagger}\circ M:T\to T$, which is relevant when considering reflection positivity, 
see e.g.\ \cite[Ch.\,6]{Glimm:1987qp}.
For a cylinder $C=S\times[0,1]$ we have that $C^{\dagger}=C$.

Let us assume that $\Sc$ is a dagger category.
We call a defect {\aQFT} $\funZ:\Borddef{\Db}\to\Sc$ \textsl{Hermitian} if 
it is compatible with the dagger structures.

In Section~\ref{sec:lattice:daqft} we give state-sum construction of defect {\aQFT}s, and in Section~\ref{sec:2dym:wilson} we discuss our main example, 2d~YM
theory with Wilson lines.

\section{State-sum construction of {\aQFT}s with defects}\label{sec:state-sum-construction}

The state-sum constructions of two-dimensional TFTs 
(see \cite{Bachas:1993lat,Fukuma:1994sts} and e.g.\ \cite{Lauda:2007oc,Davydov:2011dt}) has a straightforward generalisation to {\aQFT}s which we investigate in this section.
We start by giving the conditions on weights for plaquettes, edges and vertices in order to obtain state-sum {\aQFT} without defects, and we explain the relation of these weights to RFAs, as well as the connection to the classification of {\aQFT}s in terms of commutative RFAs assigned to $\Sb$ (Sections~\ref{sec:PLCW-dec-aqft}--\ref{sec:data-from-rfa}). Then we extend this state-sum construction to {\aQFT}s with defects and show that the weights for plaquettes traversed by defect lines are given by 
bimodules. We define the fusion of defect lines and show that it matches the tensor product of bimodules (Sections~\ref{sec:PLCW-dec-defect}--\ref{sec:fusion-of-defects}).

\subsection{PLCW decompositions with area}\label{sec:PLCW-dec-aqft}

In Section~\ref{sec:latticedata}
we will use \textsl{PLCW decompositions} \cite{Kirillov:2012pl} to build {\aQFT}s.
For a compact surface $\Sigma$ this consists of
three sets $\Sigma_0$, $\Sigma_1$ and $\Sigma_2$ whose elements are subsets of $\Sigma$.
Their elements 
are called \textsl{vertices}, \textsl{edges} and \textsl{faces}.
Faces are embeddings of polygons with $n\ge1$ edges, edges are embeddings of intervals and vertices are just points in $\Sigma$. 
Faces are glued along edges so that vertices are glued to vertices.
For example a PLCW decomposition of a cylinder $\Sb\times[0,1]$ could
consist of a rectangle with two opposite edges glued together.
{}From this one can obtain a PLCW decomposition of a torus $\Sb\times\Sb$ by glueing together the other two opposite edges.
For more details on PLCW decomposition we refer to \cite{Kirillov:2012pl} and for a short summary to \cite[Sec.\,2.2]{Runkel:2018:rs}.

We are going to need PLCW decompositions of surfaces with area, which we define now.
Let $(\Sigma,\Ac)$ be a surface with 
strictly positive area for each connected component and let 
$\Sigma_0$, $\Sigma_1$, $\Sigma_2$
be a PLCW decomposition of $\Sigma$. 
Let $\Ac_k:\Sigma_k\to\Rb_{>0}$ 
be maps for $k\in\{0,1,2\}$,
which assign to vertices, edges and faces an \textsl{area},
such that for every connected component $x\in\pi_0(\Sigma)$ the sum of the areas of
vertices, edges and faces of $x$ is equal to its area $\Ac(x)$.
A \textsl{PLCW decomposition of a surface with area} 
$(\Sigma,\Ac)$ 
consists of a choice of $\Sigma_k$ and $\Ac_k$ for $k\in\{0,1,2\}$.

	\begin{figure}[tb]
		\centering
		\def\svgwidth{16cm}
		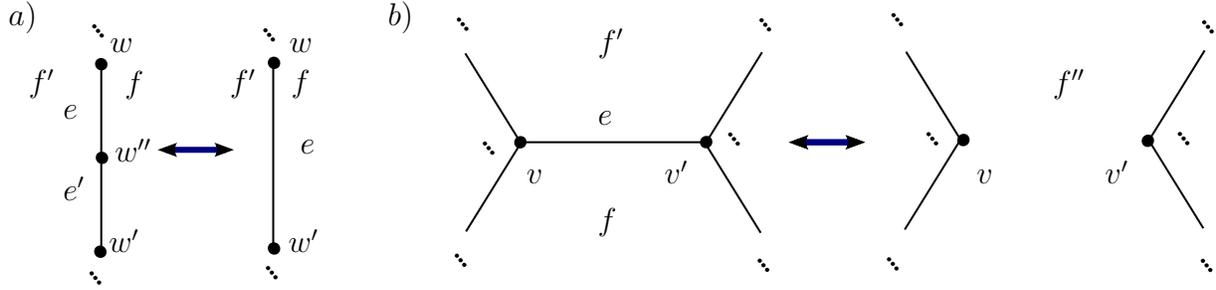
		\caption{Elementary moves of between PLCW decompositions with area.
			Figure~$a)$ shows edges $e$, $e'$ and between faces $f$ and $f'$.
			(The two faces are allowed to be the same.)
			When we remove the vertex $w''$ and the edge $e'$, the new area maps
			should be such that the area of the connected component of the surface does not change.
			In Figure~$b)$, shows an edge $e$ between two faces $f$ and $f'$.
			When we remove the edge $e$ and merge the faces $f$ and $f'$ to $f''$, the new area maps
			should be such that the area of the connected component of the surface does not change.
			}
		\label{fig:edgemove-aqft}
	\end{figure}

\begin{definition}
	An \textsl{elementary move} on a PLCW decomposition of a surface
	is either 
	\begin{itemize}
		\item removing or adding a bivalent vertex as shown in Figure~\ref{fig:edgemove-aqft}~$a$), or
		\item removing or adding an edge as shown in Figure~\ref{fig:edgemove-aqft}~$b$).
	\end{itemize}
\end{definition}

By \cite[Thm.\,7.4]{Kirillov:2012pl}, any two
PLCW decompositions can be
related by these elementary moves. 
The elementary moves in Figure~\ref{fig:edgemove-aqft} 
map PLCW decompositions with area to PLCW decompositions with area.

\subsection{State-sum construction without defects}
\label{sec:latticedata}

Let us fix a symmetric monoidal idempotent complete category $\Sc$ 
with symmetric structure $\sigma$ which has topological spaces as hom-sets 
and separately continuous composition of morphisms.

Let $A\in\Sc$ be an object and consider the following families of morphisms
\begin{align}
	\zeta_a\in\Sc(A,A)\ ,\quad\beta_a\in\Sc(A^{\otimes 2},\Ib)\quad\text{ and }\quad W_a^{n}\in\Sc(\Ib,A^{\otimes n})
	\label{eq:state-sum-data}
\end{align}
for $a\in\Rb_{>0}$ and $n\in\Zb_{\ge1}$.
We call $\beta_a$ the \textsl{contraction} and $W_a^{n}$ the \textsl{plaquette weights}.
We will use the following graphical notation for these morphisms:
\begin{align}
	\begin{aligned}
	\def\svgwidth{10cm}
\begingroup%
  \makeatletter%
  \providecommand\color[2][]{%
    \errmessage{(Inkscape) Color is used for the text in Inkscape, but the package 'color.sty' is not loaded}%
    \renewcommand\color[2][]{}%
  }%
  \providecommand\transparent[1]{%
    \errmessage{(Inkscape) Transparency is used (non-zero) for the text in Inkscape, but the package 'transparent.sty' is not loaded}%
    \renewcommand\transparent[1]{}%
  }%
  \providecommand\rotatebox[2]{#2}%
  \newcommand*\fsize{\dimexpr\f@size pt\relax}%
  \newcommand*\lineheight[1]{\fontsize{\fsize}{#1\fsize}\selectfont}%
  \ifx\svgwidth\undefined%
    \setlength{\unitlength}{319.06117344bp}%
    \ifx\svgscale\undefined%
      \relax%
    \else%
      \setlength{\unitlength}{\unitlength * \real{\svgscale}}%
    \fi%
  \else%
    \setlength{\unitlength}{\svgwidth}%
  \fi%
  \global\let\svgwidth\undefined%
  \global\let\svgscale\undefined%
  \makeatother%
  \begin{picture}(1,0.24653781)%
    \lineheight{1}%
    \setlength\tabcolsep{0pt}%
    \put(-0.00234146,0.11261132){\color[rgb]{0,0,0}\makebox(0,0)[lt]{\lineheight{0}\smash{\begin{tabular}[t]{l}$\zeta_a=$\end{tabular}}}}%
    \put(0.141225,0.12650677){\color[rgb]{0,0,0}\makebox(0,0)[lt]{\lineheight{0}\smash{\begin{tabular}[t]{l}\scriptsize{$a$}\end{tabular}}}}%
    \put(0,0){\includegraphics[width=\unitlength,page=1]{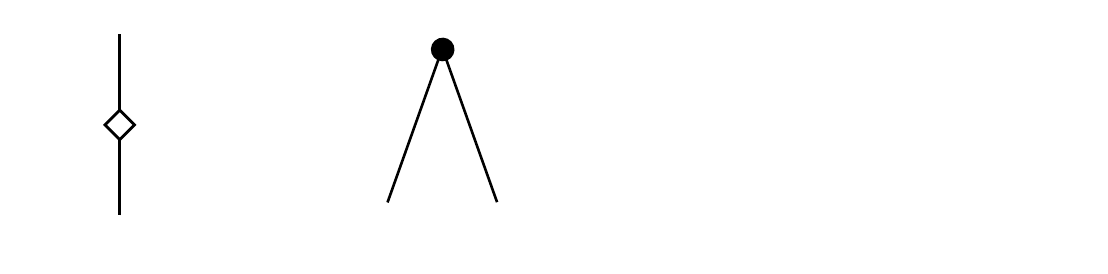}}%
    \put(0.42973087,0.19799442){\color[rgb]{0,0,0}\makebox(0,0)[lt]{\lineheight{0}\smash{\begin{tabular}[t]{l}\scriptsize{$a$}\end{tabular}}}}%
    \put(0.26366267,0.11261132){\color[rgb]{0,0,0}\makebox(0,0)[lt]{\lineheight{0}\smash{\begin{tabular}[t]{l}$\beta_a=$\end{tabular}}}}%
    \put(0,0){\includegraphics[width=\unitlength,page=2]{state-sum-data.pdf}}%
    \put(0.74405042,0.03998369){\color[rgb]{0,0,0}\makebox(0,0)[lt]{\lineheight{0}\smash{\begin{tabular}[t]{l}\scriptsize{$a;n$}\end{tabular}}}}%
    \put(0.76729096,0.11261132){\color[rgb]{0,0,0}\makebox(0,0)[lt]{\lineheight{0}\smash{\begin{tabular}[t]{l}$\dots$\end{tabular}}}}%
    \put(0.54453342,0.11261132){\color[rgb]{0,0,0}\makebox(0,0)[lt]{\lineheight{0}\smash{\begin{tabular}[t]{l}$W^n_a=$\end{tabular}}}}%
    \put(0.0913122,0.00414577){\color[rgb]{0,0,0}\makebox(0,0)[lt]{\lineheight{0}\smash{\begin{tabular}[t]{l}$A$\end{tabular}}}}%
    \put(0.33034781,0.02462999){\color[rgb]{0,0,0}\makebox(0,0)[lt]{\lineheight{0}\smash{\begin{tabular}[t]{l}$A$\end{tabular}}}}%
    \put(0.43612693,0.02462999){\color[rgb]{0,0,0}\makebox(0,0)[lt]{\lineheight{0}\smash{\begin{tabular}[t]{l}$A$\end{tabular}}}}%
    \put(0.0913122,0.22510653){\color[rgb]{0,0,0}\makebox(0,0)[lt]{\lineheight{0}\smash{\begin{tabular}[t]{l}$A$\end{tabular}}}}%
    \put(0.65546732,0.22040524){\color[rgb]{0,0,0}\makebox(0,0)[lt]{\lineheight{0}\smash{\begin{tabular}[t]{l}$A$\end{tabular}}}}%
    \put(0.69777895,0.22040524){\color[rgb]{0,0,0}\makebox(0,0)[lt]{\lineheight{0}\smash{\begin{tabular}[t]{l}$A$\end{tabular}}}}%
    \put(0.83411644,0.22040524){\color[rgb]{0,0,0}\makebox(0,0)[lt]{\lineheight{0}\smash{\begin{tabular}[t]{l}$A$\end{tabular}}}}%
  \end{picture}%
\endgroup%

	\end{aligned}
	\label{eq:state-sum-data-graphical}
\end{align}

We introduce the morphisms $P_a, D_a : A \to A$
in order to be able to state the conditions
these morphisms need to satisfy:
\begin{align}
	\begin{aligned}
	\def\svgwidth{10cm}
\begingroup%
  \makeatletter%
  \providecommand\color[2][]{%
    \errmessage{(Inkscape) Color is used for the text in Inkscape, but the package 'color.sty' is not loaded}%
    \renewcommand\color[2][]{}%
  }%
  \providecommand\transparent[1]{%
    \errmessage{(Inkscape) Transparency is used (non-zero) for the text in Inkscape, but the package 'transparent.sty' is not loaded}%
    \renewcommand\transparent[1]{}%
  }%
  \providecommand\rotatebox[2]{#2}%
  \newcommand*\fsize{\dimexpr\f@size pt\relax}%
  \newcommand*\lineheight[1]{\fontsize{\fsize}{#1\fsize}\selectfont}%
  \ifx\svgwidth\undefined%
    \setlength{\unitlength}{461.01708984bp}%
    \ifx\svgscale\undefined%
      \relax%
    \else%
      \setlength{\unitlength}{\unitlength * \real{\svgscale}}%
    \fi%
  \else%
    \setlength{\unitlength}{\svgwidth}%
  \fi%
  \global\let\svgwidth\undefined%
  \global\let\svgscale\undefined%
  \makeatother%
  \begin{picture}(1,0.16711139)%
    \lineheight{1}%
    \setlength\tabcolsep{0pt}%
    \put(0,0){\includegraphics[width=\unitlength,page=1]{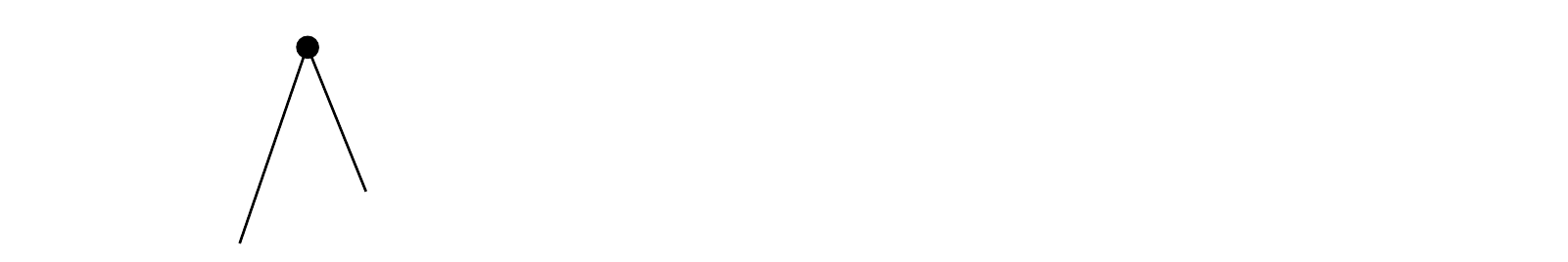}}%
    \put(-0.00162048,0.07588302){\color[rgb]{0,0,0}\makebox(0,0)[lt]{\lineheight{0}\smash{\begin{tabular}[t]{l}$P_{a_1+a_2}:=$\end{tabular}}}}%
    \put(0.21162474,0.12937897){\color[rgb]{0,0,0}\makebox(0,0)[lt]{\lineheight{0}\smash{\begin{tabular}[t]{l}\scriptsize{$a_1$}\end{tabular}}}}%
    \put(0,0){\includegraphics[width=\unitlength,page=2]{state-sum-pa.pdf}}%
    \put(0.23610754,0.01154444){\color[rgb]{0,0,0}\makebox(0,0)[lt]{\lineheight{0}\smash{\begin{tabular}[t]{l}\scriptsize{$a_2;2$}\end{tabular}}}}%
    \put(0,0){\includegraphics[width=\unitlength,page=3]{state-sum-pa.pdf}}%
    \put(0.36929852,0.07588302){\color[rgb]{0,0,0}\makebox(0,0)[lt]{\lineheight{0}\smash{\begin{tabular}[t]{l}and\end{tabular}}}}%
    \put(0.48643084,0.07588302){\color[rgb]{0,0,0}\makebox(0,0)[lt]{\lineheight{0}\smash{\begin{tabular}[t]{l}$D_{a_0+a_1+a_2+a_3}:=$\end{tabular}}}}%
    \put(0,0){\includegraphics[width=\unitlength,page=4]{state-sum-pa.pdf}}%
    \put(0.85119777,0.01156308){\color[rgb]{0,0,0}\makebox(0,0)[lt]{\lineheight{0}\smash{\begin{tabular}[t]{l}\scriptsize{$a_0;4$}\end{tabular}}}}%
    \put(0,0){\includegraphics[width=\unitlength,page=5]{state-sum-pa.pdf}}%
    \put(0.9454221,0.09516806){\color[rgb]{0,0,0}\makebox(0,0)[lt]{\lineheight{0}\smash{\begin{tabular}[t]{l}\scriptsize{$a_1$}\end{tabular}}}}%
    \put(0.87977288,0.15722327){\color[rgb]{0,0,0}\makebox(0,0)[lt]{\lineheight{0}\smash{\begin{tabular}[t]{l}\scriptsize{$a_2$}\end{tabular}}}}%
    \put(0.79796868,0.15507096){\color[rgb]{0,0,0}\makebox(0,0)[lt]{\lineheight{0}\smash{\begin{tabular}[t]{l}\scriptsize{$a_3$}\end{tabular}}}}%
  \end{picture}%
\endgroup%
\\
	\end{aligned}\ ,
	\label{eq:state-sum-data-notation}
\end{align}
for every $a_0,a_1,a_2,a_3\in\Rb_{>0}$.

Consider the following conditions on the morphisms in \eqref{eq:state-sum-data}:
for every $a,a_0,a_1,a_2,a_3\in\Rb_{>0}$,
and for every $n\in\Zb_{\ge1}$,
\begin{enumerate}
	\item 
		Cyclic symmetry:
\begin{align}
	\begin{aligned}
	\def\svgwidth{10cm}
\begingroup%
  \makeatletter%
  \providecommand\color[2][]{%
    \errmessage{(Inkscape) Color is used for the text in Inkscape, but the package 'color.sty' is not loaded}%
    \renewcommand\color[2][]{}%
  }%
  \providecommand\transparent[1]{%
    \errmessage{(Inkscape) Transparency is used (non-zero) for the text in Inkscape, but the package 'transparent.sty' is not loaded}%
    \renewcommand\transparent[1]{}%
  }%
  \providecommand\rotatebox[2]{#2}%
  \newcommand*\fsize{\dimexpr\f@size pt\relax}%
  \newcommand*\lineheight[1]{\fontsize{\fsize}{#1\fsize}\selectfont}%
  \ifx\svgwidth\undefined%
    \setlength{\unitlength}{345.0169988bp}%
    \ifx\svgscale\undefined%
      \relax%
    \else%
      \setlength{\unitlength}{\unitlength * \real{\svgscale}}%
    \fi%
  \else%
    \setlength{\unitlength}{\svgwidth}%
  \fi%
  \global\let\svgwidth\undefined%
  \global\let\svgscale\undefined%
  \makeatother%
  \begin{picture}(1,0.17237414)%
    \lineheight{1}%
    \setlength\tabcolsep{0pt}%
    \put(0,0){\includegraphics[width=\unitlength,page=1]{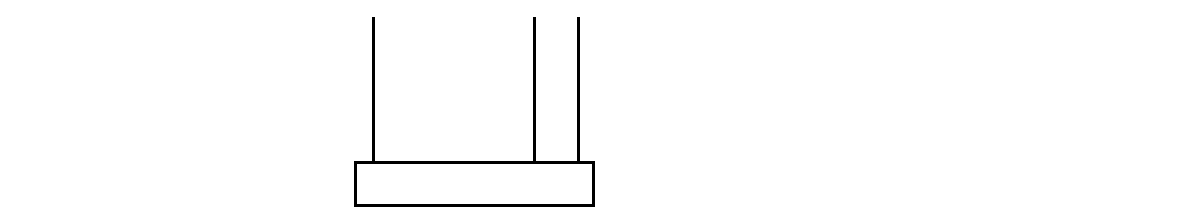}}%
    \put(0.37491691,0.01155902){\color[rgb]{0,0,0}\makebox(0,0)[lt]{\lineheight{0}\smash{\begin{tabular}[t]{l}\scriptsize{$a;n$}\end{tabular}}}}%
    \put(0.08813825,0.13930338){\color[rgb]{0,0,0}\makebox(0,0)[lt]{\lineheight{0}\smash{\begin{tabular}[t]{l}$\dots$\end{tabular}}}}%
    \put(0,0){\includegraphics[width=\unitlength,page=2]{cond-cyclicsymm.pdf}}%
    \put(0.0792793,0.01155902){\color[rgb]{0,0,0}\makebox(0,0)[lt]{\lineheight{0}\smash{\begin{tabular}[t]{l}\scriptsize{$a;n$}\end{tabular}}}}%
    \put(0.22269331,0.07872282){\color[rgb]{0,0,0}\makebox(0,0)[lt]{\lineheight{0}\smash{\begin{tabular}[t]{l}$=$\end{tabular}}}}%
    \put(0.05842369,0.04891384){\color[rgb]{0,0,0}\makebox(0,0)[lt]{\lineheight{0}\smash{\begin{tabular}[t]{l}$\dots$\end{tabular}}}}%
    \put(0.34370912,0.07872282){\color[rgb]{0,0,0}\makebox(0,0)[lt]{\lineheight{0}\smash{\begin{tabular}[t]{l}$\dots$\end{tabular}}}}%
    \put(0.56209687,0.07872282){\color[rgb]{0,0,0}\makebox(0,0)[lt]{\lineheight{0}\smash{\begin{tabular}[t]{l}and\end{tabular}}}}%
    \put(0,0){\includegraphics[width=\unitlength,page=3]{cond-cyclicsymm.pdf}}%
    \put(0.83078726,0.07872282){\color[rgb]{0,0,0}\makebox(0,0)[lt]{\lineheight{0}\smash{\begin{tabular}[t]{l}$=$\end{tabular}}}}%
    \put(0.76941983,0.15275187){\color[rgb]{0,0,0}\makebox(0,0)[lt]{\lineheight{0}\smash{\begin{tabular}[t]{l}\scriptsize{$a$}\end{tabular}}}}%
    \put(0.95303956,0.15275187){\color[rgb]{0,0,0}\makebox(0,0)[lt]{\lineheight{0}\smash{\begin{tabular}[t]{l}\scriptsize{$a$}\end{tabular}}}}%
  \end{picture}%
\endgroup%

	\end{aligned}\ .
	\label{eq:cond:cyclicsymm}
\end{align}
		\label{cond:cyclicsymm}
	\item
		Glueing plaquette weights:
\begin{align}
	\begin{aligned}
	\def\svgwidth{8cm}
\begingroup%
  \makeatletter%
  \providecommand\color[2][]{%
    \errmessage{(Inkscape) Color is used for the text in Inkscape, but the package 'color.sty' is not loaded}%
    \renewcommand\color[2][]{}%
  }%
  \providecommand\transparent[1]{%
    \errmessage{(Inkscape) Transparency is used (non-zero) for the text in Inkscape, but the package 'transparent.sty' is not loaded}%
    \renewcommand\transparent[1]{}%
  }%
  \providecommand\rotatebox[2]{#2}%
  \ifx\svgwidth\undefined%
    \setlength{\unitlength}{272.50478516bp}%
    \ifx\svgscale\undefined%
      \relax%
    \else%
      \setlength{\unitlength}{\unitlength * \real{\svgscale}}%
    \fi%
  \else%
    \setlength{\unitlength}{\svgwidth}%
  \fi%
  \global\let\svgwidth\undefined%
  \global\let\svgscale\undefined%
  \makeatother%
  \begin{picture}(1,0.22998879)%
    \put(0,0){\includegraphics[width=\unitlength]{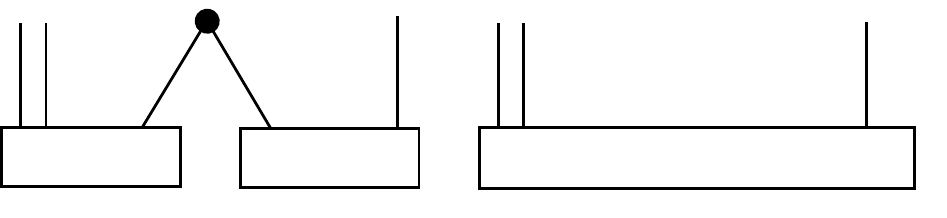}}%
    \put(0.25177122,0.20322321){\color[rgb]{0,0,0}\makebox(0,0)[lb]{\smash{\scriptsize{$a_0$}}}}%
    \put(0.29686782,0.05180685){\color[rgb]{0,0,0}\makebox(0,0)[lb]{\smash{\scriptsize{$a_2;m$}}}}%
    \put(0.04423119,0.05234426){\color[rgb]{0,0,0}\makebox(0,0)[lb]{\smash{\scriptsize{$a_1;n$}}}}%
    \put(0.0814201,0.1489883){\color[rgb]{0,0,0}\makebox(0,0)[lb]{\smash{$\dots$}}}%
    \put(0.31040686,0.1489883){\color[rgb]{0,0,0}\makebox(0,0)[lb]{\smash{$\dots$}}}%
    \put(0.53743358,0.05234426){\color[rgb]{0,0,0}\makebox(0,0)[lb]{\smash{\scriptsize
{$a_0+a_1+a_2;n+m-2$}}}}%
    \put(0.69792299,0.1489883){\color[rgb]{0,0,0}\makebox(0,0)[lb]{\smash{$\dots$}}}%
    \put(0.45719327,0.1489883){\color[rgb]{0,0,0}\makebox(0,0)[lb]{\smash{$=$}}}%
  \end{picture}%
\endgroup%

	\end{aligned}\ .
	\label{eq:cond:gluer}
\end{align}
		\label{cond:gluer}
	\item
		Removing a bubble:
\begin{align}
	\begin{aligned}
	\def\svgwidth{6cm}
\begingroup%
  \makeatletter%
  \providecommand\color[2][]{%
    \errmessage{(Inkscape) Color is used for the text in Inkscape, but the package 'color.sty' is not loaded}%
    \renewcommand\color[2][]{}%
  }%
  \providecommand\transparent[1]{%
    \errmessage{(Inkscape) Transparency is used (non-zero) for the text in Inkscape, but the package 'transparent.sty' is not loaded}%
    \renewcommand\transparent[1]{}%
  }%
  \providecommand\rotatebox[2]{#2}%
  \ifx\svgwidth\undefined%
    \setlength{\unitlength}{183.33798828bp}%
    \ifx\svgscale\undefined%
      \relax%
    \else%
      \setlength{\unitlength}{\unitlength * \real{\svgscale}}%
    \fi%
  \else%
    \setlength{\unitlength}{\svgwidth}%
  \fi%
  \global\let\svgwidth\undefined%
  \global\let\svgscale\undefined%
  \makeatother%
  \begin{picture}(1,0.27434026)%
    \put(0,0){\includegraphics[width=\unitlength]{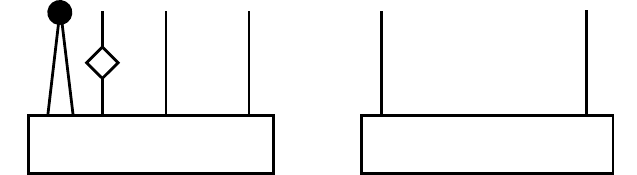}}%
    \put(0.59717074,0.03779261){\color[rgb]{0,0,0}\makebox(0,0)[lb]{\smash{\scriptsize{$a_1+a_2+a_3;n$}}}}%
    \put(0.73310133,0.17135973){\color[rgb]{0,0,0}\makebox(0,0)[lb]{\smash{$\dots$}}}%
    \put(0.46917834,0.17271261){\color[rgb]{0,0,0}\makebox(0,0)[lb]{\smash{$=$}}}%
    \put(0.13463703,0.02906556){\color[rgb]{0,0,0}\makebox(0,0)[lb]{\smash{\scriptsize{$a_3;n+2$}}}}%
    \put(0.29463732,0.17271261){\color[rgb]{0,0,0}\makebox(0,0)[lb]{\smash{$\dots$}}}%
    \put(0,0.22584611){\color[rgb]{0,0,0}\makebox(0,0)[lb]{\smash{\scriptsize{$a_1$}}}}%
    \put(0.19753476,0.18573454){\color[rgb]{0,0,0}\makebox(0,0)[lb]{\smash{\scriptsize{$a_2$}}}}%
  \end{picture}%
\endgroup%

	\end{aligned}\ .
	\label{eq:cond:selfgluer}
\end{align}
		\label{cond:selfgluer}
	\item 
		``Moving $\zeta_a$ around'':
\begin{align}
	\begin{aligned}
	\def\svgwidth{10cm}
\begingroup%
  \makeatletter%
  \providecommand\color[2][]{%
    \errmessage{(Inkscape) Color is used for the text in Inkscape, but the package 'color.sty' is not loaded}%
    \renewcommand\color[2][]{}%
  }%
  \providecommand\transparent[1]{%
    \errmessage{(Inkscape) Transparency is used (non-zero) for the text in Inkscape, but the package 'transparent.sty' is not loaded}%
    \renewcommand\transparent[1]{}%
  }%
  \providecommand\rotatebox[2]{#2}%
  \ifx\svgwidth\undefined%
    \setlength{\unitlength}{307.35424805bp}%
    \ifx\svgscale\undefined%
      \relax%
    \else%
      \setlength{\unitlength}{\unitlength * \real{\svgscale}}%
    \fi%
  \else%
    \setlength{\unitlength}{\svgwidth}%
  \fi%
  \global\let\svgwidth\undefined%
  \global\let\svgscale\undefined%
  \makeatother%
  \begin{picture}(1,0.15649588)%
    \put(0,0){\includegraphics[width=\unitlength]{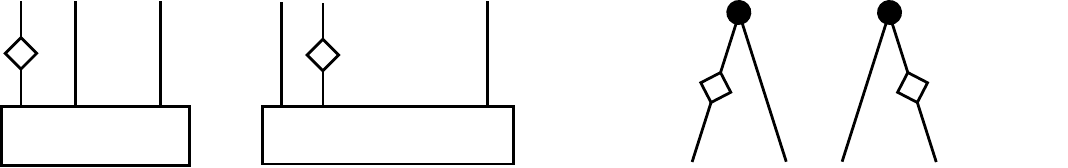}}%
    \put(0.263062,0.0189645){\color[rgb]{0,0,0}\makebox(0,0)[lb]{\smash{\scriptsize{$a_1+a_2-a_3;n$}}}}%
    \put(0.35814827,0.07963299){\color[rgb]{0,0,0}\makebox(0,0)[lb]{\smash{$\dots$}}}%
    \put(0.19111272,0.09254525){\color[rgb]{0,0,0}\makebox(0,0)[lb]{\smash{$=$}}}%
    \put(0.06003894,0.0189645){\color[rgb]{0,0,0}\makebox(0,0)[lb]{\smash{\scriptsize{$a_2;n$}}}}%
    \put(0.08974767,0.07479087){\color[rgb]{0,0,0}\makebox(0,0)[lb]{\smash{$\dots$}}}%
    \put(0.03343969,0.11968131){\color[rgb]{0,0,0}\makebox(0,0)[lb]{\smash{\scriptsize{$a_1$}}}}%
    \put(0.31507074,0.11939749){\color[rgb]{0,0,0}\makebox(0,0)[lb]{\smash{\scriptsize{$a_3$}}}}%
    \put(0.85489565,0.13749777){\color[rgb]{0,0,0}\makebox(0,0)[lb]{\smash{\scriptsize{$a_1+a_2-a_3$}}}}%
    \put(0.74827245,0.07340973){\color[rgb]{0,0,0}\makebox(0,0)[lb]{\smash{$=$}}}%
    \put(0.71298014,0.13749777){\color[rgb]{0,0,0}\makebox(0,0)[lb]{\smash{\scriptsize{$a_2$}}}}%
    \put(0.62006763,0.0636226){\color[rgb]{0,0,0}\makebox(0,0)[lb]{\smash{\scriptsize{$a_1$}}}}%
    \put(0.87466366,0.07033305){\color[rgb]{0,0,0}\makebox(0,0)[lb]{\smash{\scriptsize{
$a_
3$}}}}%
    \put(0.52216823,0.08740729){\color[rgb]{0,0,0}\makebox(0,0)[lb]{\smash{and}}}%
  \end{picture}%
\endgroup%

	\end{aligned}\ .
	\label{eq:cond:W-zeta}
\end{align}
		\label{cond:W-zeta}
	\item 
		$\lim_{a\to0}P_a=\id_A$ and 
		the assignments 
		\begin{align}
			\begin{aligned}
				(\Rb_{\ge0})^n&\to\Sc(A^{\otimes n},A^{\otimes n})\\
			(a_1,\dots,a_n)&\mapsto P_{a_1}\otimes\dots\otimes P_{a_n}
			\end{aligned}
			\label{eq:state-sum-pacont}
		\end{align}
		are jointly continuous for every $n\ge1$.
		\label{cond:approxid}
	\item The limit $\lim_{a\to0}{D}_a$ exists. \label{cond:zerocylinder}
\end{enumerate}
\begin{definition}\label{def:pdata}
	We call the family of morphisms in \eqref{eq:state-sum-data} satisfying the above conditions  
\textsl{state-sum data} and denote it with 
	\begin{align}
		\Ab=(A,\zeta_a,\beta_a,W_a^{n})\ .
		\label{eq:def:pdata}
	\end{align}
\end{definition}

\begin{lemma}\label{lem:beta-symm}
	Let $\Ab=(A,\zeta_a,\beta_a,W_a^{n})$
	denote state-sum data. 
	Then the assignments $a\mapsto\zeta_a$, $a\mapsto \beta_a$ and $a\mapsto W_a^{n}$
		are continuous for every $n\ge1$.
\end{lemma}
\begin{proof}
We only sketch that $a\mapsto W_a^n$ and $a\mapsto\zeta_a$ are
continuous.  By using Condition~\ref{cond:gluer}
	we have that 
	\begin{align}
		\begin{aligned}
			\def\svgwidth{10cm}
\begingroup%
  \makeatletter%
  \providecommand\color[2][]{%
    \errmessage{(Inkscape) Color is used for the text in Inkscape, but the package 'color.sty' is not loaded}%
    \renewcommand\color[2][]{}%
  }%
  \providecommand\transparent[1]{%
    \errmessage{(Inkscape) Transparency is used (non-zero) for the text in Inkscape, but the package 'transparent.sty' is not loaded}%
    \renewcommand\transparent[1]{}%
  }%
  \providecommand\rotatebox[2]{#2}%
  \newcommand*\fsize{\dimexpr\f@size pt\relax}%
  \newcommand*\lineheight[1]{\fontsize{\fsize}{#1\fsize}\selectfont}%
  \ifx\svgwidth\undefined%
    \setlength{\unitlength}{268.09593058bp}%
    \ifx\svgscale\undefined%
      \relax%
    \else%
      \setlength{\unitlength}{\unitlength * \real{\svgscale}}%
    \fi%
  \else%
    \setlength{\unitlength}{\svgwidth}%
  \fi%
  \global\let\svgwidth\undefined%
  \global\let\svgscale\undefined%
  \makeatother%
  \begin{picture}(1,0.18730011)%
    \lineheight{1}%
    \setlength\tabcolsep{0pt}%
    \put(0,0){\includegraphics[width=\unitlength,page=1]{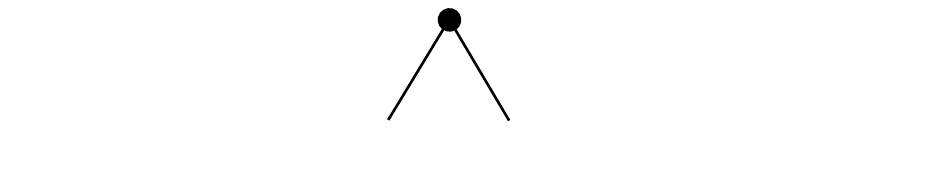}}%
    \put(0.51409698,0.16179472){\color[rgb]{0,0,0}\makebox(0,0)[lt]{\lineheight{0}\smash{\begin{tabular}[t]{l}\scriptsize{$\eps_1$}\end{tabular}}}}%
    \put(0,0){\includegraphics[width=\unitlength,page=2]{Wna-cont.pdf}}%
    \put(0.55707031,0.01750746){\color[rgb]{0,0,0}\makebox(0,0)[lt]{\lineheight{0}\smash{\begin{tabular}[t]{l}\scriptsize{$\eps_2,2$}\end{tabular}}}}%
    \put(0,0){\includegraphics[width=\unitlength,page=3]{Wna-cont.pdf}}%
    \put(0.31632851,0.01801957){\color[rgb]{0,0,0}\makebox(0,0)[lt]{\lineheight{0}\smash{\begin{tabular}[t]{l}\scriptsize{$a-\eps;n$}\end{tabular}}}}%
    \put(0,0){\includegraphics[width=\unitlength,page=4]{Wna-cont.pdf}}%
    \put(0.35176646,0.11011334){\color[rgb]{0,0,0}\makebox(0,0)[lt]{\lineheight{0}\smash{\begin{tabular}[t]{l}$\dots$\end{tabular}}}}%
    \put(0,0){\includegraphics[width=\unitlength,page=5]{Wna-cont.pdf}}%
    \put(0.08133804,0.01801957){\color[rgb]{0,0,0}\makebox(0,0)[lt]{\lineheight{0}\smash{\begin{tabular}[t]{l}\scriptsize{$a;n$}\end{tabular}}}}%
    \put(0,0){\includegraphics[width=\unitlength,page=6]{Wna-cont.pdf}}%
    \put(0.08880084,0.11011334){\color[rgb]{0,0,0}\makebox(0,0)[lt]{\lineheight{0}\smash{\begin{tabular}[t]{l}$\dots$\end{tabular}}}}%
    \put(0,0){\includegraphics[width=\unitlength,page=7]{Wna-cont.pdf}}%
    \put(0.22867615,0.07094826){\color[rgb]{0,0,0}\makebox(0,0)[lt]{\lineheight{0}\smash{\begin{tabular}[t]{l}$=$\end{tabular}}}}%
    \put(0,0){\includegraphics[width=\unitlength,page=8]{Wna-cont.pdf}}%
    \put(0.83666468,0.01801957){\color[rgb]{0,0,0}\makebox(0,0)[lt]{\lineheight{0}\smash{\begin{tabular}[t]{l}\scriptsize{$a-\eps;n$}\end{tabular}}}}%
    \put(0,0){\includegraphics[width=\unitlength,page=9]{Wna-cont.pdf}}%
    \put(0.83853247,0.11011334){\color[rgb]{0,0,0}\makebox(0,0)[lt]{\lineheight{0}\smash{\begin{tabular}[t]{l}$\dots$\end{tabular}}}}%
    \put(0,0){\includegraphics[width=\unitlength,page=10]{Wna-cont.pdf}}%
    \put(0.70984719,0.07094826){\color[rgb]{0,0,0}\makebox(0,0)[lt]{\lineheight{0}\smash{\begin{tabular}[t]{l}$=$\end{tabular}}}}%
    \put(0,0){\includegraphics[width=\unitlength,page=11]{Wna-cont.pdf}}%
    \put(0.94603225,0.10592225){\color[rgb]{0,0,0}\makebox(0,0)[lt]{\lineheight{0}\smash{\begin{tabular}[t]{l}$P_{\eps}$\end{tabular}}}}%
    \put(0,0){\includegraphics[width=\unitlength,page=12]{Wna-cont.pdf}}%
  \end{picture}%
\endgroup%

		\end{aligned}\ ,
		\label{eq:Wna-cont}
	\end{align}
	for every $a\ge\varepsilon\in\Rb_{>0}$ with $\eps=\eps_1+\eps_2$.
So by separate continuity of the composition and Condition~\ref{cond:approxid},
	the assignment $a\mapsto W_a^n$ is continuous. 
To see continuity of $\zeta_a$ we first use
Conditions~\ref{cond:gluer}~and~\ref{cond:W-zeta} and we get that
	\begin{align}
		\zeta_a\circ P_{b+c}=\zeta_{a+b}\circ P_c
		\label{eq:zeta-cont}
	\end{align}
	for every $a,b,c\in\Rb_{>0}$. Condition~\ref{cond:approxid} now allows us to take the limit $c\to0$, and continuity again follows from that of $P_{b+c}$.
\end{proof}

\medskip

Let us fix state-sum data $\Ab$
using the notation of \eqref{eq:def:pdata}.
In the rest of this section we define a symmetric monoidal functor $\funZ_{\Ab}:\Bordarea\to\Sc$
using this data.

The next lemma is best proved after having established the relation between state-sum data and RFAs in Lemma~\ref{lem:data2rfa} below, when it becomes a direct consequence of Lemma~\ref{lem:d0-split-idempot-center} and we omit the proof.

\begin{lemma}\label{lem:Da-additive}
	We have that
\begin{align}
	{D}_a\circ {D}_b={D}_{a+b}
	\label{eq:da-d0}
\end{align} 
for every $a,b\in\Rb_{\ge0}$.
In particular, the morphism ${D}_0:=\lim_{a\to0}{D}_a\in\Sc(A,A)$ is idempotent.
	\label{lem:D0-idempotent}

\end{lemma}
Recall that we assumed that $\Sc$ is idempotent complete, so the idempotent $D_0$ splits:
let $Z(A)\in\Sc$ denote its image and let us write
\begin{align}
	D_0=
	&
	\left[ A\xrightarrow{\pi_{A}}Z(A)\xrightarrow{\iota_{A}}A \right]\ , 
	&& \left[ Z(A)\xrightarrow{\iota_{A}}A\xrightarrow{\pi_{A}}Z(A) \right] =\id_{Z(A)}\ ,
	\label{eq:d0-split-idempot}
\end{align}

We define the {\aQFT} $\funZ_{\Ab}$ on objects as follows:
Let $S\in\Bordarea$. Then
\begin{align}
	\funZ_{\Ab}(S):=\bigotimes_{x\in \pi_0(S)}Z(A)^{(x)}\ ,
	\label{eq:tft:obj-aqft}
\end{align}
where $Z(A)^{(x)}=Z(A)$ and the superscript is used to label the tensor factors.

In the remainder of this section we give the definition of $\funZ_{\Ab}$ on morphisms.
Let $(\Sigma,\Ac):S\to T$ be a bordism with area
and let us assume that $(\Sigma,\Ac)$ has no component with zero area.
Choose a PLCW decomposition with area $\Sigma_k$, $\Ac_k$ for $k\in\left\{ 0,1,2 \right\}$
of the surface with area $(\Sigma,\Ac)$,
such that the PLCW decomposition has exactly 1 edge on every boundary component.
By this convention $\pi_0(S)\sqcup\pi_0(T)$
is in bijection with vertices on the boundary and with edges on the boundary.

\begin{figure}[tb]
	\centering
	\def\svgwidth{9.5cm}
	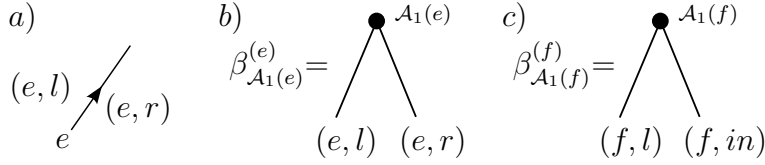
	\caption{$a)$ Left and right sides $(e,l)$ and $(e,r)$ of an inner edge $e$,
		determined by the orientation of $\Sigma$ (paper orientation) 
		and of $e$ (arrow).
		$b)$ Convention for connecting tensor factors belonging to edge sides $(e,l)$ and $(e,r)$ of an inner edge $e$ with the tensor factors belonging to the morphism $\beta_{\Ac_1(e)}^{(e)}$.
		$c)$ Conventions for the labels of the tensor factors for an ingoing boundary edge $f$ with $(f,l)\in E$.
	}
	\label{fig:connect-aqft}
\end{figure}

Let us choose an edge 
for every face before glueing, which we call \textsl{marked edge},
and let us choose an orientation of every edge.
For a face $f\in\Sigma_2$ which is an $n_f$-gon
let us write $(f,k)$, $k=1,\dots,n_f$ for the sides of $f$,
where $(f,1)$ denotes the  marked
edge of $f$, and the labeling proceeds counter-clockwise 
with respect to the orientation of $f$.
We collect the sides of all faces into a set:
\begin{align}
	F := \setc*{(f,k)}{f \in \Sigma_2 , k=1,\dots,n_f } \ .
\end{align}

We double the set of edges by considering $\Sigma_1 \times \{l,r\}$, where ``$l$'' and ``$r$'' stand for left and right, respectively. Let $E \subset \Sigma_1 \times \{l,r\}$ be the subset of all $(e,l)$ (resp.\ $(e,r)$), which 
have a face attached on the left (resp.\ right) side, cf.\ Figure~\ref{fig:connect-aqft}~$a)$. Thus for an inner edge $e\in\Sigma_1$ the set $E$ contains both $(e,l)$ and $(e,r)$,
but for a boundary edge $e'\in\Sigma_1$ the set $E$ contains either
$(e',l)$ or $(e',r)$.
By construction of $F$ and $E$ we obtain a bijection
	\begin{align}
		\Phi:F\xrightarrow{~\sim~}E
		\ , \quad (f,k)\mapsto(e,x) \ ,
		\label{eq:face-side-edge-bijection}
	\end{align}
where $e$ is the $k$'th edge on the boundary of the face $f$ lying on the side $x$ of $e$, counted counter-clockwise from the marked edge of $f$.

		For every vertex $v\in\Sigma_0$ in the interior of $\Sigma$ or
		on an 
		ingoing 
		boundary component of $\Sigma$ choose a side of an edge $(e,x)\in E$ 
		for which $v\in\partial(e)$.
Let 
\begin{align}
V : 	\Sigma_0 \setminus \pi_0(T) \to E
\label{eq:choose-vertex-edge}
\end{align}
be the resulting function.

To define $\funZ_{\Ab}(\Sigma,\Ac)$ we proceed with the following steps.
\begin{enumerate}
	\item Let us introduce the tensor products
		\begin{align}
			\begin{aligned}
			\Oc_{F}&:=\bigotimes_{(f,k) \in F}A^{(f,k)}\ ,&
			\Oc_{E}&:=\bigotimes_{(e,x) \in E}A^{(e,x)}\ ,\\
			\Oc_\mathrm{in}&:=\bigotimes_{b\in \pi_0(S)}A^{(b,in)}\ ,&
			\Oc_\mathrm{out}&:=\bigotimes_{c\in \pi_0(T)}A^{(c,out)}\ .
			\end{aligned}
			\label{eq:bigtensor-aqft}
		\end{align}
	Every tensor factor is equal to $A$, but
	the various superscripts will help us distinguish tensor factors in the source and target objects of the morphisms we define in the remaining steps.
	\item 
	Recall that by our conventions there is one edge in each boundary component and that we identified outgoing boundary edges with $\pi_0(T)$.
	Define the morphism
	\begin{align}
		\Cc:= \bigotimes_{e\in \Sigma_1\setminus \pi_0(T)} 
		\beta_{\Ac_1(e)}^{(e)}:
		\Oc_\mathrm{in}\otimes\Oc_E\to\Oc_\mathrm{out}\ ,
		\label{eq:step4-aqft}
	\end{align}
	where $\beta_{\Ac_1(e)}^{(e)}=\beta_{\Ac_1(e)}$, and where the
	tensor factors in $\Oc_\mathrm{in}\otimes\Oc_E$ are assigned to those of $\beta_{\Ac_1(e)}$ according to  Figure~\ref{fig:connect-aqft}~$b)$~and~$c)$.
	\item Define the morphism
		\begin{align}
			\Yc
			:=\prod_{v\in\Sigma_0\setminus \pi_0(T)}
			\zeta_{\Ac_0(v)}^{(V(v))}
			\in\Sc(\Oc_E,\Oc_E)\ ,
			\label{eq:step2-aqft}
		\end{align}
		where
		\begin{align}
			\zeta_a^{(e,x)} =  \id \otimes \cdots \otimes \zeta_a \otimes \cdots \otimes \id
			\in\Sc(\Oc_E,\Oc_E) \ ,
		\end{align}
		where $\zeta_a$ maps the tensor factor $A^{(e,x)}$ to itself. 
	\item 
		Assign to every face $f\in\Sigma_2$ obtained from an $n_f$-gon 
		the morphism 
		\begin{align}
			W_{\Ac_2(f)}^{f}=W_{\Ac_2(f)}^{(n_f)} :
			\Ib\to A_{(f,1)} \otimes \cdots \otimes A_{(f,n_f)}
			\label{eq:step1.1-aqft}
		\end{align}
		and take their tensor product:
		\begin{align}
			\Fc:=\bigotimes_{f\in\Sigma_2}\left( W_{\Ac_2(f)}^{f}\right):\Ib\to\Oc_{F} \ .
			\label{eq:step1-aqft}
		\end{align}
	\item 
		\label{step3-aqft}
We will now put the above morphisms together to obtain a morphism 
$\mathcal{L} : \mathcal{A}_\mathrm{in} \to \mathcal{A}_\mathrm{out}$. Denote by $\Uppi_\Phi$ the permutation of tensor factors
induced by $\Phi:F\to E$,
		\begin{align}
			\Uppi_\Phi:\Oc_F\to\Oc_E \ .
			\label{eq:step3-aqft}
		\end{align}
Using this, we define
	\begin{align}	
		\Kc&:=\left[\Ib\xrightarrow{\Fc}\Oc_F\xrightarrow{\Uppi_\Phi}\Oc_E
		\xrightarrow{\Yc}
		\Oc_E\right]\ ,\\
		\Lc&:=
		\left[ \Oc_\mathrm{in}\xrightarrow{\id_{\Oc_\mathrm{in}}\otimes\Kc} \Oc_\mathrm{in}\otimes\Oc_{E} \xrightarrow{\Cc} \Oc_\mathrm{out}  \right]\ .
		\label{eq:step6a-aqft}
	\end{align}
	\item 
Using the embedding and projection maps 
$\iota_{A}$, $\pi_{A}$ from \eqref{eq:d0-split-idempot}
we construct the morphisms:
		\begin{align}
			\Ec_\mathrm{in}:=& \bigotimes_{b\in \pi_0(S)}\iota_{A}^{(b)}:\funZ_{\Ab}(S)\to\Oc_\mathrm{in}\ ,&
			\Ec_\mathrm{out}:=&\bigotimes_{c\in \pi_0(T)}\pi_{A}^{(c)}:\Oc_\mathrm{out}\to \funZ_{\Ab}(T)\ ,
			\label{eq:step5out-aqft}
		\end{align}
		where $\iota_{A}^{(b)}=\iota_{A}:Z(A)^{(b)}\to A^{(b)}$ and
		$\pi_{A}^{(b)}=\pi_{A}:A^{(b)}\to Z(A)^{(b)}$.
		We have all ingredients to define the action of $\funZ_{\Ab}$ on morphisms:		
		\begin{align}				\funZ_{\Ab}(\Sigma,\Ac)&:=
			\left[ \funZ_{\Ab}(S)\xrightarrow{\Ec_\mathrm{in}} \Oc_\mathrm{in}\xrightarrow{\Lc} \Oc_\mathrm{out}\xrightarrow{\Ec_\mathrm{out}} \funZ_{\Ab}(T)\right]\ .
			\label{eq:step6b-aqft}
		\end{align}
\end{enumerate}

	\medskip

	Now that we defined $\funZ_{\Ab(\Db)}$ on bordisms with strictly positive area, we turn to the general case.
	Let $(\Sigma,\Ac):S\to T$ be a bordism with area and 
let $\Sigma_+:S_+\to{T}_+$ denote the connected component of $(\Sigma,\Ac)$ with strictly positive area. 
We have that in $\Bordarea$
\begin{align}
(\Sigma,\Ac)=({\Sigma}_+,\Ac_+)\sqcup(\Sigma\setminus{\Sigma}_+,0)\ ,
	\label{eq:decomp-perm-area}
\end{align}
where $\Ac_+$ denotes the restriction of $\Ac$ to $\pi_0(\Sigma_+)$.
The bordism with zero area $(\Sigma\setminus{\Sigma}_+,0)$ defines a permutation $\kappa:\pi_0(S\setminus{S}_+)\to\pi_0(T\setminus{T}_+)$.
Let $\funZ_{\Ab}(\Sigma\setminus{\Sigma}_+,0):\funZ_{\Ab}(S\setminus{S_+})\to \funZ_{\Ab}(T\setminus{T_+})$ be the induced permutation of tensor factors.
We define
\begin{align}
	\funZ_{\Ab}(\Sigma,\Ac):=\funZ_{\Ab}(\Sigma\setminus{\Sigma}_+,0)\otimes \funZ_{\Ab}({\Sigma}_+,\Ac_+)\ ,
	\label{eq:za-fulldef}
\end{align}
where $\funZ_{\Ab}({\Sigma}_+,\Ac_+)$ is defined in \eqref{eq:step6b-aqft}.

\begin{theorem} \label{thm:state-sum-aqft}
Let $\Ab$ be state-sum data.
	\begin{enumerate}
		\item The morphism defined in \eqref{eq:step6b-aqft} is 
			independent of the choice of the PLCW decomposition with area,
			the choice of marked
			edges of faces, 
			the choice of orientation of edges
			and the assignment $V$. \label{thm:aqft:1}
		\item The state-sum construction yields an 
			{\aQFT} $\funZ_{\Ab}:\Bordarea\to\Sc$
			whose action on objects and morphisms is given by 
			\eqref{eq:tft:obj-aqft} and \eqref{eq:za-fulldef}, respectively.  \label{thm:aqft:2}
	\end{enumerate}
\end{theorem}
\begin{proof}
	\textsl{Part~\ref{thm:aqft:1}:}\\
	First let us fix a PLCW decomposition with area.
	Independence on the choice of edges for faces and orientation of edges follows directly from 
	Condition~\ref{cond:cyclicsymm}.
	Independence on the assignment $V$ follows 
	from iterating Conditions~\ref{cond:cyclicsymm}~and~\ref{cond:W-zeta}.

	\begin{figure}[tb]
		\centering
		\def\svgwidth{8cm}
		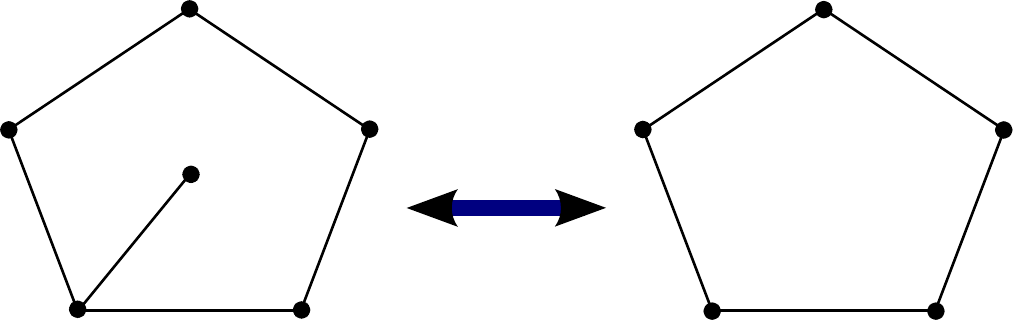
		\caption{A face with a univalent vertex.}
		\label{fig:univalent-aqft}
	\end{figure}

	In order to show independence on the PLCW decomposition with area
	first notice that all conditions on $\Ab$ depend on the sum of the parameters.
	This implies that the construction is independent of the distribution of area, 
	i.e.\ the maps $\Ac_k$ ($k\in\{0,1,2\}$).
	We need to check that the construction yields the same morphism
	for two different PLCW decompositions, but for this it is enough to
	check invariance under the elementary moves in Figure~\ref{fig:edgemove-aqft}.
	Invariance under removing or adding an edge (Figure~\ref{fig:edgemove-aqft}~$b)$)
	follows from Condition~\ref{cond:gluer}.
	To show invariance under splitting an edge by adding a vertex (Figure~\ref{fig:edgemove-aqft}~$a)$)
	we use the trick used in the proof of \cite[Lem.\,3.5]{Davydov:2011dt}.
	There the edge splitting is done inside a 2-gon, and that move in turn follows if one is allowed to add and remove univalent vertices as shown in Figure~\ref{fig:univalent-aqft} (together with adding edges as in Figure~\ref{fig:edgemove-aqft}~$b)$).
	But this follows from Condition~\ref{cond:selfgluer}.

	\medskip

\noindent
	\textsl{Part~\ref{thm:aqft:2}:}\\
	We start by showing
	that if $(S\times[0,1],\Ac):S\to S$ is an in-out cylinder
	with positive area then the assignment
	\begin{align}
		(\Ac(x))_{x\in\pi_0(S\times[0,1])}\mapsto \funZ_{\Ab}(S\times[0,1],\Ac)
		\label{eq:cylinder-cont}
	\end{align}
	is continuous and the limit
	\begin{align}
		\lim_{\Ac\to0}\funZ_{\Ab}(S\times[0,1],\Ac):\funZ_{\Ab}(S)\to \funZ_{\Ab}(S)
		\label{eq:cylinder-limit}
	\end{align}
	is a permutation of tensor factors.
	
	Let us consider one connected component of $S\times[0,1]$. 
	By Part~\ref{thm:aqft:1}, we can pick a PLCW decomposition of 
	this cylinder
	which consists of a square with two opposite edges identified, and the other two edges being the in- and outgoing boundary components.
	The morphism $\Lc$ from \eqref{eq:step6a-aqft} is exactly $D_a$ from \eqref{eq:state-sum-data-notation}, where $a$ is the area of this component.
		
	Now by looking at $\funZ_{\Ab}(S\times[0,1],\Ac)$ with different area maps, we see that
	the difference is in the $\Lc$ maps in \eqref{eq:step6a-aqft}, and is given by a
	factor of $\bigotimes_{x\in\pi_0(S)}P_{a_x}$ for some $a_x\in\Rb_{\ge_0}$.
	Therefore by separate continuity of the composition in $\Sc$ and by Condition~\ref{cond:approxid},
	the assignments in \eqref{eq:cylinder-cont} are continuous for all positive parameters.
	By Condition~\ref{cond:zerocylinder} the limits in \eqref{eq:cylinder-limit} exist,
	and we get the required permutations.

	Next we show functoriality.
	We now assume that all components of the following bordisms have positive area.
	This is not a restriction since we can always take the areas of in-out cylinders to zero to
	get arbitrary bordisms with area.  Let 
	\begin{align*}
		\left[ S\xrightarrow{(\Sigma,\Ac_\Sigma)}T\xrightarrow{(\Xi,\Ac_\Xi)} W\right]
	\end{align*}
	be two bordisms with area.
	Pick PLCW decompositions with area so that at every outgoing boundary component of $(\Sigma,\Ac_\Sigma)$
	there is a square with two opposite edges identified and one edge on the boundary.
	Applying $\funZ_{\Ab}$ on them we get
	\begin{align*}				
		\funZ_{\Ab}(\Sigma,\Ac_\Sigma)&:=
		\left[ \funZ_{\Ab}(S)\xrightarrow{\Ec_\mathrm{in}^{\Sigma}} \Oc_\mathrm{in}^{\Sigma}\xrightarrow{\psi\circ\Lc^{\Sigma}} \Oc_\mathrm{out}^{\Sigma}\xrightarrow{\Ec_\mathrm{out}^{\Sigma}} \funZ_{\Ab}(T)\right]\ ,\\
		\funZ_{\Ab}(\Xi,\Ac_\Xi)&:=
		\left[ \funZ_{\Ab}(T)\xrightarrow{\Ec_\mathrm{in}^{\Xi}} \Oc_\mathrm{in}^{\Xi}\xrightarrow{\Lc^{\Xi}} \Oc_\mathrm{out}^{\Xi}\xrightarrow{\Ec_\mathrm{out}^{\Xi}} \funZ_{\Ab}(W)\right]\ ,
	\end{align*}
	where $\psi=\bigotimes_{x\in\pi_0(T)}D_{a_x}$ for some $a_x\in\Rb_{>0}$.
	Note that $\Oc_\mathrm{out}^{\Sigma}=\Oc_\mathrm{in}^{\Xi}$.
	Composing these morphisms yields using \eqref{eq:d0-split-idempot} the morphism
	\begin{align*}				
		\funZ_{\Ab}(\Xi,\Ac_\Xi)\circ \funZ_{\Ab}(\Sigma,\Ac_\Sigma)&=
		\left[ \funZ_{\Ab}(S)\xrightarrow{\Ec_\mathrm{in}^{\Sigma}} \Oc_\mathrm{in}^{\Sigma}\xrightarrow{\psi\circ\Lc^{\Sigma}} \Oc_\mathrm{out}^{\Sigma}\xrightarrow{\psi_0 } \Oc_\mathrm{in}^{\Xi}\xrightarrow{\Lc^{\Xi}} \Oc_\mathrm{out}^{\Xi}\xrightarrow{\Ec_\mathrm{out}^{\Xi}} \funZ_{\Ab}(W)\right]\ ,
	\end{align*}
	where $\psi_0=\bigotimes_{x\in\pi_0(T)}D_{0}$. 
	For the composition of these bordisms with area $(\Xi\circ\Sigma,\Ac_{\Xi\circ\Sigma})$
	pick the PLCW decomposition with area
	obtained by glueing the two decompositions together at the boundary components corresponding to $T$.
By construction, $\Lc^{\Xi\circ\Sigma}$ contains a copy of $D_a$ for every connected component of $T$. Notice that when we compute $\funZ_{\Ab}(\Xi,\Ac_{\Xi})\circ \funZ_{\Ab}(\Sigma,\Ac_{\Sigma})$, by \eqref{eq:d0-split-idempot}, we also get a copy of $D_0$ for every connected component of $T$. 
Since $D_a \circ D_0 = D_a$ by Lemma~\ref{lem:Da-additive},  $D_0$ can be omitted and 	
	the above composition is equal to $\funZ_{\Ab}(\Xi\circ\Sigma,\Ac_{\Xi\circ\Sigma})$.

	The continuity conditions of Lemma~\ref{lem:enough-to-check-cylinders}
	hold, as we have already checked them before;
	monoidality and symmetry follow from the construction,
	so altogether we have shown that $\funZ_{\Ab}$ is indeed an {\aQFT}.
\end{proof}

\begin{remark}
	By looking at this proof we see that Conditions~\ref{cond:cyclicsymm}-\ref{cond:zerocylinder} are not only sufficient, but also necessary, at least if one requires independence
	under the elementary moves of PLCW decompositions locally, that is, for the corresponding maps $\Ib \to A^{\otimes m}$.
\end{remark}

\subsection{State-sum data from RFAs}\label{sec:data-from-rfa}

In this section we investigate the connection between state-sum data for 
	a given object in $\Sc$ and some particular RFA structures on the same object.
	We find in Theorem~\ref{thm:dataisrfa} that there is a one-to-one correspondence between these two.
	Finally we show in Theorem~\ref{thm:latticecenter} that 
	the state-sum {\aQFT} given in terms of an RFA is classified by
	the centre of this RFA, cf.\ Theorem~\ref{thm:aqftrfaequiv}.
	We will keep using the notation of the previous section.

\begin{lemma}\label{lem:data2rfa}
	The state-sum data $\Ab$
	determines a strongly separable symmetric regularised Frobenius 
	algebra structure on $A\in\Sc$ by setting
	\begin{align}
		\eta_{a_1}&:=W_{a_1}^{1}\ , & \mu_{a_1+a_2}&:=
		\left( id_{A}\otimes\tilde{B}_{a_1}^2 \right)
		\circ\left( W_{a_2}^{3}\otimes id_{A^{\otimes 2}} \right)\ , 
		\label{eq:lem:data2rfa:2}\\
		\varepsilon_{a_1+a_2}&:=\beta_{a_1}\circ\left( W_{a_2}^{1}\otimes 
		id_{A}\right)\ , &
		\Delta_{a_1+a_2}&:=\left( id_{A^{\otimes 2}}\otimes\beta_{a_1} \right)
		\circ\left( W_{a_2}^{3}\otimes id_{A} \right)\ , 
		\label{eq:lem:data2rfa:4}
	\end{align}
	for every $a_1,a_2\in\Rb_{>0}$.
	In terms of the graphical calculus these morphisms are:
		\begin{align}
		&\begin{aligned}
			\def\svgwidth{2.5cm}
\begingroup%
  \makeatletter%
  \providecommand\color[2][]{%
    \errmessage{(Inkscape) Color is used for the text in Inkscape, but the package 'color.sty' is not loaded}%
    \renewcommand\color[2][]{}%
  }%
  \providecommand\transparent[1]{%
    \errmessage{(Inkscape) Transparency is used (non-zero) for the text in Inkscape, but the package 'transparent.sty' is not loaded}%
    \renewcommand\transparent[1]{}%
  }%
  \providecommand\rotatebox[2]{#2}%
  \ifx\svgwidth\undefined%
    \setlength{\unitlength}{88.32706299bp}%
    \ifx\svgscale\undefined%
      \relax%
    \else%
      \setlength{\unitlength}{\unitlength * \real{\svgscale}}%
    \fi%
  \else%
    \setlength{\unitlength}{\svgwidth}%
  \fi%
  \global\let\svgwidth\undefined%
  \global\let\svgscale\undefined%
  \makeatother%
  \begin{picture}(1,0.57473811)%
    \put(0,0){\includegraphics[width=\unitlength]{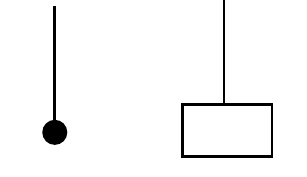}}%
    \put(0.63675488,0.10238934){\color[rgb]{0,0,0}\makebox(0,0)[lb]{\smash{\scriptsize
{$a;1$}}}}%
    \put(0.35297031,0.23752303){\color[rgb]{0,0,0}\makebox(0,0)[lb]{\smash{$=$}}}%
    \put(0,0.17405887){\color[rgb]{0,0,0}\makebox(0,0)[lb]{\smash{\scriptsize
{$a$}}}}%
  \end{picture}%
\endgroup%

		\end{aligned}\ ,
		&&\begin{aligned}
			\def\svgwidth{4.5cm}
\begingroup%
  \makeatletter%
  \providecommand\color[2][]{%
    \errmessage{(Inkscape) Color is used for the text in Inkscape, but the package 'color.sty' is not loaded}%
    \renewcommand\color[2][]{}%
  }%
  \providecommand\transparent[1]{%
    \errmessage{(Inkscape) Transparency is used (non-zero) for the text in Inkscape, but the package 'transparent.sty' is not loaded}%
    \renewcommand\transparent[1]{}%
  }%
  \providecommand\rotatebox[2]{#2}%
  \ifx\svgwidth\undefined%
    \setlength{\unitlength}{196.21150513bp}%
    \ifx\svgscale\undefined%
      \relax%
    \else%
      \setlength{\unitlength}{\unitlength * \real{\svgscale}}%
    \fi%
  \else%
    \setlength{\unitlength}{\svgwidth}%
  \fi%
  \global\let\svgwidth\undefined%
  \global\let\svgscale\undefined%
  \makeatother%
  \begin{picture}(1,0.26173462)%
    \put(0,0){\includegraphics[width=\unitlength]{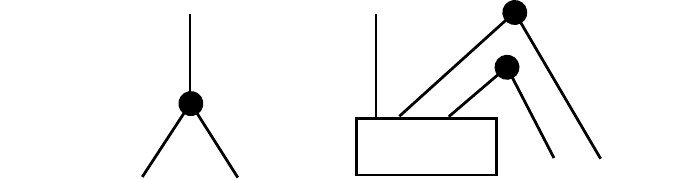}}%
    \put(0.55856967,0.02894234){\color[rgb]{0,0,0}\makebox(0,0)[lb]{\smash{\scriptsize{$a_2;3$}}}}%
    \put(0.38189371,0.1060834){\color[rgb]{0,0,0}\makebox(0,0)[lb]{\smash{$=$}}}%
    \put(0.79399933,0.22700428){\color[rgb]{0,0,0}\makebox(0,0)[lb]{\smash{\scriptsize
{$a_1/2$}}}}%
    \put(0.8321486,0.15579203){\color[rgb]{0,0,0}\makebox(0,0)[lb]{\smash{\scriptsize
{$a_1/2$}}}}%
    \put(0,0.11768238){\color[rgb]{0,0,0}\makebox(0,0)[lb]{\smash{\scriptsize
{$a_1+a_2$}}}}%
  \end{picture}%
\endgroup%

		\end{aligned}\ ,\\
		&\begin{aligned}
			\def\svgwidth{3.5cm}
\begingroup%
  \makeatletter%
  \providecommand\color[2][]{%
    \errmessage{(Inkscape) Color is used for the text in Inkscape, but the package 'color.sty' is not loaded}%
    \renewcommand\color[2][]{}%
  }%
  \providecommand\transparent[1]{%
    \errmessage{(Inkscape) Transparency is used (non-zero) for the text in Inkscape, but the package 'transparent.sty' is not loaded}%
    \renewcommand\transparent[1]{}%
  }%
  \providecommand\rotatebox[2]{#2}%
  \ifx\svgwidth\undefined%
    \setlength{\unitlength}{131.65615234bp}%
    \ifx\svgscale\undefined%
      \relax%
    \else%
      \setlength{\unitlength}{\unitlength * \real{\svgscale}}%
    \fi%
  \else%
    \setlength{\unitlength}{\svgwidth}%
  \fi%
  \global\let\svgwidth\undefined%
  \global\let\svgscale\undefined%
  \makeatother%
  \begin{picture}(1,0.3889149)%
    \put(0,0){\includegraphics[width=\unitlength]{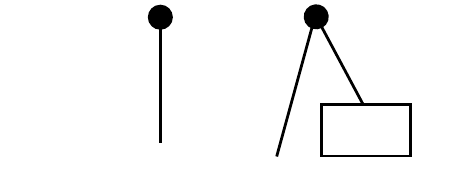}}%
    \put(0.71886263,0.07222887){\color[rgb]{0,0,0}\makebox(0,0)[lb]{\smash{\scriptsize
{$a_1;1$}}}}%
    \put(0.45555661,0.16288898){\color[rgb]{0,0,0}\makebox(0,0)[lb]{\smash{$=$}}}%
    \put(0,0.32691011){\color[rgb]{0,0,0}\makebox(0,0)[lb]{\smash{\scriptsize
{$a_1+a_2$}}}}%
    \put(0.75630174,0.35198157){\color[rgb]{0,0,0}\makebox(0,0)[lb]{\smash{\scriptsize
{$a_2$}}}}%
  \end{picture}%
\endgroup%

		\end{aligned}\ ,
		&&\begin{aligned}
			\def\svgwidth{4.5cm}
\begingroup%
  \makeatletter%
  \providecommand\color[2][]{%
    \errmessage{(Inkscape) Color is used for the text in Inkscape, but the package 'color.sty' is not loaded}%
    \renewcommand\color[2][]{}%
  }%
  \providecommand\transparent[1]{%
    \errmessage{(Inkscape) Transparency is used (non-zero) for the text in Inkscape, but the package 'transparent.sty' is not loaded}%
    \renewcommand\transparent[1]{}%
  }%
  \providecommand\rotatebox[2]{#2}%
  \ifx\svgwidth\undefined%
    \setlength{\unitlength}{195.36150513bp}%
    \ifx\svgscale\undefined%
      \relax%
    \else%
      \setlength{\unitlength}{\unitlength * \real{\svgscale}}%
    \fi%
  \else%
    \setlength{\unitlength}{\svgwidth}%
  \fi%
  \global\let\svgwidth\undefined%
  \global\let\svgscale\undefined%
  \makeatother%
  \begin{picture}(1,0.27743439)%
    \put(0,0){\includegraphics[width=\unitlength]{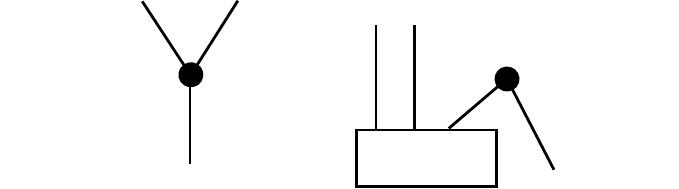}}%
    \put(0.56099996,0.02635982){\color[rgb]{0,0,0}\makebox(0,0)[lb]{\smash{\scriptsize{$a_2;3$}}}}%
    \put(0.38355529,0.14077705){\color[rgb]{0,0,0}\makebox(0,0)[lb]{\smash{$=$}}}%
    \put(0.83576921,0.17833126){\color[rgb]{0,0,0}\makebox(0,0)[lb]{\smash{\scriptsize
{$a_1$}}}}%
    \put(0,0.14207635){\color[rgb]{0,0,0}\makebox(0,0)[lb]{\smash{\scriptsize
{$a_1+a_2$}}}}%
  \end{picture}%
\endgroup%

		\end{aligned}
			\label{eq:lem:data2rfa:graphical}
		\end{align}
	Let $\kappa(\Ab)$ denote this RFA.
\end{lemma}
\begin{proof}
	We are going to show that \eqref{eq:ra:unit} holds. 
	Checking the rest of the algebraic relations of an RFA is similar and it uses the algebraic relations listed in
	Conditions~\ref{cond:cyclicsymm}-\ref{cond:W-zeta}. 
	The rhs of \eqref{eq:ra:unit} is
	\begin{align*}
		\begin{aligned}
			\def\svgwidth{11cm}
\begingroup%
  \makeatletter%
  \providecommand\color[2][]{%
    \errmessage{(Inkscape) Color is used for the text in Inkscape, but the package 'color.sty' is not loaded}%
    \renewcommand\color[2][]{}%
  }%
  \providecommand\transparent[1]{%
    \errmessage{(Inkscape) Transparency is used (non-zero) for the text in Inkscape, but the package 'transparent.sty' is not loaded}%
    \renewcommand\transparent[1]{}%
  }%
  \providecommand\rotatebox[2]{#2}%
  \ifx\svgwidth\undefined%
    \setlength{\unitlength}{344.47885742bp}%
    \ifx\svgscale\undefined%
      \relax%
    \else%
      \setlength{\unitlength}{\unitlength * \real{\svgscale}}%
    \fi%
  \else%
    \setlength{\unitlength}{\svgwidth}%
  \fi%
  \global\let\svgwidth\undefined%
  \global\let\svgscale\undefined%
  \makeatother%
  \begin{picture}(1,0.18609985)%
    \put(0,0){\includegraphics[width=\unitlength]{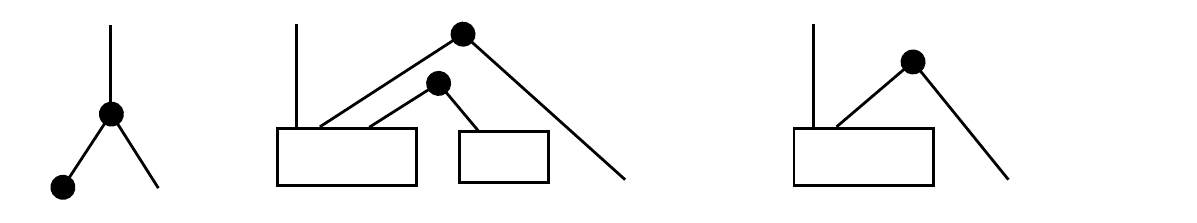}}%
    \put(0.25171455,0.04468402){\color[rgb]{0,0,0}\makebox(0,0)[lb]{\smash{\scriptsize{$a_1'';3$}}}}%
    \put(0.15108179,0.08862275){\color[rgb]{0,0,0}\makebox(0,0)[lb]{\smash{$\stackrel{\text{def.}}{=}$}}}%
    \put(0.41913141,0.17198433){\color[rgb]{0,0,0}\makebox(0,0)[lb]{\smash{\scriptsize
{$a_1'/2$}}}}%
    \put(0.46403896,0.11548762){\color[rgb]{0,0,0}\makebox(0,0)[lb]{\smash{\scriptsize
{$a_1'/2$}}}}%
    \put(0.03109784,0.09058471){\color[rgb]{0,0,0}\makebox(0,0)[lb]{\smash{\scriptsize
{$a_1$}}}}%
    \put(0,0.0276051){\color[rgb]{0,0,0}\makebox(0,0)[lb]{\smash{\scriptsize
{$a_2$}}}}%
    \put(0.39004578,0.04269587){\color[rgb]{0,0,0}\makebox(0,0)[lb]{\smash{\scriptsize
{$a_2;1$}}}}%
    \put(0.55052575,0.08862275){\color[rgb]{0,0,0}\makebox(0,0)[lb]{\smash{$\stackrel
{\text{Cond.\,\ref{cond:gluer}}}{=}$}}}%
    \put(0.68367143,0.04468402){\color[rgb]{0,0,0}\makebox(0,0)[lb]{\smash{\scriptsize{$a_1;2$}}}}%
    \put(0.83849698,0.08862275){\color[rgb]{0,0,0}\makebox(0,0)[lb]{\smash{$\stackrel
{\text{def.}}{=}
P_{a_1+a_2}$}}}%
    \put(0.78860355,0.13715677){\color[rgb]{0,0,0}\makebox(0,0)[lb]{\smash{\scriptsize
{$a_2$}}}}%
  \end{picture}%
\endgroup%

		\end{aligned}\ ,
	\end{align*}
	using Condition~\ref{cond:gluer}.
	The lhs is
	\begin{align*}
		\begin{aligned}
			\def\svgwidth{15cm}
\begingroup%
  \makeatletter%
  \providecommand\color[2][]{%
    \errmessage{(Inkscape) Color is used for the text in Inkscape, but the package 'color.sty' is not loaded}%
    \renewcommand\color[2][]{}%
  }%
  \providecommand\transparent[1]{%
    \errmessage{(Inkscape) Transparency is used (non-zero) for the text in Inkscape, but the package 'transparent.sty' is not loaded}%
    \renewcommand\transparent[1]{}%
  }%
  \providecommand\rotatebox[2]{#2}%
  \ifx\svgwidth\undefined%
    \setlength{\unitlength}{518.03914795bp}%
    \ifx\svgscale\undefined%
      \relax%
    \else%
      \setlength{\unitlength}{\unitlength * \real{\svgscale}}%
    \fi%
  \else%
    \setlength{\unitlength}{\svgwidth}%
  \fi%
  \global\let\svgwidth\undefined%
  \global\let\svgscale\undefined%
  \makeatother%
  \begin{picture}(1,0.12241261)%
    \put(0,0){\includegraphics[width=\unitlength]{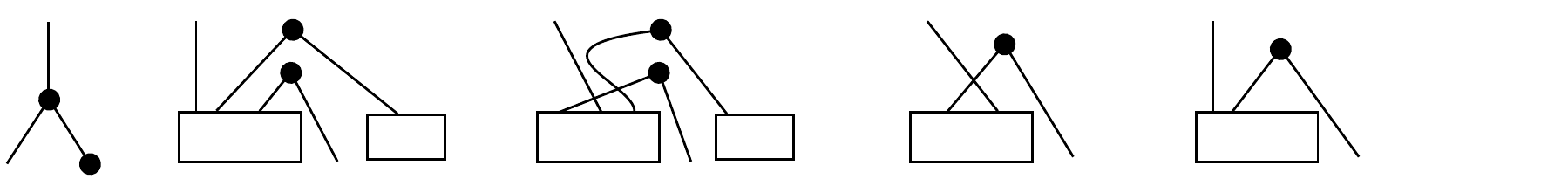}}%
    \put(0.06379919,0.05759358){\color[rgb]{0,0,0}\makebox(0,0)[lb]{\smash{$\stackrel{\text{def.}}{=}$}}}%
    \put(0,0.0655862){\color[rgb]{0,0,0}\makebox(0,0)[lb]{\smash{\scriptsize
{$a_1$}}}}%
    \put(0.07116915,0.01835648){\color[rgb]{0,0,0}\makebox(0,0)[lb]{\smash{\scriptsize
{$a_2$}}}}%
    \put(0.77622772,0.02837579){\color[rgb]{0,0,0}\makebox(0,0)[lb]{\smash{\scriptsize{$a_1;2$}}}}%
    \put(0.87918158,0.05759358){\color[rgb]{0,0,0}\makebox(0,0)[lb]{\smash{$\stackrel
{\text{def.}}{=}
P_{a_1+a_2}$}}}%
    \put(0.83364978,0.0898671){\color[rgb]{0,0,0}\makebox(0,0)[lb]{\smash{\scriptsize
{$a_2$}}}}%
    \put(0.12762809,0.02837579){\color[rgb]{0,0,0}\makebox(0,0)[lb]{\smash{\scriptsize{$a_1'';3$}}}}%
    \put(0.20806908,0.10684912){\color[rgb]{0,0,0}\makebox(0,0)[lb]{\smash{\scriptsize
{$a_1'/2$}}}}%
    \put(0.23484249,0.07545783){\color[rgb]{0,0,0}\makebox(0,0)[lb]{\smash{\scriptsize
{$a_1'/2$}}}}%
    \put(0.23814512,0.02705373){\color[rgb]{0,0,0}\makebox(0,0)[lb]{\smash{\scriptsize
{$a_2;1$}}}}%
    \put(0.35618224,0.02837579){\color[rgb]{0,0,0}\makebox(0,0)[lb]{\smash{\scriptsize{$a_1'';3$}}}}%
    \put(0.43662323,0.11302626){\color[rgb]{0,0,0}\makebox(0,0)[lb]{\smash{\scriptsize
{$a_1'/2$}}}}%
    \put(0.45413094,0.07545783){\color[rgb]{0,0,0}\makebox(0,0)[lb]{\smash{\scriptsize
{$a_1'/2$}}}}%
    \put(0.46052213,0.02705373){\color[rgb]{0,0,0}\makebox(0,0)[lb]{\smash{\scriptsize
{$a_2;1$}}}}%
    \put(0.51164179,0.05759358){\color[rgb]{0,0,0}\makebox(0,0)[lb]{\smash{$\stackrel
{\text{Cond.\,\ref{cond:gluer}}}{=}$}}}%
    \put(0.59400211,0.02837579){\color[rgb]{0,0,0}\makebox(0,0)[lb]{\smash{\scriptsize{$a_1;2$}}}}%
    \put(0.65760131,0.0898671){\color[rgb]{0,0,0}\makebox(0,0)[lb]{\smash{\scriptsize
{$a_2$}}}}%
    \put(0.29235334,0.05759358){\color[rgb]{0,0,0}\makebox(0,0)[lb]{\smash{$\stackrel
{\text{Cond.\,\ref{cond:cyclicsymm}}}{=}$}}}%
    \put(0.6938674,0.05759358){\color[rgb]{0,0,0}\makebox(0,0)[lb]{\smash{$\stackrel
{\text{Cond.\,\ref{cond:cyclicsymm}}}{=}$}}}%
  \end{picture}%
\endgroup%

		\end{aligned}\ .
	\end{align*}

	The continuity conditions for tensor products of $P_a$'s hold by
	Condition~\ref{cond:approxid}, 
	which also states that $\lim_{a\to0}P_a=\id_A$.
	We have now shown that $A$ is an RFA. 
	It is symmetric by Condition~\ref{cond:cyclicsymm}.
	To show that $A$ is strongly separable, 
	by Proposition~\ref{prop:symm-rfa-strongly-separable} we need to check that the window element of $A$ is invertible.
	Similar to the calculation above,
and using Condition~\ref{cond:selfgluer}, one checks
that $\zeta_{a_1}\circ W_{a_2}^{1}$ is inverse to the window element.

\end{proof}

\begin{lemma}\label{lem:rfa2data}
Let $A\in\Sc$ be a strongly separable symmetric regularised Frobenius algebra with separability idempotent
$e_a\in\Sc(\Ib,A^{\otimes 2})$.
Define the following families of morphisms for $a_1,a_2,a_3\in\Rb_{>0}$, $n\in\Zb_{\ge1}$,
\begin{align}
	\zeta_{a_1+a_2+a_3}&:=\left( \eps_{a_1}\otimes \mu_{a_2} \right)\circ\left( 
	e_{a_3}\otimes\id_A \right)\ ,&
	\beta_{a_1+a_2}&:=\varepsilon_{a_1}\circ\mu_{a_2}\ ,
	\label{eq:lem:rfa2data:1}\\
	W^{n}_{a_1+a_2}&:=\Delta_{a_1}^{(n)}\circ\eta_{a_2}\ ,&
\tilde{D}_{a_1+a_2+a_3}&:=\zeta_{a_1}\circ\mu_{a_2}\circ\sigma_{A,A}\circ\Delta_{a_3}\ ,
	\label{eq:lem:rfa2data:2}
\end{align}
where $\Delta_{a_1}^{(n)}$ is as in \eqref{eq:nfold-comult}.
Suppose that $\lim_{a\to_0}\tilde{D}_a$ exists.
Then the families of morphisms $\zeta_a$, $\beta_{a}$ and $W^{n}_{a}$
define state-sum data, which we denote with $\Omega(A)$.
Also, the morphism ${D}_a$ defined in \eqref{eq:state-sum-data-notation} 
	is the same as the morphism $\tilde{D}_a$ defined in \eqref{eq:lem:rfa2data:2}
	for every $a\in\Rb_{\ge0}$.
\end{lemma}
\begin{proof}
Conditions~\ref{cond:approxid}~and~\ref{cond:zerocylinder} 
	are satisfied by our assumptions.
	The algebraic conditions can be checked by direct computation, here we only give the ideas how one can do this.
	Cyclicity in Condition~\ref{cond:cyclicsymm} follows from the Frobenius relation \eqref{eq:rfa:frobrel} and $A$ being symmetric.
	The glueing condition Condition~\ref{cond:gluer} follows from the Frobenius relation \eqref{eq:rfa:frobrel} 
	from counitality \eqref{eq:rca:counit} and from coassociativity \eqref{eq:rca:coassoc}.
	Condition~\ref{cond:selfgluer} follows from $A$ being strongly separable,
	Condition~\ref{cond:W-zeta} 
	follows from the fact that the window element of $A$ is commutative.
	
\end{proof}
Let us fix an object $A\in\Sc$ and denote the sets
\begin{itemize}
	\item $\mathbf{L}:=\{$state-sum data on $A$ $\}$,
	\item $\mathbf{F}:=\{$strongly separable symmetric 
		RFA structures on $A$
	such that $\lim_{a\to0}\tilde{D}_a$ exists$\}$.
\end{itemize}
{}From a direct calculation one can show the following theorem.

\begin{theorem}\label{thm:dataisrfa}
	Let $\kappa$ and $\Omega$ denote the maps of sets
\begin{equation}
	\begin{tikzcd}
		\mathbf{L}\ar[bend left=50]{r}{\kappa}&\mathbf{F}\ar[bend left=50,swap]{l}{\Omega}
	\end{tikzcd}
\label{eq:mapofsetsreg}
\end{equation}
	defined by Lemmas~\ref{lem:data2rfa}~and~\ref{lem:rfa2data} respectively.
	Then $\kappa$ and $\Omega$ are inverse to each other.
\end{theorem}

\medskip

	In the following we make use of the notion of RFAs in order to prove some technical results used in the state-sum construction in Section~\ref{sec:latticedata}.
The following lemma is a direct generalisation of \cite[Prop.\,2.20]{Lauda:2007oc}
and was partially proved in Corollary~\ref{cor:centercycltens}.
\begin{lemma}\label{lem:d0-split-idempot-center}
	Let $\Ab$ be state-sum data and let $A$ denote the
	corresponding RFA from Lemma~\ref{lem:data2rfa}. Then
$D_a\circ D_b=D_{a+b}$ for every $a,b\in\Rb_{>0}$ and
		the image of the idempotent $D_0:=\lim_{a\to0}D_a$ is the centre $Z(A)$ of the RFA $A$.
	It is an RFA with the restricted structure maps of $A$.
\end{lemma}

Using Lemma~\ref{lem:d0-split-idempot-center} and Theorem~\ref{thm:aqftrfaequiv},
we have the direct translation of \cite[Thm.\,4.7]{Lauda:2007oc}.
\begin{theorem}\label{thm:latticecenter}
	Let $\Ab$ be
	state-sum data, 
	let $A$ denote the corresponding RFA and let $Z(A)$ denote its centre.
	Let $\funZ_{\Ab}$ denote the state-sum {\aQFT} of Theorem~\ref{thm:state-sum-aqft} and 
	let $G$ be the equivalence in \eqref{eq:aqftrfaequiv}.
	Then
	\begin{align}
		Z(A)=G(\funZ_{\Ab})\ .
		\label{eq:thm:latticecenter}
	\end{align}
\end{theorem}

The following lemma gives a concise expression for the value of a state-sum area-dependent QFT on a genus $g$ surface with $b$ outgoing boundary components.

\begin{lemma}
	Let $\Ab$ be state-sum data and $A$ the corresponding RFA.
	\begin{enumerate}
		\item We have the following identities:
			\begin{align}
				&\begin{aligned}
				\def\svgwidth{5cm}
\begingroup%
  \makeatletter%
  \providecommand\color[2][]{%
    \errmessage{(Inkscape) Color is used for the text in Inkscape, but the package 'color.sty' is not loaded}%
    \renewcommand\color[2][]{}%
  }%
  \providecommand\transparent[1]{%
    \errmessage{(Inkscape) Transparency is used (non-zero) for the text in Inkscape, but the package 'transparent.sty' is not loaded}%
    \renewcommand\transparent[1]{}%
  }%
  \providecommand\rotatebox[2]{#2}%
  \newcommand*\fsize{\dimexpr\f@size pt\relax}%
  \newcommand*\lineheight[1]{\fontsize{\fsize}{#1\fsize}\selectfont}%
  \ifx\svgwidth\undefined%
    \setlength{\unitlength}{192.74754674bp}%
    \ifx\svgscale\undefined%
      \relax%
    \else%
      \setlength{\unitlength}{\unitlength * \real{\svgscale}}%
    \fi%
  \else%
    \setlength{\unitlength}{\svgwidth}%
  \fi%
  \global\let\svgwidth\undefined%
  \global\let\svgscale\undefined%
  \makeatother%
  \begin{picture}(1,0.48757138)%
    \lineheight{1}%
    \setlength\tabcolsep{0pt}%
    \put(0,0){\includegraphics[width=\unitlength,page=1]{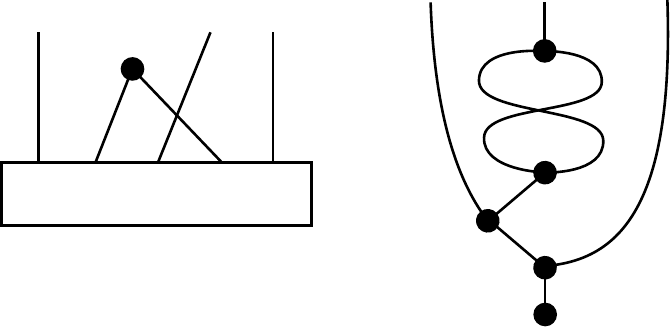}}%
    \put(0.54169236,0.27363729){\color[rgb]{0,0,0}\makebox(0,0)[lt]{\lineheight{1.25}\smash{\begin{tabular}[t]{l}=\end{tabular}}}}%
  \end{picture}%
\endgroup%

				\end{aligned}\ ,
				&&\begin{aligned}
				\def\svgwidth{6cm}
\begingroup%
  \makeatletter%
  \providecommand\color[2][]{%
    \errmessage{(Inkscape) Color is used for the text in Inkscape, but the package 'color.sty' is not loaded}%
    \renewcommand\color[2][]{}%
  }%
  \providecommand\transparent[1]{%
    \errmessage{(Inkscape) Transparency is used (non-zero) for the text in Inkscape, but the package 'transparent.sty' is not loaded}%
    \renewcommand\transparent[1]{}%
  }%
  \providecommand\rotatebox[2]{#2}%
  \newcommand*\fsize{\dimexpr\f@size pt\relax}%
  \newcommand*\lineheight[1]{\fontsize{\fsize}{#1\fsize}\selectfont}%
  \ifx\svgwidth\undefined%
    \setlength{\unitlength}{218.71682277bp}%
    \ifx\svgscale\undefined%
      \relax%
    \else%
      \setlength{\unitlength}{\unitlength * \real{\svgscale}}%
    \fi%
  \else%
    \setlength{\unitlength}{\svgwidth}%
  \fi%
  \global\let\svgwidth\undefined%
  \global\let\svgscale\undefined%
  \makeatother%
  \begin{picture}(1,0.42095869)%
    \lineheight{1}%
    \setlength\tabcolsep{0pt}%
    \put(0,0){\includegraphics[width=\unitlength,page=1]{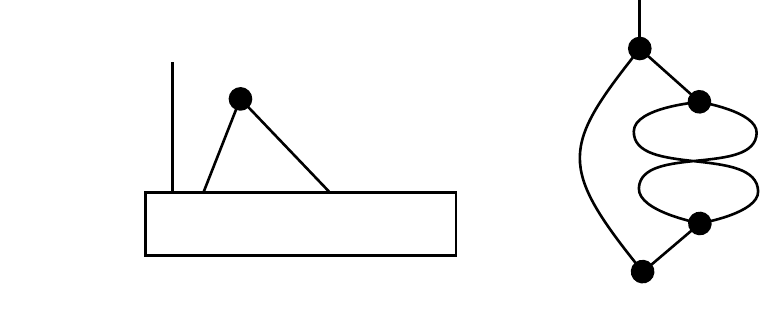}}%
    \put(0.68126401,0.19294149){\color[rgb]{0,0,0}\makebox(0,0)[lt]{\lineheight{1.25}\smash{\begin{tabular}[t]{l}=\end{tabular}}}}%
    \put(0,0){\includegraphics[width=\unitlength,page=2]{phi-rfa.pdf}}%
    \put(-0.00455426,0.19294149){\color[rgb]{0,0,0}\makebox(0,0)[lt]{\lineheight{1.25}\smash{\begin{tabular}[t]{l}$\varphi_a:=$\end{tabular}}}}%
  \end{picture}%
\endgroup%

				\end{aligned}\ ,
				\label{eq:phi-rfa}
			\end{align}
	\label{lem:ZA-from-rfa:1}
		\item Let $(\Sigma_{g,b},a):\emptyset\to(\Sb)^{\sqcup b}$ be a connected bordism 
			of genus $g$ with $b\ge1$ outgoing boundary components and area $a$. Then
			\begin{align}
				\begin{aligned}
					\funZ_{\Ab}(\Sigma_{g,b},a)=\pi^{\otimes b}\circ\Delta_{a_g}^{(b)}\circ\prod_{j=1}^{g}\varphi_{a_j}\circ
						(\zeta_{a'})^{1-b}\circ
					\eta_{a_0}\ ,
				\end{aligned}
				\label{eq:ZA-sigma-gb-rfa}
			\end{align}
			with $a=b\cdot a'+\sum_{j=0}^{g+1}a_j$.
	\label{lem:ZA-from-rfa:2}
\item Let $(C,a):\Sb\sqcup\Sb\to\emptyset$ be a cylinder with two ingoing components and area $a$. Then 
	\begin{align}
		\funZ_{\Ab}(C,a)=\varepsilon_{a_0}\circ\zeta_{a_1}\circ\mu_{a_2}\circ(\iota\otimes\iota)\ ,
		\label{eq:state-sum-in-in-cylinder}
	\end{align}
	with $a=a_0+a_1+a_2$.
	\label{lem:ZA-from-rfa:3}
	\end{enumerate}
	\label{lem:ZA-from-rfa}
\end{lemma}
\begin{proof}
Parts~\ref{lem:ZA-from-rfa:1}~and~\ref{lem:ZA-from-rfa:3} follow from a simple calculation.

	\begin{figure}[tb]
		\centering
		\def\svgwidth{5cm}
                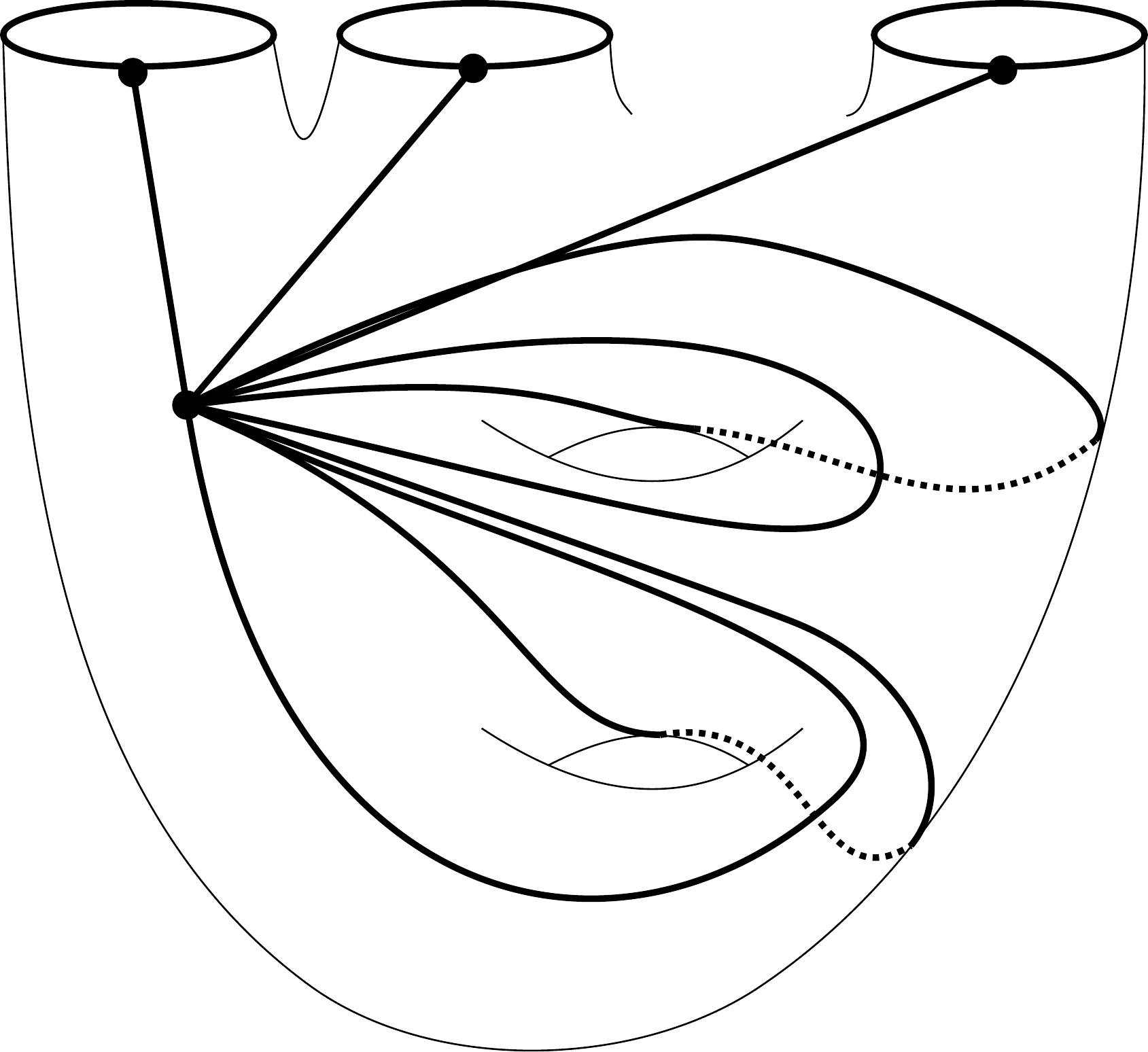
		\caption{A convenient PLCW decomposition of $\Sigma_{g,b}$ using a single face, which is a $(4g+3b)$-gon.
		The dots show the vertices and the thick lines the edges.}
		\label{fig:sigmagb-decomp}
	\end{figure}
	We will only sketch the proof for Part~\ref{lem:ZA-from-rfa:2}. 
	Pick a PLCW decomposition of $\Sigma_{g,b}$
	as shown on Figure~\ref{fig:sigmagb-decomp} and apply the state-sum construction to get 
	the morphism $\Lc$ of \eqref{eq:step6a-aqft}, which is:
	\begin{align}
		\begin{aligned}
		\def\svgwidth{7cm}
\begingroup%
  \makeatletter%
  \providecommand\color[2][]{%
    \errmessage{(Inkscape) Color is used for the text in Inkscape, but the package 'color.sty' is not loaded}%
    \renewcommand\color[2][]{}%
  }%
  \providecommand\transparent[1]{%
    \errmessage{(Inkscape) Transparency is used (non-zero) for the text in Inkscape, but the package 'transparent.sty' is not loaded}%
    \renewcommand\transparent[1]{}%
  }%
  \providecommand\rotatebox[2]{#2}%
  \ifx\svgwidth\undefined%
    \setlength{\unitlength}{308.02922363bp}%
    \ifx\svgscale\undefined%
      \relax%
    \else%
      \setlength{\unitlength}{\unitlength * \real{\svgscale}}%
    \fi%
  \else%
    \setlength{\unitlength}{\svgwidth}%
  \fi%
  \global\let\svgwidth\undefined%
  \global\let\svgscale\undefined%
  \makeatother%
  \begin{picture}(1,0.17514572)%
    \put(0,0){\includegraphics[width=\unitlength]{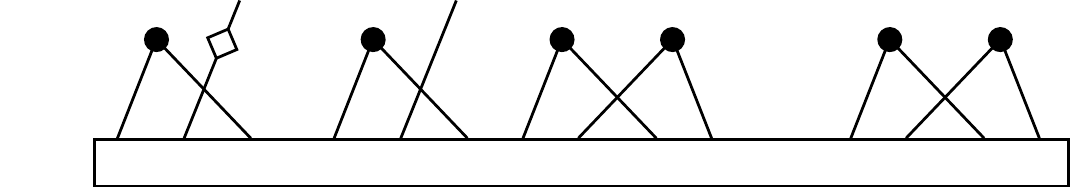}}%
    \put(-0.00344935,0.06409375){\color[rgb]{0,0,0}\makebox(0,0)[lb]{\smash{$\Lc=$}}}%
    \put(0.22577782,0.07804083){\color[rgb]{0,0,0}\makebox(0,0)[lb]{\smash{$\dots$}}}%
    \put(0.66729437,0.07804083){\color[rgb]{0,0,0}\makebox(0,0)[lb]{\smash{$\dots$}}}%
  \end{picture}%
\endgroup%

		\end{aligned}\ .
		\label{eq:Lc-state-sum-sigmagb}
	\end{align}
	The rest of the calculation is straightforward, but tedious, therefore we omit it here.
	Note that in order to get the $D_0$'s at the boundary components, which then cancel with the $\pi$'s, we need to insert $b-1$ factors of $\zeta_{a'}$'s and their inverses.
\end{proof}

In the rest of this section we discuss how one can build Hermitian {\aQFT}s via the state-sum construction.
Let us assume that $\Sc$ is equipped with a $\dagger$-structure.
The state-sum data $\Ab$ is called \textsl{Hermitian} if it satisfies
\begin{align}
	\zeta_a^{\dagger}=\zeta_a\ ,\quad
	\beta_a^{\dagger}=W_a^{2}\quad\text{ and }\quad
	\begin{aligned}
	\def\svgwidth{6cm}
\begingroup%
  \makeatletter%
  \providecommand\color[2][]{%
    \errmessage{(Inkscape) Color is used for the text in Inkscape, but the package 'color.sty' is not loaded}%
    \renewcommand\color[2][]{}%
  }%
  \providecommand\transparent[1]{%
    \errmessage{(Inkscape) Transparency is used (non-zero) for the text in Inkscape, but the package 'transparent.sty' is not loaded}%
    \renewcommand\transparent[1]{}%
  }%
  \providecommand\rotatebox[2]{#2}%
  \ifx\svgwidth\undefined%
    \setlength{\unitlength}{220.61977539bp}%
    \ifx\svgscale\undefined%
      \relax%
    \else%
      \setlength{\unitlength}{\unitlength * \real{\svgscale}}%
    \fi%
  \else%
    \setlength{\unitlength}{\svgwidth}%
  \fi%
  \global\let\svgwidth\undefined%
  \global\let\svgscale\undefined%
  \makeatother%
  \begin{picture}(1,0.35205696)%
    \put(0,0){\includegraphics[width=\unitlength]{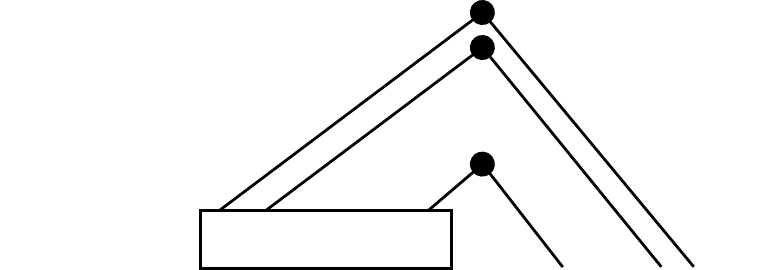}}%
    \put(0.32392654,0.02027442){\color[rgb]{0,0,0}\makebox(0,0)[lb]{\smash{\scriptsize{$a-a_1;n$}}}}%
    \put(0.44697765,0.10790021){\color[rgb]{0,0,0}\makebox(0,0)[lb]{\smash{$\dots$}}}%
    \put(0.74186114,0.02440629){\color[rgb]{0,0,0}\makebox(0,0)[lb]{\smash{$\dots$}}}%
    \put(0.81346938,0.13140654){\color[rgb]{0,0,0}\makebox(0,0)[lb]{\smash{\scriptsize
{$a_1/n$}}}}%
    \put(0.69686074,0.27035528){\color[rgb]{0,0,0}\makebox(0,0)[lb]{\smash{\scriptsize{$a_1/n$}}}}%
    \put(0.65664832,0.32843164){\color[rgb]{0,0,0}\makebox(0,0)[lb]{\smash{\scriptsize{$a_1/n$}}}}%
    \put(-0.00240799,0.15451112){\color[rgb]{0,0,0}\makebox(0,0)[lb]{\smash{$\left( W_a^{n} \right)^{\dagger}=$}}}%
  \end{picture}%
\endgroup%

	\end{aligned}\ ,
	\label{eq:cond:adjoints}
\end{align}
for every $a,a_0\in\Rb_{>0}$ and $n\in\Zb_{\ge1}$.

One can easily check the following statements.
If $\Ab$ is Hermitian, then the RFA $\kappa(\Ab)$ from Lemma~\ref{lem:data2rfa} is a $\dagger$-RFA.
Conversely, let $A$ be a $\dagger$-RFA. Then the state-sum data $\Omega(A)$ is Hermitian.
Also, if $\Ab$ is Hermitian, then the state-sum {\aQFT} $\funZ_{\Ab}$ is Hermitian.

\subsection{PLCW decompositions with defects}\label{sec:PLCW-dec-defect}

In this section we introduce a cell decomposition of bordisms with defects, which is used in Section~\ref{sec:lattice:daqft} to build defect {\aQFT}s.
We will use the notation of Sections~\ref{sec:daqft}~and~\ref{sec:PLCW-dec-aqft} and fix a set of defect conditions 
$\Db=(D_1,D_2,s,t)$.
\begin{figure}[tb]
	\centering
	\def\svgwidth{4cm}
	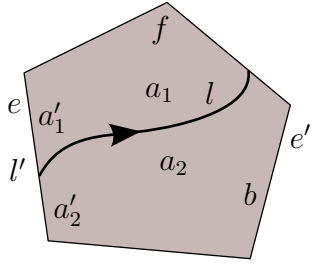
	\caption{A face with defect $f\in\Sigma_2^\mathrm{defect}$, the defect line is marked with a thick line with the arrow
	showing its orientation. 
	For example, the edge $e$ lies in $\Sigma_1^\mathrm{defect}$ and the edge $e'$ in $\Sigma_1^\mathrm{empty}$.
	The area maps have values
	$\Ac_2^\mathrm{defect}(f)=(a_1,l,a_2)$, $\Ac_1^\mathrm{defect}(e)=(a_1',l',a_2')$ and 
	$\Ac_1^\mathrm{empty}(e')=b$ for this particular face and two edges.
}
	\label{fig:face-with-defect}
\end{figure}

Let $\Sigma=\Sigma_{[1]}\cup\Sigma_{[2]}$ be a surface with defects.
A \textsl{PLCW decomposition with defects of $\Sigma$} is a PLCW decomposition 
$\Sigma_0$, $\Sigma_1$, $\Sigma_2$ 
of the surface $\Sigma$ which satisfies the following conditions.
\begin{enumerate}
	\item For every $p\in\Sigma_0$ the intersection $p \cap\Sigma_{[1]}=\emptyset$ is empty.

	\item Every intersection of an element of $\Sigma_1$ and $\Sigma_{[1]}$ is transversal.

	\item For every $e\in\Sigma_1$, $e\cap \Sigma_{[1]}$
	is either empty or consists of precisely one point.

	\item For every $f\in\Sigma_2$, 
	if $f\cap\Sigma_{[1]}\neq\emptyset$, 
	then it is diffeomorphic to an interval
		with the boundary points on edges of $f$.
	\item For every boundary component $b\in\pi_0(\partial\Sigma)$ 
	for which $b\cap \Sigma_{[1]}=\emptyset$,
		there is 1 boundary edge.

	\item For every boundary component $b\in\pi_0(\Sigma)$ for which 
	$b\cap \Sigma_{[1]}\neq \emptyset$,
		every boundary edge contains exactly one point in $\Sigma_{[1]}$.
\end{enumerate}
If the above conditions hold then the sets of faces and edges split in two disjoint sets.
For $k\in\{1,2\}$ let
$\Sigma_k^\mathrm{empty}\subseteq\Sigma_k$ be the subset of cells
which do not intersect with $\Sigma_{[1]}$ (\textsl{empty cells}) and let
$\Sigma_k^\mathrm{defect}:=\Sigma_k\setminus\Sigma_k^\mathrm{empty}$ 
be the subset of cells
which intersect with $\Sigma_{[1]}$ (\textsl{cells with a defect}).
An example is shown in Figure~\ref{fig:face-with-defect}.
Let 
	\begin{align}
		\bar{V}:\Sigma_0\to \Sigma_1 
		\label{eq:choose-edge-to-vertex-V-bar}
	\end{align}
	be a map which assigns to a vertex $v$ an edge $e$ for which $v\in\partial(e)$, similarly as in \eqref{eq:choose-vertex-edge}. This map splits $\Sigma_0$ into two disjoint sets $\Sigma_0^\mathrm{empty}:=\bar{V}^{-1}(\Sigma_1^\mathrm{empty})$ and $\Sigma_0^\mathrm{defect}:=\bar{V}^{-1}(\Sigma_1^\mathrm{defect})$.

Let $(\Sigma,\Ac,\Lc)$ be a surface with area and defects 
with strictly positive areas and lengths
and let $\Sigma_0$, $\Sigma_1$, $\Sigma_2$
be a PLCW decomposition with defects of $\Sigma=\Sigma_{[1]}\cup\Sigma_{[2]}$. 
Let 
\begin{align}
\Ac_k^\mathrm{empty}:&\Sigma_k^\mathrm{empty}\to\Rb_{>0}\ ,
	\label{eq:area-map-empty}
\end{align}
for $k\in\{0,1,2\}$ be maps which assign to empty vertices, edges and faces an 
\textsl{area}, and let
\begin{align}
	\Ac_k^\mathrm{defect}:&\Sigma_k^\mathrm{defect}\to(\Rb_{>0})^{3}\ ,
	\label{eq:area-map-defect}
\end{align}
for $k\in\{0,1,2\}$ be maps which assign to vertices, edges and faces with defects 
an \textsl{area} on the 
half edges and half faces at the 
two sides of a defect line and 
a \textsl{length} of a defect line as explained in Figure~\ref{fig:face-with-defect}.
We require of the maps in 
\eqref{eq:area-map-empty} and \eqref{eq:area-map-defect}
that for every connected component $x\in\pi_0(\Sigma_{[2]})$ the sum of the areas of corresponding
vertices, (half) edges and (half)
faces of $x$ is equal to its area $\Ac(x)$.
We require the analogous condition on lengths of defect lines.
	In the case $k=0$ in \eqref{eq:area-map-defect} the three parameters $\Ac_k^\mathrm{defect}(v)$ for a vertex $v$ contribute to the same components as those of the edge $\bar V(v)$.

In the following we will write $\Ac_k(x)$ for both $\Ac_k^\mathrm{defect}(x)$ and $\Ac_k^\mathrm{empty}(x)$
and mean the latter depending on the type of $x\in\Sigma_k$.
A \textsl{PLCW decomposition of a surface with area and defects} $(\Sigma,\Ac,\Lc)$
consists of a choice of a PLCW decomposition with defects and a choice of maps
as in \eqref{eq:area-map-empty} and \eqref{eq:area-map-defect}.

\begin{figure}[tb]
	\centering
	\def\svgwidth{12cm}
\begingroup%
  \makeatletter%
  \providecommand\color[2][]{%
    \errmessage{(Inkscape) Color is used for the text in Inkscape, but the package 'color.sty' is not loaded}%
    \renewcommand\color[2][]{}%
  }%
  \providecommand\transparent[1]{%
    \errmessage{(Inkscape) Transparency is used (non-zero) for the text in Inkscape, but the package 'transparent.sty' is not loaded}%
    \renewcommand\transparent[1]{}%
  }%
  \providecommand\rotatebox[2]{#2}%
  \newcommand*\fsize{\dimexpr\f@size pt\relax}%
  \newcommand*\lineheight[1]{\fontsize{\fsize}{#1\fsize}\selectfont}%
  \ifx\svgwidth\undefined%
    \setlength{\unitlength}{701.04560753bp}%
    \ifx\svgscale\undefined%
      \relax%
    \else%
      \setlength{\unitlength}{\unitlength * \real{\svgscale}}%
    \fi%
  \else%
    \setlength{\unitlength}{\svgwidth}%
  \fi%
  \global\let\svgwidth\undefined%
  \global\let\svgscale\undefined%
  \makeatother%
  \begin{picture}(1,0.45031346)%
    \lineheight{1}%
    \setlength\tabcolsep{0pt}%
    \put(0,0){\includegraphics[width=\unitlength,page=1]{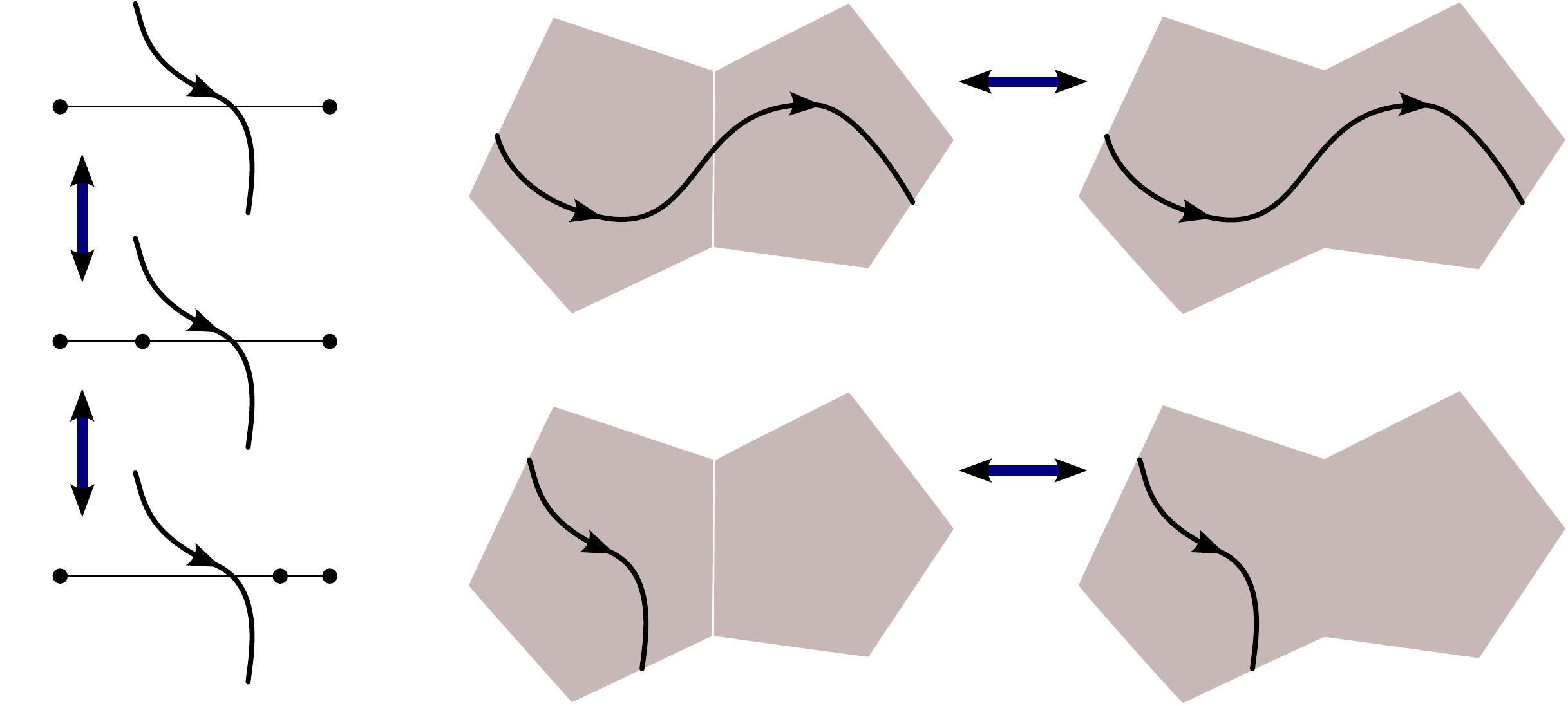}}%
    \put(-0.00106565,0.43041306){\color[rgb]{0,0,0}\makebox(0,0)[lt]{\lineheight{0}\smash{\begin{tabular}[t]{l}$a)$\end{tabular}}}}%
    \put(0.27281097,0.43041306){\color[rgb]{0,0,0}\makebox(0,0)[lt]{\lineheight{0}\smash{\begin{tabular}[t]{l}$b)$\end{tabular}}}}%
    \put(0.27281097,0.1757934){\color[rgb]{0,0,0}\makebox(0,0)[lt]{\lineheight{0}\smash{\begin{tabular}[t]{l}$c)$\end{tabular}}}}%
    \put(0,0){\includegraphics[width=\unitlength,page=2]{elementary-defect.pdf}}%
  \end{picture}%
\endgroup%

	\caption{
		Additional elementary moves of PLCW decompositions with defects.
		In Figure~$a)$, an edge which is crossed by a defect line (curved line with arrow) 
		is split in two by adding a new vertex (denoted with a dot).
		In Figure~$b)$, an edge which is crossed by a defect line is removed.
		In Figure~$c)$, an edge which is not crossed by any defect line is removed.
		Note that here only one of the faces can be crossed by defect lines.
}
	\label{fig:elementary-defect}
\end{figure}

\textsl{Elementary moves} on a PLCW decomposition with defects of a surface with defects (and area)
are elementary moves of PLCW decompositions which respect the conditions listed above.
The additional moves are shown if Figure~\ref{fig:elementary-defect}.
In \cite[Lem.\,3.6]{Davydov:2011dt} it is argued that any two PLCW decompositions with defects can be related by these elementary moves.

\subsection{State-sum construction with defects}\label{sec:lattice:daqft}

After introducing PLCW decompositions with defects let us turn to the state-sum construction of defect {\aQFT}s.
We will again use the notation of Sections~\ref{sec:daqft}~and~\ref{sec:PLCW-dec-defect}
and fix a set of defect conditions $\Db=(D_1,D_2,s,t)$.

\subsubsection*{State sum data and some preparatory notions}

As for the state-sum construction without defects in Section~\ref{sec:latticedata},
we start with giving \textsl{state-sum data with defects} $\Ab(\Db)$.
This consists of 
\begin{enumerate}
	\item state-sum data 
	$\Ab_y=(A_y,\zeta_a^{y},\beta_a^{y},W_a^{y,n})$
		for every $y\in D_2$ as in Definition~\ref{def:pdata}, 
	\item a pair of objects $X_x,\bar{X}_x\in\Sc$
		for every $x\in D_1$ together with the following families of morphisms:
\begin{equation}
\begin{gathered}
		\zeta_{a,l,b}^{x,\epsilon}\in\Sc(X_x^{\epsilon},X_x^{\epsilon})\ ,\quad
	\beta^{x}_{a,l,b}\in\Sc(X_x\otimes \bar{X}_x,\Ib)\ ,\\
	W^{x,n,m}_{a,l,b}\in\Sc(\Ib, \bar{X}_x\otimes A_{t(x)}^{\otimes n}\otimes
		X_x\otimes A_{s(x)}^{\otimes m})
\end{gathered}
	\label{eq:pdata-defect}
\end{equation}
for every $a,l,b\in\Rb_{>0}$, every $n,m\ge 0$ and $\epsilon\in\left\{ \pm \right\}$,
where we used the following  notation:
	\begin{align}
		X_x^{+}:={X}_x\quad\text{ and }\quad X_x^{-}:=\bar{X}_x\ .
		\label{eq:X-sign}
	\end{align}
\end{enumerate}
We will use the following graphical notation for these morphisms.
For $y\in D_2$,
\begin{align}
	\begin{aligned}
	\def\svgwidth{10cm}
\begingroup%
  \makeatletter%
  \providecommand\color[2][]{%
    \errmessage{(Inkscape) Color is used for the text in Inkscape, but the package 'color.sty' is not loaded}%
    \renewcommand\color[2][]{}%
  }%
  \providecommand\transparent[1]{%
    \errmessage{(Inkscape) Transparency is used (non-zero) for the text in Inkscape, but the package 'transparent.sty' is not loaded}%
    \renewcommand\transparent[1]{}%
  }%
  \providecommand\rotatebox[2]{#2}%
  \ifx\svgwidth\undefined%
    \setlength{\unitlength}{238.50266113bp}%
    \ifx\svgscale\undefined%
      \relax%
    \else%
      \setlength{\unitlength}{\unitlength * \real{\svgscale}}%
    \fi%
  \else%
    \setlength{\unitlength}{\svgwidth}%
  \fi%
  \global\let\svgwidth\undefined%
  \global\let\svgscale\undefined%
  \makeatother%
  \begin{picture}(1,0.22298422)%
    \put(0,0){\includegraphics[width=\unitlength]{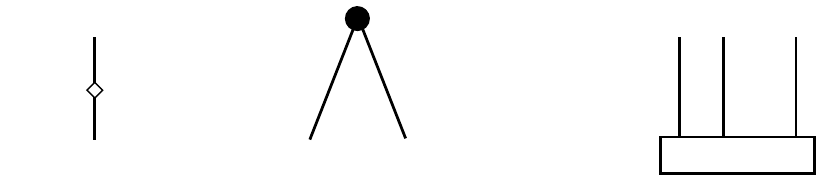}}%
    \put(0.45540767,0.19602618){\color[rgb]{0,0,0}\makebox(0,0)[lb]{\smash{\scriptsize{$(a;y)$}}}}%
    \put(0.29294747,0.11552396){\color[rgb]{0,0,0}\makebox(0,0)[lb]{\smash{$\beta^y_{a}=$}}}%
    \put(0.83911241,0.02934323){\color[rgb]{0,0,0}\makebox(0,0)[lb]{\smash{\scriptsize
{$a;y,n$}}}}%
    \put(0.90103622,0.09901158){\color[rgb]{0,0,0}\makebox(0,0)[lb]{\smash{$\dots$}}}%
    \put(0.65348551,0.11552396){\color[rgb]{0,0,0}\makebox(0,0)[lb]{\smash{$W^
{y,n}
_{a}=$}}}%
    \put(0.34316356,0.01010434){\color[rgb]{0,0,0}\makebox(0,0)[lb]{\smash{$A_y$}}}%
    \put(0.4759324,0.01010434){\color[rgb]{0,0,0}\makebox(0,0)[lb]{\smash{$A_y$}}}%
    \put(0.85203401,0.19554699){\color[rgb]{0,0,0}\makebox(0,0)[lb]{\smash{$A_
{y}$}}}%
    \put(0.93220633,0.19627646){\color[rgb]{0,0,0}\makebox(0,0)[lb]{\smash{$A_
{y}$}}}%
    \put(0.53944066,0.11552396){\color[rgb]{0,0,0}\makebox(0,0)[lb]{\smash{and}}}%
    \put(0.79836584,0.19554696){\color[rgb]{0,0,0}\makebox(0,0)[lb]{\smash{$A_
{y}$}}}%
    \put(-0.00222744,0.11552396){\color[rgb]{0,0,0}\makebox(0,0)[lb]{\smash{$\zeta^y_{a}=$}}}%
    \put(0.14370699,0.10770852){\color[rgb]{0,0,0}\makebox(0,0)[lb]{\smash{\scriptsize{$(a;y)$}}}}%
    \put(0.24426575,0.11552396){\color[rgb]{0,0,0}\makebox(0,0)[lb]{\smash{,}}}%
    \put(0.08372602,0.01370269){\color[rgb]{0,0,0}\makebox(0,0)[lb]{\smash{$A_y$}}}%
    \put(0.08271643,0.1924028){\color[rgb]{0,0,0}\makebox(0,0)[lb]{\smash{$A_y$}}}%
  \end{picture}%
\endgroup%

	\end{aligned}\ ,
	\label{eq:pdata-defect-graphical0}
\end{align}
$P_a^{(y)}\in\Sc(A_y,A_y)$ from \eqref{eq:state-sum-data-notation}
and for $x\in D_1$
\begin{align}
	\begin{aligned}
	\def\svgwidth{14cm}
	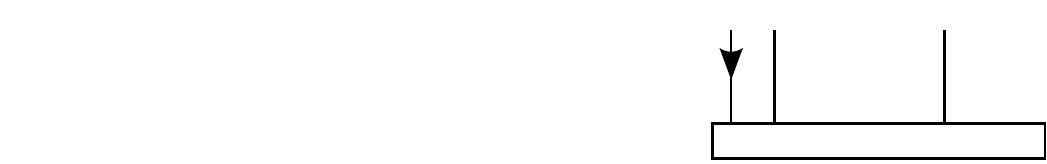
	\end{aligned}\ .
	\label{eq:pdata-defect-graphical}
\end{align}

Let us define
\begin{align}
	\begin{aligned}
		s(x,+)&:=s(x)\ ,&t(x,+)&:=t(x)\ ,\\
		s(x,-)&:=t(x)\ ,&t(x,-)&:=s(x)\ .
	\end{aligned}
	\label{eq:source-target-notation}
\end{align}
By a \textsl{defect list of length $n$} we mean an equivalence class of 
ordered lists
\begin{align}
	\underline{x}:=[(x_1,\epsilon_1,\dots,x_n,\epsilon_n)]\ ,
	\label{eq:defect-list}
\end{align}
where $x_i \in D_1$ and $\epsilon_i\in\{\pm\}$ ($i=1,\dots,n$). The $x_i,\eps_i$ have to satisfy, for $i=1,\dots,n$ and setting $x_{n+1}:=x_1$, $\epsilon_{n+1}:=\epsilon_1$,
\begin{align}
	s(x_i,\epsilon_i)=t(x_{i+1},\epsilon_{i+1})\ .
	\label{eq:source-target-simple}
\end{align}
	Two such lists $(x_1,\epsilon_1,\dots,x_n,\epsilon_n)$ and $(x_1',\epsilon_1',\dots,x_n',\epsilon_n')$ are equivalent
	if they are related by a cyclic permutation.
Let us introduce the shorthand, for a chosen representative of $\underline{x}$,
\begin{align}
	X_{\underline{x}}:=X_{x_1}^{\epsilon_1}\otimes\dots\otimes X_{x_n}^{\epsilon_n}\ .
	\label{eq:defect-list-tensor}
\end{align}
Different choices of representatives are
related by cyclic permutations of tensor factors.
Let us introduce the following morphisms:
\begin{align}
	\begin{aligned}
	\def\svgwidth{10cm}
\begingroup%
  \makeatletter%
  \providecommand\color[2][]{%
    \errmessage{(Inkscape) Color is used for the text in Inkscape, but the package 'color.sty' is not loaded}%
    \renewcommand\color[2][]{}%
  }%
  \providecommand\transparent[1]{%
    \errmessage{(Inkscape) Transparency is used (non-zero) for the text in Inkscape, but the package 'transparent.sty' is not loaded}%
    \renewcommand\transparent[1]{}%
  }%
  \providecommand\rotatebox[2]{#2}%
  \ifx\svgwidth\undefined%
    \setlength{\unitlength}{361.9671814bp}%
    \ifx\svgscale\undefined%
      \relax%
    \else%
      \setlength{\unitlength}{\unitlength * \real{\svgscale}}%
    \fi%
  \else%
    \setlength{\unitlength}{\svgwidth}%
  \fi%
  \global\let\svgwidth\undefined%
  \global\let\svgscale\undefined%
  \makeatother%
  \begin{picture}(1,0.27010487)%
    \put(0,0){\includegraphics[width=\unitlength]{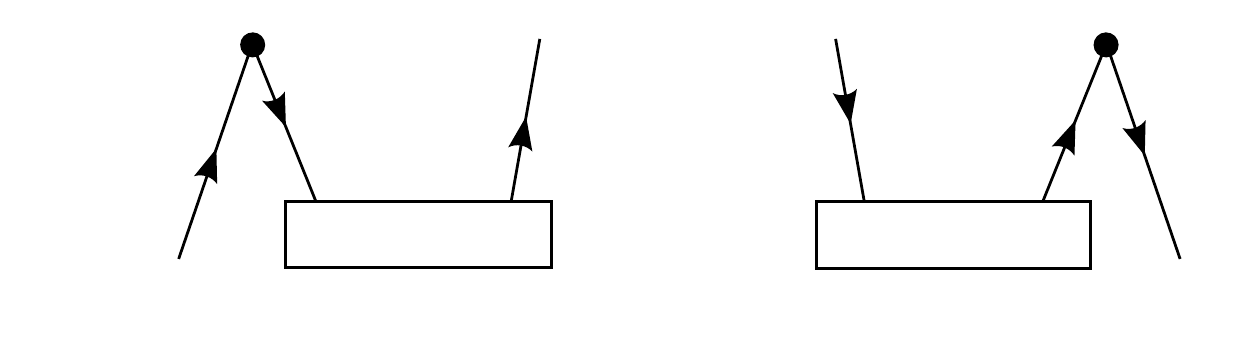}}%
    \put(0.50247993,0.14682141){\color[rgb]{0,0,0}\makebox(0,0)[lb]{\smash{$Q_{a,l,b}^{(x,-)}:=$}}}%
    \put(0.70426042,0.2178767){\color[rgb]{0,0,0}\makebox(0,0)[lb]{\smash{\scriptsize{$(a_1,l_1,b_1,x)$}}}}%
    \put(0.65927991,0.07175485){\color[rgb]{0,0,0}\makebox(0,0)[lb]{\smash{\scriptsize{$a_2,l_2,b_2;x,0,0$}}}}%
    \put(-0.00220151,0.14682141){\color[rgb]{0,0,0}\makebox(0,0)[lb]{\smash{$Q_{a,l,b}^{(x,+)}:=$}}}%
    \put(0.22347487,0.22671728){\color[rgb]{0,0,0}\makebox(0,0)[lb]{\smash{\scriptsize{$(a_1,l_1,b_1,x)$}}}}%
    \put(0.2366194,0.07162811){\color[rgb]{0,0,0}\makebox(0,0)[lb]{\smash{\scriptsize{$a_2,l_2,b_2;x,0,0$}}}}%
    \put(0.1207675,0.02627137){\color[rgb]{0,0,0}\makebox(0,0)[lb]{\smash{$X_x$}}}%
    \put(0.42107661,0.25667133){\color[rgb]{0,0,0}\makebox(0,0)[lb]{\smash{$X_x$}}}%
    \put(0.6539809,0.25667133){\color[rgb]{0,0,0}\makebox(0,0)[lb]{\smash{$\bar{X}_x$}}}%
    \put(0.93627703,0.02627137){\color[rgb]{0,0,0}\makebox(0,0)[lb]{\smash{$\bar
{X}_x$}}}%
    \put(0.47595819,0.14682141){\color[rgb]{0,0,0}\makebox(0,0)[lb]{\smash{,}}}%
  \end{picture}%
\endgroup%

	\end{aligned}\ ,
	\label{eq:state-sum-pa-defect}
\end{align}
where $a=a_1+a_2$, $b=b_1+b_2$ and $l=l_1+l_2$,
\begin{align}
	\begin{aligned}
	\def\svgwidth{12cm}
\begingroup%
  \makeatletter%
  \providecommand\color[2][]{%
    \errmessage{(Inkscape) Color is used for the text in Inkscape, but the package 'color.sty' is not loaded}%
    \renewcommand\color[2][]{}%
  }%
  \providecommand\transparent[1]{%
    \errmessage{(Inkscape) Transparency is used (non-zero) for the text in Inkscape, but the package 'transparent.sty' is not loaded}%
    \renewcommand\transparent[1]{}%
  }%
  \providecommand\rotatebox[2]{#2}%
  \ifx\svgwidth\undefined%
    \setlength{\unitlength}{387.78637085bp}%
    \ifx\svgscale\undefined%
      \relax%
    \else%
      \setlength{\unitlength}{\unitlength * \real{\svgscale}}%
    \fi%
  \else%
    \setlength{\unitlength}{\svgwidth}%
  \fi%
  \global\let\svgwidth\undefined%
  \global\let\svgscale\undefined%
  \makeatother%
  \begin{picture}(1,0.25753444)%
    \put(0,0){\includegraphics[width=\unitlength]{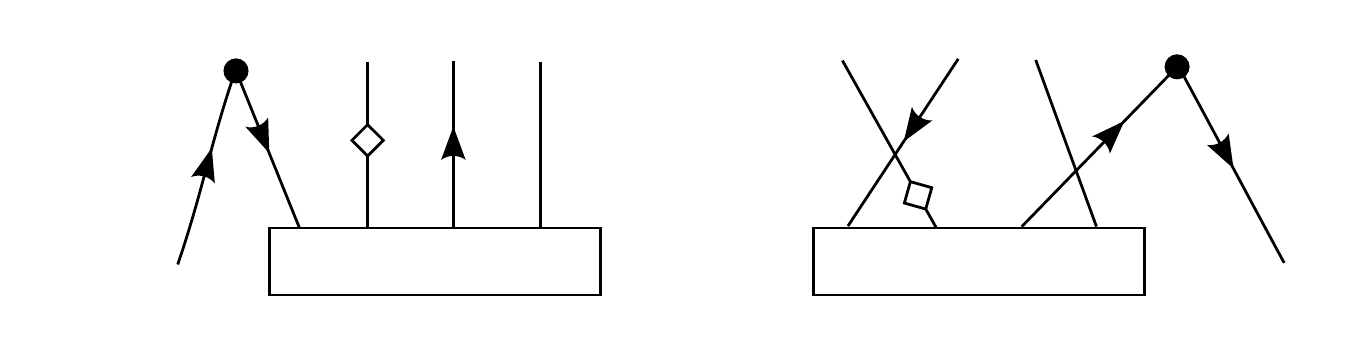}}%
    \put(-0.00205493,0.12313855){\color[rgb]{0,0,0}\makebox(0,0)[lb]{\smash{$F_{a,l,b}^{(x,+)}:=$}}}%
    \put(0.10128803,0.22567456){\color[rgb]{0,0,0}\makebox(0,0)[lb]{\smash{{\scriptsize$(a_1,l_1,b_1,x)$}}}}%
    \put(0.24562095,0.0529517){\color[rgb]{0,0,0}\makebox(0,0)[lb]{\smash{{\scriptsize
$a_2,l_2,b_2;x,1,1$}}}}%
    \put(0.10034874,0.01886682){\color[rgb]{0,0,0}\makebox(0,0)[lb]{\smash{$X_x$}}}%
    \put(0.39669063,0.22567456){\color[rgb]{0,0,0}\makebox(0,0)[lb]{\smash{$A_{s(x)}$}}}%
    \put(0.46077237,0.12313855){\color[rgb]{0,0,0}\makebox(0,0)[lb]{\smash{,}}}%
    \put(0.33019258,0.22567456){\color[rgb]{0,0,0}\makebox(0,0)[lb]{\smash{$X_x$}}}%
    \put(0.259036,0.22567456){\color[rgb]{0,0,0}\makebox(0,0)[lb]{\smash{$A_{t(x)}$}}}%
    \put(0.50131499,0.12313855){\color[rgb]{0,0,0}\makebox(0,0)[lb]{\smash{$F_{a,l,b}^{(x,-)}:=$}}}%
    \put(0.83818253,0.22567456){\color[rgb]{0,0,0}\makebox(0,0)[lb]{\smash{{\scriptsize
$(a_1,l_1
,b_1,x)$}}}}%
    \put(0.64996729,0.0529517){\color[rgb]{0,0,0}\makebox(0,0)[lb]{\smash{{\scriptsize
$a_2,l_2,b_2;x,1,1$}}}}%
    \put(0.75565109,0.22567455){\color[rgb]{0,0,0}\makebox(0,0)[lb]{\smash{$A_{s(x)}$}}}%
    \put(0.69814539,0.22567456){\color[rgb]{0,0,0}\makebox(0,0)[lb]{\smash{$\bar{X}_x$}}}%
    \put(0.60013875,0.22567456){\color[rgb]{0,0,0}\makebox(0,0)[lb]{\smash{$A_{t(x)}$}}}%
    \put(0.94051976,0.0245222){\color[rgb]{0,0,0}\makebox(0,0)[lb]{\smash{$\bar
{X}_x$}}}%
    \put(0.69357506,0.12987394){\color[rgb]{0,0,0}\makebox(0,0)[lb]{\smash{{\scriptsize
$(a_0,t(x))$}}}}%
    \put(0.30774131,0.09909911){\color[rgb]{0,0,0}\rotatebox{90}{\makebox(0,0)[lb]{\smash{{\scriptsize
$(a_0,t(x))$}}}}}%
  \end{picture}%
\endgroup%

	\end{aligned}\ ,
	\label{eq:state-sum-Fx-defect}
\end{align}
where $a=a_0+a_1+a_2$, $b=b_1+b_2$ and $l=l_1+l_2$ and
\begin{align}
	\begin{aligned}
	\def\svgwidth{10cm}
\begingroup%
  \makeatletter%
  \providecommand\color[2][]{%
    \errmessage{(Inkscape) Color is used for the text in Inkscape, but the package 'color.sty' is not loaded}%
    \renewcommand\color[2][]{}%
  }%
  \providecommand\transparent[1]{%
    \errmessage{(Inkscape) Transparency is used (non-zero) for the text in Inkscape, but the package 'transparent.sty' is not loaded}%
    \renewcommand\transparent[1]{}%
  }%
  \providecommand\rotatebox[2]{#2}%
  \ifx\svgwidth\undefined%
    \setlength{\unitlength}{316.17186279bp}%
    \ifx\svgscale\undefined%
      \relax%
    \else%
      \setlength{\unitlength}{\unitlength * \real{\svgscale}}%
    \fi%
  \else%
    \setlength{\unitlength}{\svgwidth}%
  \fi%
  \global\let\svgwidth\undefined%
  \global\let\svgscale\undefined%
  \makeatother%
  \begin{picture}(1,0.33896715)%
    \put(0,0){\includegraphics[width=\unitlength]{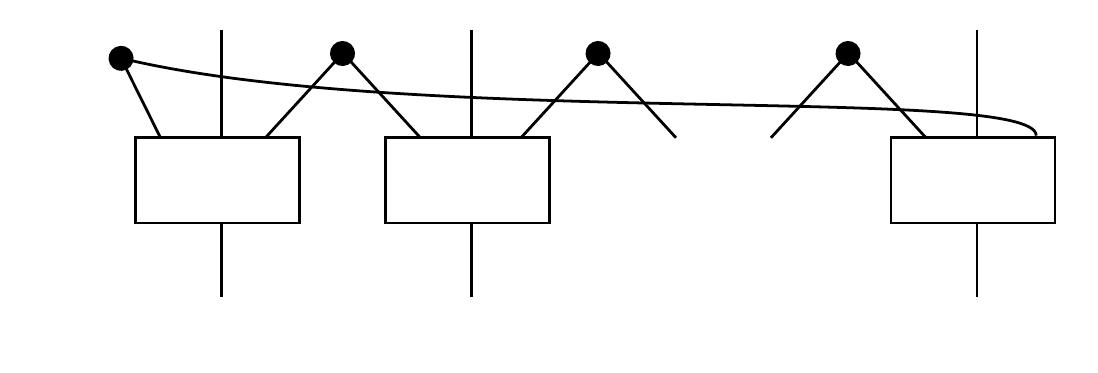}}%
    \put(0.275633,0.32021093){\color[rgb]{0,0,0}\makebox(0,0)[lb]{\smash{{\scriptsize
$(a_2',y_2)$}}}}%
    \put(0.18129688,0.0300766){\color[rgb]{0,0,0}\makebox(0,0)[lb]{\smash{$X_{x_1}^{\eps_1}$}}}%
    \put(0.13495429,0.16168134){\color[rgb]{0,0,0}\makebox(0,0)[lb]{\smash{$F_{a_1'',l,a_2}^{(x_1,\eps_1)}$}}}%
    \put(0.3626785,0.16168134){\color[rgb]{0,0,0}\makebox(0,0)[lb]{\smash{$F_{a_2'',l,a_3}^{(x_2,\eps_2)}$}}}%
    \put(0.18129688,0.32358786){\color[rgb]{0,0,0}\makebox(0,0)[lb]{\smash{$X_{x_1}
^{\eps_1}$}}}%
    \put(0.40902117,0.0300766){\color[rgb]{0,0,0}\makebox(0,0)[lb]{\smash{$X_{x_2}^{\eps_2}$}}}%
    \put(0.40902117,0.32358786){\color[rgb]{0,0,0}\makebox(0,0)[lb]{\smash{$X_{x_2}
^{\eps_2}$}}}%
    \put(0.505247,0.32021093){\color[rgb]{0,0,0}\makebox(0,0)[lb]{\smash{{\scriptsize
$(a_3',y_3)$}}}}%
    \put(0.72187329,0.32021093){\color[rgb]{0,0,0}\makebox(0,0)[lb]{\smash{{\scriptsize
$(a_n',y_n)$}}}}%
    \put(0.8695302,0.0300766){\color[rgb]{0,0,0}\makebox(0,0)[lb]{\smash{$X_{x_n}
^{\eps_n}$}}}%
    \put(0.8695302,0.32358786){\color[rgb]{0,0,0}\makebox(0,0)[lb]{\smash{$X_{x_n}
^{\eps_n}$}}}%
    \put(0.63750493,0.21386688){\color[rgb]{0,0,0}\makebox(0,0)[lb]{\smash{$\dots$}}}%
    \put(0.82318754,0.16168134){\color[rgb]{0,0,0}\makebox(0,0)[lb]{\smash{$F_{a_n'',l,a_1}
^{(x_n,\eps_n)}$}}}%
    \put(0.05485904,0.32021093){\color[rgb]{0,0,0}\makebox(0,0)[lb]{\smash{{\scriptsize
$(a_1',y_1)$}}}}%
    \put(-0.00168026,0.16168134){\color[rgb]{0,0,0}\makebox(0,0)[lb]{\smash{$E_{\underline{a}}
^{\underline{x}}:=$}}}%
  \end{picture}%
\endgroup%

	\end{aligned}\ ,
	\label{eq:state-sum-Ex-defect}
\end{align}
where $\underline{a}=(a_1+a_1'+a_1'',\dots,a_n+a_n'+a_n'',l)\in(\Rb_{>0})^{n+1}$
and the values of $y_i$ are determined by $\underline{x}$ via 
\eqref{eq:source-target-simple}, i.e.\ $y_{i}=s(x_i,\epsilon_i)=t(x_{i+1},\epsilon_{i+1})$.
Different choices of representatives  of $\underline{x}$
	induce different morphisms via \eqref{eq:state-sum-Ex-defect},
	which are related by conjugating with cyclic permutations of tensor factors.

With these preparations, we can now state the conditions 
state-sum data with defects $\Ab(\Db)$ has to satisfy. Namely, let 
$x\in D_1$, $\epsilon\in\left\{ \pm \right\}$
and let $\underline{x}$ be a defect list. Then:

\begin{enumerate}
	\item Glueing plaquette weights with defects:
\begin{align}
	\begin{aligned}
	\def\svgwidth{13cm}
\begingroup%
  \makeatletter%
  \providecommand\color[2][]{%
    \errmessage{(Inkscape) Color is used for the text in Inkscape, but the package 'color.sty' is not loaded}%
    \renewcommand\color[2][]{}%
  }%
  \providecommand\transparent[1]{%
    \errmessage{(Inkscape) Transparency is used (non-zero) for the text in Inkscape, but the package 'transparent.sty' is not loaded}%
    \renewcommand\transparent[1]{}%
  }%
  \providecommand\rotatebox[2]{#2}%
  \ifx\svgwidth\undefined%
    \setlength{\unitlength}{331.71508789bp}%
    \ifx\svgscale\undefined%
      \relax%
    \else%
      \setlength{\unitlength}{\unitlength * \real{\svgscale}}%
    \fi%
  \else%
    \setlength{\unitlength}{\svgwidth}%
  \fi%
  \global\let\svgwidth\undefined%
  \global\let\svgscale\undefined%
  \makeatother%
  \begin{picture}(1,0.17408021)%
    \put(0,0){\includegraphics[width=\unitlength]{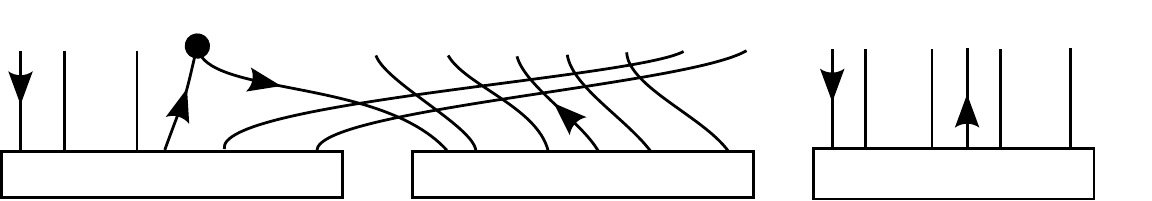}}%
    \put(0.14564803,0.15209221){\color[rgb]{0,0,0}\makebox(0,0)[lb]{\smash{\scriptsize{$(a_0,l_0,b_0;x)$}}}}%
    \put(0.66538285,0.0759882){\color[rgb]{0,0,0}\makebox(0,0)[lb]{\smash{$=$}}}%
    \put(0.05503211,0.01750321){\color[rgb]{0,0,0}\makebox(0,0)[lb]{\smash{\scriptsize{$a_1,l_1,b_1;x,n_1,m_1$}}}}%
    \put(0.07571474,0.072676){\color[rgb]{0,0,0}\makebox(0,0)[lb]{\smash{$\dots$}}}%
    \put(0.22128262,0.04967889){\color[rgb]{0,0,0}\makebox(0,0)[lb]{\smash{$\dots$}}}%
    \put(0.42643517,0.01750321){\color[rgb]{0,0,0}\makebox(0,0)[lb]{\smash{\scriptsize{$a_2,l_2,b_2;x,n_2,m_2$}}}}%
    \put(0.41695394,0.04842221){\color[rgb]{0,0,0}\makebox(0,0)[lb]{\smash{$\dots$}}}%
    \put(0.56680192,0.05324563){\color[rgb]{0,0,0}\makebox(0,0)[lb]{\smash{$\dots$}}}%
    \put(0.75254555,0.01566521){\color[rgb]{0,0,0}\makebox(0,0)[lb]{\smash{\scriptsize{$a,l,b;x,n,m$}}}}%
    \put(0.76116788,0.07494806){\color[rgb]{0,0,0}\makebox(0,0)[lb]{\smash{$\dots$}}}%
    \put(0.8847588,0.07620473){\color[rgb]{0,0,0}\makebox(0,0)[lb]{\smash{$\dots$}}}%
    \put(0.06513062,0.14433242){\color[rgb]{0,0,0}\makebox(0,0)[lb]{\smash{$A_
{t(x)}$}}}%
    \put(0.33138636,0.14090101){\color[rgb]{0,0,0}\makebox(0,0)[lb]{\smash{$A_
{t(x)}$}}}%
    \put(0.75623053,0.14705803){\color[rgb]{0,0,0}\makebox(0,0)[lb]{\smash{$A_
{t(x)}$}}}%
    \put(0.4946106,0.14227358){\color[rgb]{0,0,0}\makebox(0,0)[lb]{\smash{$A_
{s(x)}$}}}%
    \put(0.60031677,0.14295986){\color[rgb]{0,0,0}\makebox(0,0)[lb]{\smash{$A_
{s(x)}$}}}%
    \put(0.87619991,0.14568549){\color[rgb]{0,0,0}\makebox(0,0)[lb]{\smash{$A_
{s(x)}$}}}%
    \put(0.71591817,0.14857837){\color[rgb]{0,0,0}\makebox(0,0)[lb]{\smash{$\bar{X}
_x$}}}%
    \put(0.83084084,0.14698036){\color[rgb]{0,0,0}\makebox(0,0)[lb]{\smash{$X
_x$}}}%
    \put(0.43054541,0.14091661){\color[rgb]{0,0,0}\makebox(0,0)[lb]{\smash{$X
_x$}}}%
    \put(0.00666171,0.13954409){\color[rgb]{0,0,0}\makebox(0,0)[lb]{\smash{$\bar{X}
_x$}}}%
  \end{picture}%
\endgroup%

	\end{aligned}\ ,
	\label{eq:glue-defect-graphical}
\end{align}
		for every $a=a_0+a_1+a_2$, $n=n_1+n_2$, etc.
        	\label{cond:glue-defect}
        \item Glueing plaquette weights with and without defects:
\begin{align}
	\begin{aligned}
	\def\svgwidth{13cm}
\begingroup%
  \makeatletter%
  \providecommand\color[2][]{%
    \errmessage{(Inkscape) Color is used for the text in Inkscape, but the package 'color.sty' is not loaded}%
    \renewcommand\color[2][]{}%
  }%
  \providecommand\transparent[1]{%
    \errmessage{(Inkscape) Transparency is used (non-zero) for the text in Inkscape, but the package 'transparent.sty' is not loaded}%
    \renewcommand\transparent[1]{}%
  }%
  \providecommand\rotatebox[2]{#2}%
  \ifx\svgwidth\undefined%
    \setlength{\unitlength}{339.71508789bp}%
    \ifx\svgscale\undefined%
      \relax%
    \else%
      \setlength{\unitlength}{\unitlength * \real{\svgscale}}%
    \fi%
  \else%
    \setlength{\unitlength}{\svgwidth}%
  \fi%
  \global\let\svgwidth\undefined%
  \global\let\svgscale\undefined%
  \makeatother%
  \begin{picture}(1,0.16613657)%
    \put(0,0){\includegraphics[width=\unitlength]{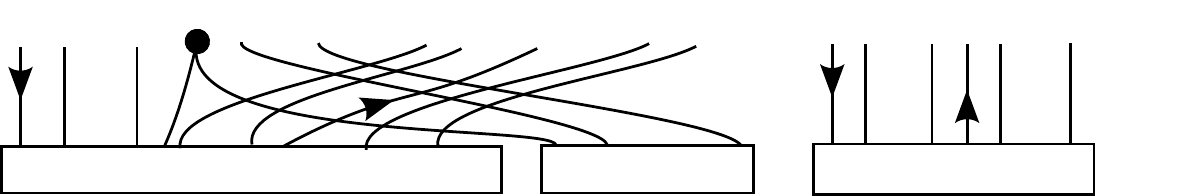}}%
    \put(0.14156703,0.1518231){\color[rgb]{0,0,0}\makebox(0,0)[lb]{\smash{\scriptsize{$(a_0;t(x))$}}}}%
    \put(0.64971365,0.07419874){\color[rgb]{0,0,0}\makebox(0,0)[lb]{\smash{$=$}}}%
    \put(0.14120145,0.01102698){\color[rgb]{0,0,0}\makebox(0,0)[lb]{\smash{\scriptsize{$a_1,l,b;x,n_1,m$}}}}%
    \put(0.06451207,0.07096455){\color[rgb]{0,0,0}\makebox(0,0)[lb]{\smash{$\dots$}}}%
    \put(0.16688042,0.04917063){\color[rgb]{0,0,0}\makebox(0,0)[lb]{\smash{$\dots$}}}%
    \put(0.49795684,0.01372143){\color[rgb]{0,0,0}\makebox(0,0)[lb]{\smash{\scriptsize{$a_2;s(x),n_2$}}}}%
    \put(0.52298471,0.04928335){\color[rgb]{0,0,0}\makebox(0,0)[lb]{\smash{$\dots$}}}%
    \put(0.75837289,0.01529631){\color[rgb]{0,0,0}\makebox(0,0)[lb]{\smash{\scriptsize{$a,l,b;x,n,m$}}}}%
    \put(0.74324302,0.07318309){\color[rgb]{0,0,0}\makebox(0,0)[lb]{\smash{$\dots$}}}%
    \put(0.86392348,0.07441018){\color[rgb]{0,0,0}\makebox(0,0)[lb]{\smash{$\dots$}}}%
    \put(0.06359685,0.14093351){\color[rgb]{0,0,0}\makebox(0,0)[lb]{\smash{$A_
{t(x)}$}}}%
    \put(0.29727066,0.1442302){\color[rgb]{0,0,0}\makebox(0,0)[lb]{\smash{$A_
{t(x)}$}}}%
    \put(0.73842195,0.14359494){\color[rgb]{0,0,0}\makebox(0,0)[lb]{\smash{$A_
{t(x)}$}}}%
    \put(0.55638744,0.14495033){\color[rgb]{0,0,0}\makebox(0,0)[lb]{\smash{$A_
{s(x)}$}}}%
    \put(0.85556615,0.14225472){\color[rgb]{0,0,0}\makebox(0,0)[lb]{\smash{$A_
{s(x)}$}}}%
    \put(0.69905891,0.14507948){\color[rgb]{0,0,0}\makebox(0,0)[lb]{\smash{$\bar{X}
_x$}}}%
    \put(0.81127525,0.1435191){\color[rgb]{0,0,0}\makebox(0,0)[lb]{\smash{$X
_x$}}}%
    \put(0.00650483,0.13625794){\color[rgb]{0,0,0}\makebox(0,0)[lb]{\smash{$\bar{X}
_x$}}}%
    \put(0.32453874,0.04775202){\color[rgb]{0,0,0}\makebox(0,0)[lb]{\smash{$\dots$}}}%
    \put(0.44775443,0.13768063){\color[rgb]{0,0,0}\makebox(0,0)[lb]{\smash{$X
_x$}}}%
  \end{picture}%
\endgroup%

	\end{aligned}\ ,
	\label{eq:glue-empty-defect-graphical}
\end{align}
for every $a=a_0+a_1+a_2$, $l,b\in\Rb_{>0}$, $n=n_1+n_2$.
		A similar condition needs to hold with the role of $s(x)$ and $t(x)$ exchanged.
        	\label{cond:glue-defect-empty}
	\item  ``Moving $\zeta$'s around'':
\begin{align}
	\begin{aligned}
	\def\svgwidth{13cm}
	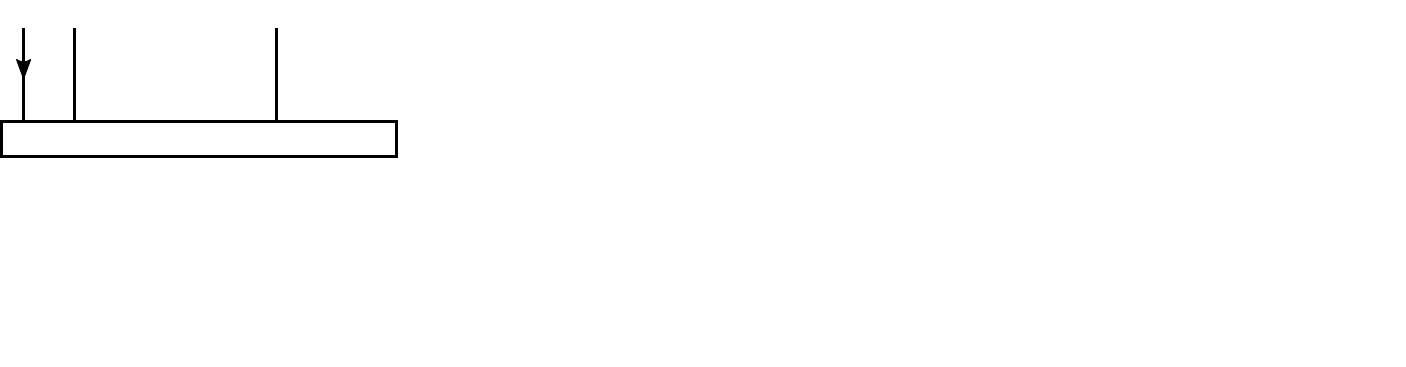
	\end{aligned}
	\label{eq:cond:W-zeta-defect}
\end{align}
for $a_1+a_2=a$, $b_1+b_2=b$, $l_1+l_2=l$, $n,m\ge0$.
		\label{cond:W-zeta-defect}
	\item The limit $\lim_{\underline{a}\to0}E_{\underline{a}}^{\underline{x}}$ exists, and $\lim_{a,b,l\to0}Q^{(x,\epsilon)}_{a,l,b}=\id_{X_x^{\epsilon}}$.
		\label{cond:defect-cylinder}

	\item For every $n,m\ge0$ with $n+m\ge1$, $(x_i,\epsilon_i)\in D_1\times\left\{ \pm \right\}$ for $i=1,\dots,n$,
		$p_j\in D_2$ for $j=1,\dots,m$
		the assignment 
		\begin{align}
			\begin{aligned}
				(\Rb_{\ge0})^{3n+m}&\to\Sc\left( \bigotimes_{i=1}^n X_{x_i}^{\epsilon_i}\otimes \bigotimes_{j=1}^{m}A_{p_j},
				\bigotimes_{i=1}^n X_{x_i}^{\epsilon_i}\otimes \bigotimes_{j=1}^{m}A_{p_j}\right)\\
				(a_1,l_1,b_1,\dots,a_n,l_n,b_n,c_1,\dots,c_m)&\mapsto
				\bigotimes_{i=1}^n Q^{(x_i,\epsilon_i)}_{a_i,l_i,b_i}\otimes \bigotimes_{j=1}^{m} P^{(p_j)}_{c_j}
			\end{aligned}
			\label{eq:cont-cond-defects}
		\end{align}
		is jointly continuous.
        	\label{cond:Q-cont}
\end{enumerate}

We have the analogue of Lemma~\ref{lem:D0-idempotent}, 
which can be proven using 
Conditions~\ref{cond:glue-defect},~\ref{cond:glue-defect-empty}~and~\ref{cond:defect-cylinder}.

\begin{lemma}
	For every defect list $\underline{x}$ of length $n\in\Zb_{\ge1}$ and $\underline{a},\underline{a}'\in(\Rb_{\ge0})^{n+1}$,
\begin{align}
	E_{\underline{a}}^{\underline{x}}\circ
	E_{\underline{a}'}^{\underline{x}}=
	E_{\underline{a}+\underline{a}'}^{\underline{x}} \ .
	\label{eq:Ea-E0}
\end{align} 
In particular, the morphism 
\begin{align}
	{E}_0^{\underline{x}}:=\lim_{\underline{a}\to0}{E}_{\underline{a}}^{\underline{x}}
	\in\Sc(X_{\underline{x}},X_{\underline{x}}) 
	\label{eq:E0-def}
\end{align}
is idempotent.
	\label{lem:E0-idempotent}
\end{lemma}

Let us fix
state-sum data with defects $\Ab(\Db)$.
In the rest of this section we define a symmetric monoidal functor $\funZ_{\Ab(\Db)}:\Borddef{\Db}\to\Sc$
using this data.

By our assumptions, the idempotents in \eqref{eq:E0-def} split.
Let $Z(X_{\underline{x}})\in\Sc$ denote the image and
write $\pi_{\underline{x}}$ and 	$\iota_{\underline{x}}$ for the projection and embedding, i.e.
\begin{align}
	\begin{aligned}
		{E}_0^{\underline{x}}=&
		\left[ X_{\underline{x}}\xrightarrow{\pi_{\underline{x}}}Z(X_{\underline{x}})\xrightarrow{\iota_{\underline{x}}}X_{\underline{x}} \right]\ ,&
		\id_{Z(X_{\underline{x}})}&= \left[ Z(X_{\underline{x}})\xrightarrow{\iota_{\underline{x}}}X_{\underline{x}}\xrightarrow{\pi_{\underline{x}}}Z(X_{\underline{x}}) \right]\ .
	\end{aligned}
	\label{eq:E0-split-idempot}
\end{align}
Note that different choices of representative in~\eqref{eq:defect-list-tensor} give the same image $Z(X_{\underline{x}})$ since the idempotents ${E}_0^{\underline{x}}$ commute with cyclic permutations.

We will also write 
$X_{()}^{(b)}=A_{d_2(b)}$,
\begin{align}
\iota_{()}^{(b)}=\iota_{A_{d_2(b)}}:
Z(A_{d_2(b)})\to A_{d_2(b)} \quad\text{ and }\quad
\pi_{()}^{(b)}=\pi_{A_{d_2(b)}}:A_{d_2(b)} \to Z(A_{d_2(b)})\ .
\label{eq:empty-list-defs}
\end{align}

\subsubsection*{Defining $\funZ_{\Ab(\Db)}$}

We define the {\aQFT} $\funZ_{\Ab(\Db)}$ on objects as follows:
Let $S\in\Borddef{\Db}$ and $c\in\pi_0(S)$. If $c\cap S_{[0]}=\emptyset$ then let $\underline{x}(c):=()$ be the empty list, and $Z(X_{()})^{(c)}:=Z(A_{d_2(c)})$.
Otherwise, for every $c\in\pi_0(S)$ let 
\begin{align}
	\underline{x}(c):=
	[(d_1(v),\epsilon(v))_{v\in c\cap S_{[0]}}]
	\label{eq:defect-list-circle}
\end{align}
be the defect list 
given by the defect labels $d_1(v)$ and orientations $\epsilon(v)$ of the defects in $c$ 
in the cyclic order determined by the orientation of $c$.
We define $\funZ_{\Ab(\Db)}$ on objects as
\begin{align}
	\funZ_{\Ab(\Db)}(S):=\bigotimes_{c\in \pi_0(S)}Z(X_{\underline{x}(c)})^{(c)}\ ,
	\label{eq:defect-aqft:obj}
\end{align}
where as in \eqref{eq:tft:obj-aqft} the superscript is used to label the tensor factors.

\medskip

The definition of $\funZ_{\Ab(\Db)}$ on morphisms is again more involved.
Let $(\Sigma,\Ac,\Lc):S\to T$ be a bordism with area and defects
and assume that it has no component with zero area or length.
Choose a PLCW decomposition with area and defects (with the same notation as in Section~\ref{sec:PLCW-dec-defect}) of
the surface with area and defects $(\Sigma,\Ac,\Lc)$.

\begin{figure}[tb]
	\centering
	\def\svgwidth{2.5cm}
\begingroup%
  \makeatletter%
  \providecommand\color[2][]{%
    \errmessage{(Inkscape) Color is used for the text in Inkscape, but the package 'color.sty' is not loaded}%
    \renewcommand\color[2][]{}%
  }%
  \providecommand\transparent[1]{%
    \errmessage{(Inkscape) Transparency is used (non-zero) for the text in Inkscape, but the package 'transparent.sty' is not loaded}%
    \renewcommand\transparent[1]{}%
  }%
  \providecommand\rotatebox[2]{#2}%
  \newcommand*\fsize{\dimexpr\f@size pt\relax}%
  \newcommand*\lineheight[1]{\fontsize{\fsize}{#1\fsize}\selectfont}%
  \ifx\svgwidth\undefined%
    \setlength{\unitlength}{208.89946722bp}%
    \ifx\svgscale\undefined%
      \relax%
    \else%
      \setlength{\unitlength}{\unitlength * \real{\svgscale}}%
    \fi%
  \else%
    \setlength{\unitlength}{\svgwidth}%
  \fi%
  \global\let\svgwidth\undefined%
  \global\let\svgscale\undefined%
  \makeatother%
  \begin{picture}(1,0.8240537)%
    \lineheight{1}%
    \setlength\tabcolsep{0pt}%
    \put(0,0){\includegraphics[width=\unitlength,page=1]{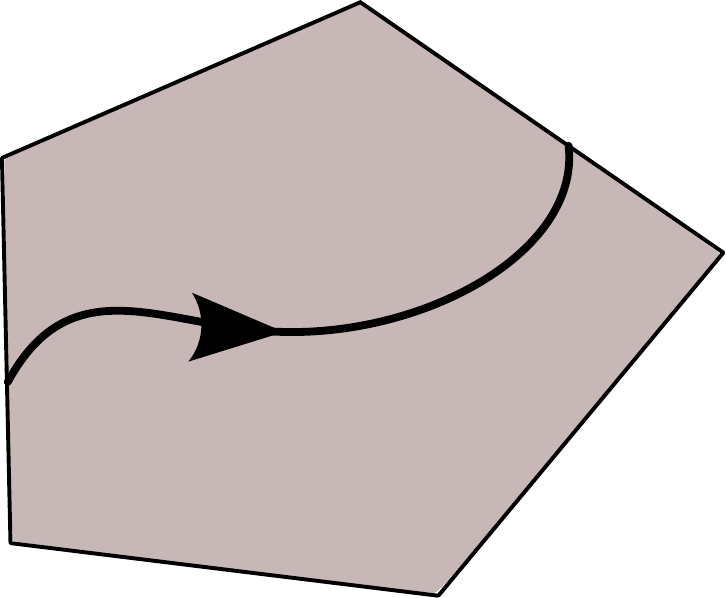}}%
    \put(0.61233935,0.23516009){\color[rgb]{0,0,0}\makebox(0,0)[lt]{\lineheight{0}\smash{\begin{tabular}[t]{l}$e$\end{tabular}}}}%
    \put(0,0){\includegraphics[width=\unitlength,page=2]{positive-crossing.pdf}}%
    \put(0.16729447,0.42234767){\color[rgb]{0,0,0}\makebox(0,0)[lt]{\lineheight{0}\smash{\begin{tabular}[t]{l}$c$\end{tabular}}}}%
  \end{picture}%
\endgroup%

	\caption{Positive crossing of an oriented edge $e$ and a defect line $c$.}
	\label{fig:positive-crossing}
\end{figure}
Let us choose a marked edge for every face in $\Sigma_2^\mathrm{empty}$ 
and for every face in $\Sigma_2^\mathrm{defect}$
let the marked edge be the one where the defect line leaves.
Also let us choose an orientation of every edge,
requiring that the orientation of edges in $\Sigma_1^\mathrm{defect}$ are such that the edges and the defect lines cross positively
as shown in Figure~\ref{fig:positive-crossing}.
	
We introduce the sets of sides of faces $F$ 
for faces and the set of sides of edges $E$ and
the bijection $\Phi:F\to E$ from \eqref{eq:face-side-edge-bijection}
as in Section~\ref{sec:latticedata}.
We choose the map $V:\Sigma_0\setminus \pi_0(T)\to E$ as in \eqref{eq:choose-vertex-edge}
so that the map $\bar{V}$ from \eqref{eq:choose-edge-to-vertex-V-bar} satisfies
	\begin{align}
		\bar{V}|_{\Sigma_0\setminus \pi_0(T)}=\left[ \Sigma_0\setminus \pi_0(T)\xrightarrow{V} E\xrightarrow{\text{forget}} \Sigma_1 \right]
		\ ,
		\label{eq:V-V-bar-compatibility}
	\end{align}
	where the map `forget' is $(e,x)\mapsto e$. In addition, $V$ has to satisfy 
that if $v$ is on the left side of the defect line crossing the edge $\bar{V}(v)$ then $V(v)=(\bar{V}(v),r)$, otherwise $V(v)=(\bar{V}(v),l)$.

\begin{figure}[tb]
	\centering
	\def\svgwidth{5cm}
	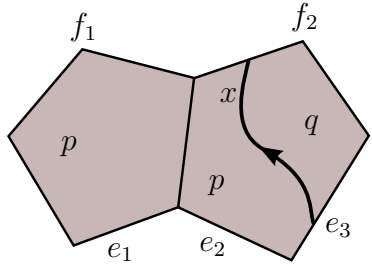
	\caption{Notation for phase labels of empty cells and defect line labels of cells with defects. The phase label of the surface component left to the defect line is $p$, the phase label of the surface component to the right is $q$ and the defect label is $x$, i.e. $t(x)=p$ and $s(x)=q$.
		The face $p_1$ and the edges $e_1$ and $e_2$ on the left are empty, i.e.\ $p_1\in\Sigma_2^\mathrm{empty}$ and $e_1,e_2\in\Sigma_1^\mathrm{empty}$.
		The corresponding phase labels are $d_2(f_1)=d_1(f_1)=d_2(e_1)=d_1(e_1)=d_2(e_2)=d_1(e_2)=p$.
			The face $f_2$ on the right and the edge $e_3$ on the right are intersected by a defect line, i.e.\ $f_2\in\Sigma_2^\mathrm{defect}$ and $e_3\in\Sigma_1^\mathrm{defect}$.
			The corresponding defect labels are $d_1(f_2)=d_1(e_3)=x$.
	}
	\label{fig:defect-label-of-cells}
\end{figure}

	It will be convenient for the state-sum construction to know
	the phase labels of surface
	components in which faces and edges that	
	are not intersected by defect lines lie.
	Similarly we will need to know the 
	defect line labels of components 
	intersected by faces and edges. Therefore
	we introduce the following for $k\in\{1,2\}$:
\begin{itemize}
	\item if $x\in\Sigma_k^\mathrm{empty}$ 
	we write $d_2(x)=d_1(x)=d_2(p)$ for the component
		$p\in\pi_0(\Sigma_{[2]})$ 
		in which $x$ lies,
	\item if $x\in\Sigma_k^\mathrm{defect}$ we write $d_1(x)=d_1(q)$ for the
		defect line $q\in\pi_0(\Sigma_{[1]})$ 
		intersected by $x$,
\end{itemize}
which we illustrate in Figure~\ref{fig:defect-label-of-cells}.

\begin{figure}[tb]
	\centering
	\def\svgwidth{4cm}
	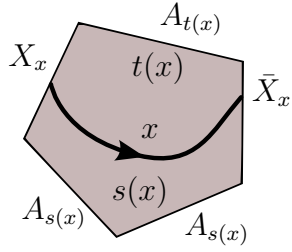
	\caption{
		Objects from the state sum data with defects assigned to edges crossed by a defect line with defect line label $x\in D_1$.}
	\label{fig:edge-labels}
\end{figure}
After introducing these notations we are ready to define $\funZ_{\Ab(\Db)}(\Sigma,\Ac,\Lc)$. We proceed with the following steps.
\begin{enumerate}
	\item Let $f\in\Sigma_2$ be a face with $n_f$ sides.
		If $f\in\Sigma_2^\mathrm{empty}$ then let $R^{(f,k)}:=A_{d_2(f)}$.
		If $f\in\Sigma_2^\mathrm{defect}$ then let $n_f^o$ be the number of the edge
		where the defect with label $x$ enters $f$. Then let 
		\begin{align}
			R^{(f,k)}:&=
			\begin{cases}
				\bar{X}_x & \text{if $k=1$,}\\
				A_{t(x)}&\text{if $1<k<n_f^o$,}\\
				X_x & \text{if $k=n_f^o$,}\\
				A_{s(x)}&\text{if $n_f^o<k$,}\\
			\end{cases}
			\label{eq:tensor-factors-R-F}
		\end{align}
		and for a side of an edge $(e,y)\in E$
		\begin{align}
			R^{(e,y)}:&=R^{\Phi^{-1}(e,y)}\ .
			\label{eq:tensor-factors-R-E}
		\end{align}
		For these conventions see Figure~\ref{fig:edge-labels}.
		
		Let us introduce the tensor products
		\begin{align}
			\begin{aligned}
			\Oc_{F}&:=\bigotimes_{(f,k) \in F}R^{(f,k)}\ ,&
			\Oc_{E}&:=\bigotimes_{(e,y) \in E}R^{(e,y)}\ ,\\
			\Oc_\mathrm{in}&:=\bigotimes_{b\in \pi_0(S)}X_{\underline{x}(b)}^{(b,in)}\ ,&
			\Oc_\mathrm{out}&:=\bigotimes_{c\in \pi_0(T)}X_{\underline{x}(c)}^{(c,out)}\ ,
			\end{aligned}
		\end{align}
	using the notation from \eqref{eq:E0-split-idempot} and \eqref{eq:tensor-factors-R-F}.
	The various superscripts will help us distinguish tensor factors 
	in the source and target objects of the morphisms we define in the remaining steps.

		\item We define the morphism
			\begin{align}
				\Cc:=
				\bigotimes_{e\in \Sigma_1\setminus \pi_0(T)} 
				\beta^{(e)}:
				\Oc_\mathrm{in}\otimes\Oc_E\to\Oc_\mathrm{out}\ ,
			\end{align}
			where $\beta^{(e)}=\beta^{d_1(e)}_{\Ac_1(e)}$ with the
			tensor factors given in Figure~\ref{fig:connect-aqft}.

	\item We define the morphism
		\begin{align}
			\Yc
			:=\prod_{v\in\Sigma_0\setminus \pi_0(T)}
			\zeta_{\Ac_0(v)}^{(V(v))}
			\in\Sc(\Oc_E,\Oc_E)\ ,
		\end{align}
		where 
		\begin{align}
			\zeta_a^{(e,y)} = 
			\begin{cases}
				\id \otimes \cdots \otimes \zeta_a^{d_2(e)} \otimes \cdots \otimes \id&\text{ ; if $e\in\Sigma_1^\mathrm{empty}$}\\
				\id \otimes \cdots \otimes \zeta_a^{d_1(e),+/-} \otimes \cdots \otimes \id&\text{ ; if $e\in\Sigma_1^\mathrm{defect}$}
			\end{cases}
			\in\Sc(\Oc_E,\Oc_E) \ ,
		\end{align}
		where $\zeta_a$ maps the tensor factor $R^{(e,y)}$ 
		to itself, and $a\in\Rb_{>0}$ or $a\in\Rb_{>0}^3$. 

	\item For $f\in\Sigma_2^\mathrm{defect}$ let
	$n_f$ and $n_f^o$ be as in step 1 and
		\begin{align}
		W_{\Ac_2(f)}^{f}:=W_{\Ac_2(f)}^{d_1(f),n_f-n_f^o,n_f^o-2}\ ; 
			\label{eq:step1a-defect}
		\end{align}
		for $f\in\Sigma_2^\mathrm{empty}$ let $n_f$ be as before and
		\begin{align}
		W_{\Ac_2(f)}^{f}:=W_{\Ac_2(f)}^{d_1(f),n_f}\ . 
			\label{eq:step1b-defect}
		\end{align}
		In both cases the labeling of 
		tensor factors is such that it matches \eqref{eq:tensor-factors-R-F}.
		Define the morphism
		\begin{align}
			\Fc:=\bigotimes_{f\in\Sigma_2}\left( W_{\Ac_2(f)}^{f}\right):\Ib\to\Oc_{F} \ .
			\label{eq:step1-defect}
		\end{align}

	\item 
		We  again put the above morphisms together as in Step~\ref{step3-aqft} of Section~\ref{sec:latticedata}:
	\begin{align}	
		\Kc&:=\left[\Ib\xrightarrow{\Fc}\Oc_F\xrightarrow{\Uppi_\Phi}\Oc_E
		\xrightarrow{\Yc}
		\Oc_E\right]\ ,\\
		\Lc&:=
		\left[ \Oc_\mathrm{in}\xrightarrow{\id_{\Oc_\mathrm{in}}\otimes\Kc} \Oc_\mathrm{in}\otimes\Oc_{E} \xrightarrow{\Cc} \Oc_\mathrm{out}  \right]\ ,
		\label{eq:step6a-defect}
	\end{align}
	where $\Uppi_\Phi$ is defined as in \eqref{eq:step3-aqft}.

	\item 
Using the embedding and projection maps from \eqref{eq:E0-split-idempot}
we construct the following morphisms:
		\begin{align}
			\Ec_\mathrm{in}:=& \bigotimes_{b\in \pi_0(S)}\iota_{\underline{x}(b)}^{(b)}:\funZ_{\Ab(\Db)}(S)\to\Oc_\mathrm{in}\ ,&
			\Ec_\mathrm{out}:=&\bigotimes_{c\in \pi_0(T)}\pi_{\underline{x}(c)}^{(c)}:\Oc_\mathrm{out}\to \funZ_{\Ab(\Db)}(T)\ .
		\end{align}
		We finally define the action of $\funZ_{\Ab(\Db)}$ on morphisms:		
		\begin{align}				\funZ_{\Ab(\Db)}(\Sigma,\Ac,\Lc)&:=
			\left[ \funZ_{\Ab(\Db)}(S)\xrightarrow{\Ec_\mathrm{in}} \Oc_\mathrm{in}\xrightarrow{\Lc} \Oc_\mathrm{out}\xrightarrow{\Ec_\mathrm{out}} \funZ_{\Ab(\Db)}(T)\right]\ .
			\label{eq:step6b-defect}
		\end{align}
\end{enumerate}

	\medskip

	We defined $\funZ_{\Ab(\Db)}$ on bordisms with defects with strictly positive area and 
	length and now we give the definition in the general case.
	Let $(\Sigma,\Ac,\Lc):S\to T$ be a bordism with area  and defects and 
let $\Sigma_+:S_+\to{T}_+$ denote the connected component of $(\Sigma,\Ac)$ with strictly positive area and length. 
The complement of $\Sigma_+$ 
again defines a permutation of tensor factors as in Section~\ref{sec:latticedata},
so we define:
\begin{align}
	\funZ_{\Ab(\Db)}(\Sigma,\Ac,\Lc):=\funZ_{\Ab(\Db)}(\Sigma\setminus{\Sigma}_+,0,0)\otimes \funZ_{\Ab(\Db)}({\Sigma}_+,\Ac_+,\Lc_+)\ ,
	\label{eq:za-defect-fulldef}
\end{align}
where $\Ac_+$ denotes the restriction of $\Ac$ to 
$\pi_0( (\Sigma_+)_{[k]})$, $k=1,2$, 
and $\Lc_+$ is defined similarly
and $\funZ_{\Ab(\Db)}({\Sigma}_+,\Ac_+,\Lc_+)$ is defined in \eqref{eq:step6b-defect}.

	We have the analogous theorem of Section~\ref{sec:latticedata}. 
\begin{theorem} \label{thm:state-sum-defect-aqft}
Let $\Ab(\Db)$ be state-sum data with defects.
	\begin{enumerate}
		\item The morphism defined in \eqref{eq:step6b-defect} is 
			independent of the choice of the PLCW decomposition with area and defects,
			the choice of marked edges of faces, the choice of orientation of edges
			and the assignment $V$. \label{thm:defect-aqft:1}
	\item The state-sum construction yields an 
		{\aQFT} $\funZ_{\Ab(\Db)}:\Borddef{\Db}\to\Sc$
		given by 
		\eqref{eq:defect-aqft:obj} and \eqref{eq:za-defect-fulldef}, respectively.  \label{thm:defect-aqft:2}
\end{enumerate}
\end{theorem}
\begin{proof}[Sketch of proof]
We only sketch some part of the proof of Part~\ref{thm:aqft:1}.
We will check invariance under the additional elementary moves in Figure~\ref{fig:elementary-defect}.
Invariance under moves $b)$ and $c)$ directly follow from 
Condition~\ref{cond:glue-defect}~and~\ref{cond:glue-defect-empty} respectively.
Invariance under move $a)$ can be shown using the same trick as in the proof of
Theorem~\ref{thm:state-sum-aqft} Part~\ref{thm:aqft:1} by combining 
the moves $b)$ and $c)$ together with the move in Figure~\ref{fig:univalent-aqft}.
We note that one needs to use Condition~\ref{cond:W-zeta-defect} to show independence of the choice of the map $V$.

Let $(C,\Ac,\Lc)$ be a cylinder over a circle with defects with defect list $\underline{x}$ and equal defect line lengths.
The morphism in \eqref{eq:step6a-defect} associated to $(C,\Ac,\Lc)$ is $E_{\underline{a}}^{\underline{x}}$ from \eqref{eq:state-sum-Ex-defect}.

The proof of Part~\ref{thm:defect-aqft:2}
goes along the same lines as the proof of Part~\ref{thm:aqft:2} of Theorem~\ref{thm:state-sum-defect-aqft}.
Joint continuity in the areas and lengths follows from Condition~\ref{cond:Q-cont}.
\end{proof}

\subsection{State-sum data with defects from bimodules}\label{sec:lattice:bimoduledata}

The purpose of this section is to give an algebraic characterisation of
state-sum data with defects. We show that given state-sum data with defects for some objects we get
a particular RFA and bimodule structure on the objects. This suggests that conversely given 
a particular RFA and bimodule structure on some objects we can get state-sum data on these objects.
As before, we keep the notation form the previous sections.

\begin{lemma}\label{lem:action-from-pdata}
For every $p\in D_2$ and $x\in D_1$ let us fix objects $A_p$ and $X_x$ in $\Sc$
and state-sum data with defects $\Ab(\Db)$ for these objects.
For $a=a_1+a_2+a_3$, $l=l_1+l_2$ and $b=b_1+b_2$ let
\begin{align}
\begin{aligned}
\def\svgwidth{14cm}
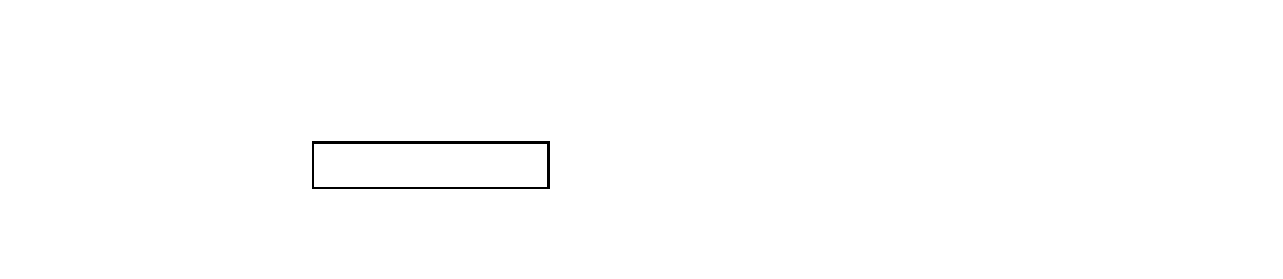
\end{aligned}\ .
\label{eq:action-from-pdata}
\end{align}

Then the state-sum data $\Ab(\Db)$ determines
\begin{itemize}
	\item a strongly separable symmetric RFA structure on
		$A_p$ for every $p\in D_2$
		as in Lemma~\ref{lem:data2rfa} and 
	\item a structure of a dual pair of 
		$A_{t(x)}$-$A_{s(x)}$-bimodules on
		$(X_x,\bar{X}_x)$ for every $x\in D_1$, where
		the actions on $X_x$ and $\bar X_x$ are given by $\rho_{a,l,b}^{x}$ and $\bar\rho_{a,l,b}^{x}$ from \eqref{eq:action-from-pdata} respectively, and
		the pairing by $\beta_{a,l,b}^{x}$ and the copairing by $\gamma_{a,l,b}^{x}:=W_{a,l,b}^{x,0,0}$.
\end{itemize}
\end{lemma}
\begin{proof}
Checking associativity \eqref{eq:ra:bimodule} can be easily done using the graphical calculus and
Conditions~\ref{cond:glue-defect}~and~\ref{cond:glue-defect-empty}.
The rest of the conditions on the action follow directly from the other conditions.
Also checking that the duality morphisms satisfy \eqref{eq:ra:dual} is straightforward.
\end{proof}
\begin{remark}
We note that contrary to the state-sum construction of topological field theories,
in general one cannot define left and right actions on the object $X_x$ as
the action \eqref{eq:action-from-pdata} always comes with three strictly positive additive parameters.
If for example the limit 
$\lim_{b\to0}\rho_{a,l,b_1}^x\circ(\id_{A_{t(x)\otimes X_x}}\otimes \eta_{b_2}^{A_{s(x)}})$ 
exists,
then one can define a left action on $X_x$, cf.\ Remark~\ref{rem:left-right-module-gives-bimodule}. 
Note, however, that these limits need not exist, see Appendix~\ref{app:bimod} for an example.
\end{remark}

The previous lemma indicates that one should be able to give state-sum data with defects from
a set of strongly separable symmetric RFAs and bimodules with duals.
The following lemma shows that if these bimodules 
satisfy some conditions pairwise,
then we indeed can obtain state-sum data, in particular the limits
$\lim_{\underline{a}\to0}E_{\underline{x}}^{\underline{a}}$ in Condition~\ref{cond:defect-cylinder} exists.

\begin{proposition}\label{prop:pdata-from-action}
For every $p\in D_2$ and $x\in D_1$ let 
$A_p$ be a strongly separable symmetric RFA and $(X_x,\bar{X}_x)$ a dual pair of
$A_{t(x)}$-$A_{s(x)}$-bimodules
with pairing $\beta_{a,l,b}^{x}$ and copairing $\gamma_{a,l,b}^{x}$.
For $n,m\in\Zb_{\ge0}$ set
\begin{align}
\begin{aligned}
\def\svgwidth{7cm}
\begingroup%
  \makeatletter%
  \providecommand\color[2][]{%
    \errmessage{(Inkscape) Color is used for the text in Inkscape, but the package 'color.sty' is not loaded}%
    \renewcommand\color[2][]{}%
  }%
  \providecommand\transparent[1]{%
    \errmessage{(Inkscape) Transparency is used (non-zero) for the text in Inkscape, but the package 'transparent.sty' is not loaded}%
    \renewcommand\transparent[1]{}%
  }%
  \providecommand\rotatebox[2]{#2}%
  \ifx\svgwidth\undefined%
    \setlength{\unitlength}{168.26191406bp}%
    \ifx\svgscale\undefined%
      \relax%
    \else%
      \setlength{\unitlength}{\unitlength * \real{\svgscale}}%
    \fi%
  \else%
    \setlength{\unitlength}{\svgwidth}%
  \fi%
  \global\let\svgwidth\undefined%
  \global\let\svgscale\undefined%
  \makeatother%
  \begin{picture}(1,0.71089978)%
    \put(0,0){\includegraphics[width=\unitlength]{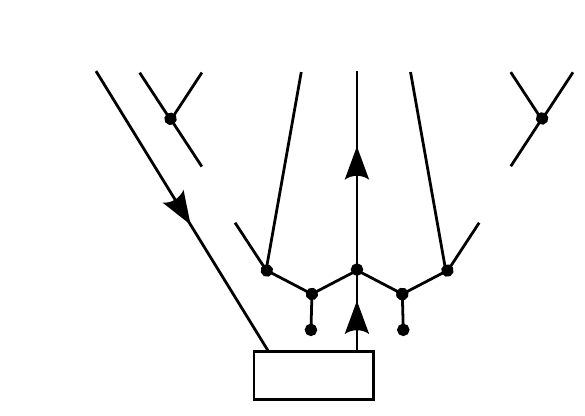}}%
    \put(0.59120803,0.6336786){\color[rgb]{0,0,0}\makebox(0,0)[lb]{\smash{$X_x$}}}%
    \put(0.45132057,1.79959482){\color[rgb]{0,0,0}\makebox(0,0)[lb]{\smash{
}}}%
    \put(0.46550592,0.05374805){\color[rgb]{0,0,0}\makebox(0,0)[lb]{\smash{$\gamma_
{a',l',b'}^{x}$}}}%
    \put(0.82387673,0.37971662){\color[rgb]{0,0,0}\makebox(0,0)[lb]{\smash{$\dots$}}}%
    \put(0.36004057,0.37261071){\color[rgb]{0,0,0}\makebox(0,0)[lb]{\smash{$\dots$}}}%
    \put(0.24298274,0.6336786){\color[rgb]{0,0,0}\makebox(0,0)[lb]{\smash{$\overbrace{{\qquad~~~}
\hphantom{6cm}}
^{A_{t(x)}^{\otimes n}}$}}}%
    \put(0.13615734,0.6336786){\color[rgb]{0,0,0}\makebox(0,0)[lb]{\smash{$\bar{X}_x$}}}%
    \put(-0.00315728,0.28563475){\color[rgb]{0,0,0}\makebox(0,0)[lb]{\smash{$W_
{a,l,b}^{x,n,m}:=$}}}%
    \put(0.70375426,0.6336786){\color[rgb]{0,0,0}\makebox(0,0)[lb]{\smash{$\overbrace
{{\qquad~~~}
\hphantom
{6cm}}
^{A_{s(x)}^
{\otimes m}}$}}}%
    \put(0.53303312,0.68200138){\color[rgb]{0,0,0}\makebox(0,0)[lb]{\smash{{}}}}%
  \end{picture}%
\endgroup%

\end{aligned}\ ,
\label{eq:pdata-from-action}
\end{align}
with some distribution of the parameters on the rhs which sums up to $a$, $b$ and $l$. 
Furthermore, let
\begin{align}
	\begin{aligned}
		\def\svgwidth{3.8cm}
\begingroup%
  \makeatletter%
  \providecommand\color[2][]{%
    \errmessage{(Inkscape) Color is used for the text in Inkscape, but the package 'color.sty' is not loaded}%
    \renewcommand\color[2][]{}%
  }%
  \providecommand\transparent[1]{%
    \errmessage{(Inkscape) Transparency is used (non-zero) for the text in Inkscape, but the package 'transparent.sty' is not loaded}%
    \renewcommand\transparent[1]{}%
  }%
  \providecommand\rotatebox[2]{#2}%
  \newcommand*\fsize{\dimexpr\f@size pt\relax}%
  \newcommand*\lineheight[1]{\fontsize{\fsize}{#1\fsize}\selectfont}%
  \ifx\svgwidth\undefined%
    \setlength{\unitlength}{99.83723551bp}%
    \ifx\svgscale\undefined%
      \relax%
    \else%
      \setlength{\unitlength}{\unitlength * \real{\svgscale}}%
    \fi%
  \else%
    \setlength{\unitlength}{\svgwidth}%
  \fi%
  \global\let\svgwidth\undefined%
  \global\let\svgscale\undefined%
  \makeatother%
  \begin{picture}(1,0.75095868)%
    \lineheight{1}%
    \setlength\tabcolsep{0pt}%
    \put(-0.00332573,0.42125253){\color[rgb]{0,0,0}\makebox(0,0)[lt]{\lineheight{0}\smash{\begin{tabular}[t]{l}$\zeta_{a,l,b}^{x,\epsilon}:=$\end{tabular}}}}%
    \put(0.54841605,0.51112782){\color[rgb]{0,0,0}\makebox(0,0)[lt]{\lineheight{0}\smash{\begin{tabular}[t]{l}{\scriptsize$(a_2,l_1,b_2;x,\epsilon)$}\end{tabular}}}}%
    \put(0,0){\includegraphics[width=\unitlength,page=1]{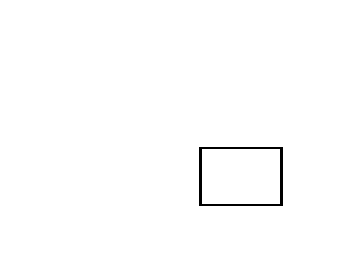}}%
    \put(0.60731005,0.21529851){\color[rgb]{0,0,0}\makebox(0,0)[lt]{\lineheight{0}\smash{\begin{tabular}[t]{l}$\tau^{-1}_{a_1}$\end{tabular}}}}%
    \put(0.44196287,0.70529842){\color[rgb]{0,0,0}\makebox(0,0)[lt]{\lineheight{0}\smash{\begin{tabular}[t]{l}$X_x^{\epsilon}$\end{tabular}}}}%
    \put(0,0){\includegraphics[width=\unitlength,page=2]{zeta-x-from-pdata.pdf}}%
    \put(0.06723323,0.16687229){\color[rgb]{0,0,0}\makebox(0,0)[lt]{\lineheight{0}\smash{\begin{tabular}[t]{l}{\scriptsize$(b_1;t(x,\epsilon))$}\end{tabular}}}}%
    \put(0.44196287,0.01417346){\color[rgb]{0,0,0}\makebox(0,0)[lt]{\lineheight{0}\smash{\begin{tabular}[t]{l}$X_x^{\epsilon}$\end{tabular}}}}%
  \end{picture}%
\endgroup%

	\end{aligned}\ ,
	\label{eq:zeta-x-from-pdata}
\end{align}
where $\tau_a^{-1}$ denotes the inverse of the window element of $A_{s(x,\epsilon)}$.
Suppose the following two conditions hold:
\begin{enumerate}
\item
Let $(x_1,\epsilon_1;x_2,\epsilon_2)\in\left(  D_1\times\left\{ \pm \right\} \right)^2$ 
be such that $s(x_1,\epsilon_{1})=t(x_{2},\epsilon_{2})$ from \eqref{eq:source-target-simple}.
Let $Y_i:=X_{x_i}^{\epsilon_i}$ for $i=1,2$ and recall the 
morphisms $D_{a,b,c,l}^{Y_i,Y_{i+1}}$ and $D_{a,l}^{Y_i}$ from \eqref{eq:D-definition-bimodule}.
We require that the limits 
\begin{align}
\lim_{a,b,c,l\to0}D_{a,b,c,l}^{Y_i,Y_{i+1}}
\quad\text{ and }\quad
\lim_{a,l\to0}D_{a,l}^{Y_i}
\label{eq:d0-bimodule}
\end{align}
exist.
\item
For every $n,m\in\Zb_{\ge0}$ with $n+m\ge1$, $(x_i,\epsilon_i)\in D_1\times\left\{ \pm \right\}$ for $i=1,\dots,n$,
$p_j\in D_2$ for $j=1,\dots,m$ the assignment 
\begin{align}
\begin{aligned}
	(\Rb_{> 0}^{3}\cup \left\{ 0 \right\})^{n} \times (\Rb_{\ge0})^{m}
&\to\Sc\left( \bigotimes_{i=1}^n X_{x_i}^{\epsilon_i}\otimes \bigotimes_{j=1}^{m}A_{p_j},
\bigotimes_{i=1}^n X_{x_i}^{\epsilon_i}\otimes \bigotimes_{j=1}^{m}A_{p_j}\right)\\
(a_1,l_1,b_1,\dots,a_n,l_n,b_n,c_1,\dots,c_m)&\mapsto
\bigotimes_{i=1}^n Q^{X_{x_i}^{\epsilon_i}}_{a_i,l_i,b_i}\otimes \bigotimes_{j=1}^{m} P^{A_{p_j}}_{c_j}
\end{aligned}
\label{eq:cont-cond-bimodules}
\end{align}
is jointly continuous.
\end{enumerate}
Then \eqref{eq:lem:rfa2data:1}, \eqref{eq:lem:rfa2data:2}, \eqref{eq:pdata-from-action},
\eqref{eq:zeta-x-from-pdata}
and $\beta_{a,l,b}^{x}$ define state-sum data with defects.
\end{proposition}
\begin{proof}
From Lemma~\ref{lem:rfa2data} we get the part of the state-sum data for elements of $D_2$,
so we turn directly to the part of the data for $D_1$.

Checking the algebraic relations of 
Conditions~\ref{cond:glue-defect},~\ref{cond:glue-defect-empty}~and~\ref{cond:W-zeta-defect}
	can be done easily using the graphical calculus, and is similar to the case of RFAs and state-sum data without defects.

	The last part of Condition~\ref{cond:defect-cylinder}
	follows directly from $X_x^{\epsilon}$ being bimodules over RFAs.
	Condition~\ref{cond:Q-cont} is just \eqref{eq:cont-cond-bimodules}.

	The only thing left to show is the first part of Condition~\ref{cond:defect-cylinder},
	namely that for every defect list $\underline{x}$ of length $n\ge1$
	the limit $\lim_{\underline{a}\to0}E_{\underline{a}}^{\underline{x}}$ exists.
	Let us introduce the shorthand $Y_i:=X_{x_i}^{\epsilon_i}$.
	If $n=1$ then $E_{\underline{a}}^{\underline{x}}=D_{a+b,l}^{Y_i}$ from \eqref{eq:D-definition-bimodule}
		and the limit $a,b,l\to0$ exists by assumption. Now let $n\ge2$.
	First rewrite $E_{\underline{a}}^{\underline{x}}$ as
	\begin{align}
		E_{\underline{a}}^{\underline{x}}=\sigma_{Y_1\otimes\dots\otimes Y_{n-1},Y_n}
		\circ \left( D_{a_1,b_1,c_1,l_i}^{Y_n,Y_1}\otimes\id \right)\circ
		\sigma_{Y_1\otimes\dots\otimes Y_{n-1},Y_n}^{-1}\circ
		\prod_{i=1}^{n-1} \id\otimes D_{a_i,b_i,c_i,l_i}^{Y_i,Y_{i+1}}\otimes\id
		\label{eq:Ex-in-terms-of-Dx}
	\end{align}
	with the appropriate distribution of the parameters.
	Since the limits in \eqref{eq:d0-bimodule} exist, we can rewrite \eqref{eq:Ex-in-terms-of-Dx} as
	\begin{align}
	E_{\underline{a}}^{\underline{x}}=
	\left(Q^{Y_{n}}_{p_n}\otimes \bigotimes_{i=1}^{n-1} Q^{Y_i}_{p_i}\right)
	\circ \widetilde{E}_0^{\underline{x}}\ ,
		\label{eq:Ex-in-terms-of-Dx-Q-D0}
	\end{align}
	with some distribution of the parameters, where $\widetilde{E}_0^{\underline{x}}$ is the morphism obtained 
	by taking the limits in the parameters of $D_{a_i,b_i,c_i,l_i}^{Y_i,Y_{i+1}}$ to 0 in \eqref{eq:Ex-in-terms-of-Dx} separately.
	The joint continuity condition in \eqref{eq:cont-cond-bimodules} together with \eqref{eq:Ex-in-terms-of-Dx-Q-D0} shows that the joint limit exists and is given by $\widetilde{E}_0^{\underline{x}}$.
\end{proof}

\begin{remark}
For strongly separable symmetric RFAs and dual pairs of bimodules in $\Hilb$, 
conditions 1 and 2 in Proposition~\ref{prop:pdata-from-action} are automatically satisfied, 
see Lemmas~\ref{lem:semigrp} and~\ref{lem:D-limits}.
\end{remark}

\subsection{Defect fusion and tensor product of bimodules}\label{sec:fusion-of-defects}

In this section we are going to assume that the state-sum data is given in terms of strongly separable symmetric
RFAs and dual pairs of bimodules for which the conditions of Proposition~\ref{prop:pdata-from-action} hold.
In Theorem~\ref{thm:state-space-tensor-product} we show that 
the state spaces \eqref{eq:defect-aqft:obj}
can be explicitly computed in terms of tensor products of the bimodules over the intermediate RFAs,
and in Theorem~\ref{thm:fusion} we give the compatibility between the tensor product of bimodules and the fusion of defect lines.

\begin{theorem}
	Let $\Ab(\Db)$ be state-sum data 
	given in terms of RFAs and bimodules 
	as in Proposition~\ref{prop:pdata-from-action}, 
	$\funZ_{\Ab(\Db)}$ the state-sum {\aQFT} from Section~\ref{sec:lattice:daqft}
	and $S\in\Borddef{\Db}$ be connected
	with corresponding defect list $\underline{x}$ of length $n\ge1$. 
	Let us assume that for every $(x_i,\epsilon_i)\in D_1\times\left\{ \pm \right\}$ ($i=1,2$)
	satisfying $s(x_1,\epsilon_1)=t(x_2,\epsilon_2)$ of \eqref{eq:source-target-simple}
	the limit 
	\begin{align}
	\lim_{a\to0}\tilde{\rho}_{a,b,c,l}^{X_{x_1}^{\epsilon_1},X_{x_2}^{\epsilon_2}}
		\label{eq:thm-state-space-limit}
	\end{align}
	of the morphism in \eqref{eq:tensor-product-action} exists.
	Then with $B_i:=A_{s(x_i,\epsilon_i)}$ we have
	\begin{align}
	\funZ_{\Ab(\Db)}(S)=Z(X_{\underline{x}}) =
	\ctimes_{B_{n}}X_{x_1}^{\epsilon_1}\otimes_{B_{1}}\dots\otimes_{B_{n-1}}X_{x_n}^{\epsilon_n}\ .
		\label{eq:state-space-tensor-product}
	\end{align}
	\label{thm:state-space-tensor-product}
\end{theorem}
\begin{proof}
	We will prove the theorem for $n=3$, for general $n$ the proof is similar.
	Let $Y_i:=X_{x_i}^{\epsilon_i}$ for $i=1,2,3$.
	Let $D^{23}:=\lim_{a,b,c,l\to0}D_{a,b,c,l}^{Y_2,Y_3}$ from \eqref{eq:D-definition-bimodule}
	and let $\pi^{23}$ and $\iota^{23}$ denote the corresponding projection and embedding of its image $Y_2\otimes_{B_2} Y_3$.
	By Proposition~\ref{prop:d0proj} $Y_2\otimes_{B_2} Y_3$ is a $B_1$-$B_3$-bimodule.
	We show that $D^{123}:=\lim_{a,b,c,l\to0}D_{a,b,c,l}^{Y_1,Y_2\otimes_{B_2} Y_3}$ exists:
\begin{align}
	\begin{aligned}
	&\lim_{a,b,c,l\to0}D_{a,b,c,l}^{Y_1,Y_2\otimes_{B_2} Y_3}=
	\begin{aligned}
		\def\svgwidth{7cm}
\begingroup%
  \makeatletter%
  \providecommand\color[2][]{%
    \errmessage{(Inkscape) Color is used for the text in Inkscape, but the package 'color.sty' is not loaded}%
    \renewcommand\color[2][]{}%
  }%
  \providecommand\transparent[1]{%
    \errmessage{(Inkscape) Transparency is used (non-zero) for the text in Inkscape, but the package 'transparent.sty' is not loaded}%
    \renewcommand\transparent[1]{}%
  }%
  \providecommand\rotatebox[2]{#2}%
  \ifx\svgwidth\undefined%
    \setlength{\unitlength}{118.69994202bp}%
    \ifx\svgscale\undefined%
      \relax%
    \else%
      \setlength{\unitlength}{\unitlength * \real{\svgscale}}%
    \fi%
  \else%
    \setlength{\unitlength}{\svgwidth}%
  \fi%
  \global\let\svgwidth\undefined%
  \global\let\svgscale\undefined%
  \makeatother%
  \begin{picture}(1,0.52214774)%
    \put(0,0){\includegraphics[width=\unitlength]{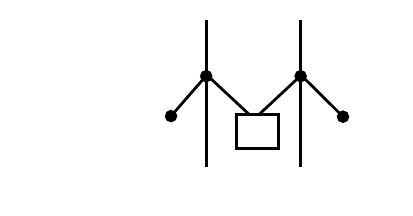}}%
    \put(-0.00335668,0.23092074){\color[rgb]{0,0,0}\makebox(0,0)[lb]{\smash{$\lim_{a,b,c,l\to0}$}}}%
    \put(0.78321845,0.33176269){\color[rgb]{0,0,0}\makebox(0,0)[lb]{\smash{\scriptsize
{$(a_3,l,c_1)$}}}}%
    \put(0.59450732,0.19022935){\color[rgb]{0,0,0}\makebox(0,0)[lb]{\smash{$e_{a_2}$}}}%
    \put(0.29162964,0.34717318){\color[rgb]{0,0,0}\makebox(0,0)[lb]{\smash{\scriptsize
{$(b_1,l,a_1)$}}}}%
    \put(0.31956525,0.23914244){\color[rgb]{0,0,0}\makebox(0,0)[lb]{\smash{\scriptsize
{$b_2$}}}}%
    \put(0.86485092,0.23705274){\color[rgb]{0,0,0}\makebox(0,0)[lb]{\smash{\scriptsize
{$c_2$}}}}%
    \put(0.62658745,0.06008454){\color[rgb]{0,0,0}\makebox(0,0)[lb]{\smash{$Y_2\otimes
_{B_2} Y_3$}}}%
    \put(0.63068211,0.49062813){\color[rgb]{0,0,0}\makebox(0,0)[lb]{\smash{$Y_2\otimes
_{B_2} Y_3$}}}%
    \put(0.464835,0.06008454){\color[rgb]{0,0,0}\makebox(0,0)[lb]{\smash{$Y_1$}}}%
    \put(0.464835,0.49142426){\color[rgb]{0,0,0}\makebox(0,0)[lb]{\smash{$Y_1$}}}%
  \end{picture}%
\endgroup%

	\end{aligned}\\&
	\begin{aligned}
		\def\svgwidth{14cm}
		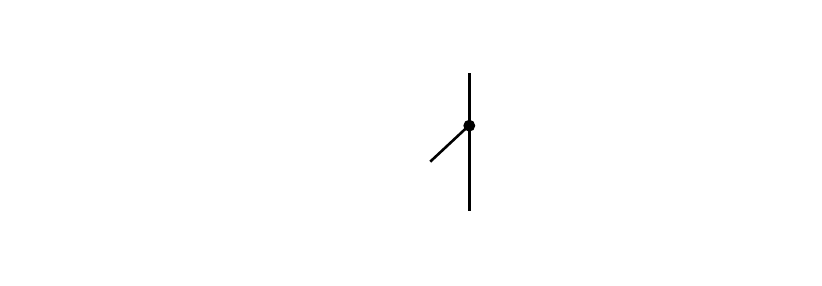
	\end{aligned}\ ,
	\end{aligned}
	\label{eq:D123}
\end{align}
	where in the last equation we used associativity of the action on $Y_2$ and took the limits separately
(the joint limit exists for the same reason as in the proof of Proposition~\ref{prop:pdata-from-action}).
	Let $\pi^{123}$ and $\iota^{123}$ denote the projection and embedding of the image of $D^{123}$ which
	is $Y_1\otimes_{B_1} Y_2 \otimes_{B_2} Y_3$.
	Note that the projectors for $(Y_1\otimes_{B_1} Y_2)\otimes_{B_2} Y_3$ and $Y_1\otimes_{B_1} (Y_2 \otimes_{B_2} Y_3)$ are the same, hence they have the same image and we can omit the brackets.

	Similarly one shows using \eqref{eq:Ex-in-terms-of-Dx} that
	\begin{align}
		\begin{aligned}
		\def\svgwidth{11cm}
\begingroup%
  \makeatletter%
  \providecommand\color[2][]{%
    \errmessage{(Inkscape) Color is used for the text in Inkscape, but the package 'color.sty' is not loaded}%
    \renewcommand\color[2][]{}%
  }%
  \providecommand\transparent[1]{%
    \errmessage{(Inkscape) Transparency is used (non-zero) for the text in Inkscape, but the package 'transparent.sty' is not loaded}%
    \renewcommand\transparent[1]{}%
  }%
  \providecommand\rotatebox[2]{#2}%
  \ifx\svgwidth\undefined%
    \setlength{\unitlength}{194.9539917bp}%
    \ifx\svgscale\undefined%
      \relax%
    \else%
      \setlength{\unitlength}{\unitlength * \real{\svgscale}}%
    \fi%
  \else%
    \setlength{\unitlength}{\svgwidth}%
  \fi%
  \global\let\svgwidth\undefined%
  \global\let\svgscale\undefined%
  \makeatother%
  \begin{picture}(1,0.48627659)%
    \put(0,0){\includegraphics[width=\unitlength]{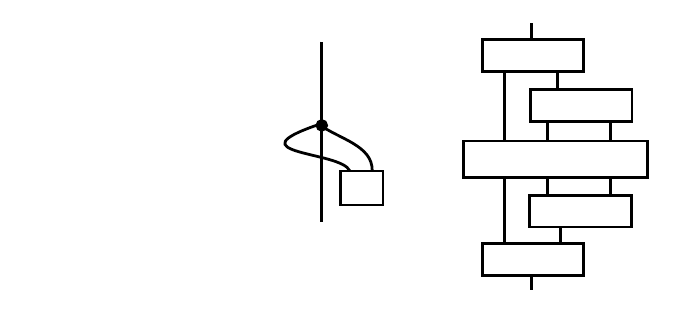}}%
    \put(0.6199792,0.25817389){\color[rgb]{0,0,0}\makebox(0,0)[lb]{\smash{=}}}%
    \put(0.76045204,0.38729259){\color[rgb]{0,0,0}\makebox(0,0)[lb]{\smash{$\pi^{123}$}}}%
    \put(0.76226725,0.08582003){\color[rgb]{0,0,0}\makebox(0,0)[lb]{\smash{$\iota^{123}$}}}%
    \put(0.67707312,0.46757025){\color[rgb]{0,0,0}\makebox(0,0)[lb]{\smash{$Y_1\otimes_{B_1}
Y_2\otimes
_{B_2} Y_3$}}}%
    \put(0.67707312,0.02438877){\color[rgb]{0,0,0}\makebox(0,0)[lb]{\smash{$Y_1\otimes_{B_1}Y_2\otimes
_{B_2} Y_3$}}}%
    \put(0.49970405,0.32289313){\color[rgb]{0,0,0}\makebox(0,0)[lb]{\smash{\scriptsize
{$(a_1,l,a_3)$}}}}%
    \put(0.507236,0.20069276){\color[rgb]{0,0,0}\makebox(0,0)[lb]{\smash{$e_{a_2}$}}}%
    \put(0,0.26038032){\color[rgb]{0,0,0}\makebox(0,0)[lb]{\smash{$D_0^{Y_1\otimes_{B_1} Y_2 \otimes_{B_2} Y_3}=
\lim_{a,l\to0}$}}}%
    \put(0.83240036,0.15697952){\color[rgb]{0,0,0}\makebox(0,0)[lb]{\smash{$\iota^{23}$}}}%
    \put(0.83169444,0.31302157){\color[rgb]{0,0,0}\makebox(0,0)[lb]{\smash{$\pi^{23}$}}}%
    \put(0.77738149,0.24015219){\color[rgb]{0,0,0}\makebox(0,0)[lb]{\smash{$E_0^{\underline{x}}$}}}%
    \put(0.36596243,0.43458677){\color[rgb]{0,0,0}\makebox(0,0)[lb]{\smash{$Y_1\otimes_{B_1}
Y_2\otimes
_{B_2} Y_3$}}}%
    \put(0.36596243,0.12271838){\color[rgb]{0,0,0}\makebox(0,0)[lb]{\smash{$Y_1\otimes_{B_1}
Y_2\otimes
_{B_2} Y_3$}}}%
  \end{picture}%
\endgroup%

		\end{aligned}\ .
		\label{eq:Dc123}
	\end{align}
	Let $\pi^{\ctimes}$ and $\iota^{\ctimes}$ denote the projection and embedding of the image of $D^{Y_1\otimes_{B_1} Y_2 \otimes_{B_2} Y_3}$ which
	is $\ctimes_{B_3}Y_1\otimes_{B_1} Y_2 \otimes_{B_2} Y_3$.
	Now a simple computation shows that 
	\begin{align}
		\begin{aligned}
		\def\svgwidth{10cm}
\begingroup%
  \makeatletter%
  \providecommand\color[2][]{%
    \errmessage{(Inkscape) Color is used for the text in Inkscape, but the package 'color.sty' is not loaded}%
    \renewcommand\color[2][]{}%
  }%
  \providecommand\transparent[1]{%
    \errmessage{(Inkscape) Transparency is used (non-zero) for the text in Inkscape, but the package 'transparent.sty' is not loaded}%
    \renewcommand\transparent[1]{}%
  }%
  \providecommand\rotatebox[2]{#2}%
  \ifx\svgwidth\undefined%
    \setlength{\unitlength}{168.88026123bp}%
    \ifx\svgscale\undefined%
      \relax%
    \else%
      \setlength{\unitlength}{\unitlength * \real{\svgscale}}%
    \fi%
  \else%
    \setlength{\unitlength}{\svgwidth}%
  \fi%
  \global\let\svgwidth\undefined%
  \global\let\svgscale\undefined%
  \makeatother%
  \begin{picture}(1,0.36206333)%
    \put(0,0){\includegraphics[width=\unitlength]{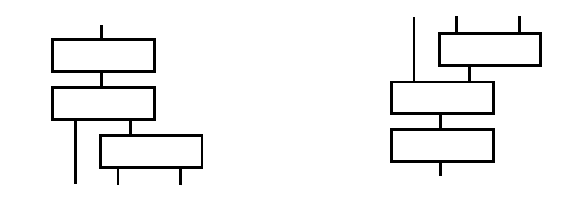}}%
    \put(0.56219573,0.15759262){\color[rgb]{0,0,0}\makebox(0,0)[lb]{\smash{$\iota:=$}}}%
    \put(0.72556326,0.17486328){\color[rgb]{0,0,0}\makebox(0,0)[lb]{\smash{$\iota^{123}$}}}%
    \put(0.62721586,0.02815419){\color[rgb]{0,0,0}\makebox(0,0)[lb]{\smash{$Y_1\otimes_{B_1}Y_2\otimes
_{B_2} Y_3$}}}%
    \put(0.80652435,0.25700922){\color[rgb]{0,0,0}\makebox(0,0)[lb]{\smash{$\iota^{23}$}}}%
    \put(0.72556326,0.09379001){\color[rgb]{0,0,0}\makebox(0,0)[lb]{\smash{$\iota^{\ctimes}$}}}%
    \put(0.68753995,0.34766703){\color[rgb]{0,0,0}\makebox(0,0)[lb]{\smash{$Y_1$}}}%
    \put(0.76143704,0.34766703){\color[rgb]{0,0,0}\makebox(0,0)[lb]{\smash{$Y_2$}}}%
    \put(0.86879876,0.34766703){\color[rgb]{0,0,0}\makebox(0,0)[lb]{\smash{$Y_3$}}}%
    \put(-0.00314572,0.15845046){\color[rgb]{0,0,0}\makebox(0,0)[lb]{\smash{$\pi:=$}}}%
    \put(0.13673393,0.16821781){\color[rgb]{0,0,0}\makebox(0,0)[lb]{\smash{$\pi^{123}$}}}%
    \put(0.0492916,0.34766703){\color[rgb]{0,0,0}\makebox(0,0)[lb]{\smash{$Y_1\otimes_{B_1}Y_2\otimes
_{B_2} Y_3$}}}%
    \put(0.21601735,0.08523325){\color[rgb]{0,0,0}\makebox(0,0)[lb]{\smash{$\pi^{23}$}}}%
    \put(0.13925051,0.24845231){\color[rgb]{0,0,0}\makebox(0,0)[lb]{\smash{$\pi^{\ctimes}$}}}%
    \put(0.10961569,0.00446877){\color[rgb]{0,0,0}\makebox(0,0)[lb]{\smash{$Y_1$}}}%
    \put(0.18351278,0.00446877){\color[rgb]{0,0,0}\makebox(0,0)[lb]{\smash{$Y_2$}}}%
    \put(0.2908745,0.00446877){\color[rgb]{0,0,0}\makebox(0,0)[lb]{\smash{$Y_3$}}}%
    \put(0.41759772,0.15759262){\color[rgb]{0,0,0}\makebox(0,0)[lb]{\smash{and}}}%
  \end{picture}%
\endgroup%

		\end{aligned}
		\label{eq:iota-pi}
	\end{align}
	satisfy $\iota\circ\pi=E_0^{\underline{x}}$ and 
	$\pi\circ\iota=\id_{\ctimes_{B_3}Y_1\otimes_{B_1} Y_2 \otimes_{B_2} Y_3}$,
	that is the image of $E_0^{\underline{x}}$ is exactly $\ctimes_{B_3}Y_1\otimes_{B_1} Y_2 \otimes_{B_2} Y_3$.
\end{proof}

	Let $\Sigma$ be a bordism with defects.
	We say that two defect lines $x_0,x_1\in\Sigma_{[1]}$ are \textsl{parallel} 
	if there is an isotopy rel boundary $t\mapsto x_t$ between them such that 
	for every $t\in(0,1)$ $x_t$ does not intersect any defect line.
Let us consider two bordisms with area and defects: 
\begin{enumerate}
	\item $(\Sigma,\Ac,\Lc)$, 
		which has two parallel defect lines 
		with length $l$ labeled with
		a $B$-$A$-bimodule $V$ and an $A$-$C$-bimodule $W$,
		and with a surface component
		with area $a$ between them and 
	\item $(\Sigma',\Ac',\Lc')$ which is the same as $(\Sigma,\Ac,\Lc)$ except that the defect line $x_0$ is removed from $\Sigma_{[1]}$ and the surface component between $x_0$ and $x_1$ is collapsed.
The remaining defect line $x_1$ is labeled by $V\otimes_A W$.
	The length of $x_1$ is $l$ and the area and length of the other defect lines and surface components are unchanged.
\end{enumerate}

\begin{theorem}\label{thm:fusion}
	Let $\Ab(\Db)$ be state-sum data satisfying the conditions of Theorem~\ref{thm:state-space-tensor-product}.
	Let us assume that, in addition, the limits
	\begin{align}
		\lim_{a\to0}\tilde{\gamma}_{a,b,c,l}^{V,W}\quad\text{ and }\quad\lim_{a\to0}\tilde{\beta}_{a,b,c,l}^{V,W}
		\label{eq:bkafdmvke}
	\end{align}
	exist for every $b,c,l\in\Rb_{>0}$
	Then we have
	\begin{align}
		\funZ_{\Ab(\Db)}(\Sigma',\Ac',\Lc')=\lim_{a\to0}\funZ_{\Ab(\Db)}(\Sigma,\Ac,\Lc).
		\label{eq:fusion}
	\end{align}
\end{theorem}
\begin{proof}[Sketch of proof]
	\begin{figure}[tb]
		\centering
		\def\svgwidth{16cm}
		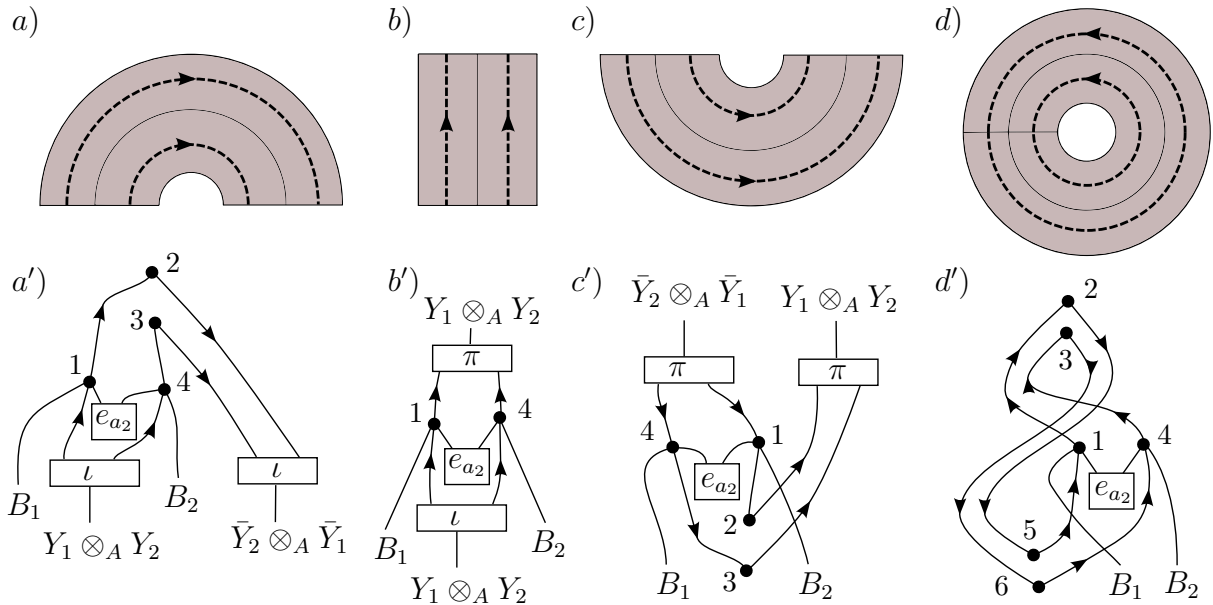
		\caption{Detail of the chosen PLCW decomposition of $\Sigma$ depending on the starting and ending point of the parallel defect lines
		(Parts~$a)$-$d)$) and the part of the corresponding morphisms given by the state-sum construction (Parts~$a')$-$d')$).
		In Part~$a)$ the defect lines start and end on an ingoing boundary component,
		in Part~$b)$ they start on an ingoing and end on an outgoing boundary component,
		in Part~$c)$ they start and end on an outgoing boundary component and
		in Part~$d)$ the defect lines are closed loops.
		In Parts~$a')$-$d')$ 
the numbers attached to the vertices for readability relate to the parameters as follows:
$1 \hat= (b_1,l_1,a_1;Y_1)$, 
$2 \hat= (b_2,l_2,a_4;Y_1)$,
$3 \hat= (a_3,l_2,c_2;Y_2)$,
$4 \hat= (a_5,l_1,c_1;Y_2)$,
$5 \hat= (b_3,l_3,a_6;Y_1)$,
$6 \hat= (a_7,l_3,c_3;Y_2)$.
	}
		\label{fig:defect-fusion-configurations}
	\end{figure}
	Let us choose a PLCW decomposition of $\Sigma$ which contains two rectangles, each of which is containing one of the parallel defect lines.
	For $\Sigma'$ we choose a PLCW decomposition which contains a rectangle containing the defect line.
	Depending on where the defect lines start and end we have 4 essentially different cases that we need to consider and which
we show in Figure~\ref{fig:defect-fusion-configurations} Parts~$a)$-$d)$.

	The corresponding detail of the morphism $\funZ_{\Ab(\Db)}(\Sigma,\Ac,\Lc)$
	is shown in Figure~\ref{fig:defect-fusion-configurations} Parts~$a')$-$d')$.
	We see that after taking the limit $a\to0$, 
	and looking at the definition of the bimodule $Y_1\otimes_A Y_2$ 
	and its dual $\bar{Y}_2\otimes_A \bar{Y}_1$
	in Definition~\ref{def:tensor-product-bimodule} and
	in Lemma~\ref{lem:dual-pair-tensor-product}, we obtain exactly
	$\funZ_{\Ab(\Db)}(\Sigma',\Ac',\Lc')$.
\end{proof}

\begin{remark}\label{rem:all-works-in-hilb}
For RFAs and bimodules in $\Hilb$ 
which are left and right modules as well and 
for which $Q_{a,l,b}$ is epi for every $a,l,b\in\Rb_{>0}$
(cf.\ Remark~\ref{rem:left-right-module-gives-bimodule} and 
Proposition~\ref{prop:tens-prod-lr-bimodules}),
the conditions of Theorem~\ref{thm:state-space-tensor-product} 
are automatically satisfied, since
the limits in \eqref{eq:d0-bimodule} exist by Lemma~\ref{lem:D-limits} and
those in \eqref{eq:thm-state-space-limit} by 
Proposition~\ref{prop:tens-prod-lr-bimodules}.
\\
We stress here that one should not necessarily expect that the limits in \eqref{eq:bkafdmvke} exist. It may happen that when one brings two defect lines near each other the correlators of the quantum field theory diverge,
for an example in conformal field theory see e.g. \cite{Bachas:2007fusion}.
\end{remark}

\section{Example: 2d Yang-Mills theory}\label{sec:2dym}

The state-sum construction of 2d Yang-Mills theory has been introduced by \cite{Migdal:1975re}, was further developed for $G=U(N)$ in \cite{Rusakov:1990wilson}, and has been summarised in \cite{Witten:1991gt}; for a review see \cite{Cordes:1995ym}.
There, partition functions and expectation values of Wilson loops were calculated.
The proof of convergence of the (Boltzmann) plaquette weights has been shown
in a different setting in \cite{Applebaum:2011cl}.
In this section we will heavily rely on the representation theory of compact Lie groups,
a standard reference is e.g.\ \cite{Knapp:2002lg}.

\subsection{Two RFAs from a compact group $G$}\label{sec:twoRFAsfromG}

Let $G$ be a compact semisimple
Lie-group and $\int\,dx$ the Haar integral
on $G$ with the normalisation $\int_{G}1\,dx=1$.
We denote with $L^2(G)$ the Hilbert space of square integrable complex functions on $G$,
where the scalar product of $f,g\in L^2(G)$ is given by
$\langle f,g\rangle:=\int f(x)^*g(x)dx$.

Let $\hat{G}$ denote representatives of isomorphism classes of 
finite dimensional simple unitary $G$-modules. 
Then for $V\in\hat{G}$ with inner product $\langle-,-\rangle_V$ and 
an orthonormal basis $\{e_i^V\}_{i=1}^{\dim(V)}$ let 
\begin{align}
	\begin{aligned}
	f_{ij}^V:G&\to \Cb\\
	g&\mapsto (\dim(V))^{1/2}\langle e_i^V,g.e_j^V\rangle_V
	\end{aligned}
	\label{eq:matrix-element-fn}
\end{align}
denote a \textsl{matrix element function} and 
let $M_V$ denote the linear span of these.
The matrix element functions are orthonormal \cite[Cor.\,4.2]{Knapp:2002lg}:
for $V,W\in\hat{G}$, $i,j\in\left\{ 1,\dots,\dim(V) \right\}$ 
and $k,l\in\left\{ 1,\dots,\dim(W) \right\}$
\begin{align}
	\langle f_{ij}^V, f_{kl}^W \rangle=\delta_{ik}\delta_{jl}\delta_{V,W}
	\label{eq:orthogonality}
\end{align}
where $\delta_{V,W}=1$ if $V=W$ and 0 otherwise.
The \textsl{character of $V$} is defined as
	\begin{align}
		\chi_V=(\dim(V))^{-1/2}\sum_{i=1}^{\dim(V)}f_{ii}^V\ .
		\label{eq:character}
	\end{align}
The Peter-Weyl theorem provides a complete orthonormal basis of $L^2(G)$ in terms of
matrix element functions and of the square integrable class functions $Cl^2(G)$ in terms of characters:
\begin{align}
	L^2(G)\cong\bigoplus_{V\in \hat{G}}M_V\quad\text{and}\quad  Cl^2(G)\cong\bigoplus_{V\in \hat{G}}\Cb .\chi_V
	\label{eq:L2-Cl2-decomposition}
\end{align}
as Hilbert space direct sums.
Note that $L^2(G)\otimes L^2(G)\cong L^2(G\times G)$ 
and $Cl^2(G)\otimes Cl^2(G)\cong Cl^2(G\times G)$
isometrically by mapping $f \otimes f'$ to the function $(g,g') \mapsto f(g)f'(g')$.
We will often use these isomorphisms without further notice.

\medskip

In the following we will define a $\dagger$-RFA structure on $L^2(G)$ and $Cl^2(G)$.
Let us start with defining the operator
\begin{align}
	\begin{aligned}
	\Delta:L^2(G)&\to L^2(G)\otimes L^2(G)\\
	f&\mapsto \Delta(f)=\left[ (x,y)\mapsto f(xy) \right]\ ,
	\end{aligned}
	\label{eq:def:delta}
\end{align}
which has norm 1. 
Let $\mu:=\Delta^{\dagger}:L^2(G)\otimes L^2(G)\to L^2(G)$ be its adjoint, 
which is given by the convolution product. For $F\in L^2(G)\otimes L^2(G)$
\begin{align}
	\begin{aligned}
		\mu(F)(y)=\int_G F(x,x^{-1}y)\,dx\ .
	\end{aligned}
	\label{eq:def:mu-conv}
\end{align}
Let $V\in\hat{G}$ and let us denote with $\sigma_V\in\Rb$ the value of the
Casimir operator of $G$ in the module $V$. We define for $a\in\Rb_{>0}$ 
\begin{align}
	\begin{aligned}
		\eta_a:\Cb&\to L^2(G)\\
		1&\mapsto\eta_a(1)=\sum_{V\in\hat{G}}e^{-a\sigma_V}\dim(V)\chi_V\ .
	\end{aligned}
	\label{eq:def:eta}
\end{align}

\begin{lemma}
	The sum in \eqref{eq:def:eta} is absolutely convergent for every $a\in\Rb_{>0}$.
	\label{lem:applebaum}
\end{lemma}
\begin{proof}
	This follows from \cite[Sec.\,3]{Applebaum:2011cl}, 
	which we explain now.
	Let us fix a maximal torus of $G$ and let $T$ denote its Lie algebra,
	let $\Lambda^{+}\subset T^*$ denote the set of dominant weights and let
	$(-,-)$ be the inner product on $T^*$ induced by the Killing form
	and $|-|$ the induced norm. 
	We will use that, since $G$ is semisimple, there is a bijection of sets
	\cite[Thm.\,5.5]{Knapp:2002lg}
	\begin{align}
		\begin{aligned}
			\hat{G}&\xrightarrow{\cong}\Lambda^{+}\\
		V&\mapsto\lambda_V\\
		V_{\lambda}&\mapsfrom\lambda
	\end{aligned}\ .
		\label{eq:irrep-highest-weight}
	\end{align}
	
	{}From \cite[(1.17)]{Sugiura:1971lg} and \cite[(3.2)]{Applebaum:2011cl}
	we have that (by the Weyl dimension formula)
	for $V\in\hat{G}$ with dominant weight $\lambda_V\in\Lambda^{+}$
	\begin{align}
		\dim(V)\le N |\lambda_V|^m\ ,
		\label{eq:dimension-estimate}
	\end{align}
	where $N\in\Rb_{>0}$ is a constant independent of $V$ and 
	$2m=\dim(G)-\rank(G)$.

	{}From \cite[Lem.\,1.1]{Sugiura:1971lg} we can express the value of the Casimir element in $V$ using the highest weight $\lambda_V$ of $V$ and
	the half sum of simple roots $\rho$ as
	\begin{align}
		\sigma_V=(\lambda_V,\lambda_V+2\rho)\ .
		\label{eq:Casimir-highest-weight}
	\end{align}
	It follows directly \cite[(3.5)]{Applebaum:2011cl} that
	\begin{align}
		|\lambda_V|^2\le\sigma_V\ .
		\label{eq:Casimir-estimate}
	\end{align}

	We can give an estimate for the norm of $\lambda_V$ as follows.
	The choice of simple roots gives a bijection $\Zb^{\rank(G)}\to\Lambda^{+}$ which we write as $n\mapsto\lambda(n)$.
	Using the proof of \cite[Lem.\,1.3]{Sugiura:1971lg} there are
	$C_1,C_2\in \Rb_{\ge0}$ such that for every 
	$n\in\Zb^{\rank(G)}$
	\begin{align}
		C_1\norm{n}\le|\lambda(n)|\le C_2 \norm{n}\ ,
		\label{eq:highest-weight-estimate}
	\end{align}
	where $\norm{n}^2=\sum_{i=1}^{\rank(G)}n_i^2$.	

	Let $b(j)$ denote the number of $n\in\Zb^{\rank(G)}$ with $\norm{n}^2=j$.
	We can easily give a (very rough)	
	estimate of this by the volume of the $\rank(G)$ dimensional cube
	with edge length $2j^{1/2}+1$:
	\begin{align}
		b(j)\le (2j^{1/2}+1)^{\rank(G)}\ .
		\label{eq:cube-estimate}
	\end{align}

	We compute the squared norm of $\eta_a$ following the computation in
	\cite[Ex.\,3.4.1]{Applebaum:2011cl}.
	\begin{align}
		\begin{aligned}
			\norm{\eta_a}^2=&\sum_{V\in\hat{G}}(\dim(V))^2 e^{-2a\sigma_V}
			\stackrel{\eqref{eq:irrep-highest-weight}}{=}
			\sum_{\lambda\in\Lambda^{+}}(\dim(V_{\lambda}))^2 e^{-2a\sigma_{V_{\lambda}}}\\
			\stackrel{\eqref{eq:dimension-estimate}}{\le}&
			N^2\sum_{\lambda\in\Lambda^{+}}|\lambda|^{2m} e^{-2a\sigma_{V_{\lambda}}}
			\stackrel{\eqref{eq:Casimir-estimate}}{\le}
			N^2\sum_{\lambda\in\Lambda^{+}}|\lambda|^{2m} e^{-2a|\lambda|^2}\\
			\stackrel{\eqref{eq:highest-weight-estimate}}{\le}&
			N^2C_2^{2m}\sum_{n\in\Zb^{\rank(G)}}\norm{n}^{2m} e^{-2aC_1\norm{n}^2}
			=N^2C_2^{2m}\sum_{j=1}^{\infty}  b(j)j^m e^{-2aC_1j}\\
			\stackrel{\eqref{eq:cube-estimate}}{\le}&
			N^2C_2^{2m}\sum_{j=1}^{\infty} (2j^{1/2}+1)^{\rank(G)}j^m e^{-2aC_1j}\ , 
		\end{aligned}
		\label{eq:eta-norm-long-estimate}
	\end{align}
	which converges.
\end{proof}
Finally we define the counit as
$\varepsilon_a:=\eta_a^{\dagger}:L^2(G)\to \Cb$.
Explicitly, for $f\in L^2(G)$,
\begin{align}
	\varepsilon_a(f)=\langle \eta_a,f\rangle=\sum_{V\in\hat{G}}e^{-a\sigma_V}\dim(V)\int_G \chi_V(x) f(x^{-1})\,dx\ .
	\label{eq:epsilon-explicit}
\end{align}
Again for $a\in\Rb_{>0}$ let
\begin{align}
	\begin{aligned}
	P_a:L^2(G)&\to L^2(G)\\
	f&\mapsto\mu(\eta_a\otimes f)\ ,
	\end{aligned}
	\label{eq:padef}
\end{align}
$\mu_a:=P_a\circ\mu$ and $\Delta_a:=\Delta\circ P_a$.

\begin{proposition}
	$L^2(G)$, together with the family of morphisms $\mu_a$, $\eta_a$,
	$\Delta_a$ and $\varepsilon_a$ for $a\in\Rb_{>0}$ defined above
	is a strongly separable symmetric $\dagger$-RFA in $\Hilb$.
	\label{prop:l2g-rfa}
\end{proposition}
Before proving this proposition let us state a technical lemma.
Let $V\in\hat{G}$ and define
\begin{align}
	\begin{aligned}
	\mu_a^{V}&:=\mu_a|_{M_V\otimes M_V}\ ,& \eta_a^V&:= e^{-a\sigma_V}\dim_V \chi_V\ ,\\
	\Delta_{a}^{V}&:=\Delta|_{M_V}\ , & \varepsilon_a^V&:=\varepsilon_a|_{M_V}\ .
	\end{aligned}
	\label{eq:MV-structure-maps}
\end{align}
	From a computation using orthogonality of the $f_{ij}^V$ 
	we can obtain the following formulas:
	\begin{align}
		\label{eq:MV-formulas-Pa}
		P_a(f_{ij}^V)&=e^{-a\sigma_V}f_{ij}^V\in M_V\ ,\\
		\label{eq:MV-formulas-mu-a}
		\mu_a(f_{ij}^V\otimes f_{kl}^V)&=\delta_{jk} e^{-a\sigma_V} (\dim(V))^{-1/2}f_{il}^V\in M_V\ ,\\
		\label{eq:MV-formulas-Delta-a}
		\Delta_a(f_{ij}^V)&= e^{-a\sigma_V} (\dim(V))^{-1/2}\sum_{k=1}^{\dim(V)}f_{ik}^V \otimes f_{kj}^V\in M_V\otimes M_V\ ,\\
		\varepsilon_a(f_{ij}^V)&= e^{-a\sigma_V} (\dim(V))^{1/2}\delta_{ij}\ .
		\label{eq:MV-formulas-epsilon-a}
	\end{align}
\begin{lemma}
	Let $V\in\hat{G}$. Then $M_V$ is a strongly separable symmetric $\dagger$-RFA in $\Hilb$
	with the structure maps in \eqref{eq:MV-structure-maps}.
	\label{lem:MV-RFA}
\end{lemma}
\begin{proof}
	Checking the algebraic relations is a straightforward calculation.
	As an example, we compute the window element of $M_V$.
	\begin{align}
		\begin{aligned}
		\mu_{a_1}^V\circ\Delta_{a_2}^V\circ\eta_{a_3}^V&=
		\sum_{l=1}^{\dim(V)}\mu_{a_1}^V\circ\Delta_{a_2}^V(f_{ll}^V)e^{-a_3 \sigma_V} (\dim(V))^{1/2}\\
		&=\sum_{k,l=1}^{\dim(V)}\mu_{a_1}^V (f_{lk}^V\otimes f_{kl}^V)e^{-(a_2+a_3) \sigma_V}\\
		&= \sum_{k,l=1}^{\dim(V)}f_{ll}^V e^{-(a_1+a_2+a_3) \sigma_V}(\dim(V))^{-1/2}\\
		&=e^{-(a_1+a_2+a_3) \sigma_V} \dim(V) \chi_V=
			\eta_{a_1+a_2+a_3}^V\ ,
		\end{aligned}
		\label{eq:MV-window-elt-calc}
	\end{align}
	which is clearly invertible.
\end{proof}

\begin{proof}[Proof of Proposition~\ref{prop:l2g-rfa}]
	Let $V\in\hat{G}$ and let us compute the following norms.
	\begin{align}
		\begin{aligned}
		\norm{\eta_a^V}^2&=e^{-2a\sigma_V}(\dim(V))^2\langle\chi_V,\chi_V\rangle=
		e^{-2a\sigma_V}\dim(V) \sum_{k,l=1}^{\dim(V)}\langle f_{kk}^{V},f_{ll}^{V}\rangle\\&=
		e^{-2a\sigma_V}(\dim(V))^2\ .
		\end{aligned}
		\label{eq:etanorm}
	\end{align}
	Let $\varphi=\sum_{i,j=1}^{\dim(V)}\varphi_{ij}f_{ij}^V\in M_V$ and compute
	\begin{align}
		\begin{aligned}
			\norm{\Delta_a^V(\varphi)}^2=e^{-2a\sigma_V}(\dim(V))^{-1}\sum_{i,j,k=1}^{\dim(V)}
			|\varphi_{ij}|^2 \norm{f_{ik}^V\otimes f_{kj}^V}^2=e^{-2a\sigma_V}\norm{\varphi}^2\ ,
		\end{aligned}
	\label{eq:deltanorm}
	\end{align}
	so $\norm{\Delta_a^V}=e^{-a\sigma_V}$.
	Since $M_V$ is a $\dagger$-RFA, $\norm{\varepsilon_a^V}=\norm{\eta_a^V}$ and $\norm{\mu_a^V}=\norm{\Delta_a^V}$.

	We now would like to take the direct sum of the RFAs $M_V$ for all $V\in \hat{G}$, so
	we check the conditions of Proposition~\ref{prop:directsumrfa}:
	the sum is convergent since it is the squared norm of $\eta_a\in L^2(G)$ 
	and the supremum is clearly bounded. Therefore $L^2(G)$ is an RFA.

	Checking that $L^2(G)$ is strongly separable, symmetric and Hermitian is straightforward 
	using Lemma~\ref{lem:MV-RFA}.
\end{proof}

Now we turn to define an RFA structure on $Cl^2(G)$. 
\begin{proposition}
	The centre of $L^2(G)$ is $Cl^2(G)$ and it is a commutative $\dagger$-RFA.
	\label{prop:center-of-L2G}
\end{proposition}
\begin{proof}
Let us compute the morphism $D_a$ from \eqref{eq:state-sum-data-notation}, 
which is the same as $\tilde{D}_a$ from \eqref{eq:lem:rfa2data:2} by Lemma~\ref{lem:rfa2data}.
For $\varphi=\sum_{V\in\hat{G}}\sum_{i,j=1}^{\dim(V)}\varphi_{ij}^V f_{ij}^V\in L^2(G)$ we find:
\begin{align}
	\begin{aligned}
		D_a(\varphi)&=\mu_{a_2}\circ \sigma_{L^2(G),L^2(G)}\circ\Delta_{a_1}(\varphi)\\
		&=\mu_{a_2}\circ \sigma_{L^2(G),L^2(G)}
		\left( \sum_{V\in\hat{G}}\sum_{i,j,k=1}^{\dim(V)}\varphi_{ij}^V f_{ik}^V\otimes f_{kj}^V \right)
		e^{-a_1\sigma_V}(\dim(V))^{-1/2}\\
		&=\sum_{V\in\hat{G}}\sum_{i,j,k=1}^{\dim(V)}\varphi_{ij}^V e^{-a\sigma_V}(\dim(V))^{-1}
		\delta_{ij}f_{kk}^V\\
		&=\sum_{V\in\hat{G}}\sum_{i=1}^{\dim(V)}\varphi_{ii}^V e^{-a\sigma_V} (\dim(V))^{-1/2}\chi_V\ .
	\end{aligned}
	\label{eq:l2g-da}
\end{align}
From this equation we immediately have that $D_a|_{Cl^2(G)}=P_a|_{Cl^2(G)}$. 
We now show that $D_a|_{(Cl^2(G))^{\perp}}=0$. Using \eqref{eq:L2-Cl2-decomposition}, we have that $\varphi\in(Cl^2(G))^{\perp}\subset L^2(G)$
if and only if for every $W\in\hat{G}$ $\langle\chi_W,\varphi\rangle=0$. We can compute this using the orthogonality relation \eqref{eq:orthogonality}
to get the following: $\varphi\in(Cl^2(G))^{\perp}$ if and only if for every $W\in\hat{G}$ we have that $\sum_{k=1}^{\dim(W)}\varphi^W_{kk}=0$.
By \eqref{eq:l2g-da} we get that $D_a(\varphi)=0$.

Altogether, this shows that the limit $\lim_{a\to0}D_a$ (in the strong operator topology)
exists and $D_0$ is an orthogonal projection onto $Cl^2(G)$.
Therefore by 
Lemma~\ref{lem:d0-split-idempot-center} the centre of $L^2(G)$ is $Cl^2(G)$.
It is a $\dagger$-RFA, since $L^2(G)$ is a $\dagger$-RFA and $D_0$ is self adjoint.
\end{proof}

For completeness we give the comultiplication $\Delta_a^{Cl^2(G)}$ of $Cl^2(G)$.
For $\varphi=\sum_{V\in\hat{G}}\varphi^V \chi_V\in Cl^2(G)$
\begin{align}
	\Delta_a^{Cl^2(G)}(\varphi)=\sum_{V\in\hat{G}}\varphi^V e^{-a\sigma_V} (\dim(V))^{-2}\chi_V\otimes\chi_V\ .
	\label{eq:Cl2-comult}
\end{align}

\begin{remark} \label{rem:l2g-a0-limits}
	Note that for both $L^2(G)$ and $Cl^2(G)$, 
	the $a\to0$ limit of the multiplication and comultiplication exists (by definition),
	but the $a\to0$ limit of the unit and counit does not.
\end{remark}

\subsection{State-sum construction of 2d Yang-Mills theory}

In this section we give state-sum data for the 2d~YM
theory following \cite{Witten:1991gt}.
The plaquette weights $W_a^k:\Cb\to(L^2(G))^{\otimes k}$ for $k\in\Zb_{\ge0}$ and $a\in\Rb>0$ are
\begin{align}
	W_a^{k}(1)(x_1,\dots,x_k)=\sum_{V\in \hat{G}}e^{-a\sigma_V} \dim(V) \chi_V(x_1\cdots x_k) \ ,
	\label{eq:boltzmannweight}
\end{align}
and the contraction and $\zeta_a$ are given by
\begin{align}
	\beta_a:=\left(W_a^{2}\right)^{\dagger}\ ,\quad\zeta_a:= P_a\ ,
	\label{eq:contraction}
\end{align}
where $P_a$ is as in \eqref{eq:padef}.

\begin{proposition}
	The morphisms \eqref{eq:boltzmannweight} and \eqref{eq:contraction} define state-sum data.
	\label{prop:2dym-state-sum-data}
\end{proposition}
\begin{proof}
	We prove this by showing that the morphisms in \eqref{eq:boltzmannweight} and \eqref{eq:contraction}
	can be obtained from the RFA $L^2(G)$ via Lemma~\ref{lem:rfa2data}.

	Since the inverse of the window element of $L^2(G)$ is simply $\eta_a$, we immediately get that $\zeta_a=P_a$.
	Let us look at the morphisms $W_a^{k}$. For $k=1$ we have $W_a^1=\eta_a$. Let us assume that for $k>1$ we have
	$W_{a_1+a_2}^{k}=\Delta_{a_1}^{(k)}\circ\eta_{a_2}$. Then for $k+1$ we see that
	\begin{align*}
		&\left( \id\otimes\dots\otimes\id\otimes\Delta_{a_3} \right)\circ W_{a_1+a_2}^{k}(1)\\=&
		\left( \id\otimes\dots\otimes\id\otimes\Delta_{a_3} \right)\left( \sum_{V\in \hat{G}}e^{-(a_1+a_2)\sigma_V} \dim(V) \chi_V(x_1\cdots x_k) \right)\\=&
		\sum_{V\in \hat{G}}e^{-(a_1+a_2+a_3)\sigma_V} \dim(V) \chi_V(x_1\cdots x_k x_{k+1})\ .
	\end{align*}
	That $\beta_a=\left(W_a^{2}\right)^{\dagger}$ follows from $L^2(G)$ being a $\dagger$-RFA.

	Finally we need to check that the limit $\lim_{a\to0}D_a$ exists, 
	which we have already checked in the proof of Proposition~\ref{prop:center-of-L2G}.
\end{proof}

After this preparation we are ready to define 2d~YM
theory, which maps $\Sb$ to the centre of $L^2(G)$.

\begin{definition}
	The \textsl{2-dimensional Yang-Mills (2d~YM)
	theory with gauge group $G$} is the area-dependent QFT
	\begin{align}
		\begin{aligned}
		\funZ_{\mathrm{YM}}^{G}:\Bordarea&\to\Hilb\\
		\Sb&\mapsto Cl^2(G)
		\end{aligned}
		\label{eq:def:2dym}
	\end{align}
	of Theorem~\ref{thm:state-sum-aqft}
	obtained from the state-sum data in \eqref{eq:boltzmannweight} and \eqref{eq:contraction}.
	\label{def:2dym}
\end{definition}

Next we compute $\funZ_{\mathrm{YM}}^{G}$ on connected surfaces with area and $b\ge 0$ outgoing boundary components. For $b=0$ the result agrees with \cite[Eqn.\,(2.51)]{Witten:1991gt} (see also \cite[Eqn.\,(27)]{Rusakov:1990wilson}).

\begin{proposition}
	Let $(\Sigma,a):(\Sb)^{\sqcup b_{\mathrm{in}}}\to(\Sb)^{\sqcup b_{\mathrm{out}}}$ 
	be a connected bordism 
	of genus $g$ with $b_{\mathrm{in}}$ ingoing and $b_{\mathrm{out}}$ outgoing boundary components and with area $a$. 
	Then for $V_j\in\hat{G}$ for $j=1,\dots,b_{\mathrm{out}}$ we have
	\begin{align}
		\begin{aligned}
			&\funZ_{\mathrm{YM}}^{G}\left( 
			\Sigma_ ,a \right) (\chi_{V_1}\otimes\dots\otimes\chi_{V_{b_{\mathrm{in}}}})\\=&
			\begin{cases}
			\sum_{V\in\hat{G}}e^{-a\sigma_V} (\dim(V))^{\chi( \Sigma)}\cdot
			(\chi_V)^{\otimes b_{\mathrm{out}}}
			&\text{if $b_{\mathrm{in}}=0$}\\
			e^{-a\sigma_{V_1}}(\dim(V_1))^{\chi(
			\Sigma)}\cdot (\chi_{V_1})^{\otimes b_{\mathrm{out}}}
			&\text{if $b_{\mathrm{in}}\ge1$ and $V_1=\dots= V_{b_{\mathrm{in}}}$}\\
			0&\text{else}
			\end{cases}
			\ ,
		\end{aligned}
		\label{eq:2dym-on-sigma-gb}
	\end{align}
where $\chi( \Sigma)=2-2g-b_{\mathrm{in}}-b_{\mathrm{out}}$ is the Euler characteristic of $\Sigma$.
For $b_{\mathrm{in}}=0$ ($b_{\mathrm{out}}=0$) the source (the target) is $\Cb$ and the factors of $\chi_V$ or $\chi_{V_j}$ are absent.
	\label{prop:2dym-on-sigma-gb}	
\end{proposition}
\begin{proof}
We first consider the case that $b:=b_{\mathrm{out}} \ge 1$ and $b_{\mathrm{in}}=0$.
	The map $\varphi_a$ from \eqref{eq:phi-rfa} is given by
	\begin{align}
		\begin{aligned}
		\varphi_a(f_{ij}^V)&=\mu\circ\left( \id\otimes (\mu\circ\sigma_{L^2(G),L^2(G)} \right)
			\left( \sum_{k,l=1}^{\dim(V)}e^{-a\sigma_V}(\dim(V))^{-1}
			f_{ik}^{V}\otimes f_{kl}^{V} \otimes f_{lj}^{V}\right)\\
			&=\sum_{k,l=1}^{\dim(V)}e^{-a\sigma_V}(\dim(V))^{-2} \delta_{jk}\delta_{kl}f_{il}^V
			=e^{-a\sigma_V}(\dim(V))^{-2}f_{ij}^V \ .
		\end{aligned}
		\label{eq:2dym-phi-a}
	\end{align}
Using this, we compute for 
	$a_0,\dots,a_{g+1}\in\Rb_{>0}$ with $a=\sum_{i=0}^{g+1}a_i$ that
	\begin{align}
		\begin{aligned}
			\Delta_{a_{g+1}}^{(b)}
			\circ\prod_{i=1}^{g}\varphi_{a_i}\circ\eta_{a_0}
			&=\Delta_{a_{g+1}}^{(b)}
			\left( \sum_{V\in\hat{G}}
		e^{-(a-a_{g+1})\sigma_V} 
			(\dim(V))^{1-2g}\chi_V \right)\\
			&=\left[ (x_1,\dots,x_b)\mapsto
			\left( \sum_{V\in\hat{G}}e^{-a\sigma_V} (\dim(V))^{1-2g}\chi_V(x_1\dots x_b)\right) \right]\ .
		\end{aligned}
		\label{eq:2dym-before-projection}
	\end{align}
	Finally, according to Part~\ref{lem:ZA-from-rfa:2} of Lemma~\ref{lem:ZA-from-rfa}, 
	we need to compose \eqref{eq:2dym-before-projection} with $\pi^{\otimes b}$ to get \eqref{eq:2dym-on-sigma-gb},
	where $\pi:L^2(G)\to Cl^2(G)$ is the projection onto the image of $D_0$.
To arrive at \eqref{eq:2dym-on-sigma-gb}, we further compute
	\begin{align}
		\begin{aligned}
			\pi^{\otimes b}(\chi_V(x_1\dots x_b))&=
			\pi^{\otimes b}\left( \sum_{k_1,\dots,k_b=1}^{\dim(V)}(\dim(V))^{-b/2}
			f_{k_1 k_2}^{V}(x_1)\dots f_{k_b k_1}^{V}(x_b)\right)\\
			&=(\dim(V))^{1-b}
			\chi_V(x_1)\dots\chi_V(x_b)\ .
		\end{aligned}
		\label{eq:pi-on-chi}
	\end{align}

	For the case $b_{\mathrm{in}}=b_{\mathrm{out}}=0$ we use functoriality. 
	Let $\Sigma'$ the surface obtained by cutting out a disk from $\Sigma$.
	Compose $\funZ_{\mathrm{YM}}^{G}\left( \Sigma' ,a-a' \right)$
	with $\varepsilon_{a'}$ and use \eqref{eq:MV-formulas-epsilon-a}.

	For the case $b_{\mathrm{in}}\neq0$ we need to turn back outgoing boundary components by composing with cylinders with two ingoing boundary components and with area $a$,
	which we denote with $(C,a)$.
	Using Part~\ref{lem:ZA-from-rfa:3} of Lemma~\ref{lem:ZA-from-rfa}, for $U,W\in\hat{G}$ we have
	\begin{align}
		\funZ_{\mathrm{YM}}^{G}(C,a)(\chi_{U},\chi_W)=e^{-a\sigma_U}\delta_{U,V}\ .
		\label{eq:2dym-in-in-cylinder}
	\end{align}
	Using the result for the $b_{\mathrm{in}}=0$ case and \eqref{eq:2dym-in-in-cylinder} we get the claimed expression.
\end{proof}

\begin{remark}
As already noted in Remark~\ref{rem:l2g-a0-limits}, $\eta_a$ and $\eps_a|_{Cl^2(G)}$, i.e.\ the value of $\funZ_{\mathrm{YM}}^{G}$ on a disc with one outgoing (resp.\ one ingoing) boundary component, do not have zero area limits. 
On the other hand,
the expression
\eqref{eq:2dym-on-sigma-gb} has a zero area limit 
if $g+\frac{b_{\mathrm{in}}+b_{\mathrm{out}}}{2} \ge 2$.
Indeed, the $\chi_V$ are orthogonal for different $V$ and have norm $\norm{\chi_V}=1$, 
and for a given $\alpha \in \Zb$ the sum $\sum_{V\in\hat{G}} (\dim(V))^{\alpha}$ converges if 
$\alpha \le -2$.
To see this, use the bijection from \eqref{eq:irrep-highest-weight} and the estimate from \eqref{eq:dimension-estimate} to get
	\begin{align}
		\sum_{V\in\hat{G}} (\dim(V))^{\alpha}\le 
		\sum_{\lambda\in\Lambda^+}(\dim(V_{\lambda}))^{\alpha}\le
		N\sum_{\lambda\in\Lambda^+}|\lambda|^{m\alpha}\ ,
		\label{eq:dim-sum-convergence}
	\end{align}
	which converges for $-m\alpha>\rank(G)$ by \cite[Lem.\,1.3]{Sugiura:1971lg}.
	Then use that $m=(\dim(G)-\rank(G))/2$ and that $3\rank(G)\le\dim(G)$ to get $\alpha<-1$, and since $\alpha$ is an integer $\alpha\le-2$.
These limits are related in \cite{Witten:1991gt} to volumes of moduli spaces of flat connections (see e.g.\ \cite{Komori:2010zf} for more results and references). For example for $G = SU(2)$ we have, for $g \ge 2$ and $b_\mathrm{in}=b_\mathrm{out}=0$,
\begin{align}
	\lim_{a \to 0} \funZ_{\mathrm{YM}}^{SU(2)}\left( 
	\Sigma ,a \right)(1) = \sum_{n=1}^\infty n^{-2g+2} = \zeta(2g-2) \ ,
\end{align}
where $\zeta$ is the Riemann zeta-function. For general $G$, these functions are also referred to as Witten zeta-functions, see e.g.\ \cite{Komori:2010zf}.
\end{remark}

\subsection{Wilson lines and other defects}\label{sec:2dym:wilson}

As we learned in Section~\ref{sec:lattice:bimoduledata}, defect lines in 
the state-sum construction can be obtained from some bimodules over
RFAs. In order to describe Wilson line observables in 2d~YM
theory, we are going to consider bimodules over $L^2(G)$ induced from finite dimensional unitary $G$-modules.

Let $V\in \hat{G}$ and consider the Hilbert space $V\otimes L^2(G)$.
Let us define a map
\begin{align}
	\begin{aligned}
	\xi:V\otimes L^2(G)&\to L^2(G)\otimes V\otimes L^2(G)\\
	v\otimes f\mapsto
	\left[ (x,y)\mapsto f(xy)\, y.v \right]
	\end{aligned}\ .
	\label{eq:left-xi}
\end{align}
One can easily check that $\norm{\xi}=1$. 
We define the left action of $L^2(G)$ on $V\otimes L^2(G)$ via the adjoint of $\xi$:
\begin{align}
	\begin{aligned}
		\rho_{0,0}^L:=\xi^{\dagger}:L^2(G)\otimes V\otimes L^2(G)&\to V\otimes L^2(G)\\
		\varphi\otimes v\otimes f&\mapsto 
		\left[ x\mapsto \int_G \varphi(y)\, y.v\, f(y^{-1}x)\,dy \right]\ ,
	\end{aligned}
\end{align}
and for $a,l\in\Rb_{>0}$
\begin{align}
	\begin{aligned}
		\rho_{a,l}^L&:=
		\rho_{0,0}^L(\eta_a\otimes -)\circ \rho_{0,0}^L\ ,
	\end{aligned}
	\label{eq:Wilson-left-action}
\end{align}
	with trivial $l$-dependence.
	In the rest of this section all length-dependence will be trivial, hence we drop the index $l$ from the notation:
	\begin{align}
		\rho_{a}^{L}:=\rho_{a,l}^{L}\ ,\quad Q_{a}^{L}:=Q_{a,l}^{L}\ ,\quad
		\text{ etc.}
		\label{eq:notation-no-l-index}
	\end{align}
In Proposition~\ref{prop:Wilson-bimodules} we prove that this is indeed an action,
however one can also understand this from a different argument.
If we consider $L^2(G)$ with pointwise multiplication and the same comultiplication $\Delta$, then it is a unital and non-counital Hopf algebra.
This Hopf algebra coacts on $V$ via $v\mapsto\left[ x\mapsto x.v \right]$, 
where we identified $L^2(G) \otimes V$ with $L^2(G,V)$.
Then $L^2(G)$ coacts on $V\otimes L^2(G)$ as in \eqref{eq:left-xi} and taking the adjoint gives the action \eqref{eq:Wilson-left-action}.

We define the right action of $L^2(G)$ on $V\otimes L^2(G)$ to be multiplication on the second factor:
\begin{align}
	\begin{aligned}
	\rho_{b}^R :V\otimes L^2(G)\otimes L^2(G)&\to V\otimes L^2(G)\\
	v\otimes f\otimes \varphi&\mapsto 
	v\otimes \mu_b(f\otimes \varphi)\ .
	\end{aligned}
	\label{eq:Wilson-right-action}
\end{align}
We will often write 
$\rho_{0}^L(\varphi\otimes v\otimes f)=\varphi.(v\otimes f)$, etc.
By acting with $\eta_a$ and $\eta_b$ from the left and right, respectively, we get
\begin{align}
	Q_{a,b} ^{V\otimes L^2(G)}(v\otimes f)(x)=
	\int_{G^2}\eta_a(y)\, y.v\, f(y^{-1}xz^{-1})\eta_b(z)\, dy\, dz\ .
	\label{eq:Wilson-explicit-Q}
\end{align}

Similarly as for $V\otimes L^2(G)$,
we define the left action of $L^2(G)$ on $L^2(G)\otimes V$ to be multiplication on the first tensor factor:
\begin{align}
	\begin{aligned}
		\bar{\rho}_{a}^L : L^2(G)\otimes L^2(G)\otimes V&\to L^2(G)\otimes V\\
	\varphi \otimes f\otimes v&\mapsto 
	\mu_a(\varphi\otimes f)\otimes v\ ,
	\end{aligned}
	\label{eq:Wilson-left-action-dual}
\end{align}
and we define the right action of $L^2(G)$ on $L^2(G)\otimes V$ as follows.
First let
\begin{align}
	\begin{aligned}
		\bar{\rho}_{0}^R :L^2(G)\otimes V\otimes L^2(G)&\to L^2(G)\otimes V\\
		f\otimes v \otimes\varphi&\mapsto 
		\left[ x\mapsto \int_G f(xy^{-1})\, y^{-1}.v\, \varphi(y)\,dy \right]\ ,
	\end{aligned}
\end{align}
and finally
\begin{align}
	\begin{aligned}
		\bar{\rho}_{b}^R&:= \bar{\rho}_{0}^R(- \otimes \eta_b)\circ \bar{\rho}_{0}^R\ .
	\end{aligned}
	\label{eq:Wilson-right-action-dual}
\end{align}

Next we define the duality morphisms
for the pair $(V\otimes L^2(G),V^*\otimes L^2(G))$ of bimodules.
Let $\{e_i^V\}_{i=1}^{\dim(V)}$ denote an orthonormal 
basis of $V$ as in Section~\ref{sec:twoRFAsfromG} and
$\{\vartheta^V_i\}_{i=1}^{\dim(V)}$ the dual basis.
Let
\begin{align}
	\begin{aligned}
	\gamma_{0,b}(1)&:= 
	\sum_{U\in\hat{G}}\sum_{k,l=1}^{\dim(U)}\sum_{j=1}^{\dim(V)}
	e^{-b\sigma_U}f_{kl}^U\otimes\vartheta^V_i \otimes e_i^V\otimes f_{lk}^U\ ,\\
	\gamma_{a,b}(1)&:= (\id_{L^2(G)\otimes V^*}\otimes \rho_{0}^L(\eta_a\otimes - ))\circ \gamma_{0,b}(1)\ ,
	\end{aligned}
	\label{eq:Wilson-copairing}
\end{align}
and
\begin{align}
	\begin{aligned}
	\beta_{0,b} (v\otimes f\otimes \vartheta \otimes g)&:=
	\vartheta(v) \int_{G^2} \eta_b(x)f(y)g(y^{-1}x^{-1})\,dy\,dx\ ,\\
	\beta_{a,b}&:=\beta_{0,b} \circ(\rho_{0}^L(\eta_a\otimes - )\otimes\id_{L^2(G)\otimes V^*})\ .
	\end{aligned}
	\label{eq:Wilson-pairing}
\end{align}

Recall that we identified 
$V\otimes L^2(G)$ with square integrable functions on $G$ with values in $V$, which we denote with $L^2(G,V)$.
We will be particularly interested in a subspace of $L^2(G,V)$ consisting of $G$-invariant functions:
\begin{align}
	L^2(G,V)^G:=\setc*{f\in L^2(G,V)}{g.f(g^{-1}xg)=f(x)\text{ for every }g,x\in G}\ .
	\label{eq:invariant-functions}
\end{align}
Note that $L^2(G,\Cb)^G=Cl^2(G)$.
\begin{proposition} \label{prop:Wilson-bimodules}
	Let $V,W\in\hat{G}$ and let $V^*$ 
	be the dual $G$-module of $V$.
	Then
	\begin{enumerate}
		\item $V\otimes L^2(G)$ is a bimodule over $L^2(G)$ via \eqref{eq:Wilson-left-action} and \eqref{eq:Wilson-right-action},
			$L^2(G)\otimes V$ is a bimodule over $L^2(G)$ via \eqref{eq:Wilson-left-action-dual} and \eqref{eq:Wilson-right-action-dual},
	\label{prop:Wilson-bimodules:1}
		\item $(V\otimes L^2(G))\otimes_{L^2(G)}(W\otimes L^2(G))=(V\otimes W)\otimes L^2(G)$,
	\label{prop:Wilson-bimodules:2}
		\item $\ctimes_{L^2(G)}(V\otimes L^2(G))= L^2(G,V)^G$,
	\label{prop:Wilson-bimodules:3}
		\item $(V\otimes L^2(G),L^2(G)\otimes V^*)$ is a dual pair of bimodules with duality morphisms
			given by \eqref{eq:Wilson-copairing} and \eqref{eq:Wilson-pairing}.
	\label{prop:Wilson-bimodules:5}
	\end{enumerate}
If furthermore $G$ is connected then
	\begin{enumerate}[resume]
		\item the bimodule $V\otimes L^2(G)$ is transmissive if and only if $V$ is the trivial $G$-module $V=\Cb$.
	\label{prop:Wilson-bimodules:4}
	\end{enumerate}
\end{proposition}
\begin{proof}
	\textsl{Part~\ref{prop:Wilson-bimodules:1}:}\\
	We only treat the case of $V\otimes L^2(G)$, the proof for $L^2(G)\otimes V$ is similar.
We start by showing associativity of the left action.
	Let $\varphi_1,\varphi_2\in L^2(G)$ and $v\otimes \psi\in V\otimes L^2(G)$ and recall that we
	abbreviate 
	$\rho_{0}^L(\varphi_1\otimes v\otimes \psi)=\varphi_1.(v\otimes \psi)$. 
	Then
	\begin{align}
		\varphi_2.(\varphi_1.(v\otimes \psi))(x)&=\int_{G^2}\varphi_2(z)\varphi_1(y)\,zy.v\,\psi(y^{-1}z^{-1}x)\,dy\,dz
		\label{eq:left-assoc-1}\ ,\\
		\mu(\varphi_2\otimes \varphi_1).(v\otimes \psi))(x)&=\int_{G^2}\varphi_2(z)\varphi_1(z^{-1}w)\,w.v\,\psi(w^{-1}x)\,dw\,dz
	\label{eq:left-assoc-2}\ .
	\end{align}
	Changing the integration variable $y=z^{-1}w$ in \eqref{eq:left-assoc-1} we get \eqref{eq:left-assoc-2}.
	Using associativity of $\rho_{0}^L$ and the unitality of $\mu$, we get that
	$\lim_{a\to0}Q_{a}^{L}=\id_{V\otimes L^2(G)}$ for $Q_{a}^{L}=\rho_{a_1}\circ\left( \eta_{a_2}\otimes - \right)$. 
	Clearly, the assignment $a\to\rho_{a}^L$ is continuous, and $\rho_{a}^L$
	satisfies the associativity \eqref{eq:ra:module}.
	Therefore $V\otimes L^2(G)$ is a left $L^2(G)$-module.

	It is easy to see that $V\otimes L^2(G)$ is also a right $L^2(G)$-module, so we are left to check two conditions.
	First, that the two actions commute as in \eqref{eq:ra:bimod}, which can be shown similarly as associativity of $\rho_{a}^L$ before.
	Second, that the two sided action is jointly continuous in the 3 parameters, which can be shown by a similar argument as in the proof of Lemma~\ref{lem:semigrp}.

	\medskip

\noindent
	\textsl{Part~\ref{prop:Wilson-bimodules:2}:}\\
	Let $\tilde{V}:=V\otimes L^2(G)$, $\tilde{W}:=W\otimes L^2(G)$, $v\otimes f\in\tilde{V}$ and $w\otimes g\in\tilde{W}$.
	We compute from \eqref{eq:D-definition-bimodule} that
	\begin{align*}
		D_{a}^{\tilde{V},\tilde{W}} (v\otimes f \otimes w\otimes g)(x,y)=
		\int_{G^2} v\otimes t.w\, f(s)\eta_a(s^{-1}xt)g(t^{-1}y)\,ds\,dt\ .
	\end{align*}
	So using that $\lim_{a\to0}P_a=\id$, we get that 
	\begin{align}
		\begin{aligned}
			D_0^{\tilde{V},\tilde{W}}(v\otimes f \otimes w\otimes g)(x,y)=
			\int_G v\otimes t.w\,f(xt)g(t^{-1}y)\,dt\ .
		\end{aligned}
		\label{eq:wilson-D0}
	\end{align}
	By Proposition~\ref{prop:d0proj}, the image of the idempotent $D_0^{\tilde{V},\tilde{W}}$
	is the tensor product $\tilde{V}\otimes_{L^2(G)} \tilde{W}$.
	Let $\pi(v\otimes f \otimes w\otimes g):=v\otimes f.(w\otimes g)$ and
	$\iota(v\otimes w\otimes f)(x,y):=v\otimes x^{-1}.w f(xy)$.
	Then we have that $\pi\circ\iota=\id_{V\otimes W\otimes L^2(G)}$ and
	\begin{align*}
			\iota\circ\pi(v\otimes f \otimes w\otimes g)(x,y)=
			v\otimes \int_G x^{-1}t.w\, f(t)g(t^{-1}xy)\,dt\ ,
	\end{align*}
	which is equal to \eqref{eq:wilson-D0} after substituting $t':=x^{-1}t$.
	We have shown that $\pi$ and $\iota$ is the projection and embedding of 
	the image of $D_0^{\tilde{V},\tilde{W}}$, so in particular the image is
	$\tilde{V}\otimes_{L^2(G)} \tilde{W}=V\otimes W\otimes L^2(G)$.
		The induced action on $V\otimes W\otimes L^2(G)$ from \eqref{eq:tensor-product-action} is
		\begin{align}
			\tilde{\rho}^{\tilde{V},\tilde{W}}_{0,0}=\iota\circ (\rho^{\tilde{V},L}_{0}\otimes \rho^{\tilde{W},R}_{0})\circ D_0^{\tilde{V},\tilde{W}}\circ\pi\ ,
			\label{eq:Wilson-tensor-product-action}
		\end{align}
		which can be shown to agree with the action on $V\otimes W\otimes L^2(G)$ by a straightforward calculation.

	\medskip

\noindent
	\textsl{Part~\ref{prop:Wilson-bimodules:3}:}\\
	Recall $(V\otimes L^2(G))^G$ from \eqref{eq:invariant-functions}.
	Let $a\in\Rb_{>0}$, $v\in V$ and $f\in L^2(G)$. Then from \eqref{eq:D-definition-bimodule} we have
	\begin{align}
		\begin{aligned}
			D_{a}^{\tilde{V}}(v\otimes f)(x)
			=&\int_{G^2}\eta_a(yz^{-1}x)\,y.v\, f(y^{-1}z)\, dy\, dz\\
			\stackrel{w=yz^{-1}}{=}&\int_{G^2}\eta_a(wx)\,y.v\, f(y^{-1}wy)\, dy\, dw\ .
		\end{aligned}
		\label{eq:D-V-tilde}
	\end{align}
	If $v\otimes f\in (V\otimes L^2(G))^G$ then 
	$D_{a}^{\tilde{V}}(v\otimes f)=(\id_V\otimes P_a)(v\otimes f)$ 
	and hence $v\otimes f\in\im(D_{0}^{\tilde{V}})$. Let $h\in G$ and compute
	\begin{align}
		\begin{aligned}
		h.D_{a}^{\tilde{V}} (v\otimes f)(h^{-1}xh)=&\int_{G^2}\eta_a(wh^{-1}xh)\,hy.v\,f(y^{-1}wy)\, dy\, dw\\
		\stackrel{\eta_a\in Cl^2(G)}{=}&\int_{G^2}\eta_a(hwh^{-1}x)\,hy.v\,f(y^{-1}wy)\, dy\, dw\\
		\stackrel{z=hwh^{-1}}{=}&\int_{G^2}\eta_a(zx)\,hy.v\,f(y^{-1}h^{-1}why)\, dy\, dz\\
		\stackrel{q=hy}{=}&\int_{G^2}\eta_a(zx)\,q.v\,f(q^{-1}wq)\, dq\, dz\\
		=&D_{a}^{\tilde{V}} (v\otimes f)(x)\ .
		\end{aligned}
	\end{align}
	Since $h.(-):V\otimes L^2(G)\to V\otimes L^2(G)$ is continuous,
	we can exchange it with $\lim_{a\to0}(-)$, so $D_{0}^{\tilde{V}}(v\otimes f)\in(V\otimes L^2(G))^G$.
	Using the identification $V\otimes L^2(G)\cong L^2(G,V)$ we arrive at 
	$\im(D_{0}^{\tilde{V}})=L^2(G,V)^G$, which is, by Proposition~\ref{prop:d0proj}, $\ctimes_{L^2(G)}(V\otimes L^2(G))$.

	\medskip

\noindent
	\textsl{Part~\ref{prop:Wilson-bimodules:5}:}\\
	It is easy to see from the definition of $\beta_{a,b}$ and $\gamma_{a,b}$ 
	that  the zig-zag identities in \eqref{eq:ra:dual} hold.
	So we only need to show that $\beta_{a,b}$ 
	intertwines the actions as in \eqref{eq:duality-compatibility}.
	We compute
	\begin{align}
		\begin{aligned}
			\beta_{0,b} (\varphi.(v\otimes f)\otimes g\otimes \vartheta)
			=& \int_{G^3} \eta_b(z)\varphi(x)\vartheta(x.v)f(x^{-1}y)g(y^{-1}z^{-1})\, dx\, dy\, dz\ ,\\
			\beta_{0,b} (v\otimes f\otimes (g\otimes \vartheta).\varphi)
			=& \int_{G^3} \eta_b(z)f(y)g(y^{-1}z^{-1}x^{-1})(x^{-1}.\vartheta)(v)\varphi(x)\, dx\, dy\, dz\\
			\stackrel{\text{$G$ acts on $V^*$}}{=}& \int_{G^3} \eta_b(z)\varphi(x)\vartheta(x.v)f(y)g(y^{-1}z^{-1}x^{-1})\, dx\, dy\, dz \\
			\stackrel{\substack{y=x^{-1}u\\z=x^{-1}wx}}{=}& 
			 \eta_b(x^{-1}wx)\varphi(x)\vartheta(x.v)f(x^{-1}u)g(u^{-1}w^{-1})\, dx\, du\, dw \\
			\stackrel{\eta_b\in Cl^2(G)}{=}& \int_{G^3} \eta_b(w)\varphi(x)\vartheta(x.v)f(x^{-1}u)g(u^{-1}w^{-1})\, dx\, du\, dw\ , \\
		\end{aligned}
		\label{eq:Wilson-pairing-compatibility}
	\end{align}
	which are equal. 
	Composing with 
	$Q_{a,0}^{V\otimes L^2(G)}\otimes Q_{a,0}^{V^*\otimes L^2(G)}$ 
	shows that \eqref{eq:duality-compatibility} holds for every 
	$a,b\in\Rb_{>0}$ too.

\medskip

\noindent
\textsl{Part~\ref{prop:Wilson-bimodules:4}:}\\
Since $\rho_{a,b}=Q_{a,b}\circ\rho_{0,0}$, it is enough to consider $Q_{a,b}$.
As we already noted in \eqref{eq:ra:modsemigrp}, 
$Q_{-,-}:(\Rb_{\ge0})^2\to\Bc(\tilde{V})$
is a two parameter strongly continuous semigroup.
This defines two one parameter semigroups 
$Q^1_a:=Q_{a,0}$, $Q^2_b:=Q_{0,b}$ and $Q_{a,b}$ depends solely on $a+b$, 
if and only if these two one parameter semigroups are the same, see also the discussion before Definition~2.4 in \cite{AlSharif:2004sg}. 
One parameter semigroups are completely determined by their generators, so we calculate these now.

Let $v\in V$ and $W\in\hat{G}$. Then 
\begin{align}
\begin{aligned}
Q_{a,b}(v\otimes f_{ij}^W)(x)&=e^{-b\sigma_W}Q_{a,0}(v\otimes f_{ij}^W)(x)\\
&=e^{-b\sigma_W}\int_G\sum_{U\in\hat{G}}e^{-a\sigma_U}\dim(U)\chi_U(s)\,s.v\, f_{ij}^W(s^{-1}x)\,ds\ .
\end{aligned}
\label{eq:wilson-line-not-transmissive}
\end{align}
Using this, and writing $H_i$ for the generator of $Q^i$ for $i=1,2$, we have
\begin{align}
\begin{aligned}
H_1(v\otimes f_{ij}^W)(x)&=\frac{\mathrm{d}}{\mathrm{d}a} Q_a^1(v\otimes f_{ij}^W)|_{a=0}(x)\\&=
\lim_{a\to0}\sum_{U\in\hat{G}}-\sigma_U e^{-a\sigma_U}\dim(U)\int_G\chi_U(s)\,s.v\, f_{ij}^W(s^{-1}x)\,ds
\end{aligned}
\label{eq:generator-Q1}
\end{align}
and $H_2(v\otimes f_{ij}^W)=-\sigma_W\, v\otimes f_{ij}^W$.

Let $v:=e_k^V$ and $W:=\Cb$. Note that $M_{\Cb}$ are constant functions. Then
\begin{align*}
H_1(e_k^V\otimes 1)=(-\sigma_V)\dim(V)\int_G \chi_V(s)\,s.e_k^V\,ds=-\sigma_V e_k^V
\otimes 1
\ ,
\end{align*}
which is nonzero if and only if $V\not\cong\Cb$. Furthermore, $H_2(e_k^V\otimes 1)=0$. 
So if $V\not\cong\Cb$ then $V\otimes L^2(G)$ is not transmissive.

Clearly, if $V\cong\Cb$ then $\Cb\otimes L^2(G)=L^2(G)$ 
and by unitality of the product on $L^2(G)$ the bimodule $\Cb\otimes L^2(G)$ is transmissive.
	\end{proof}

In terms of Section~\ref{sec:fusion-of-defects},
we can interpret these results  as follows.
Let $(\Sb,V,+)$ be a circle with a positively oriented marked point where a Wilson line with label $V\in\hat{G}$ crosses. 
Then the corresponding state space is
\be
\funZ_\mathrm{YM}^G(\Sb,V,+)=L^2(G,V)^G \ .
\ee
Let $V,W\in\hat{G}$. Furthermore, the fusion of two Wilson lines with labels $V$ and $W$ is again a Wilson line with label $V\otimes W$.

\begin{figure}[tb]
	\centering
	\def\svgwidth{4cm}
\begingroup%
  \makeatletter%
  \providecommand\color[2][]{%
    \errmessage{(Inkscape) Color is used for the text in Inkscape, but the package 'color.sty' is not loaded}%
    \renewcommand\color[2][]{}%
  }%
  \providecommand\transparent[1]{%
    \errmessage{(Inkscape) Transparency is used (non-zero) for the text in Inkscape, but the package 'transparent.sty' is not loaded}%
    \renewcommand\transparent[1]{}%
  }%
  \providecommand\rotatebox[2]{#2}%
  \newcommand*\fsize{\dimexpr\f@size pt\relax}%
  \newcommand*\lineheight[1]{\fontsize{\fsize}{#1\fsize}\selectfont}%
  \ifx\svgwidth\undefined%
    \setlength{\unitlength}{261.95178829bp}%
    \ifx\svgscale\undefined%
      \relax%
    \else%
      \setlength{\unitlength}{\unitlength * \real{\svgscale}}%
    \fi%
  \else%
    \setlength{\unitlength}{\svgwidth}%
  \fi%
  \global\let\svgwidth\undefined%
  \global\let\svgscale\undefined%
  \makeatother%
  \begin{picture}(1,0.57304225)%
    \lineheight{1}%
    \setlength\tabcolsep{0pt}%
    \put(0,0){\includegraphics[width=\unitlength,page=1]{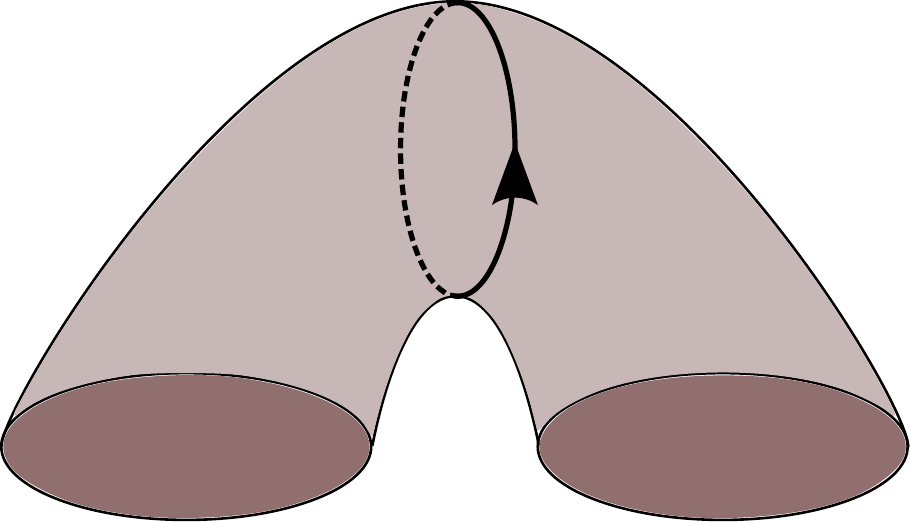}}%
    \put(0.24831201,0.25808516){\color[rgb]{0,0,0}\makebox(0,0)[lt]{\lineheight{0}\smash{\begin{tabular}[t]{l}$a^L$\end{tabular}}}}%
    \put(0.70641153,0.25808516){\color[rgb]{0,0,0}\makebox(0,0)[lt]{\lineheight{0}\smash{\begin{tabular}[t]{l}$a^R$\end{tabular}}}}%
  \end{picture}%
\endgroup%

	\caption{Cylinder $(C,a^L,a^R,V)$
	with ingoing boundaries and a Wilson line with label $V\in\hat{G}$.
	The area of the surface components left and right to the Wilson line is $a^L$ and $a^R$ respectively.
	}
	\label{fig:cut-out-cylinder}
\end{figure}

In the following we show that the value of $\funZ_\mathrm{YM}^G$ on closed surfaces with 
Wilson lines agrees with the expression in \cite[Sec.\,3.5]{Cordes:1995ym}.
Let $(\Sigma,\Ac)=(\Sigma,\Ac,\Lc)$ be a closed surface with area and defects with $\Lc=0$. 
Since $\Sigma$ is closed, the defect lines in $\Sigma$, denoted with $\Sigma_{[1]}$, are closed curves.
In order to compute $\funZ_\mathrm{YM}^G$ on $(\Sigma,\Ac)$ 
we decompose it into convenient smaller pieces as follows.
For every $x\in\Sigma_{[1]}$ with defect label $d_1(x)=V_x\otimes L^2(G)$ for some $V_x\in\hat{G}$
take a collar neighbourhood of $x$ in $\Sigma$, which is a cylinder with $x$ running around it.
Denote the corresponding bordism with area and defects with both boundary components ingoing with $(C_x,a_x^L,a_x^R,V_x)$,
where $a_x^L$ and $a_x^R$ are the area of the surface components to the left and right of $x$ respectively.
Denote with $(\Sigma',\Ac')$ 
the bordism with area with all outgoing components, 
which is formed by removing 
$\bigsqcup_{x\in\Sigma_{[1]}}(C_x,a_x^L,a_x^R,V_x)$ from $(\Sigma,\Ac)$.
We have
\begin{align}
	(\Sigma,\Ac)= \Big(\bigsqcup_{x\in\Sigma_{[1]}}(C_x,a_x^L,a_x^R,V_x) \Big) \circ (\Sigma',\Ac')\ .
	\label{eq:cut-out-Wilson-loops}
\end{align}
Note that $(\Sigma',\Ac')$ is a bordism with area but without defects, 
therefore using Proposition~\ref{prop:2dym-on-sigma-gb} and monoidality we can compute $\funZ_\mathrm{YM}^G$ on it.
The final ingredient we need is:

\begin{lemma} \label{lem:cylinder-Wilson-loop}
	Let 
	$(C,a,b,V)$ 
	be a cylinder with a Wilson line with label $V\in\hat{G}$ 
	as in Figure~\ref{fig:cut-out-cylinder}, $U\in\hat{G}$ and 
	let $U\otimes V\cong\bigoplus_{W\in\hat{G}}W^{N_{U,V}^W}$ be the decomposition into simple $G$-modules, for some integers $N_{U,V}^W$.
	Then
	\begin{align}
		\funZ_\mathrm{YM}^G(C,a,b,V)(\chi_U\otimes\chi_W)=
		e^{-a\sigma_U-b\sigma_W}
		N_{U,V}^W\ .
		\label{eq:cylinder-Wilson-loop}
	\end{align}
\end{lemma}
\begin{proof}[Sketch of proof]
	The morphism $\funZ_\mathrm{YM}^G(C,a,b,V)$ is given by the diagram
	\begin{align}
		\begin{aligned}
			\def\svgwidth{4cm}
\begingroup%
  \makeatletter%
  \providecommand\color[2][]{%
    \errmessage{(Inkscape) Color is used for the text in Inkscape, but the package 'color.sty' is not loaded}%
    \renewcommand\color[2][]{}%
  }%
  \providecommand\transparent[1]{%
    \errmessage{(Inkscape) Transparency is used (non-zero) for the text in Inkscape, but the package 'transparent.sty' is not loaded}%
    \renewcommand\transparent[1]{}%
  }%
  \providecommand\rotatebox[2]{#2}%
  \newcommand*\fsize{\dimexpr\f@size pt\relax}%
  \newcommand*\lineheight[1]{\fontsize{\fsize}{#1\fsize}\selectfont}%
  \ifx\svgwidth\undefined%
    \setlength{\unitlength}{124.19913483bp}%
    \ifx\svgscale\undefined%
      \relax%
    \else%
      \setlength{\unitlength}{\unitlength * \real{\svgscale}}%
    \fi%
  \else%
    \setlength{\unitlength}{\svgwidth}%
  \fi%
  \global\let\svgwidth\undefined%
  \global\let\svgscale\undefined%
  \makeatother%
  \begin{picture}(1,0.97147857)%
    \lineheight{1}%
    \setlength\tabcolsep{0pt}%
    \put(0,0){\includegraphics[width=\unitlength,page=1]{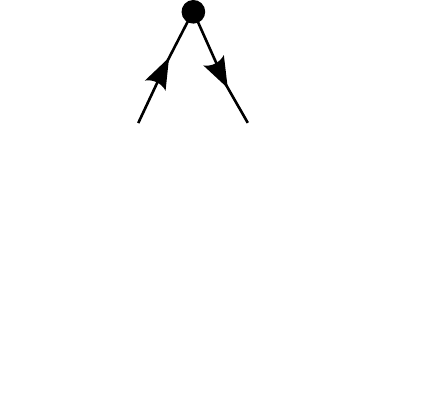}}%
    \put(0.52622862,0.92695085){\color[rgb]{0,0,0}\makebox(0,0)[lt]{\lineheight{0}\smash{\begin{tabular}[t]{l}\scriptsize{$(0,b_2;\tilde{V})$}\end{tabular}}}}%
    \put(0,0){\includegraphics[width=\unitlength,page=2]{diag-cylinder-Wilson-loop.pdf}}%
    \put(0.30460459,0.00569665){\color[rgb]{0,0,0}\makebox(0,0)[lt]{\lineheight{0}\smash{\begin{tabular}[t]{l}\scriptsize{$(0,b_1;\tilde{V})$}\end{tabular}}}}%
    \put(0,0){\includegraphics[width=\unitlength,page=3]{diag-cylinder-Wilson-loop.pdf}}%
    \put(0.64514444,0.41776828){\color[rgb]{0,0,0}\makebox(0,0)[lt]{\lineheight{0}\smash{\begin{tabular}[t]{l}\scriptsize{$(0,0;\tilde{V})$}\end{tabular}}}}%
    \put(0.00144066,0.6844723){\color[rgb]{0,0,0}\makebox(0,0)[lt]{\lineheight{0}\smash{\begin{tabular}[t]{l}\scriptsize{$(a,0;\tilde{V})$}\end{tabular}}}}%
    \put(-0.00401007,0.09645981){\color[rgb]{0,0,0}\makebox(0,0)[lt]{\lineheight{0}\smash{\begin{tabular}[t]{l}$Cl^2(G)$\end{tabular}}}}%
    \put(0.62315062,0.09841785){\color[rgb]{0,0,0}\makebox(0,0)[lt]{\lineheight{0}\smash{\begin{tabular}[t]{l}$Cl^2(G)$\end{tabular}}}}%
  \end{picture}%
\endgroup%

		\end{aligned}\ .
		\label{eq:diag-cylinder-Wilson-loop}
	\end{align}
	After a straightforward calculation and some manipulation of multiple integrals we get
	for $\varphi,\psi\in Cl^2(G)$ that
	\begin{align}
		\funZ_\mathrm{YM}^G(C,a,b,V)(\varphi\otimes\psi)=
		\int_{G^3}\eta_a(z)\varphi(z^{-1}y)\chi_V(y)\psi(y^{-1}p^{-1})\eta_b(p)\,dp\,dy\,dz\ .
	\end{align}
	Finally using
	\begin{align}
		\int_G \chi_U(y)\chi_V(y)\chi_W(y^{-1})=N_{U,V}^W\ ,
		\label{eq:character-fusion-rule}
	\end{align}
	which follows from basic properties of characters and character orthogonality,
	we get \eqref{eq:cylinder-Wilson-loop}.
\end{proof}

\begin{remark}
The computation of the defect cylinder in the above lemma allows one to interpret states of the 2d~YM
theory in terms of Wilson lines. Namely, let $(D,a,b,V)$ be a disc with outgoing 
boundary and with embedded defect circle oriented anti-clockwise and labeled by $V \in \hat G$.
The area inside the circle is $a$ and the one outside is $b$. The corresponding amplitude is
\begin{align}
	\begin{aligned}
	\big\langle \chi_W , \funZ_\mathrm{YM}^G(D,a,b,V) \rangle
	&= \funZ_\mathrm{YM}^G(C,\tfrac a2,b,V)(\eta_{\frac a2}\otimes\chi_W) 
	\\
	&= \sum_{U \in \hat G} e^{-a\sigma_U-b\sigma_W} \dim(U) N_{U,V}^W\ .
	\end{aligned}
\end{align}
For a given $W$, the sum is finite. One checks from this expression that
\be
\lim_{\substack{a\to \infty\\ b\to 0}} \funZ_\mathrm{YM}^G(D,a,b,V) = \chi_V \ .
\ee
Thus we can picture the state $\chi_V \in \funZ_\mathrm{YM}^G(\Sb)$ informally as
the disc $(D,\infty,0,V)$ with zero area outside of the circle and infinite area inside the circle (and which is hence not an allowed bordism with area).
From this point of view the action \eqref{eq:cylinder-Wilson-loop} of the cylinder is no surprise as by Theorem~\ref{thm:fusion} it amounts to the fusion of defect lines, which by Proposition~\ref{prop:Wilson-bimodules}\,(2) is given by the tensor product of $G$-representations.
\end{remark}

\begin{proposition} \label{prop:Wilson-line-expectation-value}
	For $x\in\Sigma_{[1]}$ let $\rho_R^x\in\pi_0(\Sigma')$ be the connected component which is glued to $C_x$ on the right side of $x$ in \eqref{eq:cut-out-Wilson-loops}
	and define $\rho_L^x\in\pi_0(\Sigma')$ similarly to be the connected component glued from the left.
	Using the notation from above we have
	\begin{align}
		\begin{aligned}
			\funZ_\mathrm{YM}^G(\Sigma,\Ac)=
			\prod_{\rho\in\pi_0(\Sigma')}\prod_{x\in\Sigma_{[1]}}\sum_{U_\rho\in\hat{G}}e^{-a_{\rho}\sigma_{U_\rho}}
			(\dim(U_{\rho})^{\chi(\rho)}N_{U_{\rho_L^x},V_x}^{U_{\rho_R^x}}\ ,
		\end{aligned}
		\label{eq:Wilson-line-expectation-value}
	\end{align}
	where $a_{\rho}\in\Rb_{>0}$ is the area of $\rho$.
\end{proposition}

The expression in \eqref{eq:Wilson-line-expectation-value} matches the expression in \cite[(3.28)]{Cordes:1995ym}
(see also \cite[Sec.\,5]{Rusakov:1990wilson}).

\subsubsection*{Defects from automorphisms of $G$}

Another way of obtaining bimodules is by twisting the actions on the trivial bimodule by an algebra automorphism
as we saw in Example~\ref{ex:twistedaction}. In the rest of this section we will introduce automorphisms of $L^2(G)$ (seen as an RFA) using automorphisms of $G$.

	Let $\alpha\in\Aut(G)$, $V\in\hat{G}$ and denote with ${}_\alpha V$ 
the $G$-module obtained by
	precomposing the action on $V$ with $\alpha$.
	Let $H$ denote the Haar measure on $G$ and $\alpha^*H$ the induced measure. This is a left invariant normalised measure, hence by the uniqueness of such measures
	$\alpha^*H=H$. As a consequence, the Haar integral is invariant under $\Aut(G)$.

\begin{lemma}
	Let $\alpha\in\Aut(G)$. 
	Then precomposition with $\alpha$
	is an automorphism of the RFA $L^2(G)$ and
	defines a group homomorphism
	$\Aut(G)\to\Aut_{\RFrob{\Hilb}}(L^2(G))^\mathrm{op}$.
	\label{lem:L2G-automorphism}
\end{lemma}

\begin{proof}
	Clearly, invariance of the Haar measure implies that $\alpha^*=(-)\circ\alpha$ is unitary.
	We first show that $\alpha^*$ 
	commutes with the product.
	Let $\varphi_1,\varphi_2\in L^2(G)$ and compute:
	\begin{align}
		\begin{aligned}
			\mu( (\varphi_1\circ\alpha)\otimes(\varphi_2\circ\alpha))(x)=&
			\int_{G}\varphi_1(\alpha(z))\varphi_2(\alpha(z^{-1}x))\,dz\\
			\stackrel{y=\alpha(z)}{=}&\int_{G}\varphi_1(y)\varphi_2(y^{-1}\alpha(x))\,dy 
			=\mu(\varphi_1\otimes\varphi_2)(\alpha(x))\ ,
		\end{aligned}
		\label{eq:alpha-mu}
	\end{align}
	where we used that the Haar measure on $G$ is invariant under $\alpha$.
	Next, we show that $\eta_a\circ\alpha=\eta_a$.
	\begin{align}
		\begin{aligned}
			\eta_a\circ\alpha
			=& \sum_{V\in\hat{G}}e^{-a\sigma_V}\dim(V)\chi_V\circ\alpha
			= \sum_{V\in\hat{G}}e^{-a\sigma_V}\dim(V)\chi_{\alpha V}\\
			=& \sum_{V\in\hat{G}}e^{-a\sigma_{ {}_{\alpha} V}}\dim({}_{\alpha} V)\chi_{\alpha V}
			=\eta_a\ ,
		\end{aligned}
		\label{eq:alpha-eta}
	\end{align}
	where we used that $\alpha V$ has the same dimension as $V$ and that the Casimir element is invariant under $\alpha$.
	The latter can be understood as follows. The Lie group automorphism $\alpha$ induces an automorphism on the Lie algebra of $G$,
	and the Casimir element is defined in terms of an orthonormal basis of the Lie algebra with respect to an invariant non-degenerate pairing,
	for example the Killing form.

	Since $L^2(G)$ is a $\dagger$-RFA and $\alpha^*$ is a unitary
	regularised algebra morphism, $\alpha^*$ is an RFA morphism.
\end{proof}

Let $L_{\alpha}:={}_{\alpha^*}L^2(G)_{\id}$ 
denote the transmissive twisted bimodule from Example~\ref{ex:twistedaction}.
By Examples~\ref{ex:twisted-bimodule-duals}~and~\ref{ex:twisted-bimodule-tensor-products} these bimodules have duals and can be tensored together,
i.e.\ we can label defect lines with them.
For convenience we list these results here. 
Then
\begin{itemize}
	\item $(L_{\alpha},L_{\alpha^{-1}})$ is a dual pair of bimodules,
	\item $L_{\alpha_1}\otimes_{L^2(G)}L_{\alpha_2}\cong L_{\alpha_2\circ\alpha_1}$,
		for $\alpha_1,\alpha_2\in\Aut(G)$,
	\item $\ctimes_{L^2(G)}L_{\alpha}\cong \setc*{f\in L^2(G)}{f(gx\alpha(g^{-1}))=f(x)\text{ for every $g,x\in G$}}$,
\end{itemize}
where the last equation can be be computed from $D_0^{L_{\alpha}}$ of \eqref{eq:D-definition-bimodule}.

The following lemma can be proven similarly as Lemma~\ref{lem:cylinder-Wilson-loop}.
\begin{lemma} \label{lem:twisted-bimodule-cut-out-cylinder}
Let $\alpha\in\Aut(G)$ and $(C,a,b,L_{\alpha})$ 
denote a cylinder as in Figure~\ref{fig:cut-out-cylinder} with the defect line labeled with $L_{\alpha}$.
	Then for $U,W\in\hat{G}$ we have
	\begin{align}
		\funZ_\mathrm{YM}^G(C,a,b,L_{\alpha})
		(\chi_U\otimes\chi_W)=e^{-(a+b)\sigma_W}\delta_{ {}_{\alpha} U,W}\ .
		\label{eq:twisted-bimodule-cut-out-cylinder}
	\end{align}
\end{lemma}

The following lemma shows that for some particular choices of $\alpha$, these bimodules could provide new examples.

\begin{lemma}
Let $\alpha\in\Aut(G)$ 
	and $V\in\hat{G}$. Then
	\begin{enumerate}
		\item $L_{\alpha}\cong L_{\id}=L^2(G)$ as bimodules if and only if $\alpha$ is inner,
			\label{lem:bimodule-isomorphisms:1}
	\end{enumerate}
	furthermore if $G$ is connected,
	\begin{enumerate}[resume]
		\item $L_{\alpha}\cong V\otimes L^2(G)$ as bimodules if and only if $\alpha$ is 
		inner and $V\cong \Cb$ as $G$-modules.
			\label{lem:bimodule-isomorphisms:2}
	\end{enumerate}
	\label{lem:bimodule-isomorphisms}
\end{lemma}
\begin{proof}
	Part~\ref{lem:bimodule-isomorphisms:1}:
	Let us assume that $\alpha(x)=g^{-1}xg$ for for some $g\in G$. We define
	$\varphi:L_{\alpha}\to L_{\id}$ as $\varphi(f)(x):=f(gx)$, which is clearly bounded and invertible.
	To show that it is an intertwiner calculate for $\psi\in L^2(G)$ and $f\in L_{\alpha}$:
	\begin{align}
		\begin{aligned}
			\varphi(\psi.f)(x)=\int_G\psi(g^{-1}yg)f(y^{-1}gx)=
			\substack{z=g^{-1}yg}{=}\int_G \psi(z)f(gz^{-1}x)=\psi.\varphi(f)(x)\ .
		\end{aligned}
		\label{eq:phi-intertwiner}
	\end{align}
	
	Conversely, let us assume that $L_{\alpha}\cong L_{\id}$. Let $(\Sb\times[0,1],a,b,L_{\alpha})$ 
	be a cylinder as in Lemma~\ref{lem:twisted-bimodule-cut-out-cylinder}, just with one of the boundary components being outgoing. Then we have that 
	\begin{align}
		\funZ_\mathrm{YM}^G(\Sb\times[0,1],a,b,L_{\alpha})
		(\chi_V)=e^{-(a+b)\sigma_V}\chi_{ {}_{\alpha}V}\ .
		\label{eq:twisted-defect-cylinder}
	\end{align}
	But since $L_{\alpha}\cong L_{\id}$, by a direct computation one can see that the operator in \eqref{eq:twisted-defect-cylinder} is the same as
	the operator assigned to a cylinder without defect lines and with area $a+b$, so we have for every $V\in\hat{G}$ that
	\begin{align}
		\chi_{ {}_{\alpha}V}=\chi_V\ ,
		\label{eq:equal-characters}
	\end{align}
	which is equivalent to ${}_{\alpha}V\cong V$ for every $V\in\hat{G}$.
	This means that the highest weight of ${}_{\alpha}V$ and $V$ are equal for every $V\in\hat{G}$, which 
	holds if and only if $\alpha$ corresponds to the trivial automorphism of the Dynkin diagram of $G$. This is equivalent to $\alpha$ being an inner automorphism~\cite[Ch.\,VII]{Knapp:2002lg}.

	Part~\ref{lem:bimodule-isomorphisms:2} follows directly from the fact that $L_{\alpha}$ is transmissive, 
	Part~\ref{lem:bimodule-isomorphisms:1} of this lemma and 
	Part~\ref{prop:Wilson-bimodules:4} 
	of Proposition~\ref{prop:Wilson-bimodules}.
\end{proof}

Using  Lemma~\ref{lem:twisted-bimodule-cut-out-cylinder}
and Part~\ref{lem:bimodule-isomorphisms:1} of Lemma~\ref{lem:bimodule-isomorphisms}
we can show the following proposition.

\begin{proposition}
	\label{prop:twisted-defect-expectation-values}
	Let $(\Sigma,\Ac)$ and $(\Sigma',\Ac')$ 
	be as in Proposition~\ref{prop:Wilson-line-expectation-value} with every defect line $x\in\Sigma_{[1]}$
	labeled by $L_{\alpha_x}$ for $\alpha_x\in\Aut(G)$. 
	Then
	\begin{align}
	\begin{aligned}
		\funZ_\mathrm{YM}^G(\Sigma,\Ac)=
		\prod_{\rho\in\pi_0(\Sigma')}\prod_{x\in\Sigma_{[1]}}\sum_{U_\rho\in\hat{G}}e^{-a_{\rho}\sigma_{U_\rho}}(\dim(U_{\rho})^{\chi(\rho)}\delta_{ {}_{\alpha_x} U_{\rho_L^x},U_{\rho_R^x}}\ ,
	\end{aligned}
	\label{eq:twisted-defect-expectation-value}
	\end{align}
	where $a_{\rho}\in\Rb_{>0}$ is the area of $\rho$.
	In particular, if $\alpha_x$ is inner for every $x\in\Sigma_{[1]}$ then \eqref{eq:twisted-defect-expectation-value}
	agrees with \eqref{eq:2dym-on-sigma-gb}, the value of $\funZ_\mathrm{YM}^G$ on $(\Sigma,\Ac)$ without defects.
\end{proposition}

The following is an example of a non-trivial twist-defect.

\begin{example}
	Let us assume that $G$ is furthermore simply connected. Then $\Out(G)$, the group of outer automorphisms of $G$, is isomorphic to
	the group of automorphisms of the Dynkin diagram of $G$ \cite[Ch.\,VII]{Knapp:2002lg}.

	Let $G:=SU(N)$ for $N\ge3$. Then $\Out(G)\cong\Zb_2$ and its generator, which we now denote with $\alpha$,
	corresponds to complex conjugation. We have that ${}_{\alpha} V \cong V^*$ 
	for every $V\in \hat{G}$.
	We can apply Proposition~\ref{prop:twisted-defect-expectation-values},
	so for example for a torus $T^2$ 
	with one non-contractible defect line with defect label $L_{\alpha}$ we have
	\begin{align}
	\begin{aligned}
		\funZ_\mathrm{YM}^{SU(N)}
		(T^2,a)=
		\sum_{{U\in\hat{G}\text{, }U\cong U^*}}e^{-a\sigma_U}
		\ .
	\end{aligned}
	\label{eq:torus-conjugate}
	\end{align}

\end{example}

\appendix

\section{Appendix: A bimodule with singular limits}\label{app:bimod}

	In this example we illustrate that not every bimodule over regularised algebras
	comes from a left- and right module with commuting actions. Namely, we construct two regularised algebras $A^L$ and $A^R$
	and an $A^L$-$A^R$-bimodule $M$, such that the two sided
	action $\rho_{a,l,b}$ does not provide a left module structure as in Remark~\ref{rem:left-right-module-gives-bimodule}
	since  the limit in \eqref{eq:bimod-left-module-limit} does not exist.
	
	Let $A$ be $\Cb[x]/\langle x^2\rangle$ as an algebra in $\Hilb$ with orthonormal basis $\left\{ 1,x \right\}$. 
	Let $n\in\Zb_{\ge1}$ and $M_n\in\Hilb$ be spanned by orthonormal vectors $v_0$ and $v_1$.
	We define a left $A$-module structure on $M_n$ by 
	\begin{align}
	x.v_0=e^{n^2}v_1\quad\text{ and }\quad x.v_1=0\ .
	\label{eq:Mn-module}
	\end{align}
	Since $A$ is commutative, \eqref{eq:Mn-module} defines a right $A$-module structure on $M_n$ as well and together we have an $A$-$A$-bimodule structure.
	
	Next we turn $A$ into a regularised algebra in two ways.
	Let $h^L:=x-n\in A$ and denote with $A^L_n$ the regularised algebra structure on $A$ 
	defined as in Example~\ref{ex:cxd} by setting 
	\begin{align}
	P^{A^L_n}_a(p):=e^{ah^L}p
	\label{eq:Mn-Pa-AL}
	\end{align}
	for $p\in A^L$ and $a\in\Rb_{>0}$. Let $\mu_a^{A_n^L}$ denote the product and $\eta_a^{A_n^L}=e^{-an}(1+ax)$ the unit.
	Similarly, define the regularised algebra $A^R_n$ using $h^R:=x-n^3$.

We turn the $A$-$A$-bimodule $M_n$ from above into an $A^L_n$-$A^R_n$-bimodule over regularised algebras via Proposition~\ref{prop:ramodule:findim} and by taking the $l$-dependence to be trivial. We denote the
resulting action by $\rho_{a,b}^{M_n}$. The semigroup action is given by
\begin{align}
Q_{a,b}^{M_n}(m):=e^{ah^L+bh^R}.m=
e^{-an-bn^3}(1+(a+b)x).m \ ,
\label{eq:Mn-reg-bimodule}
\end{align}
for $m\in M_n$.
	
	Let us consider 
	\begin{align}
	A^L:=\bigoplus_{n\in\Zb_{\ge1}}A^L_n\ ,\quad
	A^R:=\bigoplus_{n\in\Zb_{\ge1}}A^R_n \quad\text{ and }\quad
	M:=\bigoplus_{n\in\Zb_{\ge1}}M_n\ .
	\label{eq:Mn-direct-sums}
	\end{align}
	We claim that for every $a,b\in\Rb_{>0}$ we have 
	\begin{align}
	\sum_{n\in\Zb_{\ge1}}\norm{\eta_a^{A_n^L}}^2<\infty
	\quad\text{ and }\quad
	\sum_{n\in\Zb_{\ge1}}\norm{\eta_b^{A_n^R}}^2<\infty
	\label{eq:Mn-sums-finite}
	\end{align}
	and furthermore
	\begin{align}
	\sup_{n\in\Zb_{\ge1}}\left\{\norm{\mu_a^{A^L_n}}\right\}<\infty\ ,\quad
	\sup_{n\in\Zb_{\ge1}}\left\{\norm{\mu_b^{A^R_n}}\right\}<\infty
	\quad\text{ and }\quad
	\sup_{n\in\Zb_{\ge1}}\left\{ \norm{\rho_{a,b}^{M_n}}^2 \right\}<\infty\ . 
	\label{eq:Mn-supremum-finite}
	\end{align}
	So by Proposition~\ref{prop:directsumrfa}, $A^L$ and $A^R$ are regularised algebras and
	by Proposition~\ref{prop:directsum-bimodule} $M$ is a $A^L$-$A^R$-bimodule.
	However the limit
	\begin{align}
	\lim_{b\to0}\rho_{a,b_1}^{M}\circ\left( \id_{A^L\otimes M}\otimes \eta_{b_2}^{A^R} \right)\ ,
	\label{eq:Mn-left-action-limit}
	\end{align}
	where $b=b_1+b_2$, does not exist, i.e.\ $M$ is not a left $A^L$-module.

	Showing \eqref{eq:Mn-sums-finite} is a direct calculation and we omit it.
	We now show that \eqref{eq:Mn-supremum-finite} holds. 
	We compute for $p=p_0+p_1 x\in A_n^L$, 
	$m=m_0 v_0+m_1 v_1\in M_n$ and $q=q_0+q_1 x\in A^R_n$ that
	\begin{align}
	\begin{aligned}
	&\rho_{a,b}^{M_n}(p\otimes m\otimes q)=
	Q_{a,b}^{M_n}(pq.m)\\&=
	Q_{a,b}^{M_n}\left( p_0 q_0m_0 v_0+
	\left( (p_0 q_1 + p_1 q_0)m_0 e^{n^2}
	+p_0 q_0 m_1\right)v_1\right)\\&=
	e^{-an-bn^3}\left( p_0 q_0 m_0 v_0+
	\left(p_0 q_0 m_1+[p_0 q_0(a+b)+(p_0 q_1 + p_1 q_0)]m_0 e^{n^2}\right)v_1\right)\ .
	\end{aligned}
	\label{eq:Mn-action-calc}
	\end{align}
	Using this we compute the value of the adjoint of the action on 
	$f=f_0 v_0+f_1 v_1$ as
	\begin{align}
	\begin{aligned}
	\left( \rho_{a,b}^{M_n} \right)^{\dagger}(f)&=
	e^{-an-bn^3}
	\left(f_0+(a+b)e^{n^2} f_1\right) 1\otimes v_0\otimes 1\\
	&+
	e^{-an-bn^3}
	f_1 \left( 1\otimes v_1\otimes 1 + 
	e^{n^2}( 1\otimes v_0\otimes x+x\otimes v_0\otimes 1)\right)\ .
	\end{aligned}
	\label{eq:Mn-action-adjoint}
	\end{align}
	Let $f$ have $\norm{f}=1$ and compute the norm of \eqref{eq:Mn-action-adjoint}
	\begin{align}
	\begin{aligned}
	\norm{\left( \rho_{a,b}^{M_n} \right)^{\dagger}(f)}^2=
	e^{-2an-2bn^3}\left(
	\left| f_0+(a+b)e^{n^2} f_1 \right|^2+
	\left| f_1\right|^2\left( 1+2e^{2n^2} \right) \right)\ ,
	\end{aligned}
	\label{eq:Mn-action-adjoint-norm}
	\end{align}
	from which we get by estimating $|f_0|\le1$ and $|f_1|\le1$ that
	\begin{align}
	\norm{\rho_{a,b}^{M_n}}^2=
	\norm{\left( \rho_{a,b}^{M_n} \right)^{\dagger}}\le
	e^{-2an-2bn^3} \left(1+2(a+b)e^{n^2}+\left(  2+ (a+b)^2\right)e^{2n^2} \right)\ .
	\label{eq:Mn-action-norm-estimate}
	\end{align}
	By a similar argument, without giving the details, we obtain the following estimates:
	\begin{align}
	\norm{\mu_a^{A^L_n}}^2\le e^{-2an}\left( 2+a+a^2 \right)
	\quad\text{ and }\quad
	\norm{\mu_b^{A^R_n}}^2\le e^{-2bn^3}\left( 2+b+b^2 \right)\ .
	\label{eq:Mn-mult-norm-estimates}
	\end{align}
	From \eqref{eq:Mn-action-norm-estimate} and \eqref{eq:Mn-mult-norm-estimates} it follows that \eqref{eq:Mn-supremum-finite} holds.
	
	Finally we give a lower estimate of the norm of the morphism in \eqref{eq:Mn-left-action-limit} restricted to $A^L_n\otimes M_n$
	without the $b\to0$ limit:
	\begin{align}
	\begin{aligned}
	\norm{\rho_{a,b_1}^{M}\circ\left( \id_{A^L_n\otimes M_n} \otimes \eta_{b_2}^{A^R} \right)}^2
	&=\norm{\rho_{a,b}^{M_n}\circ\left( \id_{A^L_n\otimes M_n} \otimes 1\right)}^2\\
	&\ge e^{-2an-2bn^3}\frac{1}{2}\left( \left( 1+(a+b)e^{n^2} \right)^2+1+e^{2n^2} \right)\ .
	\end{aligned}
	\label{eq:Mn-left-action-norm-lower-estimate}
	\end{align}
	We arrived to this estimate by computing the norm of the adjoint of \eqref{eq:Mn-left-action-limit} before taking the limit $b\to0$ evaluated at $f\in M_n$, as we did in \eqref{eq:Mn-action-adjoint-norm}, and then
	by choosing $f_0=f_1=\frac{1}{\sqrt{2}}$.
Thus the $b\to0$ limit in \eqref{eq:Mn-left-action-limit} cannot give a bounded operator. 

\small

\phantomsection

\end{document}